\documentclass[notitlepage,a4paper,11pt]{report}
\usepackage[english]{babel}
\usepackage{amssymb,amsmath,latexsym,eufrak}
\usepackage[all]{xy}
\usepackage{scalefnt}
\usepackage{hyperref}
\usepackage{amsthm,amscd}
\usepackage{a4wide}

\usepackage[capitalise]{cleveref}

\usepackage{chngcntr}

\usepackage[T1]{fontenc}
\usepackage{microtype}

\DeclareFontFamily{OT1}{pzc}{}
\DeclareFontShape{OT1}{pzc}{m}{it}{<-> s * [1.200] pzcmi7t}{}
\DeclareMathAlphabet{\mathpzc}{OT1}{pzc}{m}{it}

\def\mathcat{\mathpzc}

\makeatletter
\renewenvironment{abstract}{
      \@beginparpenalty\@lowpenalty
      \begin{center}
        \bfseries \abstractname
        \@endparpenalty\@M
      \end{center}}
     {\par 
     }
\makeatother

\newtheorem{thm}{Theorem}[section]
\newtheorem{cor}[thm]{Corollary}
\newtheorem{lem}[thm]{Lemma}
\newtheorem{prop}[thm]{Proposition}
\newtheorem{comp}[thm]{Complement}
\theoremstyle{definition}

\newtheorem{examp}[thm]{Example}
\theoremstyle{remark}
\newtheorem{rema}[thm]{Remark}

\newcommand{\Aut}{\mathop{\mbox{\rm Aut}}\nolimits}
\newcommand{\Out}{\mathop{\mbox{\rm Out}}\nolimits}
\newcommand{\Gal}{\mathop{\mbox{\rm Gal}}\nolimits}
\newcommand{\End}{\mathop{\mbox{\rm End}}\nolimits}

\newcommand{\Der}{\mathop{\mbox{\rm Der}}\nolimits}
\newcommand{\Pic}{\mathop{\mbox{\rm Pic}}\nolimits}
\newcommand{\Hom}{\mathop{\mbox{\rm Hom}}\nolimits}

\newcommand{\res}{\mathop{\mbox{\rm res}}\nolimits}
\newcommand{\im}{\mathop{\mbox{\rm im}}\nolimits}

\def\ppsi{\psi}
\def\pphi{\phi}
\def\Pphi{\Phi}
\def\rrr{\sigma_{\substack{\mbox{\tiny{$G$}}}}}
\def\rrn{\sigma_{\substack{\mbox{\tiny{$N$}}}}}
\def\llambda{\kappa_G}
\def\kappaQ{\kappa_Q}
\def\kappaQI{\kappa_{Q,I}}
\def\sigmaQ{\sigma_{\substack{\mbox{\tiny{$Q$}}}}}
\def\sigmaG{\sigma_{\substack{\mbox{\tiny{$G$}}}}}
\def\fiel{\mathfrak k}
\def\bbeta{\gamma}
\def\ttheta{\pi_G}
\def\piN{\pi_N}
\def\piQ{\pi_Q}
\def\ggamma{x}

\def\omu{\mu}
\def\ome{\omega}
\def\Mx{N}
\def\MsA{M_{\substack{\mathstrut{\sigma_{\substack{\mbox{\tiny{$A$}}}}}}}}
\def\BsA{C_{\substack{\mathstrut{\sigma_{\substack{\mbox{\tiny{$A$}}}}}}}}
\def\BbB{M_{\mathrm e_{\substack{\mbox{\tiny{$Q$}}}}}}
\def\BBB{C}
\def\sta{\mathrm e_{\substack{\mbox{\tiny{$Q$}}}} ,\vartheta}
\def\XBSQ{\mathrm{XB}(S|S;\{e\},Q)}
\def\compo{\odot}
\def\mumu{\mu}
\def\dDelta{\omega_{\substack{\mbox{\tiny{$\mathcat {Pic}_{S,Q}$}}}}}
\def\jp{j_{\substack{\mbox{\tiny{$\mathcat {Pic}_{S,Q}$}}}}}
\def\mup{\mumu_{\substack{\mbox{\tiny{$\mathcat {Pic}_{S,Q}$}}}}}
\def\piEG{\pi_{\mathrm e_{\substack{\mbox{\tiny{$G$}}}}}}
\def\piEN{\pi_{\mathrm e_{\substack{\mbox{\tiny{$N$}}}}}}
\def\piEQ{\pi_{\mathrm e_{\substack{\mbox{\tiny{$Q$}}}}}}
\def\iEG{i_{\mathrm e_{\substack{\mbox{\tiny{$G$}}}}}}
\def\iEQ{i_{\mathrm e_{\substack{\mbox{\tiny{$Q$}}}}}}
\def\jEQ{j_{\mathrm e_{\substack{\mbox{\tiny{$Q$}}}}}}
\def\thetaEG{\vartheta_{\mathrm e_{\substack{\mbox{\tiny{$G$}}}}}}
\def\Ggamma{\Gamma}
\def\Egamma{K}
\def\EG{\Gamma_G}
\def\EN{\Gamma_N}
\def\EQ{\Gamma_Q}
\def\Ho{\mathrm H}
\def\Bsp{X}
\def\prin{\xi}
\def\Map{\mathrm{Map}}
\def\MZ{\mathbb Z}

\numberwithin{equation}{section}
\counterwithout{section}{chapter}
\title 
{Normality of algebras over commutative rings, crossed pairs,
and the Teichm\"uller class}
\author{Johannes Huebschmann 
\\ 
USTL, UFR de Math\'ematiques\\
CNRS-UMR 8524
\\
Labex CEMPI (ANR-11-LABX-0007-01)
\\
\newline
59655 Villeneuve d'Ascq Cedex, France\\
Johannes.Huebschmann@math.univ-lille1.fr
 }

\long
\def\MSC#1\EndMSC{\def\arg{#1}\ifx\arg\empty\relax\else
      {\par\narrower\noindent
      2010 Mathematics Subject Classification: #1\par}\fi}
\long
\def\KEY#1\EndKEY{\def\arg{#1}\ifx\arg\empty\relax\else
    {\par\narrower\noindent
      Keywords and Phrases: #1\par}\fi}

\begin{document}

\maketitle

\thispagestyle{plain}
\par\vspace*{.015\textheight}{\centering \em Dedicated to Ronnie Brown on the occasion of his 80th birthday.\par}

\bigskip
\begin{abstract}
Given a commutative ring $S$, a group $Q$ 
that acts on $S$ by ring automorphisms,
and an $S$-algebra endowed with an outer action of $Q$,
we study the associated Teichm\"uller class
in the appropriate third group cohomology group. We extend the classical 
results to this general setting.
\end{abstract}

\bigskip

\MSC 12G05 13B05 16H05 16K50 16S35 20J06
\EndMSC

\KEY Teichm\"uller cocycle, crossed module, crossed pair, 
normal algebra, crossed product,
Deuring embedding problem, group cohomology, 
Galois theory of commutative rings, Azumaya algebra, Brauer group,
Galois cohomology, non-commutative Galois theory,
non-abelian cohomology, number fields, rings of integers,
Brauer group of a topological space,
metacyclic group,
$\mathrm C^*$-algebra
\EndKEY

\tableofcontents

\addcontentsline{toc}{section}{Introduction}
\section*{Introduction}
Let $S$ be a unitary commutative ring, $Q$ a (not necessarily finite)
group, and $\kappaQ \colon 
Q\to \Aut(S)$
an action (not necessarily injective) of $Q$ on $S$ by ring automorphisms.
We refer to an algebra $A$ having $S$ as its center as
a {\em central} $S$-algebra. Given a central $S$-algebra $A$, we denote
by $\Aut(A)$ the group of ring automorphisms of $A$ and by $\Out(A)$ that of 
ring automorphisms modulo inner ones. We {\em define} a $Q$-normal
$S$-{\em algebra} to be a pair $(A, \sigma )$, where $A$ is a central
$S$-algebra and $\sigma \colon Q\to \Out (A)$ a homomorphism that lifts
the action $\kappaQ$ of $Q$ on $S$ in the sense that the composite of
$\sigma$ with 
the obvious map $\Out(A) \to \Aut(S)$ coincides with $\kappaQ$.
The special case where $S$ is a field and $Q$ a finite group of 
automorphisms of $S$ is classical;
in that case, a finite dimensional central simple $S$-algebra
$A$ having the property that
 each member of $Q$ can be extended to an automorphism of $A$
was termed $Q$-{\em normal} by Eilenberg and Mac Lane \cite{MR0025443}.
In view of the Skolem-Noether theorem, see, e.~g., Proposition \ref{skolnoet}
below for a generalization,
 the restriction map
$\Out(A) \to \Aut(S) $ is then injective, and
our definition is then therefore equivalent to 
that of Eilenberg and Mac Lane \cite{MR0025443}.
The classical case was studied already by Teich\-m\"uller \cite{MR0002858}. 
Teich\-m\"uller associated to a  $Q$-normal central simple 
algebra over a field $S$
a 3-cocycle 
of $Q$ with values in the multiplicative group $\mathrm U(S)$ of non-zero
elements of the field $S$, endowed with the $Q$-module structure
coming from the $Q$-action on $S$;
this 3-cocycle was then termed the 
{\em Teich\-m\"uller-cocycle\/}
by Eilenberg and Mac Lane \cite{MR0025443},
see also  \cite{MR0026993}, \cite{MR0025442}.

For the case of a general commutative unitary ring $S$ and a general action
$\kappaQ\colon Q \to \Aut(S)$ of $Q$ on $S$ by ring automorphisms, 
relative to the abelian group $\mathrm U(S)$ of invertible elements of 
the ring $S$, endowed with the $Q$-module structure coming from the 
$Q$-action on $S$,
to any $Q$-normal $S$-algebra
$(A, \sigma )$,
we associate
 a crossed~2-fold extension 
$
\mathrm e_{(A,  \sigma )}
$
starting at $\mathrm U(S)$ and ending at $Q$;
see Subsection~\ref{twothree} below for details.
We refer to this 
crossed~2-fold extension
as the {\em Teich\-m\"uller complex} of $(A, \sigma )$.
In view of the interpretation of 
the third group cohomology group
$\mathrm H^3(Q,\mathrm U(S))$ 
of $Q$ with coefficients in $\mathrm U(S)$
in terms 
of crossed 2-fold
extensions given in  
\cite{crossed}, see also 
\cite{historicalnote},
 $\mathrm e_{(A,  \sigma )}$ represents a class
$[\mathrm e_{(A,  \sigma )}]\in \mathrm H^3(Q,\mathrm U(S))$,  and  
we refer to this class as the 
 {\em Teich\-m\"uller class} of  $(A,  \sigma )$.
In Subsection~\ref{twothree} 
we show
that, in the cocycle description, the 
Teich\-m\"uller class is the one represented by a  Teich\-m\"uller cocycle.
Using that description of the Teichm\"uller class,
we extend the classical results related with the
Teich\-m\"uller cocycle to our general setting
and derive a number of new results.
We organize the material in three chapters.
In Chapter I, we explain the absolute case,
in Chapter II the relative case, and in Chapter III
we offer a number of examples.
Each chapter has its own abstract and introduction.

To our knowledge,
in the literature, 
apart from an interpretation of the \lq\lq Nakayama\rq\rq\ 
cocycle in \cite[\S 4]{MR0051266} 
in terms of abstract kernels,
the Teich\-m\"uller cocycle and its offspring
was never related with crossed modules
or with an equivalent notion.
In this paper, we borrow a number of constructions
from  \cite{MR1803361} and generalize them
within our setting.
It is, perhaps, worthwhile
noting that 
none of the papers by
Fr\"ohlich and Wall seems to contain a reference to Teich\-m\"uller
or to any of the follow-up papers thereof and that
none of the twelve articles that, according to the MR citations,
quote \cite{MR1803361}, contains a reference to
Teichm\"uller or to any of the follow-up papers.
Likewise, among the  articles that, according to the MR citations,
quote Teichm\"uller's original paper \cite{MR0002858},
only \cite{MR1787593} and \cite{MR1843316} 
contain a reference to the paper \cite{MR0349804} by Fr\"ohlich and Wall,
and there is no other reference to the papers  by Fr\"ohlich and Wall
related with the subject, in particular, no reference at all to
\cite{MR1803361} where a version of the Teichm\"uller
cocycle appears.
In this paper, we also explore 
the crossed pair concept, developed in \cite{MR597986}.
This concept
is closely related with that of a
pseudo-module \cite{MR0059911}; 
that paper arose out of a thesis supervised by J.H.C. Whitehead,
pseudo-modules being a generalization of crossed modules.
It seems that, thereafter, pseudo-modules were largely forgotten.
In 
\cref{pseudo}
we  make the relationship between
crossed pairs and pseudo-modules explicit.
We also note that, in class field theory, over a field,
the Teichm\"uller cocycle yields a certain obstruction
to the corresponding local-global principle
\cite{MR0051266} but is rarely spelled out 
explicitly; an explicit hint may be found, e.~g., in 
\cite[Section 11 p.~199]{MR0220697}.
We comment on this principle in 
Subsection \ref{generalremarks} below.

We  work over commutative rings 
(rather than fields on one hand or schemes on the other), and our methods 
are conceptual,
do not involve chain complexes, and avoid cocycle calculations.
It might be worthwhile to extend our methods to ringed spaces,
cf. \cite{MR0199213}.

It is a pleasure to dedicate this paper to Ronnie Brown.
In my thesis,
written with the help and
encouragement of B.~Eckmann,
 I had developed an interpretation of the group cohomology
groups in terms of crossed $n$-fold extensions, cf.
\cite{crossed}. At the time (in 1977), S. Mac Lane had suggested
I should get in contact with R. Brown, which I did,
and Ronnie got excited seeing that interpretation 
of group cohomology. It is fair to say, if I hadn't met Ronnie
at the time I might not have become a research mathematician.
Needless to point out, Ronnie has a long record of research papers
dealing with crossed modules, crossed $n$-fold extensions
and variants thereof, as well as numerous articles on applications
of these notions in topology and on foundational 
issues related with these notions.
The recent monograph \cite{MR2841564} reflects this activity and contains
a host of references.
In the present paper we show how crossed modules and
crossed pairs, a somewhat more
general notion, occur elsewhere in mathematical nature, 
in a way that is, perhaps, a bit surprising at first glance.
\chapter{The absolute case}
\label{cI}

\begin{abstract}
Let $S$ be a unitary commutative ring, $Q$ a group
that acts on $S$ by ring automorphisms, 
and let $R$ denote the subring of $S$ fixed under $Q$.
A $Q$-{\em normal\/}
$S$-algebra consists of  a central
$S$-algebra $A$ and a homomorphism 
$\sigma \colon Q\to \Out (A)$ into the group
$\Out(A)$ of outer automorphisms of $A$
that lifts
the action of $Q$ on $S$.
With respect to the abelian group $\mathrm U(S)$
of invertible elements of $S$, endowed with the $Q$-module structure
coming from the $Q$-action on $S$,
we associate to a $Q$-normal $S$-algebra
$(A, \sigma )$ a crossed~2-fold extension $\mathrm e_{(A,  \sigma )}$
starting at $\mathrm U(S)$ and ending at $Q$,
the {\em Teichm\"uller complex} of $(A, \sigma )$, and this complex, 
in turn, represents
a class, the
{\em Teichm\"uller class} of $(A, \sigma )$,
 in
the third group cohomology group
 $\mathrm H^3(Q,\mathrm U(S))$ of $Q$ 
with coefficients in $\mathrm U(S)$.
We extend some of the classical results to this general setting.
Among others, we relate the Teichm\"uller cocycle map
with the generalized Deuring embedding problem
and develop a seven term exact sequence
involving suitable generalized Brauer groups and
the generalized Teichm\"uller cocycle map.
We also relate the generalized Teichm\"uller cocycle map
with a
suitably defined abelian group
$k\mathcat{Rep}(Q,\mathcat B_{S,Q})$ 
of classes of representations of $Q$ in the $Q$-graded Brauer category
$\mathcat B_{S,Q}$ of $S$ with respect to the given action
of $Q$ on $S$.
\end{abstract}

\section{Introduction}

Exploiting the description of the Teich\-m\"uller class in terms of 
the Teich\-m\"uller complex, we show how 
the classical results
related with the Teich\-m\"uller cocycle for finite dimensional normal 
central simple algebras extend 
to general $Q$-normal algebras  
as defined above. We do not a priori assume 
$A$ to be, e.~g., an Azumaya algebra.
We shall make such an assumption only when truly necessary 
(in Subsections \ref{disc},  \ref{from}, \ref{genazalg}, 
and from Section~\ref{9} on) but
to begin with such an assumption would hide the formal
nature of the arguments.
We  now explain the results in the 
present chapter.

\smallskip

\noindent (1)  Let $(A,\sigma )$ be a $Q$-normal central $S$-algebra.
Theorem \ref{fouro} below says that $(A,\sigma )$ 
{\em has zero Teich\-m\"uller class
if and only if the $Q$-normal structure 
$\sigma_{\vert Q\vert}\colon Q \to \Out(M_{\vert Q\vert}(A))$ 
induced from $\sigma$ 
on the matrix algebra $M_{\vert Q\vert}(A)$ 
over $A$ 
is equivariant\/}, i.e., if and only if
the $Q$-normal structure 
$\sigma_{\vert Q\vert }$
arises from an ordinary action of $Q$ on $M_{\vert Q\vert}(A)$ 
by ring automorphisms. 
In particular, if $S|R$ is a Galois extension of commutative rings 
over $R=S^Q$
with
group $Q$ then an equivariant structure on  $M_{\vert Q\vert }(A)$ comes from
extension of scalars, by Galois descent, and so $(A,\sigma )$ then
{\em has zero Teich\-m\"uller class if and only if
 $(M_{\vert Q\vert }(A), \sigma_{\vert Q\vert })$ arises by extension
of scalars.}

\smallskip

\noindent (2) In order to explain our second result suppose for simplicity that
$S|R$ is a Galois extension of commutative rings 
over $R=S^Q$
with Galois group $Q$.
Corollary \ref{cor2} below
says that then {\em a central 
$S$-algebra $A$ has a $Q$-normal structure $\sigma $ with zero Teich\-m\"uller
class if and only if $A$ admits an embedding into a central $R$-algebra $C$
so that {\rm (i)} the centralizer of $S$ in $C$ is just $A$, and 
{\rm (ii)} each 
automorphism   $\kappaQ(q)$ of $S$, as $q$ ranges over $Q$, 
extends to an inner automorphism
$\alpha $ of $C$ which (in view of {\rm (i)}) maps $A$ to itself in such a
way that the class of $\alpha|A$ in $\Out (A)$ 
coincides with
 $\sigma (q)$; moreover,
if $A$ is an Azumaya $S$-algebra then $C$ may be taken to be an Azumaya $R$-algebra.}

In the classical situation, 
such an embedding problem was raised by Deu\-ring \cite{0014.20001}, 
and Teich\-m\"uller apparently invented his
cocycle in order to settle Deuring's embedding
problem.
Actually, in Theorem~\ref{4.2}below, we give a 
 somewhat more general result than that just
spelled out.

\smallskip

\noindent (3) Suitable equivalence classes of $Q$-normal Azumaya $S$-algebras
constitute an abelian  group, the {\em crossed Brauer group}, which we
 denote by $\mathrm{XB}(S,Q)$,
and the Teich\-m\"uller class depends only on the class in the crossed Brauer
group, see Section~{\ref{9}} below. Moreover, the crossed Brauer group is a functor
on a suitable category, to be
explained in Subsection~\ref{coa} below, 
the map $t\colon \mathrm{XB} (S,Q) \to \mathrm H^3(Q,\mathrm U(S))$
given by the assignment to
a $Q$-normal Azumaya $S$-algebra of its Teich\-m\"uller complex
is a natural homomorphism which we refer to henceforth as the 
{\em Teich\-m\"uller map}, and in case that $Q$ is a finite group, the kernel
of the Teich\-m\"uller map consists precisely of the classes of equivariant
$Q$-normal $S$-algebras, in view of~(1).

The results spelled out as (1), (2) and (3) extend the corresponding
results of Teich\-m\"uller 
\cite{MR0002858}.
For in the classical situation where $Q$ is a finite group of automorphisms
of a field $S$,
the crossed Brauer group $\mathrm{XB}(S,Q)$ comes down to the subgroup
$\mathrm B(S)^Q$ of the Brauer group $\mathrm B(S)$ of $S$
that consists of the classes fixed under $Q$,
and these are precisely the $Q$-normal classes. 

\smallskip

\noindent (4) An alternate approach in terms of the
\lq\lq group-like stably $Q$-graded symmetric monoidal
category\rq\rq\ $\mathcat B_{S,Q}$
associated with the symmetric monoidal
{\em Brauer\/} category $\mathcat B_S$ of the ground ring $S$
and the $Q$-action on $S$
 \cite{genbrauer},
\cite{MR0349804},
\cite{MR1803361}, see Section \ref{sgsmc}
for details,
involves the abelian group  $k\mathcat{Rep}(Q,\mathcat B_{S,Q})$
of classes of representations of $Q$ in $\mathcat B_{S,Q}$
(see Subsection \ref{sqgbc} for the notation), and we refer
to the abelian group $k\mathcat{Rep}(Q,\mathcat B_{S,Q})$
as the {\em generalized crossed Brauer group\/}
and to its members as {\em generalized $Q$-normal Azumaya algebras\/},
for the following reason: In Subsection \ref{from} below
we associate, to any generalized $Q$-normal Azumaya algebra, 
an ordinary $Q$-normal algebra whose underlying algebra is
an Azumaya algebra if and only if the subgroup $\kappaQ(Q)$ of $\Aut(S)$
is a finite group; see Theorem \ref{fundclass1} for details. 
In Subsection \ref{genazalg}
we then define the Teichm\"uller class
of the given  generalized $Q$-normal Azumaya algebra
to be the Teichm\"uller class of the associated ordinary
$Q$-normal algebra.
Here our approach in terms of general algebras
rather than just Azumaya algebras pays off
since the $Q$-normal algebra
associated to a  generalized $Q$-normal Azumaya algebra
need not be an Azumaya algebra.
The map $t\colon k\mathcat{Rep}(Q,\mathcat B_{S,Q}) \to \mathrm H^3(Q,\mathrm U(S))$
given by the assignment to 
a generalized $Q$-normal Azumaya $S$-algebra of its Teich\-m\"uller class
is a natural homomorphism which, combined with the obvious homomorphism
from the crossed Brauer group $\mathrm{XB}(S,Q)$ to
$k\mathcat{Rep}(Q,\mathcat B_{S,Q})$, 
necessarily injective,
cf. Theorem \ref{tentwentytw},
coincides with the Teichm\"uller cocycle map
from $\mathrm{XB}(S,Q)$ to $\mathrm H^3(Q,\mathrm U(S))$
defined previously. 
If, furthermore, the subgroup $\kappaQ(Q)$ of $\Aut(S)$
is a finite group,
the obvious 
 homomorphism
from the crossed Brauer group $\mathrm{XB}(S,Q)$ to
$k\mathcat{Rep}(Q,\mathcat B_{S,Q})$ is an isomorphism of abelian groups.
See Subsection \ref{gencrob} for details.
A byproduct of our reasoning is the following observation,
where $\mathrm B(S,Q)$ denotes the subgroup
of the Brauer group $\mathrm B(S)$ whose members are 
represented by Azumaya $S$-algebras $A$ having the property that
each automorphism of the kind $\kappaQ(x)$ of $S$ as $x$ ranges over $Q$
extends to an automorphism of $A$:
The canonical homomorphism from 
$\mathrm B(S,Q)$ to $\mathrm B(S)^Q=\mathrm H^0(Q,\mathrm B(S))$ is
an isomorphism when the subgroup $\kappaQ(Q)$ of $\Aut(S)$
is a finite group; see Corollary \ref{2.21} for details.

\smallskip

\noindent
(5)  With respect to the data, 
let $\mathrm{EB}(S,Q)$   
denote
the equivariant Brauer group,
see Section \ref{eleven}
below for details. 
For the particular case
where the group $Q$ is finite,
the exact sequence  \eqref{twelvet} below
involving the Teichm\"uller cocycle map $t$
yields an extension of the kind
\begin{equation*}
\ldots
\to 
(\Pic (S))^Q
\to \mathrm H^2(Q,\mathrm U(S))
\to \mathrm{EB}(S,Q) \to
\mathrm{XB}(S,Q)
\stackrel{t}\to \mathrm H^3(Q,\mathrm U(S))
\end{equation*}
of the corresponding 
classical low degree four term exact sequence
by three more terms.
If, furthermore, 
$S|R$ is a Galois extension
of commutative rings 
over $R=S^Q$
with Galois group $Q$, by Galois descent, 
the abelian group  $\mathrm{EB}(S,Q)$ is isomorphic to
the ordinary Brauer group $\mathrm{B}(R)$ of $R$.

Generalizations of the  Teich\-m\"uller cocycle map
were also developed by Childs~\cite{MR0311701}, 
Fr\"ohlich and Wall~\cite{MR1803361},
Knus~\cite{amitsuteich},
Pareigis~\cite{MR0161881},
Ulbrich~\cite{MR1011607}, \cite{MR1284782}, 
and Zelinski~\cite{MR0432621}. Our approach 
is substantially different from that in each of those papers.
Indeed, 
in those articles  
except \cite{amitsuteich} and \cite{MR1803361},
given an action $\kappa\colon Q \to \Aut(S)$ 
of a group $Q$ on the commutative ring $S$,
an $S$-Azumaya algebra $A$ is defined to be
$Q$-normal if each automorphism $\kappa(q)$ of $S$, as $q$ ranges over $Q$,
extends to an automorphism of $A$;
accordingly, the values of the Teichm\"uller cocycle map developed
in those articles do not necessarily lie in
the cohomology group $\mathrm H^3(Q, \mathrm U(S))$.
When $S$ is a field,
in view of the Skolem-Noether theorem, 
that definition is equivalent to ours but over
a general ring $S$ this is not the case; needless, perhaps, to point out that,
given the $S$-algebra $A$, with our definition involving a homomorphism
$\sigma \colon Q \to \Out(A)$,
we indeed obtain a Teichm\"uller cocycle map with values in 
$\mathrm H^3(Q, \mathrm U(S))$. 
The (unpublished) manuscript \cite{amitsuteich}
offers a purely
Amitsur complex approach.
Over a general commutative ring $S$, (without a name,)
the crossed Brauer group 
is introduced
in \cite[p.~43]{genbrauer}, see also  \cite[Section 3]{MR1803361},
denoted there by $Q\mathrm{B}(R,\Gamma)$ 
(beware: that notation $Q$ has nothing to do with our notation $Q$ for
a group)
where
$R$ corresponds to our notation $S$ and 
$\Gamma$ to our notation $Q$. 
While none of the papers by
Fr\"ohlich and Wall seems to contain a reference to Teich\-m\"uller
or to any of the follow-up papers thereof,
in \cite[Theorem 3.4 (i)]{MR1803361},
the attentive reader will discover an explicit cocycle description
of the Teichm\"uller cocycle map
(without a name)
from the crossed Brauer group $\mathrm{XB}(S,Q)$ to the corresponding
third group cohomology group,
as well as a 
cocycle description
of the Teichm\"uller cocycle map
(still without a name)
from  $k\mathcat{Rep}(Q,\mathcat B_{S,Q})$
to the corresponding
third group cohomology group.
Also the injective homomorphism
from $\mathrm{XB}(S,Q)$ to $k\mathcat{Rep}(Q,\mathcat B_{S,Q})$ is given
in \cite{MR1803361}, as well as a proof of the fact that, for
finite $Q$, this homomorphism is an isomorphism. Indeed, our 
construction of the $Q$-normal algebra associated to
a generalized $Q$-normal Azumaya algebra 
given in Subsection \ref{from} below
involves a variant of a construction
that is used in \cite{MR1803361} to establish the surjectivity of
the homomorphism from $\mathrm{XB}(S,Q)$ to
$k\mathcat{Rep}(Q,\mathcat B_{S,Q})$ in the special case where
the group $Q$ is finite.

\section{Preliminaries}
\label{one}

\subsection{Basics}
For the reader's convenience we recall some basic definitions and facts
used in the paper.

Below we  sometimes write the identity morphism on an object as $1$. 
Given a morphism $\alpha$ on an object, we  occasionally
denote the restriction to a subobject by \lq\lq $\alpha|$\rq\rq.
Given an action $\kappa \colon G\to \Aut_{\mathcat C}(C)$ of a group
$G$ on an object $C$ of a category $\mathcat C$, we  write the action as
\begin{equation*} 
G \times C \longrightarrow C,\ (x,y) \longmapsto {}^x\! y,\ x\in G,\ y \in C.
\end{equation*}

The commutative unitary ring $S$ remains fixed, unless the contrary is admitted
explicitly, and the notation $\otimes $ refers to the tensor product over $S$.
By an $S$-{\em algebra} we mean an algebra  $A$ whose center contains
$S$ as a subring; thus given an $S$-algebra $A$, 
$S$ acts faithfully on $A$. As in the introduction,
 a {\em central\/} $S$-algebra is 
an $S$-algebra such that $S$ coincides with the center
of $A$.
Given an $S$-algebra $A$ (not necessarily central), we denote by
$A^{\mathrm{op}}$ 
the {\em opposite} algebra.
Given the $S$-algebra $A$,
consider the tensor product algebra $ A\otimes A^{\mathrm{op}}$;
the map 
\begin{equation}
\eta \colon A\otimes A^{\mathrm{op}} \to \Hom_S(A,A),
\ 
\eta (a \otimes b^{\mathrm{op}}) c = acb, \ a, b, c\in A,
\label{azu}
\end{equation}
turns $A$ into an
$(A\otimes A^{\mathrm{op}})$-module.
As in 
\cite{MR0121392}, we refer to an $S$-algebra $A$ as being 
{\em separable\/} 
when $A$ is a projective 
as an $(A\otimes A^{\mathrm{op}})$-module. A central separable
$S$-algebra $A$ is also referred to as an {\em Azumaya} $S$-{\em algebra}.
This agrees with the definition in 
\cite[p.~180]{MR0217051} since,
by  
\cite[Theorem 2.1]{MR0121392}, a central $S$-algebra $A$
is separable if and only if the map $\eta $ is an isomorphism and if,
as an $S$-module, $A$ is
finitely generated and projective.

Given an action of
a group $G$ on a (not necessarily commutative)
ring $\Lambda $, 
written as
$(x,t) \mapsto {}^x t$, for $x \in G$ and $t \in \Lambda$,
the {\em twisted group ring} $\Lambda^t G$ 
has as its underlying $\Lambda $-module
the free
left $\Lambda $-module having as its basis the elements of $G$,  with
multiplication given by $sx ty =s({}^xt)xy$, where 
$s,t\in \Lambda$, $x, y\in G$.

Occasionally we  use the familiar notation $\rightarrowtail$
for an injection and $\twoheadrightarrow$
for a surjection. 

\subsection{Galois extensions of commutative rings}
\label{galext}

As in the introduction, given a group $Q$ and an
action $\kappaQ \colon Q \to \Aut(S)$ of $Q$ on $S$,
 relative to  the action,
the twisted group ring
$S{}^tQ$ is defined.
Using the notation $R = S^Q$ for the subring of $S$ 
that consists of all members of $S$
left fixed by every element of $Q$, consider the $R$-algebra homomorphism
\begin{equation*}
 j \colon S^tQ \longrightarrow \End_R(S),\ 
(j(sq))(x) = s({}^qx),\, s,x\in S,\, q\in Q.
\end{equation*}
Furthermore, we  consider the $S$-algebra
$\mathrm{Map}(G,S)$ of $S$-valued functions on $G$ 
with pointwise multiplication
and the $S$-algebra map
\begin{equation*}
h \colon S \otimes_R S \longrightarrow \mathrm{Map}(G,S),\ 
(h(s_1 \otimes s_2))(x)=s_1 x(s_2),\  s_1,s_1 \in S,\ x \in Q.
\end{equation*}

Suppose now that the structure map $\kappaQ\colon Q \to \Aut(S)$ is
injective and write $R=S^Q$.
Thus $Q$ is now a finite group of operators on the commutative
ring $S$.
Under these circumstances, $S|R$ is defined to be a 
{\em Galois extension of commutative rings with Galois group\/} $Q$ 
\cite[Def. 1.4 p.~6]{MR0195922}
if any of 
the subsequent equivalent
conditions (i)--(iv) holds:
\smallskip

{(i)} The ring $S$ is a finitely generated projective $R$-module,
and the $R$-algebra homomorphism
$j \colon S^tQ \to \End_R(S)$
is an isomorphism of $R$-algebras.

\smallskip

{(ii)} {\em (Galois descent).}
Given a left $S^tQ$-module $M$, viewed as a left $Q$-module in the
obvious way, the map 
\begin{equation*}
w\colon S\otimes_R M^Q \longrightarrow M,\quad 
w(s\otimes m) = sm, \ s\in S, \ m\in M^Q, 
\end{equation*}
is an $S$-module isomorphism, indeed, an $S^tQ$-module isomorphism relative to
the obvious $S^tQ$-module structure on  $S\otimes_R M^Q$.

\smallskip

{(iii)} Given a member $q$ of $Q$ distinct from the neutral element
and a maximal ideal $\mathfrak p$ 
of $S$, 
there exists $s = s(\mathfrak p,q)$ in $S$ with $s-{}^qs$ not in $\mathfrak p$.

\smallskip

{(iv)} The $S$-algebra map $
h \colon S \otimes_R S \longrightarrow \mathrm{Map}(G,S)$
is an isomorphism of $S$-algebras.

\smallskip

The equivalence of (i)--(iv) is established in 
\cite[Theorem 1.3 p.~4]{MR0195922}.
In 
\cite[p.~396]{MR0121392}, (i) is taken as
the definition of a Galois extension of commutative rings
with Galois group $Q$.

We  now recall a somewhat more sophisticated characterization of a 
Galois extension of commutative rings that has the advantage
of being susceptible to generalization.
To this end, let $H_Q= \mathrm{Map}(Q,R)$, endowed with the obvious structure
of a commutative $R$-Hopf algebra, the diagonal 
$\Delta \colon H_Q \to H_Q \otimes_R H_Q$
being induced by the group structure map of $Q$.
Since the group $Q$ is finite,
the $Q$-action $Q \times S \to S$ 
of $Q$ on $S$ from the left
is {\em rational\/} in the sense that it
is determined by the right
$H_Q$-comodule structure
\begin{equation}
\Delta \colon S \longrightarrow  S \otimes_R H_Q ,\ \Delta(s)
= \sum_{x \in Q} {}^xs \otimes f_x, \ f_x(y) = \delta_{x,y},
\ x,y \in Q,\ s \in S,
\end{equation}
on $S$.
Let $\mu \colon S \otimes_R S \to S$ denote the multiplication map of $S$
and consider the canonical map
\begin{equation}
S\otimes_R S \longrightarrow S \otimes_R H_Q
\label{can91}
\end{equation}
which arises as the composite
\begin{equation}
\begin{CD}
S\otimes_R S @>{S \otimes_R \Delta}>>  S \otimes _R S \otimes _R H_Q  
@>{\mu \otimes _R H_Q}>> S \otimes_R H_Q .
\end{CD}
\end{equation}

The following observation 
and the corresponding interpretation in terms of {\em Galois objects\/}
is due to
\cite[Ch. II p.~59]{MR0260724}.

\begin{prop}
The ring extension $S|R$ 
($R= S^Q$)
is a Galois extension of commutative rings
with Galois group $Q$ if and only if the following holds:

{\rm (v)}
The canonical map \eqref{can91}
is a bijection.
\end{prop}

\begin{proof}
In view of the canonical isomorphism
$S \otimes _R H_Q =S \otimes_R \mathrm{Map}(Q,R) \cong \mathrm{Map}(Q,S)$,
the bijectivity of the canonical map \eqref{can91}
is equivalent to characterization (iv) above.
\end{proof}

Given merely the finite group $Q$ of operators on the ring $S$,
maintaining the notation $R=S^Q$, we note that 
$S\otimes_R S$ acquires a left $S$-module structure
from the left-hand copy of $S$ and a right $H_Q$-comodule structure
$\Delta\colon S\otimes_R S \to S \otimes_R S \otimes_R H_Q$ 
from the  right $H_Q$-comodule structure
$\Delta\colon S \to  S \otimes_R H_Q$ 
on the right-hand copy of $S$ in $S\otimes_R S$, and
the canonical map \eqref{can91}
is a morphism of left $S$-modules and
right $H_Q$-comodules.
Hence:

\begin{prop}
The ring extension $S|R$ 
($R= S^Q$)
is a Galois extension of commutative rings
with Galois group $Q$ if and only if the following holds:

{\rm (vi)}
The canonical map \eqref{can91}
is an isomorphism of left $S$-modules and right $H_Q$-comodules. \qed
\end{prop}

Characterization (vi) above of a Galois extension of commutative rings
generalizes as follows, cf. \cite[8.1.1 Definition p.~ 123]{MR1243637}:
Fix a ground ring $\fiel$,
let $H$ be a Hopf algebra,
write the structure maps as $\Delta \colon H \to H \otimes H$,
$\mu\colon H \otimes H \to H$, $\eta\colon \fiel \to H$ (unit),
$\varepsilon \colon H \to \fiel$ (counit),
 and consider an algebra
$A$ endowed with 
a right $H$-comodule structure
$\Delta \colon A \to A \otimes H$.
Then $(A,\Delta)$ is defined to be a
a {\em right\/} $H$-{\em comodule algebra\/} if, with regard to the multiplication 
map
$\mu\colon H \otimes H \to H$ of $H$,
the multiplication map $\mu \colon A \otimes A\to A$ of $A$
is a morphism of
right $H$-comodules and if, furthermore,
the unit map $\eta\colon\fiel \to A$ of $A$ is a morphism of right $H$-comodules.
Now, consider
a right $H$-comodule algebra
$(A,\Delta)$, recall the operators 
\begin{equation}
d_0,d_1\colon A \longrightarrow \Hom(H,A),\ 
d_0(a)(x)=x q,\ d_1(a)(x) = \varepsilon (x)a,\ a \in A,\ x \in H 
\end{equation}
and let
\begin{equation}
B=A^{\mathrm{co}H}= \mathrm{ker}(d_0-d_1)\colon A \longrightarrow \Hom(H,A).
\end{equation}
Similarly as before, 
the tensor product $A\otimes_B A$ (not necessarily an algebra)
acquires
a left $A$-module structure from the left-hand copy of $A$ in $A \otimes_BA$
and a right $H$-comodule structure $A \otimes_BA \to A \otimes_BA \otimes H$
from the right $H$-comodule structure of 
the right-hand copy of
$A$ in  $A\otimes_B A$.
The extension $A|B$ is defined to be a {\em right\/} $H$-{\em Galois extension\/}
if the canonical map
\begin{equation}
\begin{CD}
A\otimes_B A @>{A \otimes \Delta}>> A\otimes_B A \otimes H
@>{\mu \otimes H}>> A\otimes H
\end{CD}
\end{equation}
is an isomorphism of left $A$-modules and right $H$-comodules.
See \cref{ringsofint}
for an example
of this more general notion of Galois extension. 

\begin{examp}
\label{exatwo}
Given a (finite) Galois extension $K|\fiel$ of algebraic number fields 
with Galois group 
$Q$,
the extension $S|R$ of the associated rings of integers is
a Galois extension of commutative rings
with Galois group $Q$ if and only if the extension
 $K|\fiel$ is unramified  \cite[Remark 1.5 (d) p.~7]{MR0195922}.
This fact is a consequence of \cite[Theorem 2.5 p.~753]{MR0106929}.
The property of being unramified is somewhat rare, however:
The field of rational numbers does not possess any unramified extension.
In general, the maximal abelian unramified extension of a number field
$\fiel $
is what is known as the {\em Hilbert class field of \/} $\fiel$.
The Galois group thereof is isomorphic to the ideal class group of $\fiel$.
See, e.~g., \cite[Phrase before III.7.9 Theorem p. 164,
III.8.8 Theorem p.~172]{MR3058613} for details.
For a general finite extension $K|\fiel$ of algebraic number fields,
the primes of $\fiel$ which ramify in $K$ are those which divide the 
discriminant of $K$ over $\fiel$. Hence only finitely many primes of
$\fiel$ ramify in $K$. Let $\mathbb S_{K|\fiel}$ 
denote the finite set of primes
that ramify in $K$,
let  $\mathbb S_K$ denote those primes of $K$ that lie over
 $\mathbb S_{K|\fiel}$, and let
$R_{\mathbb S_{K|\fiel}}\subseteq \fiel$
and
$S_{\mathbb S_{K}} \subseteq K$
denote the subrings  that arise when we invert the primes in
${\mathbb S_{K|\fiel}}$
and in ${\mathbb S_{K|\fiel}}$, respectively. We can make these constructions
precise in the language of {\em valuations\/}.
Then 
$S_{\mathbb S_{K}}|R_{\mathbb S_{K|\fiel}}$
is a Galois extension of commutative rings with Galois group $Q$.
We shall come back to this class of examples in 
\cref{ringsofint}.
\end{examp}

\begin{examp}
\label{exathr}
Given an action
$\mu \colon Y \times Q \to Y$ of a
 finite group $Q$ on a Hausdorff space
$Y$ with orbit space $X$, the ring extension $C^0(Y)|C^0(X)$
of the associated rings of continuous functions
is a Galois extension of commutative rings with Galois group $Q$
if and only if the projection $\mathrm{pr}\colon Y \to X$ is an ordinary
regular  covering map  
having the group $Q$ as its group of deck transformations
\cite[Remark 1.5 (e) p.~7]{MR0195922}.
Indeed, this is an immediate consequence of (iv) above since
the projection map $\mathrm{pr}\colon Y \to X$
is an ordinary
regular  covering map  
with $Q$ as its group of deck transformations
if and only if
 the canonical map 
$(\mathrm{pr},\mu)\colon Y \times Q \to Y \times_X Y$
is a homeomorphism.
Formally the same property, but phrased in the language of
affine varieties, characterizes affine principal bundles
relative to an affine algebraic group $Q$, the characterization
being spelled out in terms of a right
$H$-Galois extension where $H$ is the coordinate ring
of the algebraic group $Q$ under discussion.
See, e.~g., 
\cite{MR0444680}, \cite{MR1098988} for details.
Likewise, an ordinary principal bundle having as structure
group a general Lie group can as well be seen as a generalized 
Galois extension.
\end{examp}

\subsection{Crossed modules}

A {\em crossed module} $(C,\Gamma, \partial )$  
\cite[p. 453] {whitefiv}
consists of groups $C$ and $\Gamma $, an action of $\Gamma $ on 
$C$ (from the left), written as $(\gamma , x) \mapsto {}^\gamma x, \gamma \in \Gamma , x\in C$,
and a homomorphism $\partial \colon C\to \Gamma $ of $\Gamma $-groups where
$\Gamma $ acts on itself by conjugation, 
subject to the axiom
\begin{equation*}
bc b^{-1} = {}^{\partial b}c, \; b,c\in C.
\end{equation*}
Given two crossed modules $(C,\Gamma, \partial )$
and $(C',\Gamma', \partial')$, a {\em morphism\/}
\begin{equation*}
(\varphi,\psi)\colon (C,\Gamma, \partial) \longrightarrow (C',\Gamma', \partial')
\end{equation*}
{\em of crossed modules\/}
is given by a commutative diagram
\begin{displaymath}
\xymatrix{
 &C \ar[d]_{\varphi} \ar[r]^{\partial} &\Gamma \ar[d]^{\psi}& 
\\
 &C' \ar@{=} \ar[r]^{\partial'}      &\Gamma'&
}
\end{displaymath}
in the category of $\Gamma$-groups, the $\Gamma$-action
on $C'$ and $\Gamma'$ being induced by the homomorphism $\psi$.
A {\em crossed $2$-fold extension\/}
is an exact sequence of groups
\begin{equation*}
\mathrm e^2 \colon 0 \longrightarrow M \longrightarrow C 
\stackrel{\partial}\longrightarrow \Gamma \longrightarrow G \longrightarrow 1
\end{equation*}
involving a crossed module $(C, \Gamma , \partial)$
\cite{crossed}; since $M$  is then central
in $C$ it is an abelian group, and the $\Gamma$-action on $C$ induces a 
$G$-module structure
on $M$. 
Given two crossed 2-fold extensions
\begin{equation*}
\mathrm e^2 \colon 0 \to M \to C 
\stackrel{\partial}\to \Gamma \to G \longrightarrow 1,
\ 
\hat{\mathrm e}^2 \colon 0 \to \hat M \to \hat C 
\stackrel{\hat \partial}\to \hat \Gamma \to \hat G \longrightarrow 1,
\end{equation*}
a {\em morphism\/} $(\alpha,\varphi,\psi,\beta)
\colon\mathrm e^2 \to \hat{\mathrm e}^2$ of crossed 2-fold extensions
is a commutative diagram
\begin{displaymath}
\xymatrix{
0  \ar[r]   & M\ar[d]_{\alpha}\ar[r] &C \ar[d]_{\varphi} \ar[r]^{\partial} &\Gamma \ar[d]_{\psi} \ar[r]
&G \ar[d]_{\beta} \ar[r] &1 \\
0  \ar[r]   &\hat M \ar[r]  &\hat C \ar@{=} \ar[r]^{\hat \partial}&\hat \Gamma\ar[r]&\hat G \ar[r] &1 
}
\end{displaymath}
 in the category of groups such that
$(\varphi,\psi)\colon (C,\Gamma, \partial) \to (\hat C,\hat \Gamma, \hat \partial)$
is a morphism of crossed modules;
a morphism of crossed 2-fold extensions of the kind
$(1,\varphi,\psi,1)
\colon\mathrm e^2 \to \hat{\mathrm e}^2$
(so that, in particular, $\hat M = M$ and $\hat G = G$)
is referred to as a {\em congruence\/}.
The notion of congruence of group extensions with abelian kernel
is classical, cf. \cite[IV.3 p.~109]{maclaboo}.

When the group $G$ and the $G$-module $M$ are fixed, 
under the equivalence relation generated by 
congruence,
the classes of crossed $2$-fold extensions
starting at $M$ and ending at $G$
constitute an abelian group $\mathrm{Opext}^2(G,M)$, 
and this group is naturally
isomorphic to the ordinary group cohomology group $\mathrm H^3(G,M)$; this fact
is a special 
case of 
the main result in \cite[\S~7]{crossed}. 
See also 
Mac~Lane's Historical Note \cite{historicalnote}.

\section{Stably graded symmetric monoidal categories}
\label{sgsmc}
We  recall
some material from the theory of stably graded
symmetric monoidal categories \cite{genbrauer},
\cite{MR0349804},
\cite{MR1803361}
and explain, in particular, the significance of
group-like 
stably graded
symmetric monoidal categories for our purposes.
This enables us to introduce notation
needed thereafter in the paper.

\subsection{Symmetric monoidal categories}

A {\em symmetric monoidal category\/} 
\cite{MR0354798}  consists of a category
$\mathpzc C$, a covariant functor 
\[
\compo\colon \mathpzc C\times \mathpzc C \longrightarrow \mathpzc C, 
\]
an object
$E$ of $\mathpzc C$, referred to as the {\em unit object\/},
and natural structure equivalences
written as $e\colon E \compo \,\cdot\,  \to   \,\cdot\,$,
$a\colon   \,\cdot\,\compo\, (\cdot\, \compo \,\cdot\,)  \to   
 (\,\cdot\,\compo\, \cdot\,) \compo \,\cdot\,  $, and 
$c_{A,B}\colon A\compo B \to B \compo A  $,
where $A$ and $B$ range over objects of $\mathpzc C$;
these pieces of structure are subject to certain nowadays
standard axioms. The terminology in
\cite{genbrauer},
\cite{MR0349804}, \cite{MR1803361}
is \lq monoidal\rq\ for the present \lq symmetric monoidal\rq. 
We now recall some more terminology from [op. cit.].

Consider a symmetric monoidal category
$\mathpzc C$.
The notation $k \mathpzc C$
refers to the abelian monoid (with unit) 
of isomorphism classes of objects of $\mathpzc C$,
 usually required to be a set,
and $\mathrm U(\mathpzc{C})$ denotes the group
of automorphisms 
$\Aut_\mathpzc{C}(E)$
of the unit object $E$ in
$\mathpzc{C}$ \cite[\S 2 p.~235]{MR0349804}.

Given an object $C$ of $\mathpzc{C}$, for $u \in \mathrm U(\mathpzc{C})$, the
composite
\begin{equation}
\vartheta_C(u)\colon
C
\stackrel{e^{-1}_C}
\longrightarrow
E \compo C
\stackrel{u \compo 1_C}
\longrightarrow
E \compo C
\stackrel{e_C}
\longrightarrow
C
\end{equation}
defines an automorphism of $C$ and,
by
\cite[Theorem 2.2  p.~235]{MR0349804},
\begin{equation}
\vartheta_C\colon \mathrm U(\mathpzc{C}) \longrightarrow \Aut_{\mathpzc{C}}(C)
\label{varthetaC}
\end{equation}
is a homomorphism, an isomorphism if $C$ is an invertible object of
$\mathpzc{C}$ and, in general, the subgroup 
$\mathrm{im}(\vartheta_C) \subseteq \Aut_{\mathpzc{C}}(C)$
lies in the center of $\Aut_{\mathpzc{C}}(C)$.
An object $C$ of $\mathpzc C$ is {\em faithful\/} if 
$\vartheta_C$ is injective \cite[\S 6 p.~47]{genbrauer},
\cite[\S 4  p.~243]{MR0349804}.

A symmetric monoidal category $\mathpzc C$ is {\em group-like\/} if
every object and every morphism in $\mathpzc C$ is invertible
\cite[\S 2 p.~237]{MR0349804}, \cite[\S 1]{MR1803361}.
In the last reference, the definition of being
group-like is applied only to
\lq\lq precise\rq\rq\  
symmetric monoidal categories but this is presumably not intended and,
after all, in \cite{MR0349804},
the definition of being
group-like is applied to general symmetric monoidal categories.
The monoid  $k \mathpzc C$
associated with a group-like symmetric monoidal category $\mathpzc C$
is an abelian group.  
 
\subsection{The Brauer category associated with a commutative ring}
\label{bcacr}
For later reference, we recall the {\em standard construction\/} of the
{\em Brauer group\/} $\mathrm B(S)$ of the commutative ring $S$
\cite[p.~381]{MR0121392}. Two  Azumaya $S$-algebras $A_1$ and 
 $A_2$ are {\em Brauer equivalent\/}
if there are faithful finitely
generated projective 
$S$-modules $M_1$ and $M_2$
such that the Azumaya $S$-algebras
$A_1 \otimes \End_S(M_1)$
and $A_2 \otimes \End_S(M_2)$
are isomorphic $S$-algebras.
Given two faithful finitely
generated projective 
$S$-modules $M_1$ and $M_2$, the tensor product
$M_1\otimes M_2$ is again 
a faithful finitely
generated projective 
$S$-module, and the $S$-algebra
$\End_S(M_1) \otimes \End_S(M_2)$
is canonically isomorphic to
the $S$-algebra $\End_S(M_1 \otimes M_2)$.
Hence
that relation 
is indeed an  equivalence relation, 
cf. \cite[p.~381]{MR0121392}, referred to as {\em Brauer equivalence\/};
under the operation of
tensor product 
and under the assignment
to the class of an Azumaya $S$-algebra $A$
of the class of its opposite algebra $A^{\mathrm{op}}$, 
the equivalence classes 
constitute an abelian group, the {\em Brauer group\/}
$\mathrm B(S)$ of $S$,
having the class of $S$ as its unit element.
The assignment to a commutative ring of its Brauer group 
is a functor from commutative ring to abelian groups.
In particular, 
 the  
$Q$-action on $S$ induces a $Q$-action on the
Brauer group
$\mathrm B(S)$ of $S$ that turns $\mathrm B(S)$ into a $Q$-module. 
Given a homomorphism 
$f \colon S \to T$
of commutative rings, we denote by $\mathrm B(T|S)$
the kernel of the induced homomorphism
$\mathrm B(S) \to \mathrm B(T)$.
The abelian group $\mathrm B(T|S)$ really depends on the
homomorphism $f$ rather than just on $S$ and $T$ but, for intelligibility,
we  stick to the familiar notation $\mathrm B(T|S)$ 
which is classical when $S$ and $T$ are fields (and $f$  necessarily
injective).

Recall that, given two Azumaya $S$-algebras $A$ and $B$, 
a $(B,A)$-bimodule $M$ is {\em invertible\/}
if there exists an  $(A,B)$-bimodule $M'$
such that $M\otimes_A M' \cong B$ as $(B,B)$-bimodules and
$M'\otimes_B M \cong A$  as $(A,A)$-bimodules,
cf., e.~g., \cite{MR0249491}.
 The {\em Brauer category\/} $\mathpzc B_S$
of the commutative ring $S$ 
\cite[\S 3 p.~23]{genbrauer}, \cite[p.~230]{MR0349804}, \cite[\S 2]{MR1803361},
written in
\cite[\S 3 p.~23]{genbrauer} and \cite[\S 2]{MR1803361}
as $\mathpzc {B}_R$ 
and in \cite[p.~230]{MR0349804} as $\mathpzc {Br}_R$,
has as
 {\em objects\/} the
Azumaya $S$-algebras,
a {\em morphism\/} $[M]\colon A \to B$ in $\mathpzc B_S$ 
between two Azumaya algebras $A$ and $B$,
necessarily an isomorphism in $\mathpzc B_S$, 
being an isomorphism class
of an invertible $(B,A)$-bimodule $M$.
Given three Azumaya $S$-algebras $A$, $B$, $C$ and morphisms
 $[{}_BM_A]\colon A \to B$
and $[{}_AM_C]\colon C\to A$
in $\mathpzc B_{S}$, the
composite of 
 $[{}_BM_A]\colon A \to B$ with 
 $[{}_AM_C]\colon C\to A$ in $\mathpzc B_{S}$ 
is given by the morphism  
$[{}_BM_A\otimes_A {}_AM_C]\colon C\to B$  
 in $\mathpzc B_{S}$, 
where $[{}_BM_A\otimes_A {}_AM_C]$ 
refers to the isomorphism class
of the invertible
$(B,C)$-bimodule ${}_BM_A\otimes_A {}_AM_C$.
The operation of tensor product over the ground ring $S$ and the assignment
to an Azumaya $S$-algebra $A$ of its opposite algebra $A^{\mathrm{op}}$
turn $\mathcat B_S$ 
into a group-like symmetric monoidal category
\cite[\S 5 p.~247]{MR0349804}, \cite[\S 2]{MR1803361}
having $\mathrm U(\mathpzc{B}_S) = \Pic(S)$ 
\cite[\S 3]{MR1803361}.
The members of
the abelian group
$k\mathcat B_S$ 
are 
Morita equivalence classes of Azumaya $S$-algebras.
Since Morita equivalence is equivalent to Brauer equivalence,
cf., e.~g., \cite[p.~41]{MR0460308},
the canonical homomorphism from $\mathrm B(S)$ to 
the abelian group $k\mathcat B_S$ is an isomorphism.
By construction, then, given an Azumaya $S$-algebra $A$,
its group $\Aut_{\mathcat B_S}(A)$ of automorphisms in
$\mathcat B_S$
is the group of
faithful projective rank one $(A,A)$-bimodules,
and the assignment to a projective rank one $S$-module $J$ of
the $(A,A)$-bimodule $A\otimes J$ yields an isomorphism
$\Pic(S) \to \Aut_{\mathcat B_S}(A)$ of abelian groups.

\subsection{Stably graded categories}

As before, $Q$ denotes a group. We view $Q$ as a category with a 
single object. 
A $Q$-{\em graded\/} category is a pair 
$(\mathpzc C_Q,g)$ that consists of a category
$\mathpzc C_Q$ and a functor $g\colon \mathpzc C_Q \to Q$, the {\em grading\/}
\cite[\S 1 p.~2]{genbrauer},  
\cite[\S 3 p.~240]{MR0349804}, 
\cite[\S 1]{MR1803361}.
The grading $g$ 
of a  $Q$-graded category 
$(\mathpzc C_Q,g)$
is {\em stable\/} if, given an object $C$ of 
$\mathpzc{C}_Q$ and $x \in Q$, there is an equivalence $f$ in 
$\mathpzc{C}_Q$ with domain $C$ and $g(f)=x\in Q$
\cite[\S 1 p.~3]{genbrauer},  \cite[\S 3 p.~240]{MR0349804}, 
\cite[\S 1]{MR1803361}.

We  usually suppress the functor $g$ from the notation
unless it is convenient to spell it out for clarity.
Given a $Q$-graded category $\mathpzc C_Q$, the notation 
$\mathpzc{Ker}(\mathpzc C_Q)$ refers to the category having the same
objects as $\mathpzc C_Q$ but whose morphisms are only those of grade 1
\cite[\S 1 p.~2]{genbrauer}, \cite[\S 3 p.~240]{MR0349804}, 
\cite[\S 1]{MR1803361}.
Below always indicate the fact that a
$Q$-graded category is under discussion by the subscript ${-}_Q$.
Given a $Q$-graded category $\mathpzc C_Q$,
we  then use the notation $\mathpzc C$ for 
$\mathpzc{Ker}(\mathpzc C_Q)$.

Given a stably $Q$-graded category $\mathpzc{C}_Q$, the group $Q$ acts 
on $k\mathpzc{C}= k\mathpzc{Ker}(\mathpzc{C}_Q)$ 
as follows \cite[\S 1 p.~3]{genbrauer}, \cite[Lemma 3.1 p.~240]{MR0349804},
\cite[\S 1]{MR1803361}:
Given an object $C$ of $\mathpzc{C}= \mathpzc{Ker}(\mathpzc{C}_Q)$ 
and $x \in Q$, 
keeping in mind that $\mathpzc{Ker}(\mathpzc{C}_Q)$ 
and $\mathpzc C_Q$ have the same objects, 
choose a morphism
$f\colon C \to D$ in $\mathpzc{C}_Q$ of grade $x$ and define
the result of the action ${}^x[C]$ of $x$ on $[C]\in k\mathpzc{C}$
by
${}^x[C]=[D]\in k\mathpzc{C}$.
This action is well defined, cf. 
\cite[\S 3 pp.~240/41]{MR0349804}.
Denote $k\mathpzc{C}$, endowed with this $Q$-action, by 
$k_Q\mathpzc C_Q$, cf. the notation $k_{\Gamma}$ in 
\cite[Lemma 3.1 p.~240]{MR0349804}.
The association  $\mathpzc C_Q \mapsto k_Q\mathpzc C_Q$, as
$\mathpzc C_Q$ ranges over 
stably $Q$-graded symmetric monoidal categories,
defines a functor 
$k_Q$
from the category of
stably $Q$-graded symmetric monoidal categories
to the category 
$Q\mathpzc{Set}$ of $Q$-sets, cf.
 \cite[Lemma 3.1 p.~240]{MR0349804}.

A $Q$-{\em functor\/} 
\cite[\S 3 p.~241] {MR0349804}
is one over the identity map of $Q$ or, equivalently,
a functor preserving grades of morphisms.
Below {\em natural transformation of \/} $Q$-{\em functors\/} will
be required to be of grade $1$.
The {\em category\/} $\mathpzc{Rep}(Q,\mathpzc C_Q)$
of {\em representations of\/} $Q$ {\em in a\/} 
$Q$-graded category
$(\mathpzc C_Q,g)$
is the category of $Q$-functors $F\colon Q \to \mathpzc C_Q$,
that is,
functors $F$ from $Q$ to $\mathpzc C_Q$
such that the composite $g \circ F$ 
of $F$ with the grade functor $g$
is the identity functor on $Q$,
and natural transformations of grade $1$
 \cite[\S 1 p.~2]{genbrauer}, 
\cite[Introduction, \S 3 p.~241] {MR0349804}, 
\cite[\S 1]{MR1803361}.
Thus
an object of 
$\mathpzc{Rep}(Q,\mathpzc C_Q)$ 
is a representation $h\colon Q\to \Aut_{\mathpzc C_Q}(A)$,
in the category $\mathpzc C_Q$,
of $Q$ on an object $A$ of the category 
$\mathpzc C=\mathpzc{Ker}(\mathpzc C_Q)$.

Given two $Q$-graded categories $(\mathpzc C_Q,g)$ and  $(\mathpzc C'_Q,g')$,
write  $\mathpzc C_Q \times_Q \mathpzc C'_Q$ for the {\em pull back
category\/} of $(g,g')$. This pull back category
acquires an obvious $Q$-grading, and this grading is stable
if $g$ and $g'$ are. The pull back category yields the {\em product\/}
in the category of $Q$-graded categories.

\subsection{Stably graded symmetric monoidal categories}
\label{stablygraded}

A {\em stably\/} $Q$-{\em graded symmetric monoidal category\/}
\cite[\S 3 p.~240] {MR0349804}, 
\cite[\S 1]{MR1803361}
consists of a
stably $Q$-graded category
$(\mathpzc C_Q,g)$, a covariant $Q$-functor 
$\odot\colon \mathpzc C_Q \times_Q \mathpzc C_Q
\to \mathpzc C_Q$, that is, 
the composite $h_1 \odot h_2$ of two morphisms $h_1$ and $h_2$
is defined only when $h_1$ and $h_2$ have the same grade, and
$g(h_1 \odot h_2)=g(h_1)=g(h_2)$,
a covariant $Q$-functor $E\colon Q \to \mathpzc C_Q$, referred to as
the {\em unit object\/} of $\mathpzc C_Q$,
 and  grade 1 natural 
equivalences 
$e\colon E \compo \,\cdot\,  \to   \,\cdot\,$,
$a\colon   \,\cdot\,\compo\, (\cdot\, \compo \,\cdot\,)  \to   
 (\,\cdot\,\compo\, \cdot\,) \compo \,\cdot\,  $, and 
$c_{A,B}\colon A\compo B \to B \compo A  $,
where $A$ and $B$ range over objects of $\mathpzc C_Q$;
these equivalences are subject to the standard axioms.

Let $\mathpzc C_Q$ be a stably $Q$-graded symmetric monoidal category.
The category
$\mathpzc C =\mathpzc {Ker}(\mathpzc C_Q)$
acquires an obvious symmetric monoidal category structure having,
in particular, $E(e)$ as its unit object.
(N.B.: $E$ is a $Q$-functor, 
and $e\in Q$ refers to the neutral element of $Q$.)
By definition, the {\em unit group\/} $\mathrm U(\mathpzc C_Q)$
of $\mathpzc C_Q$ 
is the (abelian) unit group 
\[
\mathrm U(\mathpzc C) =\mathrm U(\mathpzc {Ker}(\mathpzc C_Q))
=\Aut_{\mathpzc C}(E(e))
\]
of the category $\mathpzc C =\mathpzc {Ker}(\mathpzc C_Q)$
\cite[(2.4) p.~8]{genbrauer}, 
\cite[\S 3 p.~241] {MR0349804}, 
\cite[\S 1]{MR1803361}.

The grading induces a surjective homomorphism $\Aut_{\mathpzc C_Q}(E(e))\to Q$
which fits into a group extension
\begin{equation}
\mathrm e_{\mathcat C_Q}^{\mathrm U(\mathpzc C)}\colon
1
\longrightarrow
\mathrm U(\mathpzc C)
\longrightarrow
\Aut_{\mathpzc C_Q}(E(e))
\longrightarrow
Q
\longrightarrow
1, 
\label{eECC}
\end{equation}
and $E \colon Q \to \Aut_{\mathpzc C_Q}(E(e))$ 
splits this group extension. Thus, via the $Q$-functor $E$, 
the group $\Aut_{\mathpzc C_Q}(E(e))$ decomposes as the semi-direct product
\begin{equation}
\Aut_{\mathpzc C_Q}(E(e)) = \mathrm U(\mathpzc C) \rtimes Q,
\end{equation}
and this decomposition induces a $Q$-module structure on
$\mathrm U(\mathpzc C)$
\cite[\S 2 p.~8]{genbrauer},
 \cite[\S 3 p.~241]{MR0349804}, \cite[\S 1]{MR1803361}.
We  use the notation
$\mathrm U(\mathpzc C_Q)$ for
$\mathrm U(\mathpzc C)$, endowed with the $Q$-module structure just explained.

Given an object $C$ of $\mathpzc C_Q$, the 
group $\Aut_{\mathpzc C_Q}(C)$ of automorphisms of $C$ in
$\mathpzc C_Q$ has the 
subgroup of grade $1$ automorphisms as a normal subgroup,
this subgroup is canonically isomorphic
to the group $\Aut_{\mathpzc C}(C)$ of automorphisms
of $C$ in $\mathpzc C= \mathpzc {Ker}(\mathpzc C_Q)$, and
the above homomorphism 
$\vartheta_C\colon \mathrm U(\mathpzc C)
\longrightarrow
\Aut_{\mathpzc C}(C)$ is available,  cf. \eqref{varthetaC}.
An object $C$ of $\mathpzc C_Q$ is
{\em faithful\/} if it is faithful as on object of
$\mathpzc C= \mathpzc {Ker}(\mathpzc C_Q)$ or, equivalently, if,
with a slight abuse of the notation $\vartheta_C$,
the induced homomorphism
\begin{equation}
\vartheta_C\colon \mathrm U(\mathpzc C)
\longrightarrow
\Aut_{\mathpzc C_Q}(C)
\label{varthetaqCC}
\end{equation}
is injective.

The category $\mathpzc C_Q$ being a 
stably $Q$-graded symmetric monoidal category,
the values of the functor $k_Q$ now lie in the category of
$Q$-monoids (monoids endowed with a $Q$-action that is compatible with
the monoid structure).
Given an  object $C$ of $\mathpzc C_Q$, the grade homomorphism
$\Aut_{\mathpzc C_Q}(C) \to Q$
is surjective if and only if
the class $[C]\in 
k_Q\mathpzc C_Q$ is fixed under $Q$.
Hence
the group $\Aut_{\mathpzc C_Q}(C)$ of 
automorphisms in $\mathpzc C_Q$ of
a faithful {\em invertible\/} object $C$ of $\mathpzc C_Q$
whose class $[C]\in 
k_Q\mathpzc C_Q$ 
is fixed under $Q$ fits into a group extension
\begin{equation}
\mathrm e_C^{\mathrm U(\mathpzc C)}\colon
1
\longrightarrow
\mathrm U(\mathpzc C)
\stackrel{\vartheta_C}
\longrightarrow
\Aut_{\mathpzc C_Q}(C)
\longrightarrow
Q
\longrightarrow
1.
\label{eUCC}
\end{equation}

The category $\mathpzc{Rep}(Q,\mathpzc C_Q)$ 
of representations of $Q$ in $\mathpzc C_Q$ acquires 
an obvious
symmetric monoidal category structure.
In particular, the constituent $E$
in the definition of a stably $Q$-graded symmetric monoidal
category $\mathpzc C_Q$ is the unit object of 
$\mathpzc{Rep}(Q,\mathpzc C_Q)$
 \cite[\S 3 p.~241]{MR0349804}, and
$\mathrm U(\mathpzc {Rep}(Q, \mathpzc C_Q))\cong 
\mathrm H^0(Q,\mathrm U(\mathpzc C_Q))$
\cite[(3.5) p.~242]{MR0349804}.
By definition, a member of $\mathpzc{Rep}(Q,\mathpzc C_Q)$, that is,
a representation $F\colon Q \to\Aut_{\mathpzc C_Q}(C)$
of $Q$ by automorphisms in $\mathpzc C_Q$ 
of an object $C$ of $\mathpzc C_Q$, combined with
the grade homomorphism $\Aut_{\mathpzc C_Q}(C) \to Q$,
yields the identity map of $Q$ whence the class $[C]\in k_Q\mathpzc C_Q$
is fixed under $Q$; if furthermore, $C$ is faithful, the representation
$F$ splits the associated group extension \eqref{eUCC}.

The notion of inverse 
and the definition of
$\mathpzc C_Q$ being {\em group-like\/} extend to
stably $Q$-graded symmetric monoidal categories in an obvious way.
When the category $\mathpzc C_Q$ is group-like
and has every object faithful,
 the assignment to
an object $C$  of $\mathpzc C_Q$
of the group extension \eqref{eUCC}
induces a homomorphism
\begin{equation}
\omega_{\mathpzc C_Q}\colon 
\mathrm H^0(Q,k_Q\mathpzc C_Q)
\longrightarrow
\mathrm H^2(Q,\mathrm U(\mathpzc C))
\label{omegac}
\end{equation}
of abelian groups.
By  \cite[Lemma 3.2 p.~241]{MR0349804}, if 
the given stably $Q$-graded symmetric monoidal category
$\mathpzc C_Q$ is group-like, so is the category
$\mathpzc {Rep}(Q, \mathpzc C_Q)$. 
In particular, in the proof of
\cite[Lemma 3.2 p.~241]{MR0349804},
an explicit construction is given for the
inverse in the category $\mathpzc {Rep}(Q, \mathpzc C_Q)$
associated with a group-like
stably $Q$-graded symmetric monoidal category
$\mathpzc C_Q$.

The forgetful functor
$\mathpzc {Rep}(Q, \mathpzc C_Q) \to \mathpzc {Rep}(\{e\}, \mathpzc C_Q) 
\cong \mathpzc C_Q$
induces a monoid homomorphism 
$\mumu_{\mathpzc C_Q}\colon k\mathpzc {Rep}(Q, \mathpzc C_Q) 
\to k_Q\mathpzc C_Q$
whose values lie in 
$\mathrm H^0(Q,k_Q\mathpzc C_Q)$ since,
given an object $F$ of  $\mathpzc {Rep}(Q, \mathpzc C_Q)$,
that is, a 
representation 
$F \colon Q \to \Aut_{\mathpzc C_Q}(C)$
of $Q$ on an object
$C$ of $\mathpzc C_Q$, 
the existence of the homomorphism $F$
plainly entails that the grade homomorphism from 
$\Aut_{\mathpzc C_Q}(C)$ to $Q$ is surjective; in fact, when $C$
is a faithful invertible object, $F$ splits the group extension
\eqref{eUCC}.

Let $\mathpzc E$
denote the full symmetric monoidal subcategory
of $\mathpzc C=\mathpzc {Ker}(\mathpzc C_Q)$
that has $E(e)$ as its single object.
This category is isomorphic to the abelian group 
$\Aut_{\mathpzc C}(E(e))\cong \mathrm U(\mathpzc C)$,
viewed as a category with a single object.
Let $\mathpzc E_Q$ 
denote the associated stably $Q$-graded symmetric monoidal category
or, equivalently, the stably $Q$-graded symmetric monoidal subcategory
of $\mathpzc C_Q$ having the single object $E(e)$; this category is
isomorphic to the group 
$\Aut_{\mathpzc C_Q}(E(e))\cong \mathrm U(\mathpzc C) \rtimes Q$, viewed as a category
with a single object.
The standard interpretation of $\mathrm H^1(Q, \mathrm U(\mathpzc C))$
as classes of sections of
\eqref{eECC}, two sections being identified 
whenever they differ by conjugation
in $\Aut_{\mathpzc C_Q}(E(e))$ by a member of
$\mathrm U(\mathpzc C)$,
induces a canonical isomorphism
\begin{equation}
\mathrm H^1(Q,\mathrm U(\mathpzc C)) \longrightarrow
\mathpzc {Rep}(Q,\mathpzc E_Q)
\label{inter1}
\end{equation}
of abelian monoids whence, since
$\mathrm H^1(Q,\mathrm U(\mathpzc C))$ is an abelian group, so is
$\mathpzc {Rep}(Q,\mathpzc E_Q)$ \cite[(4.2) p.~243]{MR0349804}.
The obvious injection
$k\mathpzc {Rep}(Q,\mathpzc E_Q) \to k\mathpzc {Rep}(Q,\mathpzc C_Q)$ 
of abelian monoids yields an injection
\begin{equation}
j_{\mathpzc C_Q}\colon \mathrm H^1(Q,\mathrm U(\mathpzc C_Q))
\longrightarrow
k\mathpzc {Rep}(Q,\mathpzc C_Q)
\label{inj1}
\end{equation}
of abelian monoids.
When $\mathpzc C_Q$ is group-like
and has
all objects faithful, the sequence
\begin{equation}
0
\longrightarrow
\mathrm H^1(Q,\mathrm U(\mathpzc C_Q))
\stackrel{j_{\mathpzc C_Q}}\longrightarrow
k \mathpzc{Rep}(Q,\mathpzc C_Q)
\stackrel{\mumu_{\mathpzc C_Q}}
\longrightarrow
\mathrm H^0(Q,k_Q \mathpzc C_Q)
\stackrel{\omega_{\mathpzc C_Q}}
\longrightarrow
\mathrm H^2(Q,\mathrm U(\mathpzc C_Q)) 
\label{fittwelve}
\end{equation}
is an exact sequence of abelian groups
\cite[Theorem 5 \S 6 p.~48]{genbrauer}, 
\cite[Theorem 4.5 p.~244]{MR0349804}, \cite[\S 1]{MR1803361}.

\subsection{The  stably $Q$-graded Brauer category  
associated with a
commutative ring endowed with a $Q$-action}
\label{sqgbc}

Given two Azumaya algebras $A$ and $B$, a $(B,A)$-{\em bimodule
grade\/} $x \in Q$ is a $(B,A)$-bimodule $M$ so that
the left $S$-module structure of $M$ via $S \to B$ and the right
$S$-module structure via $S \to A$ are connected by the identity
\begin{equation*}
y s= ({}^xs) y,\ y \in M,\ s \in S.
\end{equation*} 
The  $Q$-{\em graded Brauer category\/} 
$\mathpzc B_{S,Q}$ associated with
the commutative ring $S$ and the $Q$-action $\kappaQ\colon Q \to \Aut(S)$
on $S$
\cite[\S 3 p.~23]{genbrauer}, \cite[p.~230]{MR0349804}, \cite[\S 2]{MR1803361},
written in
\cite[\S 3 p.~23]{genbrauer} and \cite[\S 2]{MR1803361}
as $\mathpzc {B}_R$ 
and in \cite[p.~230]{MR0349804} as $\mathpzc {Br}_R$,
has as
 {\em objects\/} the
Azumaya $S$-algebras,
 a {\em morphism\/} $([M],x)\colon A \to B$ in 
$\mathpzc B_{S,Q}$ {\em of grade\/} 
$x\in Q$ between two 
Azumaya algebras $A$ and $B$,
necessarily an isomorphism in $\mathpzc B_{S,Q}$, 
being a pair
$([M],x)$ where $[M]$ is an isomorphism class of an invertible
$(B,A)$-bimodule $M$ of grade  $x\in Q$.
Given three Azumaya algebras $A$, $B$, $C$ and morphisms
 $([{}_BM_A],x)\colon A \to B$
and $([{}_AM_C],x)\colon C\to A$
in $\mathpzc B_{S,Q}$, the
composite of 
 $([{}_BM_A],x)\colon A \to B$ with 
 $([{}_AM_C],x)\colon C\to A$ in $\mathpzc B_{S,Q}$ 
is given by the morphism  
$([{}_BM_A\otimes_A {}_AM_C],x)\colon C\to B$  
of grade $x\in Q$ in $\mathpzc B_{S,Q}$, 
where $[{}_BM_A\otimes_A {}_AM_C]$ 
refers to the isomorphism class
of the invertible
$(B,C)$-bimodule ${}_BM_A\otimes_A {}_AM_C$.
The operation of tensor product over the ground ring $S$ and the assignment
to an Azumaya $S$-algebra $A$ of its opposite algebra $A^{\mathrm{op}}$
turn $\mathcat B_{S,Q}$ 
into a group-like symmetric monoidal category
\cite[\S 5 p.~247]{MR0349804}, \cite[\S 2]{MR1803361}
having $(S,\kappaQ\colon Q \to \Aut(S))$ as unit object and
\[
\mathrm U(\mathpzc B_{S,Q}) = 
\mathrm U(\mathpzc B_S) = 
\Aut_{\mathpzc B_S}(S)
=\Pic(S)
\] 
\cite[\S 3]{MR1803361},
the assignment to an object $A$ of $\mathcat B_{S,Q}$ of the neutral 
element $e$ of $Q$ and that to a morphism in $\mathcat B_{S,Q}$
of its grade in $Q$ yields a functor $g$ from  $\mathcat B_{S,Q}$ to $Q$,
viewed as a category with a single object, i.~e., 
turn $\mathcat B_{S,Q}$ into a $Q$-graded 
symmetric monoidal
category, and the grading is stable.
In particular, the induced $Q$-action on
 $\mathrm U(\mathpzc B_{S,Q}) = \Pic(S)$
is the standard $Q$-action on $\Pic(S)$.
Since the category 
$\mathpzc B_{S,Q}$
is group-like, so is 
$\mathpzc {Rep}(Q, \mathpzc B_{S,Q})$,
and thence 
$k\mathpzc {Rep}(Q, \mathpzc B_{S,Q})$
is an abelian group. However,
apart from trivial cases, this group is {\em not\/}
the equivariant Brauer group of $S$ relative to $Q$,
this group being defined as the obvious equivariant generalization of the
ordinary Brauer group and
explored
in Section \ref{eleven} below.

By construction, then, 
the assignment to an automorphism in $\mathcat B_{S,Q}$
of an Azumaya algebra $A$  of its grade in $Q$
yields a homomorphism 
\begin{equation}
\pi^{\substack{\mbox{\tiny{$\Aut_{\mathcat B_{S,Q}}(A)$}}}}
\colon \Aut_{\mathcat B_{S,Q}}(A) \longrightarrow Q
\label{eBpi}
\end{equation}
which
is surjective if and only if the Brauer class
$[A]\in \mathrm B(S)$ of $A$ in $\mathrm B(S)$ is fixed under $Q$, and
the group $\Aut_{\mathcat B_{S,Q}}(A)$ associated to 
an Azumaya $S$-algebra 
$A$ whose Brauer class $[A]$ is fixed under $Q$ fits into a group 
extension of the kind \eqref{eUCC}, viz.
\begin{equation}
\mathrm e^{\substack{\mbox{\tiny{$\Pic(S)$}}}}_A\colon
1
\longrightarrow
\Pic (S)
\longrightarrow
\Aut_{\mathcat B_{S,Q}}(A)
\stackrel{\pi^{\substack{\mbox{\tiny{$\Aut_{\mathcat B_{S,Q}}(A)$}}}}}
\longrightarrow
Q
\longrightarrow
1
\label{eApic}
\end{equation}
with abelian kernel in such a way that the assignment to $A$ of 
$\mathrm e^{\substack{\mbox{\tiny{$\Pic(S)$}}}}_A$
yields a homomorphism 
\begin{equation}
\ome_{\mathcat B_{S,Q}}\colon \mathrm H^0(Q,\mathrm B(S)) \longrightarrow \mathrm H^2(Q,\Pic(S)).
\label{d2pic}
\end{equation}
The sequence \eqref{fittwelve} now takes the  form
\begin{equation}
0 \longrightarrow \mathrm H^1(Q, \Pic (S)) 
\stackrel{j_{\mathpzc B_{S,Q}}}\longrightarrow 
k\mathcat{Rep}(Q,\mathcat B_{S,Q})
\stackrel{\omu_{\mathpzc B_{S,Q}}}\longrightarrow \mathrm B(S)^Q
\stackrel{\ome_{\mathpzc B_{S,Q}}}
\longrightarrow \mathrm H^2(Q, \Pic (S)) 
\label{FW}
\end{equation}
and 
is an exact sequence of abelian groups 
since the category $\mathcat B_{S,Q}$ is group-like. 
This sequence is spelled out in
\cite[Corollary 1 p.~51]{genbrauer} as a sequence of abelian monoids,
where the notation 
$k\mathcat{Rep} \tilde {\mathrm B}_S $
corresponds to our notation
$k\mathcat{Rep}(Q,\mathcat B_{S,Q})$.
It is also a special case of the exact sequence of abelian groups
given in
\cite[Corollary (4.6) p.~245] {MR0349804}.
The tilde notation in the setting of \cite{genbrauer} 
refers to the additional structure of an involution
which, however, is not present in our
approach. 

The following is immediate; we spell it out for later reference.
\begin{prop}
\label{isomsq}
Given two Azumaya $S$-algebras $A$ and $B$ and an invertible
$(B,A)$-bimodule $M$ so that the isomorphism class of $M$
yields an isomorphism $A \to B$ in $\mathcat B_S$,
the induced
isomorphism
$\Aut_{\mathcat B_{S,Q}}(A) \to \Aut_{\mathcat B_{S,Q}}(B)$ 
of groups 
is given by the assignment to an $(A,A)$-bimodule $N_x$ of grade $x\in Q$
of the $(B,B)$-bimodule 
\begin{equation*}
M\otimes_A N_x \otimes_A M^* 
\cong \Hom_A(M,M\otimes_A N_x)
\end{equation*}
of grade $x$. \qed
\end{prop}

\subsection{Picard categories}
\label{piccat}
The {\em Picard category\/}  $\mathpzc{Pic}_S$ 
associated with the commutative ring $S$, written in 
\cite[\S 2 p.~17]{genbrauer}, \cite[\S 2]{MR1803361} as 
$\mathpzc{C}_R$,
has as
objects the faithful finitely generated invertible projective
$S$-modules, that is, the faithful finitely generated projective rank one
$S$-modules,
a morphism in $\mathpzc{Pic}_S$ an isomorphism 
between two $S$-modules
in  $\mathpzc{Pic}_S$.
Given three  faithful finitely generated projective rank one
$S$-modules and two morphisms between them,
composition 
in  $\mathpzc{Pic}_S$
is defined in the obvious way, that is, via 
ordinary composition of $S$-linear maps.
The operation of tensor product over the ground ring $S$
and the assignment to a faithful finitely generated projective rank one
$S$-module of its $S$-dual turn
$\mathpzc{Pic}_S$ into
a group-like symmetric monoidal category
having the ground ring $S$ ,
viewed as a free rank one $S$-module,
as its unit object and
$\mathrm U(\mathpzc{Pic}_S)=\Aut_{\mathpzc{Pic}_S}(S)= \mathrm U(S)$, 
the group of units of the ground ring $S$.
The abelian group $k\mathpzc{Pic}_S$ is canonically isomorphic to
the ordinary Picard group $\Pic(S)$ of $S$.

The  $Q$-{\em graded Picard category\/} 
$\mathpzc {Pic}_{S,Q}$ associated with
the commutative ring $S$ and the $Q$-action $\kappaQ\colon Q \to \Aut(S)$
on $S$, written in \cite[\S 3]{MR1803361} as
$\mathpzc{C}_{R}$,
has the same objects as $\mathpzc {Pic}_S$,
 a {\em morphism\/} $(f,x)\colon J_1 \to J_2$ in 
$\mathpzc {Pic}_{S,Q}$ {\em of grade\/} 
$x\in Q$ between two 
faithful projective rank $S$-modules
$J_1$ and $J_2$,
necessarily an isomorphism in $\mathpzc {Pic}_{S,Q}$, 
being a pair
$(f,x)$ where $f\colon J_1 \to J_2$ is an isomorphism 
over $R=S^Q$ such that
$f(sy)={}^x\!sf(y)$, for $s \in S$ and $y \in J_1$.
Given three  faithful finitely generated projective rank one
$S$-modules and two morphisms between them,
composition 
in  $\mathpzc{Pic}_{S,Q}$
is defined in the obvious way, that is, via 
ordinary composition of $S$-linear maps.
The operation of tensor product over the ground ring $S$
and the assignment to a faithful finitely generated projective rank one
$S$-module of its $S$-dual turn
$\mathpzc{Pic}_{S,Q}$ into
a group-like  symmetric monoidal category
having $(S,\kappaQ\colon Q \to \Aut(S))$ as unit object, and
$\mathrm U(\mathpzc{Pic}_{S,Q})=\mathrm U(\mathpzc{Pic}_S)
=\Aut_{\mathpzc{Pic}_S}(S)= \mathrm U(S)$, 
the group of units of the ground ring $S$.
The assignment to an object of 
$\mathpzc{Pic}_{S,Q}$ 
of the neutral element $e$ of $Q$ and that to a morphism of its
grade in $Q$ turn
$\mathpzc{Pic}_{S,Q}$ into a $Q$-graded symmetric monoidal category,
and the grading is stable.
In particular, the induced $Q$-action on
 $\mathrm U(\mathpzc{Pic}_{S,Q}) = \mathrm U(S)$
is the standard $Q$-action on $\mathrm U(S)$,
 endowed with its
induced $Q$-module structure. 
Since the category $\mathpzc {Pic}_{S,Q}$ is group-like,
so is the category $\mathpzc {Rep}(Q, \mathpzc {Pic}_{S,Q})$,
and thence 
$k\mathpzc {Rep}(Q, \mathpzc {Pic}_{S,Q})$
is an abelian group.
This group is canonically isomorphic to
the equivariant Picard group $\mathrm{EPic}(S,Q)$ of $S$
with respect to the $Q$-action $\kappaQ\colon Q \to \Aut(S)$
\cite[\S 5 p.~38]{genbrauer}, \cite{MR0409424},
\cite[\S 3]{MR1803361} (written as $C(R,\Gamma)$
and referred to as the equivariant class group). 
The sequence \eqref{fittwelve} now takes the  form
\begin{equation}
0 \longrightarrow \mathrm H^1(Q, \mathrm U (S)) 
\stackrel{\jp}
\longrightarrow 
\mathrm{EPic}(S,Q)
\stackrel{\mup}
\longrightarrow \Pic(S)^Q
\stackrel{\dDelta}
\longrightarrow \mathrm H^2(Q, \mathrm U (S)) 
\label{ldes}
\end{equation}
and 
is an exact sequence of abelian groups 
since the category $\mathcat {Pic}_{S,Q}$ is group-like. 
This construction recovers 
the classical four-term exact sequence associated with the data;
this sequence can, of course, be obtained by straightforward ad
hoc constructions. For example,
the $Q$-action $\kappaQ\colon Q \to \Aut(S)$ splits
the group extension 
\begin{equation}
\mathrm e_{\substack{\mbox{\tiny{${\mathcat {Pic}_{S,Q}}$}}}}
\colon
0
\longrightarrow
\mathrm U(S) \longrightarrow 
\Aut_{\mathpzc {Pic}_{S,Q}}(S)
\stackrel{\pi^{\substack{\mbox{\tiny{$\Aut_{\mathcat {Pic}_{S,Q}}(A)$}}}}}
\longrightarrow
Q
\longrightarrow 1,
\end{equation}
and the
homomorphism $j_{\mathpzc {Pic}_{S,Q}}$ is induced
by the assignment to a derivation $d\colon Q \to \mathrm U(S)$
of the associated section 
for $\pi^{\substack{\mbox{\tiny{$\Aut_{\mathcat {Pic}_{S,Q}}(A)$}}}}$.
For later reference we note that the exactness of \eqref{ldes} 
at $\mathrm{EPic}(S,Q)$
says that
\begin{equation}
j_{\mathpzc {Pic}_{S,Q}} \colon 
\mathrm H^1(Q, \mathrm U (S))  \longrightarrow 
\mathrm{EPic}(S|S,Q)
\label{picf}
\end{equation}
is an isomorphism from $\mathrm H^1(Q, \mathrm U (S))$
onto the subgroup $\mathrm{EPic}(S|S,Q)$
of $\mathrm{EPic}(S,Q)$ which consists of classes of
objects in 
$\mathpzc {Rep}(Q, \mathpzc {Pic}_{S,Q})$
whose underlying $S$-modules are free of rank 1.

\subsection{Change of actions}
\label{coa}
We define the {\em change of actions category\/} $\mathcat {Change}$
as follows:
The {\em objects} of $\mathcat{Change}$ 
are triples $(S,Q,\kappa)$ that consist of a commutative
ring $S$,
a group $Q$, and an action 
$\kappa \colon Q\to \Aut(S)$ of $Q$ on $S$;
given two objects $(S,Q,\kappaQ)$ and $(T,G,\llambda )$, a 
{\em morphism}
$
(f,\varphi ) \colon (S,Q, \kappa) \longrightarrow (T,G,\llambda )
$
in $\mathcat{Change}$ consists of a ring homomorphism $f \colon S \to T$ and a group 
homomorphism $\varphi \colon G \to Q$ such that, given $s\in S$ and $x\in G$,
\begin{equation}
f( {}^{\varphi (x)}s) = {}^x(f(s)).
\label{changeaxiom}
\end{equation}
In this category, composition of morphisms is defined in the obvious way.

To describe the morphisms in a somewhat more picturesque way,
given the ring homomorphism $f \colon S \to T$, let
$\Aut^S(T)$ denote the subgroup of $\Aut(S)\times \Aut(T)$
that consists of those pairs $(\alpha,\beta)$ of automorphisms 
which have the property
 that
$\beta \circ f = f \circ \alpha$, and
let $\Aut(T|S)$ denote the kernel of
the obvious homomorphism $\Aut^S(T) \to \Aut(S)$.
In the special case where $f$ is injective,
the group $\Aut(T|S)$ amounts to the ordinary group
of automorphisms of $T$ that leave $S$ elementwise fixed
whence the notation. In the general case,
given, furthermore, the homomorphism $\varphi\colon G \to Q$,
the condition \eqref{changeaxiom} says that
$(\kappaQ\circ \varphi,\llambda)\colon G \to \Aut^S(T)$
yields a commutative diagram
\begin{equation}
\begin{CD}
1
@>>>
\mathrm{ker}(\varphi)
@>>>
G
@>{\varphi}>>
Q
\\
@.
@VVV
@V{(\kappaQ\circ \varphi,\llambda)}VV
@V{\kappaQ}VV
\\
1
@>>>
\Aut(T|S)
@>>>
\Aut^S(T) 
@>>>
\Aut(S)
\label{changea2}
\end{CD}
\end{equation}
in the category of groups with exact rows.

The stably $Q$-graded categories $\mathcat B_{S,Q}$
and $\mathcat {Pic}_{S,Q}$ behave functorially on $\mathcat{Change}$
in an obvious way with respect to the variable $(S,Q)$.
Likewise,
group cohomology $\mathrm H^*(Q,\mathrm U(S))$
 is a covariant functor on the change of actions category $\mathcat{Change}$
with respect to the variable $(S,Q)$,
and so are $k\mathcat{Rep}(Q,\mathcat B_{S,Q})$ and
$\mathrm {EPic}(S,Q)$.
More examples will show up later.

\section{Normal algebras and their Teich\-m\"uller complexes}
\label{2}
\subsection{Normal algebras}
\label{2.1}
Let $A$ be a central $S$-algebra. Denote by $\Aut(A)$ the group of ring
automorphisms
of $A$ and by $\mathrm U(A)$ the group of units of $A$. The 
obvious homomorphism
$\partial \colon \mathrm U(A) \to \Aut (A)$
assigns to a unit of $A$ the associated inner automorphism of $A$, 
and the obvious action of $\Aut(A)$
on $\mathrm U(A)$ turns the triple $(\mathrm U(A),\Aut(A), \partial )$ into a crossed module.
Moreover, $\ker (\partial) = \mathrm U(S)$, the group of units of $S$, and
 $\partial (\mathrm U(A))$ is a normal subgroup of $\Aut(A)$. 
As in the introduction,
write 
$\Out(A) = \mathrm{coker} (\partial) $.

Each inner automorphism of $A$ leaves $S$ elementwise fixed whence
 the  
restriction map $\Aut(A) \to \Aut(S)$ induces a homomorphism 
${\Out(A) \to \Aut(S)}$, and the data fit into the crossed 2-fold extension
\begin{equation}
\mathrm e_A \colon 0 \longrightarrow \mathrm U(S) \longrightarrow \mathrm U(A) \stackrel{\partial}\longrightarrow \Aut(A) \longrightarrow \Out(A) 
\longrightarrow 1.
\label{eA}
\end{equation}
As in the introduction, let 
$Q$ be a group and
$\kappaQ \colon Q\to \Aut(S)$ an action of  $Q$ on
$S$ by ring automorphisms.
With respect to
$\kappaQ \colon Q \to \Aut(S)$, let
\[
\Aut (A,Q)=\Aut(A) \times_{\Aut(S)} Q,\ \Out (A,Q)=\Out(A) \times_{\Aut(S)} Q
\] 
denote the indicated fiber product groups. 
The group $\Aut (A,Q)$ acts on $\mathrm U(A)$ in the
obvious way,  and this action, together with
the obvious map 
$\partial\colon  \mathrm U(A) \longrightarrow \Aut(A,Q)$, yields 
a crossed 2-fold extension
\begin{equation}
{\mathrm e}_{(A,Q)} \colon 0 \longrightarrow \mathrm U(S) \longrightarrow \mathrm U(A) \stackrel{\partial}\longrightarrow \Aut(A,Q) 
\longrightarrow \Out(A,Q) \longrightarrow 1.
\label{eAQ}
\end{equation}
 As in the introduction, we define a $Q$-{\em normal structure\/} on 
the central $S$-algebra $A$ 
{\em relative to the action 
$\kappaQ \colon Q \to \Aut(S)$
of $Q$ on\/} $S$
to be a
homomorphism $\sigma\colon Q \to \Out(A)$ that lifts the action 
$\kappaQ \colon Q \to \Aut(S)$ 
of $Q$ on $S$
in the sense that the composite of
$\sigma$ with 
the obvious restriction map $\res\colon \Out(A) \to \Aut(S)$ coincides with $\kappaQ$.
A $Q$-{\em normal algebra\/} is, then, a central $S$-algebra
$A$ together with a $Q$-normal structure $\sigma\colon Q \to \Out(A)$.

Occasionally we refer to the canonical homomorphism
$\Aut(A,Q) \to Q$ as well as to
the canonical homomorphism
$\Out(A,Q) \to Q$ as a {\em grade homomorphism\/}
and to the value in $Q$ of a member of
$\Aut(A,Q)$ and, likewise, of a member of
$\Out(A,Q)$, as the {\em grade\/} of that member.

\subsection{Discussion of normality and generalized normality}
\label{disc}

Given a central $S$-algebra $A$, we  say that $A$ is {\em weakly 
$Q$-normal\/}
if each automorphism $\kappaQ(x)$ of $S$, as $x$ ranges over $Q$,
extends to a ring automorphism of $A$.
For Azumaya algebras,
this is the notion of $Q$-normality
used by
Childs~\cite{MR0311701}, Pareigis~\cite{MR0161881},
Ulbrich~\cite{MR1011607}, \cite{ MR1284782}, 
and Zelinski~\cite{MR0432621}. Equivalently, a central $S$-algebra $A$
is weakly $Q$-normal if and only if the canonical homomorphism
$\pi^{\substack{\mbox{\tiny{$\Out(A,Q)$}}}}\colon
\Out(A,Q) \to Q$ is surjective.

Let $A$ be an Azumaya $S$-algebra and let $\lambda$ be automorphisms of $A$
over the center $S$ of $A$. Let ${}_{\lambda}A$ be the
$(A\otimes A^{\mathrm{op}})$-module 
which, as an $S$-module, is just $A$, and
whose structure map
is given by
\begin{equation*}
(A\otimes A^{\mathrm{op}}) \otimes {}_{\lambda}A 
\longrightarrow {}_{\lambda},
\ 
(a \otimes b)\otimes y \longmapsto \lambda(a)y b,\ a,b,y \in A.
\end{equation*} 
Then 
\[
 {}_{\lambda}J=\Hom_{A^{\mathrm e}}(A,{}_{\lambda}A)\cong
\{a \in A; \lambda(x) a = ax,\  \text{for\ all}\  x \in A\}
\]
is a faithful projective rank one $S$-module in such a way that
the canonical evaluation map
\begin{equation*}
\Hom_{A^{\mathrm e}}(A,{}_{\lambda}A) \otimes A \to {}_{\lambda}A
\end{equation*}
is an isomorphism of $A^{\mathrm e}$-modules, the
 $A^{\mathrm e}$-module structure on the left-hand side being the one 
induced by the 
canonical $A^{\mathrm e}$-module structure on $A$.

Let $\Pic(A)$ denote the abelian group (under the operation
of taking tensor products) that consists of left $A$-isomorphism classes
of left $(A \otimes A^{\mathrm{op}})$-modules $P$
which, as $S$-modules, are finitely generated and projective,
 such that, for every maximal ideal $\mathfrak m$ in $S$, the
$S_{\mathfrak m}$-module $P \otimes  S_{\mathfrak m}$ is isomorphic to
$A \otimes  S_{\mathfrak m}$.

Recall the following generalization of the Skolem-Noether theorem
\cite[Theorem 5]{MR0148709}.

\begin{prop}
\label{skolnoet}
Given an Azumaya $S$-algebra $A$, 
the assignment
to an automorphism $\lambda$ of $A$ over $S$ of the
class $\alpha(\lambda)=[{}_{\lambda}J]\in \Pic(S)$ 
and the assignment to the class $[J] \in \Pic(S)$ 
of the class $\beta([J])= [A \otimes J] \in \Pic(A)$
yields an exact sequence
\begin{equation*}
1
\longrightarrow
 \Out(A|S) 
\stackrel{\alpha} \longrightarrow  \Pic(S)
\stackrel{\beta} \longrightarrow  \Pic(A)
\longrightarrow
1
\end{equation*}
of abelian groups. In particular, the group $\Out(A|S)$
is an abelian group,
and the image  $\alpha(\Out(A|S))\subseteq \Pic(S)$
consists of isomorphism classes of projective rank one modules
$J$ such that $J\otimes A$ is isomorphic to $A$ as a left $A$-module. 
\end{prop}

The following is immediate.

\begin{cor}
A weakly $Q$-normal Azumaya algebra $A$ admits a $Q$-normal structure
if and only if the associated group extension 
\begin{equation}
\mathrm e^{\substack{\mbox{\tiny{$\Out(A|S)$}}}}_A
\colon 1
\longrightarrow
 \Out(A|S) 
\longrightarrow  \Out(A,Q)
\stackrel{\pi^{\substack{\mbox{\tiny{$\Out(A,Q)$}}}}}
\longrightarrow  Q
\longrightarrow
1
\label{eAO}
\end{equation}
(with abelian kernel) splits, and $Q$-normal structures on $A$ 
are then in one-one 
correspondence with sections $\sigma\colon Q \to  \Out(A,Q)$
for 
$\pi^{\substack{\mbox{\tiny{$\Out(A,Q)$}}}}$.
\end{cor}

The $Q$-invariant Brauer classes  $\mathrm B(S)^Q$ in $\mathrm B(S)$
constitute a subgroup of $\mathrm B(S)$,
and the classes of weakly $Q$-normal Azumaya $S$-algebras in 
$\mathrm B(S)$ constitute a subgroup 
of $\mathrm B(S)^Q$;
 we denote by
$\mathrm B(S,Q)$ that subgroup of $\mathrm B(S)^Q$.
When the class $[A]\in \mathrm B(S)$ of an Azumaya $S$-algebra
$A$ is fixed under $Q$, the algebra $A$ need not be weakly $Q$-normal, however.

Given a ring automorphism $f\colon B \to B$ of an $S$-algebra $B$,
we  denote by $B_f$ the $(B,B)$-bimodule 
which, as a left $S$-module, is just $B$ and
whose structure map
is given by
\begin{equation*}
B \otimes B_f \otimes B \longrightarrow B_f,\ b_1 \cdot b \cdot b_2=
b_1b f(b_2),\ 
b,b_1,b_2 \in B.
\end{equation*}

Given an Azumaya $S$-algebra $A$,
the assignment to an automorphism $\alpha_x$ of $A$ that extends
the automorphism $\kappaQ(x)$ of $S$, as $x$ ranges over 
the image 
$
\pi^{\substack{\mbox{\tiny{$\Out(A,Q)$}}}}(\Out(A,Q))
\subseteq Q
$ 
of $\Out(A,Q)$ in $Q$ under $\pi^{\substack{\mbox{\tiny{$\Out(A,Q)$}}}}$,
of the  invertible 
$(A,A)$-bimodule $A_{\alpha_x}$ induces an injective 
homomorphism 
\begin{equation}
\Theta\colon \Out(A,Q) \longrightarrow
\Aut_{\mathcat B_{S,Q}}(A)
\label{Thetain}
\end{equation}
such that the composite
\begin{equation}
\Out(A,Q)
\stackrel{\Theta}
\longrightarrow
\Aut_{\mathcat B_{S,Q}}(A)
\stackrel{\pi^{\substack{\mbox{\tiny{$\Aut_{\mathcat B_{S,Q}}(A)$}}}}}
\longrightarrow
Q
\label{compo99}
\end{equation}
coincides with
$\pi^{\substack{\mbox{\tiny{$\Out(A,Q)$}}}}\colon
\Out(A,Q) \to Q$.
Hence:

\begin{prop}
\label{difference}
An Azumaya $S$-algebra $A$
whose Brauer class
$[A]\in \mathrm B(S)$ in $\mathrm B(S)$ is fixed under $Q$
is weakly $Q$-normal if and only if the composite
{\rm\eqref{compo99}} is surjective. \qed
\end{prop}

Thus, given a weakly $Q$-normal Azumaya $S$-algebra $A$,
the diagram
\begin{equation*}
\begin{CD}
\mathrm e^{\substack{\mbox{\tiny{$\Out(A|S)$}}}}_A
\colon
1
@>>>
\Out(A|S)
@>>>
\Out(A,Q)
@>{\pi^{\substack{\mbox{\tiny{$\Out(A,Q)$}}}}}>>
Q
@>>>
1
\\
@.
@V{\alpha}VV
@V{\Theta}VV
@|
@.
\\
\mathrm e^{\substack{\mbox{\tiny{$\Pic(S)$}}}}_A\,\,\,\,\, \colon
1
@>>>
\Pic(S)
@>>>
\Aut_{\mathcat B_{S,Q}}(A)
@>{\pi^{\substack{\mbox{\tiny{$\Aut_{\mathcat B_{S,Q}}(A)$}}}}}>>
Q
@>>>
1
\end{CD}
\end{equation*}
is commutative with exact rows.

\begin{rema} In particular, 
an Azumaya algebra $A$ having $\alpha \colon \Out(A|S) \to \Pic(S)$
surjective is weakly $Q$-normal  if its Brauer class
$[A]\in \mathrm B(S)$ is fixed under $Q$. When $S$ is a field, 
$\Pic(S)$ is trivial whence then
an Azumaya algebra $A$ is weakly $Q$-normal, in fact even normal, if and only 
if its Brauer class
$[A]\in \mathrm B(S)$ is fixed under $Q$.
\end{rema}

\begin{rema}
\label{fundamental}
Given a weakly $Q$-normal Azumaya $S$-algebra $A$,
there is a fundamental difference between the two
group extensions $\mathrm e^{\substack{\mbox{\tiny{$\Out(A|S)$}}}}_A$
and
$\mathrm e^{\substack{\mbox{\tiny{$\Pic(S)$}}}}_A$:
The kernel of the former depends on $A$ whereas that of the latter does not.
\end{rema}

A $Q$-normal Azumaya $S$-algebra $(A,\sigma)$
represents a member of $\mathrm B(S,Q) (\subseteq \mathrm B(S)^Q)$
in such a way that $\sigma$ splits the associated group extension
\eqref{eAO}, and the composite 
\begin{equation}
\Theta_{\sigma}\colon 
Q 
\stackrel{\sigma} \longrightarrow
\Out(A,Q)
\stackrel{\Theta} \longrightarrow
\Aut_{\mathcat B_{S,Q}}(A)
\label{Thetas}
\end{equation}
of $\sigma$ with the
homomorphism
$\Theta$ from $\Out(A,Q)$ to
$\Aut_{\mathcat B_{S,Q}}(A)$ 
given above as \eqref{Thetain}
yields the object 
$(A,\Theta_{\sigma})$
of 
$\mathcat{Rep}(Q,\mathcat B_{S,Q})$
and hence a member of 
$k\mathcat{Rep}(Q,\mathcat B_{S,Q})$.
We therefore refer to an object of
$\mathcat{Rep}(Q,\mathcat B_{S,Q})$ 
as a {\em generalized\/} $Q$-{\em normal Azumaya algebra\/}.
In Subsection \ref{from} we shall show that, when the group $Q$ is finite,
each member of
$k\mathcat{Rep}(Q,\mathcat B_{S,Q})$ arises in this manner
from a $Q$-normal Azumaya $S$-algebra.

\subsection{Equivariant algebras and scalar extension}
\label{twotwo}  
Let $Q$ be a group and $\kappaQ\colon Q \to \Aut(S)$ an action of $Q$ on $S$.
A $Q$-{\em equivariant} $S$-algebra $(A, \tau )$ 
consists of a
central $S$-algebra $A$ together with a homomorphism 
$\tau \colon Q \to \Aut (A)$ that induces $\kappaQ$.

Let $R = S^Q$ be the subring of $S$ which is elementwise fixed
under the $Q$-action. Given  a central $R$-algebra $B$, 
{\em scalar extension} 
$B \mapsto A = B\otimes_R S$ yields the central $S$-algebra $A$;
then the action of $Q$ on $A= B\otimes_R S$ induced by the action $\kappaQ$ 
of $Q$ on $S$
yields a $Q$-equivariant
structure  $\tau_0 \colon Q \to \Aut (A)$ and hence a $Q$-normal 
structure $\sigma_0 \colon Q \to \Out(A)$, and we  say that 
$(A, \tau_0)$ and $(A,\sigma_0)$ arise from $B$ by
{\em scalar extension.}
By Galois descent, cf. Subsection \ref{galext}~(ii), 
if $S| R$ is a Galois extension of commutative 
rings with
Galois group $Q$, any $Q$-equivariant $S$-algebra arises by scalar
extension.

\subsection{The Teich\-m\"uller class of a $Q$-normal algebra}
\label{twothree}
As in the classical approach of
 Teich\-m\"uller \cite{MR0002858} and  Eilenberg-Mac Lane \cite{MR0025443},
we seek to classify $Q$-normal $S$-algebras modulo those which
are obtained by  extension of scalars, rather than modulo the equivariant
ones (in case of algebras over fields this makes no difference by
Galois descent), by means of certain 3-dimensional cohomology classes.

Let $(A, \sigma )$ be a $Q$-normal $S$-algebra,
 with respect to
$\sigma\colon Q \to \Out(A)$, let $B^\sigma $ denote the fiber
product group ${\Aut (A) \times_{\Out (A)}Q}$, let $B^\sigma $
act on $\mathrm U(A)$ in the obvious way, 
that is, via the canonical homomorphism from 
$B^\sigma $ to $\Aut(A)$,
and let
$\partial ^\sigma \colon \mathrm U(A) \to B^\sigma$ denote the homomorphism
induced by $\partial \colon \mathrm U(A) \to \Aut(A)$.
Pulling back the
crossed 2-fold extension 
 \eqref{eA} 
above
yields the crossed 2-fold extension
\begin{equation}
\mathrm e_{(A,\sigma)} \colon 0 \longrightarrow \mathrm U(S) 
\longrightarrow \mathrm U(A) 
\stackrel{\partial^\sigma}\longrightarrow B^\sigma 
\longrightarrow Q \longrightarrow 1,
\label{pb1}
\end{equation}
uniquely determined by $(A, \sigma )$. By construction, $B^\sigma $ may be
identified with a certain subgroup of $\Aut(A,Q)$. 
The crossed 2-fold extension
$\mathrm e_{(A,\sigma)}$ represents a class
$[\mathrm e_{(A,\sigma)}] \in \mathrm H^3 (Q, \mathrm U(S))$
\cite[Section 7]{crossed},
cf. Section \ref{one} above. We refer to 
$\mathrm e_{(A,\sigma )}$ as 
the {\em Teich\-m\"uller complex} of $(A,\sigma )$
and to the corresponding class in $\mathrm H^3(Q,\mathrm U(S))$ as the 
{\em Teich\-m\"uller class} of $(A,\sigma )$.

For completeness, and for future reference, we indicate how the
Teich\-m\"uller complex is related with the classical 
\lq\lq Teich\-m\"uller cocycle\rq\rq: Let $\mathrm e_Q$ denote the 
crossed standard resolution
of $Q$ introduced in   
\cite[Section~9]{crossed}, and lift
the identity map to a commutative diagram of the kind
\begin{displaymath}
\xymatrix{
\mathrm e_Q \colon \ldots\ar[r] &C_2\ar[d]^\xi \ar[r]   &C_1 \ar[d]^\beta \ar[r]
                   &F \ar[d]^\alpha \ar[r]    &Q \ar@{=}[d] \ar[r] &1 \\
\mathrm e_ {(A, \sigma)} \colon0\ar[r]&\mathrm U(S)\ar[r] &\mathrm U(A)\ar[r]
&B^\sigma \ar[r]&Q\ar[r]&1,}
\end{displaymath}
so that $(\beta , \alpha )$ is a morphism of crossed modules, 
cf. \cite[Section 5]{crossed}. In view of 
\cite[Section 9 $({}^{\ast \ast})$]{crossed},
the map $\xi $ is a 3-cocycle of $Q$ with values in $\mathrm U(S)$.
By the main Theorem in \cite{crossed}, the class 
$[\mathrm e_{(A,\sigma )}]\in \mathrm H^3(Q,\mathrm U(S))$ 
coincides with the class which, in the cocycle description, is
 represented by $\xi$;
more precisely, 
the isomorphism 
$\mathrm{Opext}^2(Q,\mathrm U(S)) \to \mathrm H^3(Q,\mathrm U(S))$ 
results from the assignment to
a crossed 2-fold extension of a 3-cocycle by the method
just explained. It is manifest that,
in case $S$ is a field, $Q$ a finite group of automorphisms of $S$
and $A$ a finite dimensional central simple $S$-algebra,
the 3-cocycle $\xi $ is a Teich\-m\"uller cocycle for $A$,
 cf. Teich\-m\"uller \cite{MR0002858} 
and Eilenberg and Mac Lane \cite{MR0025443}, and so $[\mathrm e_{(A,\sigma )}]$
then comes down to the Teich\-m\"uller class represented by the Teich\-m\"uller
cocycle.

Henceforth we  denote by $\mathrm e_0$ the crossed 2-fold extension
\begin{equation*}
\mathrm e_0 \colon 0 \longrightarrow \mathrm U(S) \stackrel{=}\longrightarrow \mathrm U(S)  \stackrel{0}\longrightarrow Q \stackrel{=}\longrightarrow
Q \longrightarrow 1,
\end{equation*}
the crossed module structure being given by 
the $Q$-module structure on $\mathrm U(S)$.

\begin{prop}
\label{2.3.1}
The Teich\-m\"uller class  $[\mathrm e_{(A,\sigma )}]$ of an equivariant
$Q$-normal $S$-algebra $(A,\sigma )$ is zero.
\end{prop}

\begin{proof}
There is a congruence morphism $(1,\cdot,\cdot , 1) \colon \mathrm e _0 \to \mathrm e_{(A,\sigma )}$ of 
crossed 2-fold extensions, and $[\mathrm e_0] = 0 \in \mathrm H^3(Q,\mathrm U(S))$,
see  
\cite{crossed}.
\end{proof}

\subsection{Opposite algebras}
\label{2.4}
Given an algebra $A$, we denote its opposite algebra by $A^{\mathrm{op}}$
as usual.
Let $A$ be a central $S$-algebra. The association
\begin{equation*} 
\alpha \mapsto \hat\alpha,\   
\hat\alpha (a^{\mathrm{op}}) = (\alpha a)^{\mathrm{op}},\  a\in A,
\alpha \in \Aut (A),
\end{equation*}
yields an isomorphism 
$\,{}\widehat{} \,\colon \Aut (A) \to \Aut (A^{\mathrm{op}})$
and, with an abuse of notation, we denote
by $\,{}\widehat{} \, \colon \Out (A) \to \Out (A^{\mathrm{op}})$
the induced isomorphism as well.
Moreover,
inversion yields an isomorphism 
$({^\mathrm{op}})^{-1} \colon  \mathrm U(A) \to \mathrm U(A^{\mathrm{op}})$, 
and so we get an isomorphism
\begin{equation*}
(-1, ({^\mathrm{op}})^{-1},\,\widehat \,, \,{}\widehat{} \,) \colon \mathrm e_A \longrightarrow {\mathrm e}_{A^{\mathrm{op}}}
\end{equation*}
of crossed 2-fold extensions.

Given a $Q$-normal $S$-algebra $(A,\sigma )$, 
we equip $A^{\mathrm{op}}$
with a $Q$-normal structure by setting 
$\sigma^{\mathrm{op}} =  \,{}\widehat{} \,\circ\sigma $;
we refer to $(A^{\mathrm{op}}, \sigma^{\mathrm{op}})$ as the {\em opposite} of $(A,\sigma )$.
Likewise, given   a $Q$-equivariant $S$-algebra $(A,\tau )$,
letting
$\tau^{\mathrm{op}} = \,{}{\widehat{}}\, \circ \tau $, we obtain a $Q$-equivariant structure on 
$A^{\mathrm{op}}$, and we refer to 
$(A^{\mathrm{op}},\tau^{\mathrm{op}})$ is the {\em opposite} of $(A, \tau )$.

\begin{prop}
\label{2.4.1}
For a
 $Q$-normal $S$-algebra $(A,\sigma )$, the isomorphism $\,\widehat{}\, $
 induces an isomorphism
\begin{equation*}
(-1,\cdot,\cdot ,1) \colon \mathrm e_{(A,\sigma )} \longrightarrow  \mathrm e_{(A^{\mathrm{op}},\sigma ^{\mathrm{op}})}
\end{equation*}
of crossed $2$-fold extensions, and therefore
\begin{equation*}
[\mathrm e_{(A,\sigma )}] +  [\mathrm e_{(A^{\mathrm{op}},\sigma ^{\mathrm{op}})}] = 0 \in \mathrm H^3(Q,\mathrm U(S)).
\end{equation*}
\end{prop}

\subsection{Matrix algebras}
\label{2.4.22}
Now let $A$ be a central $S$-algebra, and let $\mathrm M_I(A)$ be a matrix algebra
over $A$; if $I$ is not finite, we interpret $\mathrm M_I(A)$ 
as being
the endomorphism ring
of $\oplus_I A^{\mathrm{op}}$. The algebra $\mathrm M_I(A)$ is again a central $S$-algebra.
It is obvious that an automorphism of $A$ yields one of $\mathrm M_I(A)$
in a unique way, and the obvious map $A \to \mathrm M_I(A)$ is a ring homomorphism.
Hence a $Q$-normal structure $\sigma \colon Q \to \Out(A)$ on $A$
determines one on $\mathrm M_I(A)$, 
and we
denote this structure by 
\[
\sigma_I \colon Q \longrightarrow \Out(\mathrm M_I(A));
\]
likewise a $Q$-equivariant structure $\tau \colon Q \to \Aut(A)$ 
on $A$ determines an
obvious $Q$-equivariant structure on $\mathrm M_I(A)$,
and we denote this structure by
\[
\tau _I\colon Q \longrightarrow \Aut(\mathrm M_I(A)).
\]
A special case is $A = S$ and $\sigma = \tau = \kappaQ$;
then we get the obvious equivariant structure $\kappaQI$ on $\mathrm M_I(S)$.

\begin{prop}
\label{2.4.2}
Given a $Q$-normal $S$-algebra $(A,\sigma )$,
the obvious algebra map ${A \to \mathrm M_I(A)}$ extends to a morphism 
$(A, \sigma ) \to (\mathrm M_I(A), \sigma_I)$ of $Q$-normal $S$-algebras, 
and there
is an induced congruence 
$(1,\cdot,\cdot, 1) \colon 
\mathrm e_{(A,\sigma)} \longrightarrow \mathrm e_{(\mathrm M_I(A),\sigma_I)}$
involving the corresponding crossed $2$-fold extensions. Hence
\begin{equation*}
[\mathrm e_{(A,\sigma)}] =  [\mathrm e_{(\mathrm M_I(A),\sigma_I)}] \in \mathrm H^3(Q,\mathrm U(S)).
\end{equation*}
\end{prop}

\subsection{Tensor products}
\label{2.4.33}
Given two $Q$-normal $S$-algebras $(A_1,\sigma_1)$ and  $(A_2,\sigma_2)$,
the $Q$-normal structures determine an obvious homomorphism 
$\sigma_1 \otimes \sigma_2 \colon Q \longrightarrow \Out(A_1 \otimes A_2)$
so that $(A_1 \otimes A_2, \sigma_1 \otimes \sigma_2)$ is a $Q$-normal
$S$-algebra;
likewise, given two $Q$-equivariant
$S$-algebras $(A_1, \tau_1)$ and $(A_2, \tau_2)$, 
the $Q$-equivariant structures determine an obvious
homomorphism 
$\tau_1 \otimes \tau_2 \colon Q \longrightarrow \Aut (A_1 \otimes A_2)$
in such a way that 
$(A_1 \otimes A_2, \tau_1 \otimes \tau_2)$ is a $Q$-equivariant
$S$-algebra.

Given two $Q$-normal $S$-algebras $(A_1,\sigma_1)$ and  $(A_2,\sigma_2)$,
consider the Baer sum
\begin{equation*} 
{\mathrm e}_{(A_1,\sigma_1)} +{\mathrm e}_{(A_2,\sigma_2)} \colon 0 \longrightarrow \mathrm U(S) \longrightarrow \mathrm U(A_1) \times^{\mathrm U(S)} \mathrm U(A_2) \longrightarrow
B^{\sigma_1} \times_Q B^{\sigma_2}\longrightarrow Q \longrightarrow 1
\end{equation*}
of the crossed 2-fold extensions
${\mathrm e}_{(A_1,\sigma_1)}$ and ${\mathrm e}_{(A_2,\sigma_2)}$,
cf. \cite{crossed} for details.
This is a crossed
2-fold extension 
that
 represents the sum 
 $[{\mathrm e}_{(A_1,\sigma_1)}] + [{\mathrm e}_{(A_2,\sigma_2)}]$
in $\mathrm H^3(Q,\mathrm U(S))$.

\begin{prop}
\label{2.4.3}
There is an obvious congruence 
\begin{equation*}
(1,\cdot ,\cdot ,1) \colon {\mathrm e}_{(A_1,\sigma_1)} + {\mathrm e}_{(A_2,\sigma_2)} \longrightarrow {\mathrm e}_{(A_1\otimes A_2, \sigma_1\otimes \sigma_2)}
\end{equation*}
of crossed $2$-fold extensions. Hence
\begin{equation*}
[{\mathrm e}_{(A_1\otimes A_2, \sigma_1\otimes \sigma_2)}]
= [{\mathrm e}_{(A_1,\sigma_1)}] + [{\mathrm e}_{(A_2,\sigma_2)}] \in \mathrm H^3(Q,\mathrm U(S)).   \qed
\end{equation*}
\end{prop}

Combining Propositions \ref{2.4.1} and \ref{2.4.3} with the observation  
that the Teich\-m\"uller class of $(S,\kappaQ)$ is zero, we see
that the Teich\-m\"uller classes of $Q$-normal $S$-algebras constitute a
subgroup of $\mathrm H^3(Q,\mathrm U(S))$. 

\subsection{Behavior under change of actions}
\label{2.59}

\begin{prop}
\label{2.5.1}
Let
$(f,\varphi ) \colon (S,Q,\kappa) \to (T,G,\lambda )$ be a morphism
in $\mathcat{Change}$ between two given
objects $(S,Q,\kappa)$ and $(T,G,\lambda )$ of $\mathcat{Change}$, so that
the group $G$ acts on $S$ via $\varphi \colon G \to Q$.

\noindent
{\rm (i)}
Given a $Q$-normal $S$-algebra $(A,\sigma )$, the structure map $\sigma $
and $(f, \varphi )$ induce a canonical $G$-normal structure 
$\sigma_{(f,\varphi)}\colon G \to \Out(T \otimes A)$ on 
$T \otimes A$ that is compatible with
the operations of taking opposite algebras and tensor products.

\noindent 
{\rm (ii)}
Given a
$Q$-equivariant $S$-algebra $(A,\tau )$, the structure map
$\tau $ and $(f, \varphi)$ induce a canonical $G$-equivariant structure
$\tau_{(f,\varphi )}\colon G \to \Aut(T \otimes A)$ on $T \otimes A$ 
that is compatible with the operations of taking
opposite algebras and tensor products.

\noindent 
{\rm (iii)}
Given
a $Q$-normal $S$-algebra $(A,\sigma )$, the structure map $\sigma $
and $(f, \varphi )$ induce, in a canonical way, morphisms
\begin{equation*}
(1, \cdot , \cdot ,  \varphi) \colon {\mathrm e}_{(A,\sigma \varphi) } 
\to {\mathrm e}_{(A,\sigma ) } \;
\textrm{and} \; 
(f ,\cdot, \cdot,  1) \colon {\mathrm e}_{(A,\sigma \varphi )} \to {\mathrm e}_{(T\otimes A,\sigma_{(f,\varphi )} )} 
\end{equation*}
of crossed $2$-fold extensions. 
Consequently,
\begin{equation*}
[{\mathrm e}_{(T\otimes A, \sigma_{(f,\varphi  )}) } ]
= (f, \varphi )_{\ast } [{\mathrm e}_{(A,\sigma )}]\in \mathrm H^3(G,\mathrm U(T) ),
\end{equation*}
where $(f,\varphi )_\ast $ denotes the map induced on cohomology. \qed
\end{prop}

\subsection{Embedding algebras into algebras with smaller center}
\label{eaas}
The title of this subsection is intended to remind the reader of
Deuring's paper  \cite{0014.20001} having the same title.

As before,
$Q$  denotes a group and $\kappaQ\colon Q \to \Aut(S)$ an action of $Q$ on 
the commutative ring $S$.
Let $A$ be a central $S$-algebra, 
let $R$ denote the subring $S^Q$ of elements of $S$
that are elementwise fixed under $Q$,
and let $C$ be an $R$-algebra
containing $A$ as a subalgebra.
We  refer to the embedding of $A$ into $C$ 
as a {\em Deuring embedding
with respect to the action\/} $\kappaQ\colon Q \to \Aut(S)$
of $Q$ on $A$ if each automorphism 
$\kappaQ(x)$, as $x$ ranges over $Q$, extends to an inner automorphism
of $C$ that maps $A$ to itself. 
In the special case where $A$ coincides with the centralizer of $T$ in $C$,
the requirement that each inner automorphism of $C$
that lifts an automorphism of $S$ of the kind $\kappaQ(x)$
as $x$ ranges over $Q$ map $A$ to itself is redundant.

For technical reasons, we need a stronger notion of Deuring embedding.
We  now prepare for the description of this stronger notion.

Thus, let $C$ be an $R$-algebra
that contains $A$ as a subalgebra.
Let  
$N^{\mathrm U(C)}(A)$ denote the {\em normalizer of} 
$A$ in the group 
$\mathrm U(C)$ of invertible elements of $C$; the group 
$N^{\mathrm U(C)}(A)$ consists
of those $u\in \mathrm U(C)$ such that, for each $a\in A$, 
the member $u a u^{-1}$ of $C$ already lies in  $A$.
Denote by $i \colon \mathrm U(A) \to N^{\mathrm U(C)}(A)$ the inclusion.

\begin{prop}
\label{twosixo}
{\rm (i)} 
Conjugation in $C$ induces a morphism
\begin{equation*}
(1,\eta ) \colon (\mathrm U(A), N^{\mathrm U(C)}(A),i) \longrightarrow (\mathrm U(A), \Aut (A), \partial)
\end{equation*}
of crossed modules,
the requisite action of
$N^{\mathrm U(C)}(A)$ on $\mathrm U(A)$ being given by conjugation,
 and hence an 
$(N^{\mathrm U(C)}(A)/\mathrm U(A))$-normal structure 
\begin{equation}
\eta_{\sharp} \colon N^{\mathrm U(C)}(A)/\mathrm U(A) \longrightarrow \Out(A)
\label{homo1}
\end{equation}
on $A$.
Explicitly, the 
homomorphism $\eta\colon N^{\mathrm U(C)}(A) \to \Aut (A)$ is given by
\begin{equation*}
(\eta(u))(a)= u a u^{-1} \in A,\ u \in  N^{\mathrm U(C)}(A),\ a \in A.
\end{equation*}

\noindent
{\rm (ii)}
The induced homomorphism 
\begin{equation*}
\begin{CD}
\eta_{\flat}\colon
N^{\mathrm U(C)}(A)/\mathrm U(A) @>{\eta_{\sharp}}>> 
\Out(A) @>{\res}>>  \Aut(S)
\end{CD}
\end{equation*}
maps onto the subgroup
$\kappaQ(Q) \subseteq \Aut(S)$
if and only if the embedding of $A$ into $C$ is a Deuring embedding,
that is, if and only if each automorphism $\kappaQ(q)$
of $S$, as $q$ ranges over
$Q$, extends
to an inner automorphism of $C$ that
normalizes $A$.

\noindent
{\rm (iii)}
If $A$ 
coincides with  the centralizer of $S$ in $C$, then the induced homomorphism
$\eta_{\flat}$ from
 $N^{\mathrm U(C)}(A)/\mathrm U(A)$ to $\Aut(S)$ 
spelled out in {\rm (ii)} above is injective.

\noindent
{\rm (iv)}
Suppose that  the given action $\kappaQ\colon Q \to \Aut(S)$ of $Q$ on $S$
lifts to
a homomorphism ${\chi \colon Q \to N^{\mathrm U(C)}(A)/\mathrm U(A)}$ 
in the sense
that the
combined map 
\begin{equation}
\begin{CD}
Q @>{\chi}>> N^{\mathrm U(C)}(A)/\mathrm U(A) @>{\eta_{\flat}}>> \Aut(S)
\end{CD}
\label{compelev}
\end{equation}
coincides with $\kappaQ\colon Q \to \Aut(S)$. 
Then  the composite
\begin{equation}
\begin{CD}
\sigma \colon Q @>{\chi}>> N^{\mathrm U(C)}(A)/\mathrm U(A) @>{\eta_{\sharp}}>> \Out(A)
\end{CD}
\label{comp1}
\end{equation}
of $\chi$ with the homomorphism
{\rm\eqref{homo1}} 
in {\rm{(i)}} above
yields  a $Q$-normal structure $\sigma$ 
on $A$; in particular, a strong Deuring embedding structure map
$\chi $ exists
and is uniquely determined if $A$ coincides with
 the centralizer of $S$ in $C$.

\noindent
{\rm (v)} 
Given a lift
 $\chi \colon Q \to N^{\mathrm U(C)}(A)/\mathrm U(A)$
in the sense
that the
composite {\rm \eqref{compelev}} thereof with  
$\eta_{\flat}$ coincides
with
 the structure map $\kappaQ\colon Q \to \Aut(S)$,
the obvious map from the 
fiber product group ${\Gamma = N^{\mathrm U(C)}(A) 
\times_{N^{\mathrm U(C)}(A)/\mathrm U(A)} Q}$
to $Q$ yields a group extension
\begin{equation}
1 \longrightarrow \mathrm U(A) \stackrel{j}\longrightarrow \Gamma 
\longrightarrow Q \longrightarrow 1,
\label{ext2}
\end{equation}
the injection $j$ being the obvious homomorphism
from $\mathrm U(A)$ to $\Gamma$, 
and the obvious action $\vartheta\colon \Gamma \to \Aut(A)$
of $\Gamma$ on $A$
(induced by conjugation in $C$ or, equivalently, by the
homomorphism $\eta\colon N^{\mathrm U(C)}(A) \to \Aut(A)$ in {\rm{(i)}} above,)
yields a morphism
\begin{equation}
(1,\vartheta) \colon (\mathrm U(A), \Gamma ,j) \longrightarrow 
(\mathrm U(A), \Aut (A), \partial )
\label{mor1}
\end{equation}
of crossed modules,
the requisite action of $\Gamma$
on $\mathrm U(A)$ being given by conjugation.
The morphism $(1,\vartheta)$ of crossed
modules, in turn,
 induces the $Q$-normal structure
{\rm \eqref{comp1}} on $A$. 
\end{prop}

Consider a central $S$-algebra $A$;
given an algebra $C$ over $R=S^Q$,
we define a {\em strong Deuring embedding} of 
$A$
into $C$ {\em relative to\/} the action $\kappaQ\colon Q \to \Aut(S)$ 
of $Q$ on $S$
to be
an embedding of $A$ into $C$ together with a homomorphism 
$\chi \colon Q \to N^{\mathrm U(C)}(A)/\mathrm U(A)$ 
such that
 the combined map 
\begin{equation}
\begin{CD}
Q @>{\chi}>> N^{\mathrm U(C)}(A)/\mathrm U(A) @>{\eta_{\flat}}>> \Aut(S) 
\end{CD}
\label{comb1}
\end{equation}
coincides with 
the structure map $\kappaQ\colon Q \to \Aut(S)$.

Given a strong 
Deuring embedding $(A\subseteq C,\chi)$ of the $S$-algebra $A$ into an algebra
$C$ relative to the action $\kappaQ\colon Q \to \Aut(S)$
of $Q$ on the commutative ring $S$,
the homomorphism $\sigma\colon Q \to \Out(A)$
given as \eqref{comp1} above yields a $Q$-normal structure
on $A$, and we 
 say that
this $Q$-normal structure {\em arises from the strong Deuring embedding\/} 
 $(A\subseteq C,\chi)$ of $A$ into $C$ relative to $\kappaQ$
or that {\em the strong Deuring embedding\/} 
 $(A\subseteq C,\chi)$ {\em of\/} $A$ {\em into\/} $C$ {\em relative to\/} $\kappaQ$
{\em induces the\/} $Q$-normal structure $\sigma\colon Q \to \Out(A)$
on $A$.

\begin{prop}
\label{deucent}
Let $C$ be an algebra over $R=S^Q$,  let $A\subseteq C$
be an embedding such that $A$ coincides with
the centralizer of $S$ in $C$, and
suppose that the embedding
is a Deuring embedding
with respect to the action $\kappaQ\colon Q \to \Aut(S)$
of $Q$ on $S$.
Then the data determine a unique group homomorphism
$\chi \colon Q \to N^{\mathrm U(C)}(A)/\mathrm U(A)$  that
turns the embedding $A \subseteq C$ into a strong Deuring embedding
with respect to $\kappaQ$.
 \end{prop}

\begin{proof}
Since $A$ coincides with the centralizer of $S$ in $C$,
the normalizer $N^{\mathrm U(C)}(A)$ 
of $A$ in $C$ coincides with the normalizer
$N^{\mathrm U(C)}(S)$ of $S$ in $C$.
The hypothesis that 
the embedding of $A$ into $C$ be a Deuring embedding 
with respect to the action $\kappaQ\colon Q \to \Aut(S)$
of $Q$ on $S$, that is, that
each 
automorphism $\kappaQ(q)$ 
of $S$, as $q$ ranges over $Q$, 
extends to an inner automorphism of $C$
entails that 
the canonical homomorphism
from 
the group $N^{\mathrm U(C)}(S)=N^{\mathrm U(C)}(A)$ 
to $\Aut(S)$ is a surjective homomorphism onto
the subgroup $\kappaQ(Q)$ of $\Aut(S)$.

Since $A$ coincides with the centralizer of $S$ in $C$, 
 by 
Proposition \ref{twosixo}(iii),
the homomorphism 
\[
 \eta_{\flat}\colon N^{\mathrm U(C)}(A)/\mathrm U(A) \longrightarrow \Aut(S)
\] 
is 
injective. Consequently the structure map $\kappaQ\colon Q \to \Aut(S)$ lifts
to a
uniquely determined homomorphism
$\chi \colon Q \to N^{\mathrm U(C)}(A)/\mathrm U(A)$ in the sense 
that the composite \eqref{comb1}
coincides with $\kappaQ$.
\end{proof}

Recall that the Teich\-m\"uller class
of a $Q$-normal $S$-algebra $(A,\sigma)$ is represented by the associated
crossed 2-fold extension ${\mathrm e}_{(A,\sigma  )}$
introduced as \eqref{pb1} above.

\begin{prop}
\label{2.6.5}
If a $Q$-normal structure $\sigma \colon Q \to \Out(A)$ 
on a central $S$-algebra $A$
arises from a strong Deuring embedding $(C,\chi)$ of $A$ into an algebra $C$ 
over $R=S^Q$,
then
the Teich\-m\"uller class 
$[{\mathrm e}_{(A,\sigma  )}]\in \mathrm H^3 (Q,\mathrm U(S))$ 
of $(A,\sigma  )$
is zero.
\end{prop}

\begin{proof}
This  
is a consequence of
Proposition \ref{twosixo}(v).
Indeed, 
$(\mathrm U(S)\times \mathrm U(A),\Gamma,\partial)$
being endowed with the obvious crossed module structure,
let
 $\mathrm e$ denote the associated crossed 2-fold extension
\begin{equation*}
\mathrm e \colon 0 \longrightarrow \mathrm U(S) 
    \longrightarrow  
\mathrm U(S)\times \mathrm U(A)\stackrel{\partial}
\longrightarrow \Gamma 
\longrightarrow Q
\longrightarrow 1.
\end{equation*}
There are 
obvious congruences
$(1,\cdot ,\cdot, 1) \colon {\mathrm e} \to {\mathrm e}_{(A,\sigma )}$ 
and
$(1,\cdot ,\cdot, 1) \colon {\mathrm e} \to {\mathrm e}_0$
of crossed 2-fold extensions 
whence the assertion.
\end{proof}

\begin{rema}
\label{rema1}
The morphism \eqref{mor1}
of crossed modules induces 
the morphism
\begin{equation*}
(\mathrm{Id}, \,\cdot\,) \colon
(\mathrm U(A),\Gamma,j) \longrightarrow (\mathrm U(A), B^{\psi},\partial^{\psi})
\end{equation*}
of crossed modules.
This morphism of  crossed modules displays the fact that,
in the language of abstract kernels, cf. 
\cite[Section IV.8 p.~124]{maclaboo},
the abstract kernel associated to the crossed module
$(\mathrm U(A),B^{\psi},\partial^{\psi})$
is an extendible kernel. Consequently
its obstruction class in 
$\mathrm H^3(Q,\mathrm U(S))$ vanishes.
This obstruction class coincides with the
associated Teich\-m\"uller class.
\end{rema}

\begin{rema}
For the converse of Proposition \ref{2.6.5}, see Theorem \ref{4.2} below.
\end{rema}

It seems worthwhile spelling out a special case of Proposition \ref{2.6.5}.

\begin{prop}
\label{2.6.6}
Let $C$ be an algebra over $R=S^Q$ that contains $A$
as a subalgebra in such a way that $A$ coincides with
the centralizer of $S$ in $C$, and
suppose that the embedding of $A$ into $C$ is a Deuring embedding
with respect to the action $\kappaQ\colon Q \to \Aut(S)$
of $Q$ on $S$.
Then the $Q$-normal structure $\sigma\colon Q \to \Out(A)$ on $A$
induced by the Deuring embedding 
of $A$ into $C$
via the associated homomorphism 
$\chi \colon Q \to N^{\mathrm U(C)}(A)/\mathrm U(A)$
in Proposition {\rm \ref{deucent}} above
has zero Teich\-m\"uller class 
$[{\mathrm e}_{(A,\sigma )}] \in \mathrm H^3(Q,\mathrm U(S))$.
\end{prop}

\subsection{The normal algebra associated to
a generalized normal Azumaya algebra}
\label{from}
Until the end of this section, given an algebra $A$, a left $A$-module
${}_AM$ and a right $A$-module $M_A$, we  use the notation
${}_A\End({}_AM)$ for the algebra of left $A$-endomorphisms and
$\End_A(M_A)$ for the algebra of right $A$-endomorphism
of  $M_A$; accordingly we use the notation
${}_A\Hom(\,\cdot\, ,\,\cdot\,)$ and $\Hom_A(\,\cdot\, ,\,\cdot\,)$.

As before, let $R=S^Q \subseteq S$.
Consider an Azumaya $S$-algebra $A$, viewed as an object of $\mathcat B_{S,Q}$,
let $Q_A\subseteq Q$ denote the image of the canonical homomorphism 
$\Aut_{\mathcat B_{S,Q}}(A)
\longrightarrow Q$,
and consider the associated group extension
\begin{equation}
\mathrm e^{\substack{\mbox{\tiny{$\Pic(S)$}}}}_A\colon
1
\longrightarrow
\Pic (S)
\longrightarrow
\Aut_{\mathcat B_{S,Q}}(A)
\longrightarrow
Q_A
\longrightarrow
1
\label{eApictilde}
\end{equation}
of the kind \eqref{eApic},
with $Q_A$ rather than $Q$.
Then
$A$ represents
a member of $\mathrm B(S,Q_A)$, and
$Q_A=Q$ if and only if $A$ represents
a member of $\mathrm B(S,Q)$.
Let $\sigma_A\colon Q_A \to \Aut_{\mathcat B_{S,Q}}(A)$
be a section for \eqref{eApictilde} of the underlying sets
that sends the neutral element of $Q_A$ to  the neutral element of
$\Aut_{\mathcat B_{S,Q}}(A)$, not necessarily 
a homomorphism. 

To simplify the exposition, we  now suppose that $\kappaQ$ is injective.
The general case then results 
from the special case
with the subgroup $\kappaQ(Q)$ of $\Aut(S)$
substituted for $Q$.

For each $x \in Q_A$, the value 
$\sigma_A(x)\in \Aut_{\mathcat B_{S,Q}}(A)$ 
is an isomorphism class
of an invertible $(A,A)$-bimodule
of grade $x\in Q$, and
we choose an
invertible $(A,A)$-bimodule $M_x$ in $\sigma_A(x)$;
while $M_x$ depends on $\sigma_A$,
we do not indicate this dependence in notation, to simplify the exposition.
In particular,  $M_e$ denotes the algebra $A$, viewed as an
$(A,A)$-bimodule in the canonical way.
Let $\MsA=\oplus _{z\in Q_A} M_z$ and ${\BsA}={}_A\End(\MsA)$,
the algebra of left $A$-endomorphisms of $\MsA$.
When the group $Q_A$ is finite, ${\BsA}$ (as well as $\BsA^{\mathrm{op}}$)
is an Azumaya $S$-algebra.

\begin{rema} The construction of $\BsA$ may be found in the proof
of \cite[Theorem 4.1 (ii)]{MR1803361} where it is the basic tool to 
establish the surjectivity of a homomorphism spelled out below
as \eqref{theta} in the case where the group $Q$ is finite.
See also Remark \ref{frwall} below.
\end{rema}

Recall that, given a ring automorphism $\alpha$ of $\BsA^{\mathrm{op}}$,
the notation $(\MsA)_{\alpha}$ refers to the right
$\BsA^{\mathrm{op}}$-module structure on $\MsA$ given by
\begin{equation*}
\MsA \otimes \BsA^{\mathrm{op}} \longrightarrow \MsA,\
m \cdot b = m \alpha(b),\ m \in \MsA, \ b \in \BsA^{\mathrm{op}}.
\end{equation*}

\begin{prop}
\label{prop8} 
Given a member $\alpha$ of $\Aut(\BsA^{\mathrm{op}},Q)$,
the operation
\begin{equation*}
A \otimes \Hom_{\BsA^{\mathrm{op}}}(\MsA, (\MsA)_\alpha) \otimes A \longrightarrow 
\Hom_{\BsA^{\mathrm{op}}}(\MsA, (\MsA)_\alpha)
\end{equation*}
given by
$a_1\otimes \varphi \otimes a_2 \longmapsto 
a_1\circ \varphi \circ a_2 $ ($a_1,a_2 \in A$, $\varphi \in  \Hom_{\BsA^{\mathrm{op}}}(\MsA, (\MsA)_\alpha)$) where, for $m \in \MsA$, the value
$(a_1\circ \varphi \circ a_2)(m)$ equals $a_1(\varphi(a_2m))$,
yields
an obvious invertible $(A,A)$-bimodule structure
on \begin{equation}
N_\alpha=
\Hom_{\BsA^{\mathrm{op}}}(\MsA, (\MsA)_\alpha )
\label{bimod1}
\end{equation}
of grade equal to the grade of $\alpha$ in such a way that
the assignment to $\alpha$ of 
$N_{\alpha}$
induces 
an injective 
homomorphism $\Theta_{\BsA}\colon \Out({\BsA},Q) \to \Aut_{\mathcat B_{S,Q}}(A)$
that is compatible with the grades in $Q_A$.
\end{prop}

We postpone the proof after Theorem \ref{fundclass1}
below.

\begin{rema}
Suppose that the group $Q_A$ is finite.
Then, by construction, the $(A,\BsA^{\mathrm{op}})$-module $\MsA$ yields an 
isomorphism $\MsA \colon \BsA^{\mathrm{op}} \to A$ in $\mathcat B_{S,Q}$.
By Proposition \ref{isomsq},
given a $(\BsA^{\mathrm{op}},\BsA^{\mathrm{op}})$-bimodule
$N_x$ of grade $x\in Q$,
the $(A,A)$-bimodule
\begin{equation}
\MsA\otimes_{\BsA^{\mathrm{op}}} N_x \otimes_{\BsA^{\mathrm{op}}} \MsA^* 
\cong \Hom_{{\BsA^{\mathrm{op}}}}(\MsA, \MsA\otimes_{\BsA^{\mathrm{op}}} N_x)
\label{bimod14}
\end{equation}
of grade $x$ spelled out in Proposition \ref{isomsq}
represents the image of $[N_x] \in \Aut_{\mathcat B_{S,Q}}(\BsA^{\mathrm{op}})$
in  $\Aut_{\mathcat B_{S,Q}}(A)$
under the isomorphism 
$(\MsA)_*\colon \Aut_{\mathcat B_{S,Q}}(\BsA^{\mathrm{op}}) \longrightarrow \Aut_{\mathcat B_{S,Q}}(A)$
induced by the isomorphism $[\MsA] \colon \BsA^{\mathrm{op}} \to A$
in $\mathcat B_{S,Q}$.
\end{rema}

\begin{thm}
\label{fundclass1}
The central $S$-algebras ${\BsA}$ and $\BsA^{\mathrm{op}}$ 
are weakly $Q_A$-normal.
Furthermore,
if the
section
$\sigma_A\colon Q_A \to \Aut_{\mathcat B_{S,Q}}(A)$,
of the underlying sets,
for {\rm{\eqref{eApictilde}}} 
is a homomorphism and hence
defines a generalized
$Q_A$-normal structure on $A$,
this section
$\sigma_A$ 
determines
a $Q_A$-normal structure 
\[
\sigma_{\BsA}
\colon Q_A \longrightarrow 
\Out({\BsA},Q)
\]
on ${\BsA}$
such that
\begin{equation*}
\Theta_{\BsA} \circ \sigma_{\BsA} 
=\sigma_A\colon Q_A \to \Aut_{\mathcat B_{S,Q}}(A),
\end{equation*}
and the $Q_A$-normal $S$-algebra 
$(\BsA,\sigma_{\BsA})$
is then determined by
$(A,\sigma_A)$ up to within isomorphism.
\end{thm}

\begin{cor}
\label{2.21}
When the group
$\kappaQ(Q)$ is finite,
the injection $\mathrm B(S,Q) \to \mathrm H^0(Q,\mathrm B(S))$
is the identity. \qed
\end{cor}

\begin{rema}
The main difference between $A$ and ${\BsA}$ is that, while the canonical 
homomorphism from
$\Out(A,Q_A)$ to $Q_A$ is not necessarily surjective,
the (grade) homomorphism $\Out(\BsA,Q_A) \to Q_A$ is surjective,
and $Q_A$ coincides with the image $Q_{\BsA}\subseteq Q$
of the grade homomorphism $\Out({\BsA},Q) \to Q$.
\end{rema}

Given a generalized $Q$-normal Azumaya algebra
$(A,\sigma_A\colon Q \to \Aut_{\mathcat B_{S,Q}}(A))$,
we refer to a $Q$-normal algebra of the kind
\begin{equation*}
(\BsA^{\mathrm{op}},\sigma_{\BsA^{\mathrm{op}}}\colon Q \to \Out(\BsA^{\mathrm{op}},Q))=
({}_A\End(\MsA)^{\mathrm{op}},\sigma_{\BsA^{\mathrm{op}}}),
\end{equation*}
the notation being that in Theorem \ref{fundclass1},
as the $Q$-{\em normal $S$-algebra associated to\/} $(A,\sigma_A)$,
the definite article being justified by the fact that
$(\BsA^{\mathrm{op}},\sigma_{\BsA^{\mathrm{op}}})$
is uniquely determined by 
$(A,\sigma_A)$ up to within isomorphism.

\begin{proof}[Proof of the first assertion of Theorem {\rm \ref{fundclass1}}]
From the group extension \eqref{eApic} of $Q$ by $\Pic(S)$ we deduce 
that,
given $x,y\in Q_A$, there is a projective rank one $S$-module
$J_{x,y}$ and an isomorphism
\begin{equation*}
f_{x,y}\colon M_x \otimes_A M_y \longrightarrow J_{x,y} \otimes M_{xy}
\end{equation*}
of $(A,A)$-bimodules of grade $xy$.
Let $x \in Q_A$. Then
\begin{equation*}
M_x\otimes_A \MsA= \oplus_{y\in Q_A} M_x\otimes_A M_y.
\end{equation*}
Define an isomorphism
\begin{equation*}
\begin{aligned}
\beta_x \colon M_x\otimes_A \MsA &\longrightarrow \oplus_{y\in Q_A}
J_{x,y} \otimes M_{xy}
\end{aligned}
\end{equation*}
of $(A,\BsA^{\mathrm{op}})$-modules by
\begin{equation*}
\beta_x(m_x\otimes m_y)= f_{x,y}(m_x\otimes m_y) \in J_{x,y} \otimes M_{xy},
\ m_x \in M_x, m_y \in M_y,\ y \in Q_A.
\end{equation*}
The canonical isomorphisms
\[
\mathrm{can}_1\colon {\BsA} \longrightarrow
{}_A\End(M_x\otimes_A \MsA),\ 
\mathrm{can}_2 \colon {\BsA} \to
{}_A\End( \oplus_{y\in Q_A}J_{x,y} \otimes M_{xy})
\]
that are induced by the invertible
$(A, \BsA^{\mathrm{op}})$-bimodule structures
are given by
\begin{equation*}
\mathrm{can}_1(f)(m_x \otimes m_y)=m_x \otimes f(m_y),
\
\mathrm{can}_2(f)(j \otimes m_z)\ \, = j \otimes f(m_z),
\end{equation*}
where $f\in {\BsA}={}_A\End(\MsA)$, $m_x \in M_x$, $m_y \in M_y$, 
$m_z \in M_z$, $j \in J_{x,y}$, $y,z \in Q$.
The automorphism 
$\alpha_{x}$
of $\BsA^{\mathrm{op}}$ that makes the diagram 
\begin{equation*}
\begin{CD}
\BsA^{\mathrm{op}}
@>{\mathrm{can}_1}>>
{}_A\End(M_x\otimes_A \MsA)
\\
@V{\alpha_x}VV
@V{\beta_{x,*}}VV
\\
\BsA^{\mathrm{op}}
@>{\mathrm{can}_2}>>
{}_A\End( \oplus_{y\in Q_A}J_{x,y} \otimes M_{xy})
\end{CD}
\end{equation*}
commutative yields an automorphism of $\BsA^{\mathrm{op}}$
that extends the automorphism $\kappaQ(x)$ of $S$.
Indeed, given $\varphi \in{}_A\End(M_x\otimes_A \MsA)$, the value
$\beta_{x,*}(\varphi) \in {}_A\End( \oplus_{y\in Q_A}J_{x,y} \otimes M_{xy})$ 
makes the diagram
\begin{equation*}
\begin{CD}
M_x\otimes_A \MsA
@>{\varphi}>>
M_x\otimes_A \MsA
\\
@V{\beta_{x}}VV
@V{\beta_{x}}VV
\\
\oplus_{y\in Q_A}J_{x,y} \otimes M_{xy}
@>{\beta_{x,*}(\varphi)}>>
\oplus_{y\in Q_A}J_{x,y} \otimes M_{xy}
\end{CD}
\end{equation*}
commutative.
In the standard manner, 
the  assignment to $s\in S$ of the $A$-linear
endomorphism $f_s\colon \MsA \to \MsA$
given by $f_s(m)=sm$ (via the left $A$-module structure on $\MsA$)
as $m$ ranges over $\MsA$
embeds $S$ into ${\BsA}={}_A\End(\MsA)$.
Since, given $y \in Q_A$, the isomorphism
$f_{x,y}$
is one of $(A,A)$-bimodules
of grade $xy$,
given $s \in S$ and $m_x \in M_x$, $m_y \in M_y$, $j \in J_{x,y}$,
we find
\begin{align*}
\mathrm{can}_1(f_s)(m_x \otimes m_y)&=m_x \otimes f_s(m_y)=
m_x \otimes (sm_y) =m_x s \otimes m_y  =({}^x s) m_x  \otimes m_y 
\\
\beta_{x}(\mathrm{can}_1(f_s)(m_x \otimes m_y)) &
=\beta_{x}(({}^x s) m_x  \otimes m_y )= f_{x,y}(({}^x s) m_x  \otimes m_y)
=({}^x s)  f_{x,y}( m_x \otimes m_y)
\\
\mathrm{can}_2(f_{{}^xs})(j \otimes m_z)\ \, &= j \otimes ({}^x s)m_z= 
({}^x s)(j \otimes m_z) .
\end{align*}
Consequently
\[
\beta_{x,*}(\mathrm{can}_1(f_s))= f_{{}^xs} .
\]
Hence the automorphism  $\alpha_{x}$ 
of ${\BsA}$ extends the automorphism
$\kappaQ(x)$ of $S$.
Since $x\in Q_A$ is arbitrary,  the algebras
${\BsA}$ and $\BsA^{\mathrm{op}}$ are weakly $Q_A$-normal.
\end{proof}

We now prepare for the proof of the \lq\lq Furthermore\rq\rq\ statement
of Theorem \ref{fundclass1}.
We view $\MsA$ as 
an
$(A,\BsA^{\mathrm{op}})$-bimodule in the obvious manner.
By assumption, for each $x\in Q_A$, the left $A$-module structure
on $M_x$ induces an isomorphism
$A\to \End_{A^{\mathrm{op}}}(M_x)$  and
the right $A$-module structure
on $M_x$ induces an isomorphism
$A^{\mathrm{op}}\to {}_A\End(M_x)$.
These right $A$-module structures induce an injection
$A^{\mathrm{op}}\to {\BsA}$ of $S$-algebras.

Given $x,y \in Q_A$, the $(A,A)$-bimodule
${}_A\Hom(M_x,M_y) \cong M^*_x \otimes_AM_y $, necessarily  of grade 
$x^{-1}y$, is an $S$-submodule
of ${\BsA}$ 
in an obvious manner.
Let $x,y,z\in Q_A$; given
$h_{y,x}\colon M_x \to M_y$ and
$h_{z,y}\colon M_y \to M_z$,
the composite $h_{z,y}\circ h_{y,x}\colon M_x \to M_z$ is defined
so that,
as $S$-modules,
\begin{equation}
 {\BsA}=\prod_{u\in Q_A} \bigoplus_{v\in Q_A} {}_A\Hom(M_u,M_v)
\cong \prod_{u\in Q_A}{}_A\Hom(M_u,\MsA).
\label{desc3}
\end{equation}
Thus, with a grain of salt, we can think of the members of ${\BsA}$
as being matrices whose columns have only finitely many
non-zero entries. 

\begin{lem}
\label{lem8}
The injection $A \to {}_{\BsA}\End(\MsA)= \End_{\BsA^{\mathrm{op}}}(\MsA)$ given by
the assignment to $a \in A$ of $f_a \in {}_{\BsA}\End(\MsA)$ where $f_a(m)=am$ ($m \in \MsA$)
is an isomorphism.
\end{lem}

In the case where the group $Q_A$ is finite the claim 
of the lemma
is immediate, but for
general $Q_A$ we must be more circumspect.

\begin{proof}
Let $f\colon \MsA \to \MsA$ be ${\BsA}$-linear and, given $x,w,z\in Q_A$, consider 
the restrictions $f \colon M_x \to M_z$ 
and $f \colon M_y \to M_w$. Given $y \in Q$ and
$h_{x,y} \in {}_A\Hom(M_x,M_y)$, since $f$ is ${\BsA}$-linear, the diagram
\begin{equation*}
\begin{CD}
M_x
@>f>>
M_z
\\
@V{h_{x,y}}VV
@V{h_{x,y}}VV
\\
M_y
@>f>>
M_w
\end{CD}
\end{equation*}
is commutative.

Composition of endomorphisms yields a canonical isomorphism
\begin{equation*}
 {}_A\Hom(M_x,M_y)
\otimes_A
 ({}_A\Hom(M_y,M_x))
\longrightarrow
{}_A\End(M_x) \cong A^{\mathrm{op}} 
\end{equation*}
of $(A,A)$-bimodules
whence there are finitely many $h^j_{x,y} \in  {}_A\Hom(M_x,M_y)$
and 
\linebreak
$\tilde h^j_{y,x} \in  {}_A\Hom(M_y,M_x)$
such that
$\sum \tilde h^j_{y,x} \circ  h^j_{x,y} = \mathrm{Id}_{M_x}$. 
Accordingly, the diagram
\begin{equation*}
\begin{CD}
M_x
@>f>>
M_z
\\
@V{\sum \tilde h^j_{y,x} \circ  h^j_{x,y}}VV
@VV{\sum \tilde h^j_{y,x} \circ  h^j_{x,y}}V
\\
M_x
@>f>>
M_z
\end{CD}
\end{equation*}
is commutative. However the right-hand vertical arrow is zero unless $x = z$.
Thus only the components of $f$ of the kind
$M_x \to M_x$ are non-zero and, since $f$ is ${\BsA}$-linear,
relative to the embedding of $A^{\mathrm{op}}$ into ${\BsA}$,
the endomorphism
$f\colon M_x \to M_x$ is $A^{\mathrm{op}}$-linear and hence,
in view of the $S$-algebra isomorphism $A\to \End_{A^{\mathrm{op}}}(M_x)$,
on $M_x$, the endomorphism $f$ is given by left multiplication
by a uniquely determined $a_x \in A$ so that
$f(m)=a_xm$.

Now, given $x,y \in Q_A$, 
and
$h_{x,y} \in {}_A\Hom(M_x,M_y)$, the diagram
\begin{equation}
\begin{CD}
M_x
@>{f_{a_x}}>>
M_x
\\
@V{h_{x,y}}VV
@V{h_{x,y}}VV
\\
M_y
@>{f_{a_y}}>>
M_y
\end{CD}
\label{diag8}
\end{equation}
is commutative.

However, the evaluation map
\begin{equation*}
{}_A\Hom(M_x,M_y) \otimes_A M_x \longrightarrow M_y
\end{equation*}
is an isomorphism. Hence, given $m \in M_y$, there are
$b_1, \ldots, b_k \in {}_A\Hom(M_x,M_y) $ and
$m_1, \ldots, m_k \in M_x$ such that
$m = \sum b_j m_j$.
Since the diagram \eqref{diag8} is commutative, 
\begin{equation*}
f_{a_y}(m) =\sum f_{a_y} b_j m_j  =\sum  b_j f_{a_x} m_j.
\end{equation*}
On the other hand, 
\begin{equation*}
\sum  b_j f_{a_x} m_j = \sum   f_{a_x} b_j m_j =  f_{a_x} \sum b_j m_j
=f_{a_x} (m).
\end{equation*}
Since $m\in M_y$ is arbitrary, we conclude $f_{a_x} = f_{a_y}$,
whence the injection $A \to \End_{\BsA}(\MsA)$ is surjective as asserted.
\end{proof}

\begin{cor}
\label{cor9}
Given an invertible $(A,A)$-bimodule $\Mx$,
the injection 
\begin{equation*}
\iota_{\Mx}\colon 
\Mx \longrightarrow \Hom_{\BsA^{\mathrm{op}}}(\MsA, \Mx\otimes _A \MsA),
\ m \mapsto f_m,
\ f_m(w)= m \otimes w,
\ m \in \Mx,\ w \in \MsA, 
\end{equation*}
is an isomorphism
of $(A,A)$-bimodules.
\end{cor}

\begin{proof}
The injection $\iota_{\Mx}$ is the image of $\mathrm{Id}_{\Mx\otimes _A \MsA}$
under
the adjointness isomorphism
\begin{equation}
\mathrm{ad}\colon \End_{\BsA^{\mathrm{op}}}(\Mx\otimes _A \MsA)
\longrightarrow
\Hom_A(\Mx, \Hom_{\BsA^{\mathrm{op}}}(\MsA, \Mx\otimes _A \MsA)).
\label{adj1}
\end{equation}
Since, as a right $A$-module, $\Mx$ is finitely generated projective,
evaluation yields an isomorphism
\begin{equation}
\mathrm{ev}\colon
\Mx \otimes_A (\Hom_A(\Mx, \Hom_{\BsA^{\mathrm{op}}}(\MsA, \Mx\otimes _A \MsA)))
\longrightarrow \Hom_{\BsA^{\mathrm{op}}}(\MsA, \Mx\otimes _A \MsA)
\label{evalu1}
\end{equation}
of $(A,A)$-bimodules.
In view of Lemma \ref{lem8},
the canonical isomorphism 
\[
\End_{\BsA^{\mathrm{op}}}(\Mx\otimes _A \MsA)\cong  
\End_{\BsA^{\mathrm{op}}}(\MsA),
\]
 combined with the isomorphism
$A \to \End_{\BsA^{\mathrm{op}}}(\MsA)$, identifies
$\End_{\BsA^{\mathrm{op}}}(\Mx\otimes _A \MsA)$ with $A$.
Up to this identification,
the injection $\iota_{\Mx}$ is the composite 
\begin{equation*}
\Mx \otimes_A \End_{\BsA^{\mathrm{op}}}(\Mx\otimes _A \MsA)
\stackrel{\mathrm{ev}\circ(\mathrm{Id}_{\Mx}\otimes_A\mathrm{ad})}\longrightarrow
 \Hom_{\BsA^{\mathrm{op}}}(\MsA, \Mx\otimes _A \MsA) .
\end{equation*}
\end{proof}

\begin{proof}[Proof of Proposition {\rm\ref{prop8}}]
Given $\varphi \in N_{\alpha}=\Hom_{\BsA^{\mathrm{op}}}(\MsA,(\MsA)_{\alpha})$,
$m \in \MsA$, and $s \in S$,
\begin{equation*}
(\varphi \circ s)(m) =\varphi (s(m)); 
\end{equation*}
however, $f_s(m) =sm$ ($m \in \MsA$) defines a member $f_s$ of
$\BsA^{\mathrm{op}}={}_A\End(\MsA)^{\mathrm{op}}$ and, by the definition of $(\MsA)_{\alpha}$,
\[
\varphi(sm)=\varphi(mb_s)=
\varphi (f_s(m))= f_{\alpha(s)}\varphi(m)= \varphi(m)b_{\alpha(s)}
=\alpha(s)\varphi(m)
\]
whence $\varphi \circ s= \alpha(s)\circ \varphi$.
Furthermore,
\begin{align*}
\Hom_{\BsA^{\mathrm{op}}}(\MsA, (\MsA)_\alpha ) \otimes_A 
\Hom_{\BsA^{\mathrm{op}}}((\MsA)_\alpha,\MsA) &\cong \End_{\BsA^{\mathrm{op}}}(\MsA) \cong A,
\\ 
\Hom_{\BsA^{\mathrm{op}}}((\MsA)_\alpha,\MsA) \otimes_A 
\Hom_{\BsA^{\mathrm{op}}}(\MsA, (\MsA)_\alpha) &\cong \End_{\BsA^{\mathrm{op}}}((\MsA)_{\alpha})
\cong A.
\end{align*}
Consequently $N_{\alpha}$ is an invertible $(A,A)$-bimodule
of grade equal to the grade of $\alpha$ as asserted.
We leave the proof of the remaining claims to the reader.
\end{proof}

\begin{proof}[Proof of the \lq\lq Furthermore\rq\rq\ statement of 
Theorem {\rm \ref{fundclass1}}]
Suppose that the given section 
$\sigma_A$
\linebreak
 from $ Q$ to $\Aut_{\mathcat B_{S,Q}}(A)$
is a homomorphism of groups.
Fix $x \in Q_A$. 
The construction of the extension $\alpha_x\colon {\BsA} \to {\BsA}$
of the automorphism $\kappaQ(x)$ of $S$
in the proof of the first statement
simplifies since now the constituents $J_{x,y}$ 
are trivial, i.~e., copies of the commutative ring $S$.
Thus, given $y\in Q_A$, there is an isomorphism
\begin{equation*}
f_{x,y}\colon M_x \otimes_A M_y \longrightarrow M_{xy}
\end{equation*}
of (invertible) $(A,A)$-bimodules.
Then
\begin{equation*}
M_x\otimes_A \MsA= \oplus_{y\in Q_A} M_x\otimes_A M_y.
\end{equation*}
Define an isomorphism
\begin{equation*}
\beta_{x} \colon M_x\otimes_A \MsA \longrightarrow \oplus_{y\in Q_A} M_{xy} = \MsA,
\end{equation*}
of $(A,\BsA^{\mathrm{op}})$-modules by
\begin{equation*}
\beta_{x}(m_x\otimes m_y)= f_{x,y}(m_x\otimes m_y) \in M_{xy},
\ 
m_x \in M_x,\, m_y \in M_y,\,y \in Q_A.
\end{equation*}
The canonical isomorphism
\[
\mathrm{can}_x\colon {\BsA} \longrightarrow
{}_A\End(M_x\otimes_A \MsA) 
\]
that is induced by the invertible
$(A, \BsA^{\mathrm{op}})$-bimodule structure on $M_x\otimes_A \MsA$
is given by
\begin{equation*}
\mathrm{can}_x(f)(m_x \otimes m_y)=m_x \otimes f(m_y),\ m_x \in M_x,\ m_y \in M_y,\ f\in {\BsA}={}_A\End(\MsA).
\end{equation*}
The automorphism $\alpha_{x}$ of ${\BsA}$ that makes the diagram 
\begin{equation*}
\begin{CD}
{\BsA}
@>{\mathrm{can}_x}>>
{}_A\End(M_x\otimes_A \MsA)
\\
@V{\alpha_{x}}VV
@V{\beta_{x,*}}VV
\\
{\BsA}
@=
{}_A\End(\MsA)
\end{CD}
\end{equation*}
commutative yields an automorphism of ${\BsA}$
that extends the automorphism $\kappaQ(x)$ of $S$.
Now, given $\varphi \in {}_A\End(M_x\otimes_A \MsA)$,
the value $\beta_{x,*}(\varphi)$ makes the diagram
\begin{equation*}
\begin{CD}
M_x \otimes_A \MsA
@>{\varphi}>>
M_x \otimes_A \MsA
\\
@V{\beta_x}VV
@V{\beta_x}VV
\\
\MsA
@>{\beta_{x,*}(\varphi)}>>
\MsA
\end{CD}
\end{equation*}
commutative.
Hence, given $b \in {\BsA}={}_A\End(\MsA)$,
\begin{equation}
(\alpha_x(b))(\beta_x(m_x \otimes m)) = \beta_x(m_x\otimes bm).
\label{fund5}
\end{equation}
Consequently, the ${\BsA}$-module structure on $M_x \otimes _A \MsA$ being given by
\begin{equation*}
{\BsA} \otimes (M_x \otimes _A \MsA) \longrightarrow M_x \otimes _A \MsA,
\ 
f \otimes m_x \otimes m \longmapsto m_x \otimes f(m),
\end{equation*}
the isomorphism $\beta_x$ 
is one of left ${\BsA}$-modules from
$M_x \otimes _A \MsA$ onto ${}_{\alpha_x} \MsA$ or, equivalently,
 one of right $\BsA^{\mathrm{op}}$-modules from
$M_x \otimes _A \MsA$ onto $(\MsA)_{\alpha_x}$.
We define the value 
\linebreak
$\sigma_{\substack{\mbox{\tiny{$\BsA$}}}  }(x)\in \Out({\BsA},Q)$ to be the
class of $\alpha_x$ in $\Out({\BsA},Q)$.

By Proposition \ref{prop8}, the homomorphism 
$\Theta_{\BsA}\colon \Out({\BsA},Q) \to \Aut_{\mathcat B_{S,Q}}(A)$
sends the class of $\alpha_x$ to the isomorphism class of the
invertible $(A,A)$-bimodule 
\[
N_{\alpha_x}=\Hom_{\BsA^{\mathrm{op}}}(\MsA,(\MsA)_{\alpha_x})
\cong
\Hom_{\BsA^{\mathrm{op}}}(\MsA,M_x \otimes _A \MsA)
\cong M_x, 
\]
the last isomorphism being a consequence of
Corollary \ref{cor9}.
Hence 
$\Theta_{\BsA}$ sends the class of $\alpha_x$ to the isomorphism class of the
$(A,A)$-bimodule $M_x$. Consequently
$\Theta_{\BsA} \circ \sigma_{\BsA} = \sigma_A$.
Furthermore, since $\Theta_{\BsA}$ is injective and since
$\sigma_A$ is a homomorphism, 
the section $\sigma_{\substack{\mbox{\tiny{$\BsA$}}}  }$, at first one of the underlying sets,
is a homomorphism. 

We leave the proof of the claim that the
 $Q_A$-normal $S$-algebra 
$(\BsA,\sigma_{\BsA})$
is determined by
$(A,\sigma_A)$ up to within isomorphism
to the reader.
\end{proof}

\subsection{The Teichm\"uller class of a generalized $Q$-normal
Azumaya algebra}
\label{genazalg}

Let $(A,\sigma_A\colon Q \to \Aut_{\mathcat B_{S,Q}}(A))$
be a generalized $Q$-normal Azumaya algebra, and let
\begin{equation*}
(\BBB^{\mathrm{op}},\sigma_{\BBB^{\mathrm{op}}})= (\BsA^{\mathrm{op}},\sigma_{\BsA^{\mathrm{op}}}\colon Q \to \Out(\BsA^{\mathrm{op}},Q))=
({}_A\End(\MsA)^{\mathrm{op}},\sigma_{\BsA^{\mathrm{op}}})
\end{equation*}
be the $Q$-normal $S$-algebra associated to $(A,\sigma_A)$.
We then refer to the Teichm\"uller complex 
$\mathrm e_{(\BBB^{\mathrm{op}},\sigma_{\BBB^{\mathrm{op}}})}$
of the $Q$-normal $S$-algebra $(\BBB^{\mathrm{op}},\sigma_{\BBB^{\mathrm{op}}})$,
cf. \eqref{pb1}, as the {\em Teichm\"uller complex\/}
of $(A,\sigma_A)$
 and to
the class
$[\mathrm e_{(\BBB^{\mathrm{op}},\sigma_{\BBB^{\mathrm{op}}})}] \in \mathrm H^3(Q,\mathrm U(S))$
as the {\em Teichm\"uller class\/}
of $(A,\sigma_A)$.

\begin{thm}
\label{genaz}
The assignment to a generalized $Q$-normal Azumaya algebra
$(A,\sigma_A)$ of its Teichm\"uller class
$[\mathrm e_{(\BBB^{\mathrm{op}},\sigma_{\BBB^{\mathrm{op}}})}] \in \mathrm H^3(Q,\mathrm U(S))$
yields a homomorphism 
\begin{equation}
t \colon k\mathcat{Rep}(Q,\mathcat B_{S,Q}) \longrightarrow \mathrm H^3(Q,\mathrm U(S))
\label{teich1}
\end{equation}
of abelian groups such that, when 
the generalized $Q$-normal Azumaya $S$-algebra
$(A,\sigma_A)$ 
arises from an ordinary
$Q$-normal algebra structure $\sigma\colon Q \to \Out(A,Q)$ on $A$,
\begin{equation}
[\mathrm e_{(\BBB^{\mathrm{op}},\sigma_{\BBB^{\mathrm{op}}})}] 
=
[\mathrm e_{(A,\sigma)}]
\in \mathrm H^3(Q,\mathrm U(S)) .
\label{identity1}
\end{equation}

\end{thm}

\begin{rema}
A variant of the homomorphism $t$, written there as $\chi$, 
is given in 
\cite[Theorem 3.4 (ii)]{MR1803361}
by the assignment to a generalized $Q$-normal Azumaya $S$-algebra
of an explicit 3-cocycle of $Q$ with values in $\mathrm U(S)$.
The identity \eqref{identity1} is equivalent to the statement of
\cite[Theorem 3.4 (iii)]{MR1803361}.
\end{rema}

\begin{proof}[Proof of the first assertion Theorem {\rm \ref{genaz}}]
Proposition \ref{2.4.3}, suitably adjusted to the present situation,
entails that $t$ is a homomorphism. We leave the details to the reader.
\end{proof}

We postpone the proof of the second assertion of Theorem \ref{genaz}
to  Subsection \ref{ptgenaz} below.

\section{Crossed products with normal algebras}
\label{three}

As before,
$Q$  denotes a group, $\kappaQ\colon Q \to \Aut(S)$ an action of $Q$ on 
the commutative ring $S$,
and $R$ the subring $R = S^Q \subset S$ of $S$ that consists
of the elements fixed under $Q$.
In this section, we  explore certain crossed products of $Q$ with
$Q$-normal $S$-algebras. 
The special case  that involves certain algebras 
over fields is due to 
Teich\-m\"uller \cite{MR0002858}, 
the corresponding crossed product algebras being referred to in
\cite{MR0002858} as \lq\lq Normalringe\rq\rq.

Consider a central $S$-algebra $A$.
Suppose that there is a group
extension
\begin{equation}
\mathrm e_Q \colon 1 \longrightarrow \Egamma \stackrel{j}\longrightarrow 
\Gamma \stackrel{\pi}
\longrightarrow Q \longrightarrow 1
\label{getw}
\end{equation}
together with a morphism 
\begin{equation}
(i,\vartheta) \colon (\Egamma,\Gamma , j) \longrightarrow (\mathrm U(A), \Aut (A), \partial )
\label{morthr}
\end{equation}
of crossed modules having $i$ injective.
Such a morphism of crossed modules, in turn, induces a $Q$-normal structure
${\sigma_{\vartheta}}\colon Q \to \Aut(A)$ on $A$.
Conversely, given
a $Q$-normal structure 
${\sigma\colon Q \to \Out(A)}$ on $A$,
an extension of the kind \eqref{getw} 
with $\Egamma=\mathrm U(A)$
together with a morphism of crossed
modules of the kind \eqref{morthr} 
inducing $\sigma$
exists if and only if the Teich\-m\"uller 
class $[\mathrm e_{(A,\sigma)}] \in \mathrm H^3(Q, \mathrm U(S))$
of $(A,\sigma)$ is zero, cf. the proof of Proposition \ref{2.6.5}
as well as Remark \ref{rema1} above.
Thus we consider a $Q$-normal algebra
$(A,\sigma)$ having
zero Teich\-m\"uller class and  fix an extension of the kind
\eqref{getw} together with a morphism of crossed modules of the kind
\eqref{morthr}.

We remind the reader that the second cohomology group
$\mathrm H^2(Q, Z_\Egamma)$
of $Q$ with values in the center $Z_\Egamma$ 
of $\Egamma$ acts faithfully and transitively
on the congruence classes of extensions of the kind 
\eqref{getw} that have the same outer action $Q \to \Out(\Egamma)$, 
cf. \cite[Section IV.8 Theorem 8.8 p.~128]{maclaboo}.

\subsection{First crossed product algebra construction}
\label{first}
Relative to the action of $\Gamma$ on $A$ given by the homomorphism
$\vartheta\colon \Gamma \to \Aut(A)$,
let $A^t\Gamma $ be the twisted group ring of $\Gamma$ with twisted 
coefficients in $A$, and let
$<y-j(y),y\in \Egamma>$ denote the two-sided ideal in 
$A^t\Gamma $ generated by
the elements  $y-j(y)\in A^t \Egamma\,(\subseteq A^t\Gamma)$ 
as $y$ ranges over $\Egamma$.
Define the algebra $(A,Q,\mathrm e_Q,\vartheta)$ to be the quotient algebra
\begin{equation*}
(A,Q,\mathrm e_Q,\vartheta) = A^t \Gamma / <y-j(y),y\in \Egamma>.
\end{equation*}
It is obvious that the ring $R\,(=S^Q)$ lies in the center
of $(A,Q,\mathrm e_Q,\vartheta)$ 
whence  $(A,Q,\mathrm e_Q,\vartheta)$ is
an $R$-algebra.
We refer to the algebra $(A,Q,\mathrm e_Q,\vartheta )$ as the {\em crossed product}
of $A$ and $Q$, {\em with respect to $\mathrm e_Q$ and} $\vartheta$.

\subsection{Second crossed product algebra construction}
\label{second}

Relative to the group extension \eqref{getw},
let $v \colon Q \to  \Gamma $ be a section for $\pi$  of the underlying sets,
i.e., a section which is not necessarily a homomorphism.
Assume for convenience that $v(1) = 1$ and, for $q\in Q$, write
$v_q = v(q)$. Let $\varphi \colon Q\times Q \to \Egamma$ be 
a corresponding normalized
2-cocycle  relative to 
$v$ so that
\begin{equation*}
v_p v_q = \varphi (p,q)v_{pq}, \ p,q\in Q.
\end{equation*}
Then the {\em crossed product\/} $(A,Q,\mathrm e_Q,\vartheta )$ of $A$ and $Q$,
{\em with respect to the\/} 2-{\em cocycle\/} $\varphi $ 
{\em and the homomorphism\/} 
$\vartheta \colon \Gamma \to \Aut(A)$, is the algebra
having
$\oplus_Q Av_q$
as its underlying left $A$-module and whose multiplicative structure 
is given by
\begin{equation*}
v_q a = ({}^{\vartheta(v_q)}a)v_q,\  v_pv_q = \varphi (p,q)v_{pq},
\ a\in A,\ p, q\in Q.
\end{equation*}
In the special case where the group $Q$ is finite,  
where $A = S$, and where $S|R$ is
a Galois extension of commutative rings with Galois 
group $Q$, that notion of crossed
product comes down to the ordinary crossed product algebra.
The more general construction of $(A,Q,\varphi , \vartheta)$ was given,
in the classical situation 
(i.e., $S$ a field etc.),
in \cite[p.~145]{MR0002858} as well as 
in
\cite[p.~9]{MR0025443}, 
and for Azumaya algebras 
in \cite[p.~13]{MR0311701} where 
$S| R$ 
was still assumed to be a Galois extension of commutative rings with Galois
group $Q$.

\subsection{Equivalence of the two crossed product algebra constructions}

The two notions of crossed product algebra given above are equivalent:

\begin{prop}
\label{3.1}
The association 
\begin{equation}
(A,Q,\mathrm e_Q,\vartheta ) \longrightarrow (A,Q,\varphi , \vartheta ),
\ 
x \longmapsto (x v^{-1}_{\pi(x)})v_{\pi(x)},\ x \in \Gamma,
\label{mor4}
\end{equation}
where $x v^{-1}_{\pi(x)} \in \Egamma$
is to be interpreted as a member of $A$,
yields a morphism of algebras and, likewise,
the association
$v_q \mapsto v_q$ ($q \in Q$) 
yields a morphism
${(A,Q,\varphi, \vartheta ) \to (A,Q, \mathrm e_Q, \vartheta )}$
of algebras; 
these morphisms
preserve the structures and are inverse to each
other.  Hence the 
section $v \colon Q\to \Gamma $, in the category of sets, 
for the surjection $\pi$ in the extension {\rm \eqref{getw}}
yields a basis of the left $A$-module that underlies
the crossed product algebra $(A,Q,\mathrm e_Q,\vartheta )$.
\end{prop}

\begin{proof}
The morphism \eqref{mor4}
of algebras is plainly well defined and surjective.

By construction, the
$A^t\Egamma$-module that underlies the $R$-algebra $A^t\Gamma$
is free, having as basis the family $\{v_q; q \in Q\} \subseteq \Gamma$. 
Furthermore, 
write the injection $\Egamma \to A$ as $u \mapsto a_u$ 
($u \in \Egamma$);
the kernel of the canonical surjection
$A^t\Egamma \to A$
of algebras given by the association
$A^t\Egamma \ni au \mapsto aa_u \in A$
is the two-sided ideal in $A^t\Egamma$
generated by
the elements  $y-j(y)\in A^t\Egamma$ as $y$ ranges over $\Egamma$.
Now, the $A$-module that underlies the algebra
$(A,Q,\mathrm e_Q,\vartheta )$
arises from $A^t\Gamma$, viewed as
an $A^t\Egamma$-module, 
through the surjection
$A^t\Egamma \to A$ of algebras.
By construction,
as an $A$-module,
the algebra
$(A,Q,\varphi,\vartheta )$
has likewise the family $\{v_q; q \in Q\} \subseteq \Gamma$
as basis 
whence \eqref{mor4}
is an isomorphism of algebras.
\end{proof}

\begin{rema}
The classical fact that, up to within
isomorphism, 
the crossed product algebra depends only on the congruence 
class of the corresponding
group extension extends to our more general situation in an obvious
way;
we leave the details to the reader.
\end{rema}

\subsection{Properties of the crossed product algebra}

We  write  $(A,Q,\mathrm e_Q,\vartheta )$ or  $(A,Q,\varphi , \vartheta )$ 
according as
which construction of the crossed product algebra 
is more convenient for the particular situation under discussion. 
In the special case
where $A = S$, the action $\vartheta\colon \Gamma \to \Aut(S)$
of $\Gamma$ on $S$
necessarily coincides with the composite $\kappaQ \circ \pi$. 
Proofs of the statements below
are straightforward; we leave most of them to the reader.

\begin{prop}
\label{threet}
{\rm(i)}
The algebra $A = Av_1$ is a subalgebra of the crossed product algebra
 $(A,Q,\varphi , \vartheta)$
and lies in the centralizer of $S$. 

\noindent {\rm(ii)}
If
$S| R$ is a Galois extension of commutative rings
with Galois group $Q$, the algebra $A$ coincides with the centralizer of $S$.
In particular,
when $A = S$, this comes down to the familiar
fact that $S$ is a maximal commutative subring of 
the algebra $(S,Q,\varphi,\kappaQ\circ \pi )$.

\noindent {\rm(iii)}
If $S| R$ is a Galois extension of commutative rings with Galois group $Q$,
the ring $R$ coincides with the center of $(A,Q,\varphi , \vartheta )$.

\noindent {\rm(iv)}
The group $\Gamma $ embeds canonically into the
normalizer $N^{\mathrm U(A,Q,\varphi,\vartheta)}(A)$ of $A$ 
in 
$\mathrm U(A,Q,\varphi,\vartheta)$ so that, given 
$a = av_1 \in A$ and $x \in \Gamma $,
\begin{equation*}
x a x^{-1} = {}^{\vartheta (x)} a.
\end{equation*}
In particular, each automorphism $\kappaQ(q)$ of $S = S v_1$,
as $q$ ranges over $Q$, extends to an inner automorphism of 
$(A,Q,\varphi , \vartheta )$
that normalizes $A$. If $S| R$ is a Galois extension of commutative 
rings with Galois group $Q$ and if the injection $i\colon \Egamma \to \mathrm U(A)$
is an isomorphism of groups,
then ${\Gamma = N^{\mathrm U(A,Q,\varphi,\vartheta)}(A)}$.
\end{prop}

\begin{proof}[Proof of {\rm (ii)}]
Let $x = \sum a_q v_q \in (A,Q,\varphi , \vartheta)$
and suppose that $xs = sx$ 
for each $s\in S$. This implies that $a_qv_q s = sv_q a_q$
and therefore 
$({}^qs-s)a_q = 0$, for any $s\in S$ and $q\in Q$. 
Consider $q \in Q$ distinct from
$e\in Q$. If $a_q \neq 0$ then $S a_q \subset A$
is a cyclic $S$-submodule; let $J$ denote the annihilator ideal of 
$Sa_q$. Since $S| R$ is assumed to be a 
Galois extension of commutative rings with Galois group $Q$,
by characterization (iii) of a Galois extension
in Subsection \eqref{galext} above, 
there is an $s\in S$  with ${}^qs-s \not\in J$, and so
$({}^q s-s)a_q \ne  0$.
\end{proof}

To simplify the notation, we denote by
$\BbB$ the left $A$-module that underlies the crossed product algebra
$(A,Q,\mathrm e_Q, \vartheta )$; by construction, 
as an $S$-module, $\BbB \cong \oplus_Q Av_q$.
The crossed product algebra structure on $\BbB$
yields some additional structure on the central
$S$-algebra ${}_A\End (\BbB) \cong \mathrm M_{\vert Q\vert }(A^{\mathrm{op}})$:

\begin{prop}
\label{tpf}
{\rm (i)} 
With respect to the action of the group $\Gamma$
on $S$ via the combined map $\kappaQ\circ \pi\colon \Gamma \to Q \to \Aut(S)$,
the association
\begin{equation}
\Gamma \times \BbB \longrightarrow \BbB,\ 
(x , b) \longmapsto x b \in (A,Q,\varphi , \vartheta ),\ 
x \in \Gamma,\ b\in \BbB,
\label{assoc9}
\end{equation}
yields an $S^t\Gamma $-module structure on $\BbB$.

\noindent{\rm (ii)}
The induced action $\beta_1 \colon \Gamma \to \Aut ({}_A\End (\BbB))$ 
on ${}_A\End (\BbB)$ given
by 
\begin{equation*}
({}^{\beta_1(x) } f)b = x f(x^{-1} b),\  x \in \Gamma ,
\ b\in \BbB, \ f\in {}_A\End(\BbB),
\end{equation*}
is trivial on $\Egamma \subset \Gamma $ and 
hence induces, on ${}_A\End (\BbB)$,  a $Q$-equivariant
structure 
\begin{equation}
\tau_{\mathrm e_Q}\colon Q \longrightarrow \Aut({}_A\End (\BbB)) .  
\label{tauo}
\end{equation}

\noindent
{\rm (iii)}
Setting
\begin{equation*}
{}^x (av_q) = ({}^{\vartheta(x) } a)v_q,
\ a\in A, 
\ x \in \Gamma,
\ q\in Q, 
\end{equation*}
we obtain another $S^t\Gamma $-module structure on $\BbB$.

\noindent
{\rm (iv)}
The action $\beta_2 \colon \Gamma \to \Aut ({}_A\End (\BbB))$ on ${}_A\End (\BbB)$ 
induced
by the $S^t\Gamma $-module structure in  
{\rm (iii)}
is given by the identity
\begin{equation*}
({}^{\beta_{2}(x )}f)(av_q) =
a\,{}^x (f(v_q)), 
\ x \in \Gamma , 
\ a\in A, 
\ f\in {}_A\End (\BbB), q\in Q,
\end{equation*}
and induces a $Q$-normal structure 
$\sigma_{\sta}
\colon Q \to \Out({}_A\End (\BbB))$
on ${}_A\End (\BbB)$ which,
under the isomorphism
 ${}_A\End (\BbB) \cong \mathrm M_{\vert Q\vert}(A^{\mathrm{op}})$ 
(cf. Subsections~{\rm \ref{2.4}} and {\rm \ref{2.4.22}} above),
corresponds to the $Q$-normal structure
$\sigma^{\mathrm{op}}_{\vartheta,\vert Q\vert }$
on $\mathrm M_{\vert Q\vert}(A^{\mathrm{op}})$, 
i.~e., to the $Q$-normal structure  on $\mathrm M_{\vert Q\vert}(A^{\mathrm{op}})$
induced by the $Q$-normal structure ${\sigma_{\vartheta}\colon Q \to \Aut(A)}$
on $A$.

\noindent
{\rm (v)}
Given  $x \in \Gamma$,
the association 
$v_q \longmapsto x v_q \in (A,Q,\varphi , \vartheta)$,
as $q$ ranges over $Q$,
yields an 
$A$-linear automorphism $i_x \colon \BbB \to \BbB$, i.e., a member of
\begin{equation*}
\mathrm U({}_A\End (\BbB)) = {}_A\Aut(\BbB) 
\cong \mathrm{GL}_{\vert Q\vert }(A^{\mathrm{op}}).
\end{equation*}

\noindent
{\rm (vi)}
The two $S^t\Gamma $-structures on $\BbB$ given in 
{\rm (i)} and {\rm (iii)}
are related by the identity
\begin{equation*}
x b = {}^x(i_x (b)), 
\ x \in \Gamma , 
\ b\in \BbB.
\end{equation*}

\noindent
{\rm (vii)}
Given $x \in \Gamma $ and $f\in \End_A(\BbB)$,
\begin{equation*}
{}^{\beta_2 (x^{-1})}({}^{\beta_1 (x)}f) = i_x f i^{-1}_x .
\end{equation*}
Thus, in view of  
{\rm (ii)}, {\rm (iv)}, and {\rm (v)},
the $Q$-normal structure 
$\sigma_{\sta}$
on ${}_A\End (\BbB)$ factors through the
$Q$-equivariant
structure  
{\rm \eqref{tauo}}
on ${}_A\End (\BbB)$ and is therefore $Q$-equivariant.

\noindent
{\rm (viii)}
Given $u \in (A,Q,\varphi, \vartheta)$, define $f_u\in \End_A(\BbB)$ by
$f_u (b) = bu \in (A,Q, \varphi , \vartheta)$,
 for $b\in \BbB$.  
The assignment to 
$u^{\mathrm{op}}\in (A,Q,\varphi , \vartheta)^{\mathrm{op}}$  of $f_u$ 
yields an injection 
\begin{equation*}
(A,Q,\varphi , \vartheta)^{\mathrm{op}} \longrightarrow \End_A(\BbB) 
\end{equation*}
which
identifies the algebra
$(A,Q,\varphi , \vartheta)^{\mathrm{op}} $ with 
the subalgebra ${}_A\End (\BbB)^Q$ of  ${}_A\End (\BbB)$, i.~e., 
with the subalgebra that consists 
of the elements fixed under the $Q$-equivariant structure 
{\rm \eqref{tauo}}.

\noindent
{\rm (ix)}
If $S|R$ is a Galois extension of commutative rings with Galois
group $Q$, then the obvious map 
\begin{equation*}
\alpha \colon S\otimes_R (A,Q,\varphi , \vartheta)^{\mathrm{op}} \longrightarrow {}_A\End (\BbB)
\end{equation*}
given by 
$[\alpha (s\otimes u^{\mathrm{op}})]b = sbu$, 
for $s\in S$, 
 $b\in \BbB$,
$u \in (A, Q, \varphi, \vartheta)$, 
is an isomorphism of $S$-algebras as well as,
relative to the $Q$-equivariant structure {\rm \eqref{tauo}},
an isomorphism of $S^tQ$-modules.
In particular, when $A = S$, this statement 
recovers the familiar fact that 
the crossed product $R$-algebra
$(S,Q,\varphi,\kappaQ\circ \pi )$ is split by $S$.

\noindent
{\rm (x)}
If $S|R$ is a Galois extension of commutative rings
with Galois group $Q$, then the induced
isomorphism
\begin{equation*} 
\alpha_{\sharp} \colon \Aut (S\otimes_R (A, Q, \varphi , \vartheta)^{\mathrm{op}}) \longrightarrow \Aut ({}_A\End (\BbB))
\end{equation*}
identifies the obvious $Q$-equivariant structure on 
$S\otimes_R (A, Q, \varphi , \vartheta)^{\mathrm{op}}$ which comes from scalar extension
with the $Q$-equivariant structure 
{\rm\eqref{tauo}} on ${}_A\End (\BbB)$.
Consequently the $Q$-normal algebra
${({}_A\End (\BbB),\sigma_{\sta})}$
then arises from the
$R$-algebra $(A,Q,\varphi,\vartheta)^{\mathrm{op}}$ by scalar extension.

\noindent
{\rm (xi)}
If $S|R$ is a Galois extension of commutative rings with Galois group
$Q$ and if $A$ is an Azumaya $S$-algebra, then $(A,Q,\varphi , \vartheta )$
is an  Azumaya $R$-algebra.
\end{prop}

Concerning statement (x), we note that 
statement (vii)
already implies that that the $Q$-normal algebra 
$({}_A\End (\BbB), 
\sigma_{\sta})$
arises by scalar extension but  statement (x) is more precise.

\begin{proof} 
(vii):
Let $a \in A$, $x \in \Gamma$ and $q \in Q$. In view of the definitions of the various actions,
\begin{align*}
\left({}^{\beta_2 (x^{-1})}({}^{\beta_1 (x)}f)\right)(av_q) &= 
a\, {}^{x^{-1}}\negthickspace\left(({}^{\beta_1 (x)}f)(v_q)\right)
=a\, {}^{x^{-1}}\negthickspace \left(x f(x^{-1}v_q)\right)
\\
&=
a (i_x f i^{-1}_x)(v_q) 
\end{align*}
since, with $b= f(x^{-1}v_q) \in \BbB$,
 in view of (vi),
\begin{equation*}
 {}^{x^{-1}}\negthinspace(x b) =i_x(b).
\end{equation*}

\noindent
(viii): The action of $Q$ on ${}_A\End(\BbB)$ in 
(ii) is given 
by 
\begin{equation*}
({}^q f)(b) = x f(x^{-1} b),\ f\in {}_A\End(\BbB),\
q \in Q,
\end{equation*}
 where $x \in \Gamma $ is a 
pre-image of $q\in Q$; hence given $u \in (A,Q,\varphi , \vartheta)$,
\begin{equation*}
({}^qf_u)(b) = x f_u(x^{-1} b)
=
xx^{-1}bu = bu = f_u b.
\end{equation*}

On the other hand, let $f \colon \BbB \to \BbB$ be an $A$-linear 
endomorphism of $\BbB$ that is
fixed under $Q$.
Let $q\in Q$ and choose a pre-image $x \in \Gamma $ of $q$. Then,
for any $a\in A$,
\begin{equation*}
f(av_q) = a f(v_q) = ax f(x^{-1}v_q)
        = ax^{-1}x v_q f(1) = a v_q f(1).
\end{equation*}
Hence $f$ is given by multiplication in $(A,Q,\varphi, \vartheta)$ by $f(1)$.
The argument really comes down to the standard fact that, for any
ring $\Lambda$, the algebra  ${}_{\Lambda}\End(\Lambda ) $ is canonically
isomorphic to $\Lambda^{\mathrm{op}}$.

\noindent
(ix): This follows from  
Proposition \ref{tpf}(viii)
since, by Galois descent, cf. Subsection ~\ref{galext}~(ii), the canonical map
$S\otimes_R {}_A\End (\BbB)^Q \to {}_A\End (\BbB)$ is an isomorphism
of $S$-algebras. 

\noindent
(xi):
In view of 
Proposition \ref{tpf}(ix), 
$S \otimes_R (A,Q,\varphi, \vartheta)$ is an  Azumaya $S$-algebra,
i.e., a central separable $S$-algebra. By  
\cite[Corollary A.5 p.~398]{MR0121392},
the ring $S$ is separable over $R$ and therefore, by
 \cite[Theorem 2.3~p.~374]{MR0121392},
the algebra
 $S\otimes_R (A,Q,\varphi ,\vartheta)$ is
separable over $R$. By  \cite[Proposition A.3]{MR0121392}, 
as an $R$-module, the ring $R$ is a direct summand
of $S$  and hence, by  
\cite[Proposition 1.7]{MR0121392},
$(A,Q,\varphi, \vartheta )$ is separable over $R$.  
By Proposition \ref{threet}(iii) above,
the ring $R$ coincides with 
the center of $(A, Q, \varphi , \vartheta )$. 
\end{proof}

\begin{rema} 
Proposition \ref{tpf}(xi)
is related with  
\cite[Corollary 5.2 p.~13]{MR0311701} 
and with 
\cite[Theorem~4 p.~110]{MR0168605}.
\end{rema}

\begin{rema}
\label{specialAS}
Consider the special case where $(A,\sigma)$ is the ground ring $S$, 
endowed with 
the given $Q$-action $\kappaQ\colon Q \to \Aut(S)$, and
suppose that the group $Q$ is finite and that
$\kappaQ$ is injective so that $Q$ is a finite group of operators on $S$.
Then the above crossed product construction is classical,
and the assignment to a group extension 
$\mathrm e \colon \mathrm U(S)\rightarrowtail \Gamma 
\stackrel{\pi^{\Gamma}}\twoheadrightarrow Q$
of its associated crossed product algebra
$(S,Q,\mathrm e, \kappaQ\circ \pi^{\Gamma})$
yields the familiar homomorphism
\begin{equation}
\mathrm H^2(Q,\mathrm U(S))
\longrightarrow
\mathrm B(S|R)
\label{ordinarycp}
\end{equation}
into the subgroup $\mathrm B(S|R)$ of the Brauer group
$\mathrm B(R)$ of $R=S^Q$ that consists of the Brauer classes split by $S$.
When $S$ is, furthermore, 
a field, the homomorphism \eqref{ordinarycp}
is an isomorphism.
In that case,
the inverse of \eqref{ordinarycp} is induced by the assignment to 
a central simple $R$-algebra $A$ having $S$ as a maximal commutative subfield
of the group extension
\begin{equation}
\begin{CD}
\mathrm e_A\colon
0
@>>>
\mathrm U(S)
@>>>
N^{\mathrm U(A)}(S)
@>>>
Q
@>>>
1.
\end{CD}
\end{equation}
Here, similarly as before, 
$N^{\mathrm U(A)}(S)$ designates the normalizer
of $S$ in the group $\mathrm U(A)$ of invertible elements of $A$, 
and the canonical homomorphism
$N^{\mathrm U(A)}(S)\to Q$ is surjective, in view of the
Skolem-Noether theorem. Thus, when $S|R$ is a Galois extension
of general 
commutative rings with Galois group $Q$, given an Azumaya $R$-algebra
$A$ that contains $S$ as a maximal commutative subalgebra, the algebra
$A$ can be written as a crossed product algebra with respect to the $Q$-action
on $S$ if and only if 
the canonical homomorphism
$N^{\mathrm U(A)}(S)\to Q$ is surjective.
\end{rema}

\begin{rema}
Formally similar crossed products have been explored in the literature
in the framework of $\mathrm C^*$-algebras.
See \cref{dynamical}
for details.
\end{rema}

\section{Normal algebras with zero Teich\-m\"uller class}
\label{nazet}

As before, $S$ denotes a commutative ring and $\kappaQ\colon Q \to \Aut(S)$
an action of a group $Q$ on $S$.

\begin{thm}
\label{fouro}  
The Teich\-m\"uller class
of a $Q$-normal $S$-algebra $(A,\sigma )$
is zero if and only if, for $I = Q$, the $Q$-normal structure $\sigma_I$
on the matrix algebra $\mathrm M_I(A)$ is equivariant.
\end{thm}

\begin{proof}
In view of Propositions \ref{2.3.1} and \ref{2.4.2},
the condition is sufficient.

To show that the condition is necessary, consider a $Q$-normal algebra
$(A,\sigma )$
having zero Teich\-m\"uller class;
then
 so has the opposite algebra $(A^{\mathrm{op}}, \sigma^{\mathrm{op}})$,
by Proposition \ref{2.4.1}. 
Consider the crossed module
$(\mathrm U(A^{\mathrm{op}}),B^{\sigma^{\mathrm{op}}},\partial)$ 
which defines the crossed 2-fold extension 
$\mathrm e_ {(A^{\mathrm{op}},\sigma^{\mathrm{op}})}$
associated to 
$(A^{\mathrm{op}}, \sigma ^{\mathrm{op}})$.
By 
\cite[\S 10 Theorem]{crossed}, there is a group extension
\begin{equation*}
\mathrm e \colon 1 \longrightarrow \mathrm U(A^{\mathrm{op}}) \stackrel{j}\longrightarrow\Gamma \longrightarrow Q \longrightarrow 1
\end{equation*}
 together with a
morphism
\begin{equation*}
(1,\hat \vartheta) \colon (\mathrm U(A^{\mathrm{op}}), \Gamma ,j) 
\longrightarrow 
(\mathrm U(A^{\mathrm{op}}), B^{\sigma^{\mathrm{op}}}, \partial)
\end{equation*}
of crossed modules inducing the identity map of $Q$.
Via the canonical homomorphism from
$B^{\sigma^{\mathrm{op}}}$ to $\Aut(A^{\mathrm{op}})$,
the morphism $(1,\hat \vartheta)$ of crossed modules yields a morphism
\begin{equation*}
(1,\vartheta) \colon (\mathrm U(A^{\mathrm{op}}), \Gamma , j) 
\longrightarrow (\mathrm U(A^{\mathrm{op}}), \Aut (A^{\mathrm{op}}), \partial )
\end{equation*}
of crossed modules that induces $\sigma^{\mathrm{op}}$. Then
$\End_{A^{\mathrm{op}}}(A^{\mathrm{op}},Q,\mathrm e , \vartheta) 
= \mathrm M_{\vert Q\vert}(A)$ 
and, by 
Proposition \ref{tpf}(vii), the $Q$-normal structure
$\sigma_{\vert Q\vert } \colon Q \to \Aut(\mathrm M_{\vert Q\vert}(A))$ is equivariant.
\end{proof}

When $S|R$ is a Galois extension of commutative rings 
with Galois group $Q$, Theorem~\ref{fouro} has an obvious
consequence to be phrased in terms of extensions of scalars, by
Galois descent, cf.
Propositions \ref{2.3.1}, 
\ref{tpf}(ix) and \ref{tpf}(x). 
We content ourselves with the following simplified
form which in the classical case comes down to the corresponding result 
of Teich\-m\"uller \cite{MR0002858}, 
see also Eilenberg-Mac Lane \cite{MR0025443}.

\begin{cor}
\label{cor1}
Suppose that $S|R$ is a Galois extension of commutative rings with Galois
group $Q$.
Then the Teich\-m\"uller class of a $Q$-normal $S$-algebra $(A, \sigma )$ is 
zero if and only if, for some matrix algebra $\mathrm M_I(A)$, the $Q$-normal structure
arises by scalar extension.
\end{cor}

\begin{rema}
In the corollary one cannot in general assert that, if the Teich\-m\"uller class
of a $Q$-normal $S$-algebra $(A,\sigma )$ is zero, 
the $Q$-normal structure
$\sigma\colon Q \to \Out(A) $ itself comes
from scalar extension, see 
\cite[\S~14]{MR0025443}. 
Likewise,
the $Q$-normal structure $\sigma $ in Theorem~\ref{fouro}  not in general 
itself be $Q$-equivariant.
\end{rema}

The Teich\-m\"uller class of a $Q$-normal $S$-algebra $(A,\sigma)$
is the obstruction to the existence of a strong Deuring embedding 
relative to $\kappaQ\colon Q \to \Aut(S)$
of $A$ into an 
algebra $C$ over $R = S^Q$, in the following sense:

\begin{thm}
\label{4.2}
The
Teich\-m\"uller class 
of a $Q$-normal $S$-algebra $(A,\sigma )$
is zero if and only if $A$ admits a strong Deuring embedding
\[
(C, \chi \colon Q \to N^{\mathrm U(C)}(A)/\mathrm U(A))
\]
relative to $\kappaQ\colon Q \to \Aut(S)$
into an algebra $C$ over $R=S^Q$  that
induces the $Q$-normal structure $\sigma$ on $A$
(in the sense that 
the associated homomorphism 
{\rm \eqref{comp1}}
coincides with
$\sigma $, cf. 
Proposition {\rm \ref{twosixo}(v)}).
Hence a central $S$-algebra $A$ admits a strong
Deuring embedding relative to
$\kappaQ\colon Q \to \Aut(S)$
into some algebra $C$  over $R=S^Q$ 
if and only if $A$ admits
a $Q$-normal structure with zero Teich\-m\"uller class.
\end{thm}

\begin{proof}
Proposition \ref{2.6.5} entails that a $Q$-normal structure arising
from a strong Deuring embedding (relative to
the structure map $\kappaQ\colon Q \to \Aut(S)$)
has zero Teich\-m\"uller class. 

Conversely,
consider a  $Q$-normal $S$-algebra
$(A, \sigma )$ with zero Teich\-m\"uller class.
Since the Teich\-m\"uller class is zero,
there is a group
extension
\begin{equation}
\mathrm e \colon 1 \longrightarrow \mathrm U(A) \stackrel{j}\longrightarrow \Gamma \stackrel{\pi}
\longrightarrow Q \longrightarrow 1,
\end{equation}
cf. \eqref{getw} above,
together with a morphism 
\begin{equation}
(1,\vartheta) \colon (\mathrm U(A),\Gamma , j) \longrightarrow (\mathrm U(A), \Aut (A), \partial )
\end{equation}
of crossed modules that induces $\sigma $,
cf.
\eqref{morthr} above.
The embedding of $A$ into
the crossed product algebra  $C = (A,Q,\mathrm e,\vartheta)$,
cf. \cref{threet}(i),
together with the
homomorphism $\chi \colon Q \to N^{\mathrm U(C)}(A)/\mathrm U(A)$
induced by the injection $\Gamma \to N^{\mathrm U(C)}(A)$, 
cf. \cref{threet}(iv),
is a strong Deuring embedding of $A$ into $C$ relative to 
$\kappaQ\colon Q \to \Aut(S)$ that induces the
$Q$-normal structure $\sigma$ on $A$.
\end{proof}

\begin{rema}
Under the circumstances
of Theorem \ref{4.2}, since
$\Gamma$ injects into $N^{\mathrm U(C)}(A)$, the structure map
$\chi \colon Q \to N^{\mathrm U(C)}(A)/\mathrm U(A)$
is
injective as well.
\end{rema}

Theorem \ref{4.2} generalizes the corresponding results
of Teich\-m\"uller \cite{MR0002858}, in view of the following.

\begin{cor}
\label{cor2}
Let $S|R$ be a Galois extension of commutative rings with Galois group $Q$, and let
$(A,\sigma )$ be a $Q$-normal $S$-algebra.
Then the Teich\-m\"uller class of $(A,\sigma )$ is zero if and only
if there is a central $R$-algebra $C$ which contains $A$ in such a way
that

\noindent
{\rm (i)} the centralizer 
of $S$ in $C$ coincides with $A$, and 

\noindent
{\rm (ii)} each automorphism of $S|R$, i.e., each member $q$ of $Q$, 
extends to an inner automorphism $\alpha $ of $C$ which (in view of {\rm (i)})
maps $A$ to itself, in such a way that 
the class of $\alpha | A$ in $\Out (A)$
coincides with $\sigma (q)$. Moreover, if $A$ is an Azumaya $S$-algebra then $C$ may
be taken to be an Azumaya $R$-algebra.
\end{cor}

\begin{proof}
The \lq\lq if\rq\rq\ part of the first assertion was given in 
Proposition \ref{2.6.6}. On the other
hand, given a $Q$-normal $S$-algebra 
$(A,\sigma )$
with zero Teich\-m\"uller
class, the algebra $C$ in the proof of Theorem \ref{4.2} will have the desired
properties, by Propositions \ref{threet}
and \ref{tpf}(xi). 
\end{proof}

In the classical case considered by Deuring \cite{0014.20001} and
Teich\-m\"uller \cite{MR0002858} there was no need to spell out condition (ii) in
Corollary \ref{cor2}, in view of the Skolem-Noether theorem.

The crossed product construction yields sort of a \lq\lq generic\rq\rq\  
solution of
the Deuring embedding problem, provided the solution exists, that is, the 
obstruction vanishes:

\begin{thm}
\label{4.3}
Let $(A,\sigma )$ be a $Q$-normal $S$-algebra.
Suppose that $A$ admits a strong Deuring embedding 
$(C, \chi \colon Q \to N^{\mathrm U(C)}(A)/\mathrm U(A))$ 
relative to the action
$\kappaQ\colon Q \to \Aut(S)$ 
of $Q$ on $S$
into an algebra
$C$ over $R=S^Q$ 
that induces
the $Q$-normal structure $\sigma $ on $A$
in the sense that 
the composite of $\chi$ with the homomorphism 
$\eta_{\sharp}\colon N^{\mathrm U(C)}(A)/\mathrm U(A) \to \Out(A)$
induced by conjugation in the normalizer 
$N^{\mathrm U(C)}(A)$ of $A$ in the group $\mathrm U(C)$ of invertible 
elements of $C$ 
coincides with the $Q$-normal structure
$\sigma \colon Q \to \Out(A)$ on $A$, cf. Theorem {\rm \ref{4.2}} above.
Then the
strong Deuring embedding of $A$ into $C$ determines a group extension
\begin{equation*}\mathrm e \colon 1 \longrightarrow \mathrm U(A) 
\stackrel{j} \longrightarrow \Gamma \longrightarrow Q \longrightarrow 1,
\end{equation*}
together with
a morphism
\begin{equation*}(1,\vartheta) \colon (\mathrm U(A), \Gamma , j) \longrightarrow (\mathrm U(A), \Aut (A), \partial)\end{equation*}
of crossed modules that induces $\sigma $, and 
the data induce
a morphism
\begin{equation}
(A,Q,\mathrm e ,\vartheta ) \longrightarrow C
\label{mor9}
\end{equation}
of algebras over $R=S^Q$ which is compatible with the strong 
Deuring embeddings.
\end{thm}

\begin{proof}
Let $\mathrm e$ be the group extension \eqref{ext2}
and $(1,\vartheta)$ 
the morphism \eqref{mor1}
of crossed modules.
Recall that $N^{\mathrm U(C)}(A)$ denotes the normalizer
of $A$ in the group $\mathrm U(C)$ of invertible elements
of $C$.
By construction, the group $\Gamma$ is the fiber product group
$\Gamma = N^{\mathrm U(C)}(A) 
\times_{N^{\mathrm U(C)}(A)/\mathrm U(A)} Q$, the requisite homomorphism
from $Q$ to $N^{\mathrm U(C)}(A)/\mathrm U(A)$ being the strong
Deuring embedding 
structure
map $\chi \colon Q \to N^{\mathrm U(C)}(A)/\mathrm U(A)$.

The canonical morphism
$A^t\Gamma \to C$
of algebras induces the morphism
$(A,Q,\mathrm e ,\vartheta ) \to C$ 
of algebras we seek.
In particular, this morphism of algebras restricts to a homomorphism
$\mathrm U(A,Q,\mathrm e ,\vartheta ) \to \mathrm U(C)$
between the groups of invertible elements, and this homomorphism, in turn,
restricts to a homomorphism
\[
N^{\mathrm U(A,Q,\mathrm e ,\vartheta )}(A) \longrightarrow 
N^{\mathrm U(C)}(A)
\]
from the normalizer 
$N^{\mathrm U(A,Q,\mathrm e ,\vartheta )}(A)$
of $A$ in  $\mathrm U(A,Q,\mathrm e ,\vartheta )$
to the normalizer $N^{\mathrm U(C)}(A)$ of $A$ in 
$\mathrm U(C)$.
The composite
\begin{equation}
\begin{CD}
Q @>>> N^{\mathrm U(A,Q,\mathrm e ,\vartheta )}(A)/\mathrm U(A)
@>>> N^{\mathrm U(C)}(A)/\mathrm U(A)
\end{CD}
\end{equation}
of the strong Deuring embedding structure map
with respect to
$(A,Q,\mathrm e ,\vartheta )$
and the induced homomorphism
$
N^{\mathrm U(A,Q,\mathrm e ,\vartheta )}(A)/\mathrm U(A)
\to N^{\mathrm U(C)}(A)/\mathrm U(A)
$
yields the strong Deuring embedding structure map
$\chi$.
\end{proof}

\section{Induced normal and equivariant structures}
\label{eight}

As before, $S$ denotes a commutative ring
and $\kappaQ\colon Q \to \Aut(S)$ an action of a group $Q$ on $S$, 
and let $R=S^Q$.
Recall that a {\em split algebra} over $S$ is an algebra which
is isomorphic to an algebra of endomorphisms of some faithful
$S$-module $M$. If $S$ is a field, any split 
algebra admits an obvious $Q$-equivariant structure. 
However, if $S$ is an arbitrary ring and $M$ a faithful $S$-module,
some more structure on $M$ is necessary 
to guarantee the existence
of a  
$Q$-equivariant
structure on $A = \End_S(M)$
 or at least of a  $Q$-normal structure
as we now explain.

\subsection{Induced $Q$-normal structures}
\label{indqn}

Consider 
an $S$-module $M$. Suppose that $M$ admits an 
$S^t\Gamma $-module structure for some group $\Gamma $ which maps onto
$Q$ in such a way that $\Gamma $ acts on the coefficients from 
$S$ via the projection
 $\pi \colon \Gamma \to Q$, that is to say, $\Gamma $ acts on
$M$ by semi-linear transformations in the sense that
\begin{equation}
{}^x (sy) = 
({}^{\pi (x)}s) ({}^x y),\  x \in \Gamma, \, s\in S, \, y\in M.
\label{semil1}
\end{equation}
The group $\mathrm U(\End_S(M))$ of invertible elements of 
$\End_S(M)$ coincides with the group $\Aut_S(M)$.
The action of $\Gamma$ on $M$ 
restricts to 
an ordinary representation $\alpha\colon 
\ker(\pi) \to \Aut_S(M)$ of $\ker(\pi)$ 
on $M$ by $S$-linear transformations.
The inclusion
$i \colon \ker(\pi) \to \Gamma $ and the action of
$\Gamma$  on $\ker(\pi)$ constitute a crossed module, and
the action of $\Gamma $ on $M$ induces an action
${\beta\colon \Gamma \to \Aut(\End_S(M))}$ of $\Gamma$ on 
$\End_S(M)$ in such a way that
\begin{equation}
(\alpha , \beta ) \colon (\ker(\pi), \Gamma, i)
 \longrightarrow (\mathrm U(\End_S(M)),\Aut (\End_S(M)), \partial) 
\label{mor11}
\end{equation}
is a morphism of crossed modules; 
this morphism of crossed modules induces
a $Q$-normal structure
$\sigma= \sigma_{(\alpha , \beta )}\colon Q \to \Out(\End_S(M))$ 
on $\End_S(M)$, and 
we refer
to such a structure as
an {\em induced\/} $Q$-{\em normal structure\/} on $\End_S(M)$. 
Accordingly we define an {\em induced $Q$-normal split algebra\/}
to be a $Q$-normal algebra of the kind $(\End_S(M),\sigma) $
for some faithful $S$-module $M$, where $\sigma$ is an induced
$Q$-normal structure on $\End_S(M)$.

Let $M$ be an $S$-module. 
Define the subgroup $\Aut(M,Q)$
of $\Aut_R(M)\times Q$ by
\begin{equation}
\Aut(M,Q) = 
\{(\alpha,x);\  
 \alpha (sy) = {}^x\!s\,\alpha(y),\,s\in S, \,y\in M\}
\subseteq \Aut_R(M)\times Q,
\label{eightthirtythree}
\end{equation}
and let
$\pi^{\substack{\mbox{\tiny{$\Aut(M,Q)$}}}}\colon \Aut(M,Q) \to Q$
denote the obvious homomorphism.
The following is immediate, and we spell it out for later reference.

 \begin{prop}
\label{8.3} {\rm (i)} The homomorphism
$\pi^{\substack{\mbox{\tiny{$\Aut(M,Q)$}}}}\colon \Aut(M,Q) \to Q$
has $\Aut_S(M)$ as its kernel.

\noindent
{\rm (ii)}
The homomorphism $\pi^{\substack{\mbox{\tiny{$\Aut(M,Q)$}}}}$ 
is surjective if and only if
$M$ admits an 
$S^t\Gamma $-module structure for some group $\Gamma $ which maps onto
$Q$ in such a way that $\Gamma $ acts on the coefficients from 
$S$ via the projection
 $\pi \colon \Gamma \to Q$, that is to say, $\Gamma $ acts on
$M$ by semi-linear transformations in the sense that
{\rm \eqref{semil1}} holds.
If this happens to be the case, the group $\Aut(M,Q)$
fits into a group extension
\begin{equation}
1
\longrightarrow
\Aut_S(M) 
\longrightarrow
\Aut(M,Q)
\stackrel{\pi^{\substack{\mbox{\tiny{$\Aut(M,Q)$}}}}}
\longrightarrow
Q
\longrightarrow
1 .
\label{extseventyone}
\end{equation}

\noindent
{\rm (iii)} Suppose that 
the homomorphism $\pi^{\substack{\mbox{\tiny{$\Aut(M,Q)$}}}}$ 
is surjective. 
Then the group extension {\rm \eqref{extseventyone}} splits if an only 
if $M$ admits an  $S^tQ$-module structure, and the 
$S^tQ$-module structures on $M$ are in bijective correspondence with
the sections 
$Q \to
\Aut(M,Q)
$
for
$\pi^{\substack{\mbox{\tiny{$\Aut(M,Q)$}}}}$.

\noindent
{\rm (iv)} In particular, a free $S$-module $M$ admits 
an  $S^tQ$-module structure and, if $M$ has finite rank,
the $S^tQ$-module structures on $M$ are parametrized by $S$-bases. 

\noindent
{\rm (v)} Suppose that $M$ is a free $S$-module of finite rank.
Then every $Q$-equivariant structure on
$\End_S(M)$, necessarily an induced one, arises from
an  $S^tQ$-module structure on $M$,
and
the $Q$-equivariant structures 
on the central $S$-algebra $\End_S(M)$
are parametrized by classes of
$S^tQ$-module structures 
$s\colon Q \to
\Aut(M,Q)$ (sections for $\pi^{\substack{\mbox{\tiny{$\Aut(M,Q)$}}}}$)
 on $M$; indeed, two 
 sections $s_1,s_2\colon Q \to \Aut(M,Q)$
for $\pi^{\substack{\mbox{\tiny{$\Aut(M,Q)$}}}}$
induce the same
$Q$-equivariant
structure on the $S$-algebra $\End_S(M)$ if and only if,
relative to the obvious injection $\mathrm U(S) \to \Aut_S(M)$,
the two sections $s_1$ and $s_2$ differ by a derivation 
$Q \to \mathrm U(S)$.
Furthermore, $\End_S(M)$ acquires a unique induced
$Q$-normal
structure
$\sigma\colon Q \to \Out(\End_S(M),Q)$, necessarily an equivariant one.
\end{prop}

\begin{proof} We sketch an argument for statement (v).
When $M$ is free of finite rank,
the homomorphism
$\beta\colon \Pic(S) \to \Pic(\End_S(M))$
in the generalized Skolem-Noether theorem, i.~e.,
in Proposition \ref{skolnoet} above,
is an isomorphism and hence every $S$-linear automorphism of
$\End_S(M)$ is an inner automorphism. This observation implies the
assertion. We leave the details to the reader.
\end{proof}

\begin{rema} Proposition \ref{8.3}(ii) shows that, in the definition
of an induced $Q$-normal structure on a split algebra
$\End_S(M)$ over some faithful $S$-module $M$, we may take the group
$\Gamma$ to be the group
$\Aut(M,Q)$.
\end{rema}

\begin{prop}
\label{8.2}
Given an $S$-module $M$,
the Teich\-m\"uller class $[\mathrm e_{(A,\sigma)}]\in \mathrm H^3(Q,\mathrm U(S))$
of an induced $Q$-normal structure
$\sigma\colon Q \to \Out(A)$ on $A=\End_S(M)$ is zero.
\end{prop}

\begin{proof}
The $Q$-normal structure $\sigma$ is induced by a semi-linear
action of some group $\Gamma$ on $M$ that maps onto $Q$ in such a way that
$\Gamma$ acts on the coefficients via this projection, combined with
$\kappaQ\colon Q \to \Aut(S)$.
The corresponding 
morphism \eqref{mor11} of crossed modules induces a commutative diagram
\begin{equation*}
\begin{CD}
1 @>>> \mathrm{ker}(\pi) @>>> \Gamma @>{\pi}>> Q @>>> 1
\\
@.
@V{\alpha}VV
@V{\beta^{\sigma}}VV
@|
@.
\\
@. 
\mathrm U(A)
@>{\partial^{\sigma}}>>
B^{\sigma}
@>>>
Q
@>>>
1 
\end{CD}
\end{equation*}
with exact rows,
cf. \eqref{pb1} 
for the notation $\partial^{\sigma}$ etc.
This implies that
the Teich\-m\"uller class $[\mathrm e_{(A,\sigma)}]\in \mathrm H^3(Q,\mathrm U(S))$
is zero.
\end{proof}

\subsection{Proof of the second assertion of Theorem {\rm \ref{genaz}}}
\label{ptgenaz}

Suppose that the generalized $Q$-normal Azumaya $S$-algebra
$(A,\sigma_A)$ 
arises from an ordinary
$Q$-normal algebra structure $\sigma\colon Q \to \Out(A,Q)$ on $A$.
Then the construction of  
$\MsA$ and $\BsA$, cf. Subsection \ref{from} above, 
simplifies as follows: For each $x \in Q$, let
$\varepsilon_x \colon A \to A$ be an automorphism 
of grade $x$
that lifts
the automorphism $\kappaQ(x)$ of $S$,
and take $M_x$ to be the $(A,A)$-bimodule
$A_{\varepsilon(x)}$, that is, the algebra $A$ itself,
viewed as an invertible $(A,A)$-bimodule of grade $x$
via the structure map
\begin{equation*}
A \otimes A_{\varepsilon(x)} \otimes A \longrightarrow A_{\varepsilon(x)},
\ 
a_1\cdot a \cdot a_2 =a_1 a (\varepsilon(x) (a_2)),\ a, a_1, a_2 \in A.
\end{equation*}
Furthermore, given $x,y\in Q$, the isomorphism
$f_{x,y}\colon  A_{\varepsilon(x)} \otimes_A A_{\varepsilon(y)} \to
 A_{\varepsilon(xy)}
$ may be taken to be an inner automorphism of $A$ in the sense that, 
for some invertible member $u$ of $A$, given
$a_1 \otimes a_2 \in  A_{\varepsilon(x)} \otimes_A A_{\varepsilon(y)}$,
the value
$f_{x,y}(a_1 \otimes a_2)$ equals $u a_1 a_2 u^{-1}$. Now
$\MsA= \oplus_{z\in Q} A_{\varepsilon(z)}$ is a free left $A$-module.

From the obvious isomorphisms
\begin{equation}
\End_S(\MsA)=\prod_{u\in Q_A} \bigoplus_{v\in Q_A} \Hom_S(M_u,M_v)
\cong \prod_{u\in Q_A}\Hom_S(M_u,\MsA)
\label{desc2}
\end{equation}
we deduce that the canonical homomorphism
$A \otimes \BBB \to \End_S(\MsA)$ is an isomorphism of $S$-algebras.
(This is immediate when the group $\kappaQ(Q)\subseteq \Aut(S)$
is finite.) Under this isomorphism, the tensor product $Q$-normal structure
$\sigma\otimes \sigma_{\BBB}$
on $A \otimes \BBB$ corresponds to an induced $Q$-normal structure
on $\End_S(\MsA)$. In view of Proposition \ref{2.4.3},
Proposition \ref{8.2} entails that the sum
$[\mathrm e_{(A,\sigma)}] +[\mathrm e_{(\BBB,\sigma_{\BBB})}] 
\in \mathrm H^3(Q,\mathrm U(S))$ is zero whence,
in view of Proposition \ref{2.4.1},
 $[\mathrm e_{(\BBB^{\mathrm{op}},\sigma_{\BBB^{\mathrm{op}}})}] 
=
[\mathrm e_{(A,\sigma)}]
\in \mathrm H^3(Q,\mathrm U(S))$
as asserted.

\subsection{Induced $Q$-equivariant structures}
\label{indqe}

Consider an $S$-module $M$.
Suppose that $M$ admits an 
$S^t\Gamma $-module structure for some group $\Gamma $ which maps onto
$Q$ in such a way that $\Gamma $ acts on the coefficients from 
$S$ via the projection
 $\pi \colon \Gamma \to Q$,
consider the associated morphism
\eqref{mor11} of crossed modules, and
suppoas that
$\alpha$ maps $\ker(\pi)$ to $\mathrm U(S)\subseteq \mathrm U(\End_S(M))$. 
Then the homomorphism $\beta $ induces  a $Q$-equivariant structure 
$\tau = \tau_{(\alpha , \beta )}\colon Q \to \Aut(\End_S(M))$ on $\End_S(M)$, 
and we refer 
to such a structure as an 
{\em induced\/} $Q$-equivariant structure  on $\End_S(M)$.
Accordingly,
we define an {\em induced $Q$-equivariant split algebra\/}
to be a $Q$-equivariant algebra of the kind $(\End_S(M),\tau)$
for some faithful $S$-module $M$, where $\tau$ is an induced
$Q$-equivariant structure on $\End_S(M)$.

Let $M$ be an $S$-module
let $\tau\colon Q \to \Aut(\End_S(M))$ be a $Q$-equivariant structure
on $\End_S(M)$,
let $\Aut(M,Q,\tau)$
denote the subgroup of $\Aut(M,Q)$ defined by
\begin{equation*}
\Aut(M,Q,\tau) = 
\{(\alpha,x);\  
 \alpha (ay) = ({}^{\tau (x)}a)\alpha (y),\,a\in \End_S(M), \,y\in M\}
\subseteq \Aut_R(M)\times Q,
\end{equation*}
and let
 $\pi^{\substack{\mbox{\tiny{$\Aut(M,Q,\tau)$}}}}\colon \Aut(M,Q,\tau) \to Q$
denote the canonical homomorphism. The following is immediate.

\begin{prop}
\label{8.0}
{\rm (i)} The homomorphism
$\pi^{\substack{\mbox{\tiny{$\Aut(M,Q,\tau)$}}}}\colon \Aut(M,Q,\tau) \to Q$
has $\mathrm U(S)$ as its kernel.

\noindent
{\rm (ii)}
The homomorphism $\pi^{\substack{\mbox{\tiny{$\Aut(M,Q,\tau)$}}}}$ 
is surjective if and only if
$M$ admits an 
$S^t\Gamma $-module structure for some group $\Gamma $ which maps onto
$Q$ in such a way that $\Gamma $ acts on
$M$ by semi-linear transformations 
via the projection $\pi\colon \Gamma \to Q$
in the sense that
{\rm \eqref{semil1}} holds
and such that $\mathrm{ker}(\pi)$ maps to $\mathrm U(S)$.
If this happens to be the case, the group $\Aut(M,Q,\tau)$
fits into a group extension
\begin{equation}
\mathrm e^{M,\tau}\colon
1
\longrightarrow
\mathrm U(S) 
\longrightarrow
\Aut(M,Q,\tau)
\stackrel{\pi^{\substack{\mbox{\tiny{$\Aut(M,Q,\tau)$}}}}}
\longrightarrow
Q
\longrightarrow
1 .
\label{ext711} 
\end{equation}
\end{prop}

\begin{rema} Let $M$ be a faithful $S$-module. \cref{8.0}(ii) 
shows that, in the definition of an induced $Q$-equivariant structure on 
the split algebra $\End_S(M)$ over  $M$, given a $Q$-equivariant structure 
$\tau\colon Q \to \Aut(\End_S(M))$ on $\End_S(M)$, we may take the group
$\Gamma$ to be the group $\Aut(M,Q,\tau)$.
\end{rema}

\begin{lem}
\label{8.00}
Given a faithful $S$-module $M$,
let $\tau \colon Q \to \Aut(\End_S(M))$ be a $Q$-equivariant structure
on the central $S$-algebra $\End_S(M)$, and suppose that
$\pi^{\substack{\mbox{\tiny{$\Aut(M,Q)$}}}}\colon
\Aut(M,Q) \to Q$ is surjective. Then
$\pi^{\substack{\mbox{\tiny{$\Aut(M,Q,\tau)$}}}}\colon
\Aut(M,Q,\tau) \to Q$ is surjective as well, whence
$\tau$ is then an induced $Q$-equivariant structure.
\end{lem}

\begin{proof}
By \cref{8.3}(ii), the homomorphism
 $\pi^{\substack{\mbox{\tiny{$\Aut(M,Q)$}}}}\colon \Aut(M,Q) \to Q$ 
is surjective and induces,
via the associated action $\beta\colon  \Aut(M,Q) \to  \Aut(\End_S(M))$
of $\Aut(M,Q)$ on $\End_S(M)$, the
$Q$-normal structure $\sigma_\tau \colon Q \to \Out(\End_S(M))$ on $\End_S(M)$
associated to $\tau$.

The canonical homomorphism $\Aut(M,Q,\tau)\to Q$
is surjective.
Indeed,
let $x\in \Aut(M,Q)$. Then the automorphisms $\beta(x)$ and
$\tau(\pi^{\substack{\mbox{\tiny{$\Aut(M,Q)$}}}}(x))$
of $\End_S(M)$ have the same image in $\Out(\End_S(M))$
whence there exists an $S$-linear
automorphism $\tilde\alpha_{\ggamma} $ of $M$ so that, given 
$a\in \End_S(M)$, 
\begin{equation*}
\tilde\alpha_{\ggamma} ({}^{\beta (\ggamma)} a)\tilde\alpha_{\ggamma}^{-1} = {}^{\tau (\pi(\ggamma))}a .
\end{equation*}
Consequently, given $q\in Q$ and a pre-image 
$\ggamma \in \Aut(M,Q)$ of $q$, 
\begin{equation*}
\tilde\alpha_{\ggamma}({}^\ggamma (a y)) = ({}^{\tau (q)}a) \tilde\alpha_{\ggamma} ({}^\ggamma y),
\ a\in \End_S(M),\, y \in M;
\end{equation*}
thus the automorphism  
$\alpha_x \in \Aut_R(M)$ given by $\alpha_x(y) = \tilde\alpha_{\ggamma}({}^xy)$
then yields a pre-image $(\alpha_x,q)\in \Aut(M,Q,\tau)$ of $q \in Q$.
Hence
 the canonical homomorphism from $\Aut(M,Q,\tau)$ to $Q$
is surjective.
By construction, the induced action of $\Aut(M,Q,\tau)$ on
$\End_S(M)$ 
induces
the given $Q$-equivariant structure $\tau $ on $\End_S(M)$. 
\end{proof}

\begin{cor}
\label{eighto}
Given 
a $Q$-equivariant structure
$\tau\colon Q \to\Aut(\End_S(M))$ on the central $S$-algebra
$\End_S(M)$ over a 
faithful $S$-module $M$,
suppose  that the associated
$Q$-normal structure $\sigma_\tau \colon Q \to \Out(\End_S(M))$ on $\End_S(M)$
associated to $\tau$ 
is induced as a $Q$-normal structure. Then 
the $Q$-equivariant structure
$\tau $ on $\End_S(M)$ is an induced $Q$-equivariant structure.
\end{cor}

\subsection{Induced $Q$-equivariant structures and crossed product
algebras}
 
Suppose that the group $Q$ is a finite group.
Let $M$ be a faithful 
finitely generated projective
$S$-module  and 
$\tau\colon Q \to\Aut(\End_S(M))$ 
an induced $Q$-equivariant structure 
on the split central $S$-algebra
$\End_S(M)$. With respect to an associated group
extension 
$\mathrm e\colon 
\mathrm U(S) \stackrel{i}\rightarrowtail \Gamma \stackrel{\pi^\Gamma}\twoheadrightarrow Q$
and  morphism
\begin{equation*}
(j,\beta)\colon (\mathrm U(S),\Gamma,i) \longrightarrow
(\Aut_S(M),\Aut(\End_S(M)),\partial)
\end{equation*}
of crossed modules
inducing the $Q$-equivariant structure $\tau$,
let $M_{\mathrm e}$ denote the $S^t\Gamma$-module
that underlies the crossed product algebra 
$(S,Q,\mathrm e, \kappa_Q \circ \pi^\Gamma)$,
by construction, free as an $S$-module,
let $\tau_{\mathrm e}\colon Q \to \Aut(\End_S(M_{\mathrm e}))$
denote the associated induced $Q$-equivariant structure on 
$\End_S(M_{\mathrm e})$,
and consider the $S$-module
$\Hom_S(M,M_{\mathrm e})
$, necessarily finitely generated projective and
faithful. 
The association
\begin{equation}
\End(M)^{\mathrm{op}}\otimes \End_S(M_{\mathrm e}) \otimes\Hom_S(M,M_{\mathrm e})
\longrightarrow
\Hom_S(M,M_{\mathrm e})
\label{action8}
\end{equation}
which, for $h\in \End_S(M)$, $f \in \End_S(M_{\mathrm e})$, 
$\varphi\in \Hom_S(M, M_{\mathrm e})$, is given by
\begin{equation*}
h \otimes f \otimes \varphi
\longmapsto
f \circ \varphi \circ h,
\end{equation*}
identifies the central $S$-algebras
$\End_S(M)^{\mathrm{op}}\otimes \End_S(M_{\mathrm e})$ and
$\End_S(\Hom_S(M,M_{\mathrm e}))$.

\begin{prop}
\label{diagact}
The diagonal action 
of $\Gamma$ on  $\Hom_S(M,M_{\mathrm e})$
given by the association
\[                                                                         
(\alpha,\varphi) \longmapsto \alpha(\varphi)=
\alpha \circ \varphi \circ \alpha^{-1},\ \alpha \in \Gamma,\ \varphi \in \Hom_S(M,M_{\mathrm e}),
\]
is trivial on $\mathrm U(S)=\mathrm{ker}(\pi^\Gamma)$ and hence descends to an
$S^tQ$-module structure
\begin{equation}
Q \times \Hom_S(M,M_{\mathrm e}) \longrightarrow
\Hom_S(M,M_{\mathrm e})
\label{stq}
\end{equation}
on  $\Hom_S(M,M_{\mathrm e})$. 
Consequently,
in terms of the notation $\tau_0$ for
the $Q$-equivariant structure on
$(\End_S(\Hom_S(M,M_{\mathrm e}))$
induced by {\rm \eqref{stq}},
\[
{(\End_S(M)^{\mathrm{op}}\otimes \End_S(M_{\mathrm e}), \tau^{\mathrm{op}} 
\otimes \tau_{\mathrm e})} 
\cong (\End_S(\Hom_S(M,M_{\mathrm e})),\tau_0)
\]
as $Q$-equivariant central $S$-algebras. 
\end{prop}

\begin{proof}
For both $M$ and $M_{\mathrm e}$, the
restriction of the $\Gamma$-action 
 to the kernel $\mathrm U(S)$ 
of $\pi^\Gamma$
 coincides with the $\mathrm U(S)$-module structure
coming from multiplication by members of the coefficient ring $S$.
This implies the assertion.
\end{proof}

\section{Crossed Brauer group, generalized crossed Brauer group, and Picard group}
\label{9}

\subsection{Crossed Brauer group}
\label{cbg}
As before, $S$ denotes a commutative ring and 
$\kappaQ\colon Q \to \Aut(S)$ an action of a group $Q$ on $S$. 
Given a faithful finitely generated projective  {$S$-module},
the central $S$-algebra $\End_S(M)$ is an Azumaya algebra.
We  say that
two {$Q$-normal} Azumaya $S$-algebras $(A_1,\sigma_1)$ and 
 $(A_2,\sigma_2)$ are {\em normally Brauer equivalent\/}
if there are faithful finitely
generated projective  $S$-modules
modules $M_1$ and $M_2$ 
together with induced $Q$-normal structures
\[
\rho_1\colon Q \to \Out(B_1),
\
B_1 = \End_S(M_1),
\
\rho_2\colon Q \to \Out(B_2),
\
B_2 = \End_S(M_2),
\]
such that 
$(A_1 \otimes B_1, \sigma_1 \otimes \rho_1 )$
and $(A_2 \otimes B_2, \sigma_2 \otimes \rho_2)$
are isomorphic $Q$-normal {$S$-algebras}.
Since the tensor
product of two induced $Q$-normal split algebras is again an induced 
$Q$-normal
split algebra in an obvious manner, 
that relation, 
referred to henceforth as {\em normal Brauer equivalence\/}, 
is indeed an equivalence relation, similarly as
in 
\cite[p.~381]{MR0121392}, and this equivalence relation
is compatible with the operation of taking
tensor products. Hence the equivalence classes constitute an abelian monoid;
moreover, given a $Q$-normal Azumaya $S$-algebra $(A,\sigma)$, 
the map \eqref{azu} being an isomorphism,
the $S$-algebra $(A \otimes A^{\mathrm{op}},
\sigma \otimes \sigma^{\mathrm{op}})$ is an induced $Q$-normal split algebra,
the requisite semi-linear action 
on the $S$-module that underlies $A$
being given by the action
$B^{\sigma} \to \Aut(A)$ of the fiber product group
$B^{\sigma}=\Aut(A)\times_Q \Out(A)$ with respect to $\sigma\colon Q \to \Out(A)$
that yields the crossed 2-fold extension \eqref{pb1},
and so, taking the class of $(A^{\mathrm{op}}, \sigma^{\mathrm{op}})$ as the 
inverse of
the class of $(A,\sigma )$, we get in fact an abelian group,
the identity element being
the equivalence class of
induced $Q$-normal split algebras $(\End_S(M),\sigma)$
where $M$ ranges over faithful 
finitely generated projective $S$-modules
having the property that the obvious homomorphism
$\pi^{\substack{\mbox{\tiny{$\Aut(M,Q)$}}}}$ from $\Aut(M,Q)$ to $Q$
is surjective, cf. Proposition \ref{8.3} above.
In particular, $(S,\kappaQ)$ represents the identity element.
We refer to that group
as the {\em crossed Brauer group} of $S$ 
relative to $\kappaQ \colon Q \to \Aut(S)$,
denote this group 
by $\mathrm{XB}(S,Q)$,
and we refer to the construction just given as the 
{\em standard construction\/}.
An equivalent construction of the crossed Brauer group 
(as the cokernel of a homomorphism of certain abelian monoids)
is given in
\cite[Theorem 4 p.~43]{genbrauer},
\cite[Section 3, a few lines before Theorem 3.2]{MR1803361}
(written as $QB(R,\Gamma)$), cf. 
\eqref{Seq2} below.

\begin{thm}
\label{nineo}
The crossed Brauer group is a functor on the 
change of actions category $\mathcat{Change}$  introduced
in Subsection {\rm\ref{coa}} in such a way that the following hold. 

\noindent
{\rm (i)}
The assignment to a $Q$-normal Azumaya $S$-algebra
of its Teich\-m\"uller complex yields a natural homomorphism
\begin{equation}t \colon \mathrm{XB} (S, Q) \longrightarrow \mathrm H^3 (Q, \mathrm U(S))
\label{teich2}
\end{equation}
of abelian groups.

\noindent
{\rm (ii)}
The class $[(A,\sigma)] \in \mathrm{XB}(S,Q)$ 
of a $Q$-normal Azumaya {$S$-algebra} $(A,\sigma )$ is zero
if and only if,  for some faithful
finitely generated projective $S$-module
$M$, the algebra $A$ is isomorphic to
$\End_S(M)$
in such a way that $\sigma$
is an induced $Q$-normal structure.
\end{thm}

\begin{proof}
Functoriality of the crossed Brauer group is a consequence of
\cref{2.5.1}(i) together
with the fact that induced structures on split algebras are preserved
under change of rings. Statement (i) is a consequence
of \cref{8.2}, combined
with Propositions \ref{2.4.1}, \ref{2.4.3} and \ref{2.5.1}(iii). 
We leave the details 
to the reader. 
As for Statement (ii),
the \lq\lq if\rq\rq\ statement is manifest but the 
\lq\lq only if\rq\rq\ statement is not.
We shall complete the proof thereof in the next subsection.
\end{proof}

\begin{rema}
A variant of the homomorphism \eqref{teich2}, written there as $\rho$, 
is given in
\cite[Theorem 3.4 (i)]{MR1803361}
by the assignment to a $Q$-normal Azumaya $S$-algebra
of an explicit 3-cocycle of $Q$ with values in $\mathrm U(S)$.
The statement of Theorem \ref{nineo}(ii) generalizes  
\cite[Proposition~5.3]{MR0121392}.
\end{rema}

\subsection{Crossed Brauer group and Picard group}
\label{10.29}

Given a morphism $(f, \varphi ) \colon (S, Q, \kappa)  \to (T,G, \lambda )$ 
in the change of actions category $\mathcat{Change}$
introduced 
in Subsection \ref{coa}, we denote by $\mathrm{XB}(T|S;G,Q)$ 
the kernel of the 
induced homomorphism from $\mathrm{XB}(S,Q)$ to $\mathrm{XB}(T,G)$.
Thus $\mathrm{XB}(S|S;Q,Q)$ is the trivial group whereas
$\XBSQ$ is the kernel of the forgetful homomorphism 
from $\mathrm{XB}(S,Q)$ to $\mathrm{B}(S)$.
The notation $\mathrm{XB}(\,\cdot \, ;\,\cdot \,   ,\,\cdot \,)$
might look a bit heavy to the reader;
we use it for the sake of consistency with 
a similar notation for the equivariant case
which we introduce in 
\cref{releqb}.

We  now define a homomorphism
from $\XBSQ$ to $\mathrm H^1(Q, \Pic (S))$.
To this end, let $M$ be a faithful finitely generated projective 
$S$-module, and consider the split $S$-algebra $\End_S(M)$.
Given an algebra 
automorphism $\alpha$ of $\End_S(M)$, 
let
${}^\alpha M$ denote the $\End_S(M)$-module whose $\End_S(M)$-module structure is given 
by the formula
\begin{equation*}
a\cdot x = ({}^\alpha a) y \ , \ a\in \End_S(M), \, y\in M ;
\end{equation*}
in particular, $S$ then acts on ${}^\alpha M$ by
\begin{equation*}
s\cdot y = ({}^{\alpha |} s) y, \ s\in S, \ y\in M ,
\end{equation*}
and so the association $a\mapsto {}^\alpha a $ yields an isomorphism 
$\End_S(M) \to \End_S( {}^\alpha M)$ of $S$-algebras. Consequently,
$J (\alpha ) = \Hom_{\End_S(M)}({}^\alpha M, M)$ is a 
faithful
finitely generated 
projective rank one $S$-module, cf., e.~g., 
\cite[Lemma~9]{MR0148709}, and 
the evaluation map
\[
 \Hom ({}^\alpha M, M)\otimes {}_{\End_S(M)}{}^\alpha M \longrightarrow M
\]
is an isomorphism of $S$-modules \cite[Prop.~A.6]{MR0117252}.

Now, let $\sigma $ be a $Q$-normal structure on $\End_S(M)$. Then
$(\End_S(M),\sigma )$ represents a member of $\XBSQ$.
Let $w \colon Q\to  \Aut (\End_S(M))$ be a morphism of the underlying sets 
which lifts $\sigma$.

\begin{lem}
Let $w' \colon Q \to \Out (\End_S(M))$ be another lifting of $\sigma $. 
Given $q\in Q$, there is an $S$-linear automorphism 
$\alpha_q \colon M \to M$ so that
\begin{equation*}
{}^{w(q)}a = \alpha_q ({}^{w'(q)}a)\alpha_q^{-1} 
\colon M \longrightarrow M, \ a\in \End_S(M),
\end{equation*}
whence the class $[J(w(q))]\in \Pic(S)$
depends only on $M$ and $\sigma$ and not on the choice of $w$.
Hence
the map
\begin{equation}
d_{(M,\sigma)} \colon Q \longrightarrow \Pic(S), 
\ 
d_{(M,\sigma)}(q) = [J (w(q))]\in \Pic (S),\  q\in Q,
\label{der1}
\end{equation}
is well-defined in the sense that it does not depend on the choice
of $w$. 
\end{lem}

\begin{proof}
This is an immediate consequence of the exactness of the sequence
\[
\Aut_S(M) \longrightarrow
\Aut (\End_S(M))
\longrightarrow
\Out (\End_S(M))
\longrightarrow
1.
\]
\end{proof}

The following is immediate.

\begin{prop}
\label{10.4}
Given $q, r\in Q$, the $S$-modules $J(w(qr))$ and 
$J(w(q)) \otimes {}^q J(w(r))$ are isomorphic in an obvious way.
Consequently the map $d_{(M,\sigma)}$ defined by {\rm \eqref{der1}}
is a derivation on $Q$ with values in $\Pic(S)$. \qed
\end{prop}
\begin{prop}
\label{10.5}
The class of $d_{(M,\sigma)}$ in $\mathrm H^1(Q, \Pic (S))$ depends only on
$(\End_S(M),\sigma)$ and not on a particular choice of $M$.
\end{prop}

\begin{proof}
Let $M'$ be another faithful finitely generated 
projective $S$-module having the property that $\End_S(M)$ and $\End_S(M')$
are isomorphic central $S$-algebras. Then 
\[
J = \Hom_{\End_S(M)}(M,M') 
\]
is a faithful finitely generated 
projective rank one
$S$-module in such a way that 
$M'$ and $M\otimes J$ are isomorphic ${\End_S(M)}$-modules under the
obvious map, the ${\End_S(M)}$-action on the latter being given by
\begin{equation*}
a(y\otimes f) = ay\otimes f,\  a\in {\End_S(M)}, \, y \in M,\ f\in J.
\end{equation*}
Given $q\in Q$, the module  
$\Hom_S({}^qJ,S)$ represents ${}^q[J]^{-1}\in \Pic (S)$,
and so $d_{(M,\sigma)}$ and $d_{(M',\sigma)}$ differ by the inner derivation
\begin{equation*}
Q \longrightarrow\Pic (S),\ q \longmapsto [J] ({}^q[J]^{-1}), \ q \in Q.
\end{equation*}
Hence the class of $d_{(M,\sigma)}$ in $\mathrm H^1(Q, \Pic (S))$ depends only on
$({\End_S(M)},\sigma)$ and not on the choice of $M$.
\end{proof}

\begin{prop}
\label{10.6}
Suppose that the $Q$-normal structure $\sigma $ on $\End_S(M)$
is induced. Then $d_{(M,\sigma)}$ is zero.
\end{prop}

\begin{proof}
Some semi-linear action of some group $\Gamma $
on $M$  maps onto $Q$ via (say) $\pi \colon \Gamma \to Q$ and
yields an action
$\beta \colon \Gamma \to \Aut (\End_S(M))$ 
of $\Gamma$ on $\End_S(M)$
which, in turn, induces $\sigma $.
Let $w' = Q \to \Gamma$ be a section 
for $\pi $ of the underlying sets, and let
$w = \beta w'$. Then each finitely generated projective rank one $S$-module
$J(w(q)) = \Hom_{\End_S(M)} ({}^{w(q)}M,M)$,
as $q$ ranges over $Q$,
 is isomorphic to $S$, 
and so $d_{(M,\sigma)}$ is zero.
\end{proof}

\begin{prop}
\label{10.7}
Suppose that $d_{(M,\sigma)}$ is an inner derivation, that is to say, 
there is a faithful finitely
generated projective rank one $S$-module $J$ so that, 
for each $q\in Q$, the $S$-modules
$J(w(q))$ and ${}^q J \otimes \Hom_S(J,S)$ are isomorphic.
Then 
$({\End_S(M)},\sigma )$ is an induced $Q$-normal 
split algebra.
\end{prop}

\begin{proof}
For each $q\in Q$,
\begin{equation*}
\Hom_{\End_S(M)} ({}^{w(q)}(M\otimes J), M\otimes J) 
\cong J(w(q)) \otimes \Hom_S ({}^q J,S) \otimes J \cong S.
\end{equation*}
Hence, replacing $M$ by $M\otimes J$, we may assume 
that $\End_S(M)$ has the property that each 
projective rank one $S$-module $\Hom_{\End_S(M)}({}^{w(q)}M,M)$ is a free
$S$-module $Su_q$ on a single generator $u_q$, necessarily 
an isomorphism $u_q \colon {}^{w(q)} M \to M$  of ${\End_S(M)}$-modules. 
For each $q\in Q$, the isomorphism $u_q$ then satisfies the identity 
\begin{equation*}
(u_q)^{-1}(ay) = {}^{w(q)} a (u_q)^{-1}(y),\ a \in {\End_S(M)},\, y\in M.
\end{equation*}
Thus each $q\in Q$ extends to a semi-linear transformation 
of $M$, and $({\End_S(M)},\sigma )$ is therefore an induced $Q$-normal 
split algebra.
\end{proof}

Given a $Q$-normal Azumaya $S$-algebra $(A,\sigma)$
that represents a member of
$\mathrm{XB}(S|S;{e},Q)$, 
in view of normal Brauer equivalence,
there are faithful finitely generated projective $S$-modules
$M_1$ and $M_2$ together with an induced $Q$-normal structure
$\sigma_1\colon Q \to \Out(\End_S(M_1))$ on $\End_S(M_1)$
such that 
$A \otimes \End_S(M_1)$ and $\End_S(M_2)$ are isomorphic as central 
$S$-algebras and, 
under this isomorphism, the $Q$-normal structure
$\sigma\otimes \sigma_1$ on $A \otimes \End_S(M_1)$
corresponds to a $Q$-normal structure
$\sigma_2\colon Q \to \End_S(M_2)$ on $\End_S(M_2)$.

\begin{thm}
\label{10.2} 
The assignment to the class  $[(A,\sigma)] \in \mathrm {XB}(S|S;\{e\},Q)$
of a $Q$-normal Azumaya $S$-algebra $(A,\sigma)$
that represents a member of
$\mathrm{XB}(S|S;{e},Q)$
of $[d_{M_2,\sigma_2}] \in \mathrm H^1(Q,\Pic(S))$
yields an injective homomorphism
\begin{equation}
\label{tenthr}
\iota\colon      
\XBSQ \longrightarrow \mathrm H^1(Q, \Pic (S))
\end{equation}
which is, furthermore, natural on the change of actions category $\mathcat{Change}$ introduced in 
Subsection {\rm \ref{coa}} above.
\end{thm}

\begin{proof}
Propositions \ref{10.5} and \ref{10.6}
entail that \eqref{tenthr} is well-defined, and
Proposition \ref{10.7} implies that
the map \eqref{tenthr} is injective.
We leave the proofs that the map 
\eqref{tenthr} is a natural homomorphism of abelian groups to the reader.
\end{proof}

\begin{rema}
The injectivity of \eqref{tenthr} may be found
in \cite[Theorem 3.3]{MR1803361}.
\end{rema}

\begin{proof}[Proof of Theorem~{\rm \ref{nineo} (ii)}]
Suppose that $[(A,\sigma)] =0 \in \mathrm {XB}(S,Q)$.
By \cite[Proposition~5.3]{MR0121392}, there is a faithful finitely 
generated projective $S$-module $M$ such that $A \cong \End_S(M)$,
and $[(\End_S(M),\sigma)]=0 \in \mathrm {XB}(S|S;\{e\},Q)$.
Then the derivation $d_{(M,\sigma)}$ is an inner derivation
whence, by Proposition \ref{10.7},
$(\End_S(M),\sigma)$ is an induced $Q$-normal split algebra.
\end{proof}

\subsection{The generalized crossed Brauer group}
\label{gencrob}

Inspection shows that the assignment to a $Q$-normal Azumaya algebra
$(A,\sigma)$ of the associated generalized $Q$-normal Azumaya algebra
of $\mathcat{Rep}(Q,\mathcat B_{S,Q})$,
cf. {\rm \eqref{Thetas}}, yields a homomorphism
\begin{equation}
\theta
\colon
\mathrm{XB}(S,Q) \longrightarrow
k\mathcat{Rep}(Q,\mathcat B_{S,Q})
\label{theta}
\end{equation}
of abelian groups, cf.
\cite[Section 3, a few lines before Theorem 3.2; Theorem 3.3]{MR1803361}.
We  therefore refer to
$k\mathcat{Rep}(Q,\mathcat B_{S,Q})$ as the 
{\em generalized crossed Brauer group\/} of $S$ 
relative to $\kappaQ \colon Q \to \Aut(S)$.
Let $\mathrm{can}\colon \mathrm B(S,Q)\to \mathrm H^0(Q,\mathrm B(S))$ 
denote the canonical injection.

\begin{thm}
\label{tentwentytw}
\noindent
{\rm (i)} 
The diagram
\begin{equation}
\begin{gathered}
\xymatrix{
0 \ar[r] 
&\XBSQ \ar[d]^{\iota}  \ar[r] 
&\mathrm{XB}(S,Q) \ar[d]^{\theta}  \ar[r]
&\mathrm B(S,Q)\ar[d]^{\mathrm{can}}  
\\ 
0 \ar[r]
&\mathrm H^1(Q,\Pic (S)) \ar[r]^{j_{\mathcat B_{S,Q}}} 
&k\mathcat{Rep}(Q,\mathcat B_{S,Q}) \ar[r]^{\mumu_{\mathcat B_{S,Q}}} 
&\mathrm H^0(Q,\mathrm B(S))
}
\end{gathered}
\label{CD02}
\end{equation}
is a commutative diagram of abelian groups with exact rows whence, 
in particular, the homomorphism $\theta$ is injective;
here $j_{\mathcat B_{S,Q}}$ is the homomorphism of abelian group 
{\rm \eqref{inj1}} above for $\mathcat C_Q=\mathcat B_{S,Q}$,
and $\mumu_{\mathcat B_{S,Q}}$ refers to the canonical homomorphism, 
cf., e.~g., the sequence \eqref{FW} above.

\noindent
{\rm (ii)} The composite of $\theta$ with {\rm \eqref{teich1}}
coincides with {\rm \eqref{teich2}}.

\noindent
{\rm (iii)} If the group $Q$ is finite, the homomorphism
$\theta$ is surjective and hence an isomorphism
whence
$\iota$ is then an isomorphism as well.

\noindent
{\rm (iv)} If the group $Q$ is finite, the homomorphism
$\mathrm{can}\colon \mathrm{B}(S,Q) \to \mathrm H^0(Q,\mathrm B(S))$
is the identity.
\end{thm}

In particular,
when the group $Q$ is a finite group,
the exact sequence \eqref{FW}
is valid with $\mathrm {XB}(S,Q)$
substituted for 
$k\mathcat{Rep}(Q,\mathcat B_{S,Q})$.

\begin{rema}
\label{frwall}
The injectivity of \eqref{theta} may be found
in \cite[Theorem 4 p.~43]{genbrauer}, \cite[Theorem 3.3]{MR1803361},
but spelled out for monoids rather than groups.
The commutative diagram \eqref{CD02} is essentially the diagram
in \cite[Theorem 3.3]{MR1803361}.
Statement (ii) above is equivalent to the statement of
\cite[Theorem 3.4 (iii)]{MR1803361}.
Statement (iii) above is equivalent to the statement of
\cite[Theorem 4.1 (ii)]{MR1803361}.
\end{rema}

\begin{proof}
Commutativity of the diagram \eqref{CD02} is straightforward, 
and the injectivity
of $\iota$, established in Theorem \ref{10.2} above, 
entails that of $\theta$.
 Statement (ii) follows at once from 
Theorem \ref{genaz}. 
Statement (iii) is an immediate consequence of the fact that 
in case the group $Q$ is finite,
given the generalized $Q$-normal Azumaya $S$-algebra
$(A,\sigma)$,
the associated $Q$-normal algebra 
$(\BBB^{\mathrm{op}},\sigma_{\substack{\mbox{\tiny{$\BBB^{\mathrm{op}}$}}}})$, cf.
Theorem \ref{fundclass1}, is a $Q$-normal Azumaya $S$-algebra
that represents a member of $\mathrm{XB}(S,Q)$
which, under $\theta$, goes to the class of $(A,\sigma)$ 
in $k\mathcat{Rep}(Q,\mathcat B_{S,Q})$.
Statement (iv) is the statement of Corollary \ref{2.21}.
\end{proof}

\section{The equivariant Brauer group}
\label{eleven}

As before, $S$ denotes a commutative ring and 
$\kappaQ\colon Q \to \Aut(S)$ an action of a group $Q$ on $S$. 

\subsection{The construction}

Given a split algebra  $\End_S(M)$ for some faithful
$S$-module $M$,
we  refer to an induced $Q$-equivariant structure
$\phi\colon Q \to \Aut(\End_S(M))$ on $\End_S(M)$ 
that arises from an $S^tQ$-module structure
on $M$ as a {\em trivially induced\/} $Q$-{\em equivariant
structure\/}, and we then refer to
\linebreak 
$(\End_S(M),\phi) $ as a
{\em trivially induced $Q$-equivariant split algebra\/}.
We  say that
two {$Q$-equivariant} Azumaya $S$-algebras $(A_1,\tau_1)$ and 
 $(A_2,\tau_2)$ are {\em equivariantly Brauer equivalent\/}
if there are 
$S^tQ$-modules $M_1$ and $M_2$
whose  
underlying $S$-modules are
faithful finitely
generated projective such that, relative to the associated
$Q$-equivariant structures
\[
\phi_1\colon Q \to \Aut(B_1),
\
B_1 = \End_S(M_1),
\
\phi_2\colon Q \to \Aut(B_2),
\
B_2 = \End_S(M_2),
\]
the algebras
$(A_1 \otimes B_1, \tau_1 \otimes \phi_1 )$
and $(A_2 \otimes B_2, \tau_2 \otimes \phi_2)$
are isomorphic $Q$-equivariant {$S$-algebras}.
Since the tensor
product of two trivially induced $Q$-equivariant 
split algebras is again a trivially induced 
$Q$-equivariant
split algebra in an obvious manner, 
that relation 
is an equivalence relation, 
referred to henceforth as {\em equivariant Brauer equivalence\/},
cf. Section~\ref{9} above or \cite[p.~381]{MR0121392};
under the operation of
tensor product 
and under the assignment
to the class of an equivariant algebra $(A,\tau)$
of the class of $(A^{\mathrm{op}},\tau^{\mathrm{op}})$, 
the equivalence classes 
constitute an abelian group,
the identity element being
the equivalence class of
trivially induced $Q$-equivariant split algebras $(\End_S(M),\tau)$
where $M$ ranges over $S^tQ$-modules
whose underlying $S$-module is faithful and finitely generated projective.
This group is the 
{\em equivariant Brauer group\/} of $S$ with respect to
$\kappaQ\colon Q \to \Aut(S)$,  introduced 
in \cite{MR0409424} and 
\cite[p.~40]{genbrauer}.
We  denote this group by
$\mathrm{EB}(S,Q)$,
and we refer to the construction just given as the 
{\em standard construction\/}.

\subsection{Some properties of the equivariant Brauer group}

Let $R = S^Q$.
It is manifest that extension of scalars yields an obvious homomorphism
\linebreak
$\mathrm B(R) \to \mathrm{EB}(S,Q)$. If 
$S|R$ is a Galois extension of commutative rings with Galois
group $Q$, by Galois descent, cf. Subsection \ref{galext}~(ii),
that homomorphism is actually an isomorphism.

In the general case of an arbitrary action $\kappaQ \colon Q \to \Aut(S)$
of $Q$ on $S$, the following holds.

\begin{thm}
\label{elevo}
{\rm (i)} The equivariant Brauer group is a functor on the 
change of actions category $\mathcat{Change}$ introduced in
Subsection {\rm \ref{coa}}.

\noindent
{\rm (ii)} 
The assignment to a $Q$-equivariant Azumaya $S$-algebra of its
canonically associated $Q$-normal $S$-algebra yields a natural homomorphism
\begin{equation*}\res \colon \mathrm{EB}(S,Q) \longrightarrow \mathrm{XB}(S,Q).\end{equation*}

\noindent
{\rm (iii)} 
The composite
\begin{equation}
\label{11.2}
\mathrm{EB}(S,Q) \stackrel{\res}\longrightarrow 
\mathrm{XB}(S,Q) \stackrel{t}\longrightarrow 
\mathrm H^3(Q,\mathrm U(S)) 
\end{equation}
is zero. If furthermore,
the group $Q$ is finite, the sequence {\rm \eqref{11.2}}
is exact and, furthermore, necessarily natural in the data.
\end{thm}

\begin{proof}
This follows from the observation that induced equivariant structures
on split algebras are preserved under change of rings, together with 
Proposition~\ref{2.5.1}(ii)
and Theorem~\ref{fouro}; we only note that in case that
$Q$ is a finite group the matrix algebra
$\mathrm M_{|Q|}(A)$ over an Azumaya $S$-algebra is again an Azumaya 
$S$-algebra. \end{proof}

\subsection{Group extensions and equivariant Azumaya 
algebras}

We maintain the notation $R=S^Q$. Let
$
\mathrm e_Q \colon \mathrm U(S) \stackrel{\iEQ}
\rightarrowtail \EQ 
\stackrel{\piEQ}
\twoheadrightarrow Q
$
be a group extension.
Then the crossed product 
 $R$-algebra
$(S,Q,{\mathrm e}_Q,\kappaQ \circ \piEQ)$ is defined.
By construction,
the left 
$S$-module  
$\BbB$
that underlies the algebra $(S,Q,{\mathrm e}_Q,\kappaQ \circ \piEQ)$
is free with basis in one-one 
correspondence with the elements of $Q$,
and the $Q$-equivariant structure \eqref{tauo}, viz.
\[
\tau_{\mathrm e_Q}\colon Q \longrightarrow \Aut({}_S\End (\BbB)),
\]
is defined.
When the group $Q$ is finite, the algebra 
${}_S\End(\BbB)$ is an Azumaya $S$-algebra.

\begin{prop}
Suppose that the group $Q$ is finite. Then the assignment to a group extension
$\mathrm e_Q$ of $Q$ by $\mathrm U(S)$ of the $Q$-equivariant algebra 
$({}_S\End(\BbB)^{\mathrm{op}},\tau^{\mathrm{op}}_{\mathrm e_Q})$
yields a homomorphism
\begin{equation}
\mathrm{cpr} \colon 
\mathrm H^2(Q,\mathrm U(S)) \longrightarrow \mathrm{EB}(S,Q)
\label{inftwty}
\end{equation}
of abelian groups that is natural on the change of actions category
$\mathcat{Change}$. \qed
\end{prop}

\begin{rema} When $S|R$ is a Galois extension of rings
with Galois group $Q$, Galois descent, cf. Subsection \ref{galext}~(ii), 
 yields an isomorphism
$\mathrm{EB}(S,Q) \cong \mathrm B(R)$, and \eqref{inftwty}
amounts to the composite of \eqref{ordinarycp} with the injection
$\mathrm B(S|R) \to \mathrm B(R)$.
\end{rema}

\section{The seven term exact sequence}
\label{12}

\begin{thm}
\label{twelvetw}
Suppose that the group $Q$ is finite. 
Then the extension
\begin{equation}
\ldots \longrightarrow 
(\Pic (S))^Q
\stackrel{\dDelta}\longrightarrow \mathrm H^2(Q,\mathrm U(S))
\stackrel{\mathrm{cpr}} \longrightarrow \mathrm{EB}(S,Q) \stackrel{\res}\longrightarrow
\mathrm{XB}(S,Q)
\stackrel{t}\longrightarrow \mathrm H^3(Q,\mathrm U(S))
\label{twelvet}
\end{equation}
of the exact sequence {\rm \eqref{ldes}}
is defined and yields a seven term exact sequence that is natural
in terms of the data.
If, furthermore, $S|R$ is a Galois extension of commutative rings 
over $R=S^Q$ with 
group $Q$, then,
with $\Pic (S|R)$,
$\Pic (R)$ and $\mathrm B(R)$ substituted for,
respectively 
$\mathrm H^1(Q,\mathrm U(S))$,
$\mathrm{EPic} (S,Q)$ and $\mathrm{EB}(S,Q)$,
the homomorphisms {\rm cpr} and {\rm res} being
modified accordingly,
the sequence is exact as well.
\end{thm}

\begin{rema}
The lower long sequence in
\cite[Theorem 4.2]{MR1803361}
yields a long exact sequence of the kind
\eqref{twelvet}.
\end{rema}

\begin{proof}
The {\em exactness at\/} $\mathrm{XB}(S,Q)$ follows from 
Theorem~\ref{fouro} or
Theorem \ref{elevo}.

\noindent{\em Exactness at\/} $\mathrm H^2(Q,\mathrm U(S))$:
Let $J$ be a finitely generated projective rank one $S$-module
representing a class in $(\Pic (S))^Q$, 
consider the group $\Aut(J,Q)\cong \Aut_{\mathcat{Pic}_{S,Q}}(J)$, 
i.~e., the group \eqref{eightthirtythree}
with $J$ substituted for $M$,
and let
\begin{equation}
\mathrm e_J \colon 0 \longrightarrow 
\mathrm U(S) \longrightarrow 
\Aut(J,Q) 
\stackrel{\pi_J}\longrightarrow Q \longrightarrow 1
\label{eJ}
\end{equation}
be the associated group extension 
\eqref{eUCC} with $\mathcat{Pic}_{S,Q}$ substituted for $\mathcat C_Q$
so that 
\[
\dDelta [J] = [\mathrm e_J]\in \mathrm H^2(Q,\mathrm U(S)).
\]
Consider
the crossed product $R$-algebra $(S,Q,{\mathrm e}_J,\kappaQ \circ\pi_J)$, 
let 
$M_{\mathrm e_J}$ denote the free $S$-module
that underlies  $(S,Q,{\mathrm e}_J,\kappaQ \circ\pi_J)$, and recall that
the corresponding association \eqref{assoc9}, now of the kind
$\Aut(J,Q) \times M_{\mathrm e_J} \to M_{\mathrm e_J}$,
induces, on  the  Azumaya
$S$-algebra $\End_S(M_{\mathrm e_J})$, 
via the association
\[
(\alpha,f) \longmapsto \alpha(f)=\alpha\circ f \circ \alpha^{-1}, \ 
\alpha \in \Aut(J,Q),\ f \in \End_S(M_{\mathrm e_J}),
\]
an induced $Q$-equivariant structure
$\tau_{\mathrm e_J}\colon Q \to \Aut(\End_S(M_{\mathrm e_J}))$
where we do not distinguish in notation between
$\alpha \in \Aut(J,Q)$ and 
the associated semi-linear  automorphism 
$\alpha\colon M_{\mathrm e_J} \to M_{\mathrm e_J}$.
By construction, then, the value $\mathrm {cpr}([\mathrm e_J]) \in 
\mathrm {EB}(S|S;Q,Q) \subseteq \mathrm {EB}(S,Q)$
is represented by the $Q$-equivariant Azumaya algebra
$(\End(M_{\mathrm e_J})^{\mathrm{op}},\tau_{\mathrm e_J}^{\mathrm{op}})$.

Since $J$ is a faithful finitely generated 
projective rank one $S$-module, the operation of
composition $(f,\varphi)\longmapsto f_{\sharp}(\varphi)=f \circ \varphi$, 
as $f$ ranges over $\End(M_{\mathrm e_J})$
and $\varphi$ over  $\Hom_S(J,M_{\mathrm e_J})$,
induces an isomorphism
\begin{equation}
\End_S(M_{\mathrm e_J}) \longrightarrow
\End_S(\Hom_S(J,M_{\mathrm e_J})),\ 
f \longmapsto f_{\sharp},
\label{iso9}
\end{equation}
of $S$-algebras.
Since the restriction
of the $\Aut(J,Q)$-action on $M_{\mathrm e_J}$
 to the kernel $\mathrm U(S)$ 
of $\pi_J$
 coincides with the $\mathrm U(S)$-module structure
coming from multiplication by members of the coefficient ring $S$,
the diagonal action
of $\Aut(J,Q)$ on $\Hom_S(J,M_{\mathrm e_J})$
given by the assignment to
$(\alpha,\varphi)$  of 
$\alpha_J(\varphi)=\alpha\circ \varphi \circ \alpha^{-1}$,
where $\alpha \in \Aut(J,Q)$ and $\varphi \in \Hom_S(J,M_{\mathrm e_J})$,
descends to an action
\begin{equation}
Q \times \Hom_S(J,M_{\mathrm e_J}) \longrightarrow
\Hom_S(J,M_{\mathrm e_J})
\label{desc1}
\end{equation}
of $Q$ on $\Hom_S(J,M_{\mathrm e_J})$
that turns  $\Hom_S(J,M_{\mathrm e_J})$
into an $S^tQ$-module.
The 
$S^tQ$-module structure \eqref{desc1} on
$\Hom_S(J,M_{\mathrm e_J})$, in turn,
induces 
a trivially induced $Q$-equivariant structure on 
$\End_S(\Hom_S(J,M_{\mathrm e_J}))$
via the association
\[
(\alpha,h) \longmapsto \alpha_J(h)=\alpha_J\circ h \circ \alpha_J^{-1}, \ 
\alpha \in \Aut(J,Q),\ h\in \End_S(\Hom_S(J,M_{\mathrm e_J})).
\]
By construction, given $\varphi \in \Hom_S(J,M_{\mathrm e_J})$,
 $f \in \End_S(M_{\mathrm e_J})$, 
 and
$\alpha \in \Aut(J,Q)$,
\begin{align*}
(\alpha_J(f_{\sharp})) (\varphi) &=
(\alpha_J\circ f_{\sharp} \circ \alpha_J^{-1})(\varphi)=
\alpha_J (f_{\sharp} (\alpha_J^{-1} (\varphi)))
=
\alpha_J (f_{\sharp} (\alpha^{-1} \circ \varphi \circ \alpha))
\\
&=\alpha_J (f \circ\alpha^{-1} \circ \varphi \circ \alpha)
=\alpha\circ (f \circ\alpha^{-1} \circ \varphi \circ \alpha)\alpha^{-1}
=\alpha \circ f \circ\alpha^{-1} \circ \varphi
\end{align*}
whence
the induced $Q$-equivariant structure 
$\tau_{\mathrm e_J}\colon Q \to
\Aut(\End_S(M_{\mathrm e_J}))$ 
gets identified, under the isomorphism \eqref{iso9}, with the 
trivially induced
$Q$-equivariant structure induced by the $S^tQ$-module structure
\eqref{desc1} on 
$\Hom_S(J,M_{\mathrm e_J})$.
Consequently $(\End_S(M_{\mathrm e_J}),\tau_{\mathrm e_J})$
represents zero in $\mathrm{EB}(S,Q)$.

Conversely, let
$\mathrm e \colon \mathrm U(S) \rightarrowtail \Ggamma\stackrel{\pi} 
\twoheadrightarrow Q$
be a group extension, 
consider the associated crossed product algebra
$(S,Q,\mathrm e, \kappaQ\circ \pi)$, 
let $M_{\mathrm e}$ denote the free $S$-module that underlies the algebra
$(S,Q,\mathrm e, \kappaQ\circ \pi)$,
let $\tau_{\mathrm e}\colon Q \to \Aut(\End_S(M_{\mathrm e}))$
denote the corresponding induced $Q$-equivariant structure,
and suppose that $(\End_S(M_{\mathrm e}), \tau_{\mathrm e})$ represents
zero in $\mathrm{EB}(S,Q)$. 
In view of equivariant Brauer equivalence,
there are 
$S^tQ$-modules $M_1$ and $M_2$ whose
underlying $S$-modules
are faithful
and finitely generated
projective so that,
in terms of the notation
$\tau_1\colon Q \to \Aut(\End_S(M_1))$
and
$\tau_2\colon Q \to \Aut(\End_S(M_2))$
for the associated trivially induced
$Q$-equivariant structures,
$(\End_S(M_{\mathrm e}),\tau_{\mathrm e})\otimes 
(\End_S(M_1),\tau_1)$ and $(\End_S(M_2),\tau_2)$ are isomorphic 
as $Q$-equivariant $S$-algebras.
Let $J 
= \Hom_{\End_S(M_{\mathrm e}\otimes M_1)}(M_{\mathrm e}\otimes M_1,M_2)$;
this is a
faithful
 finitely generated projective rank one $S$-module.
Moreover, the association
\begin{equation*}
\Ggamma \times J \longrightarrow J, \ 
(x,f)\mapsto \pi(x) f x^{-1} \colon M_{\mathrm e}\otimes M_1 
\longrightarrow M_2 ,
\ x\in \Ggamma,\ f\in J,
\end{equation*}
where we do not distinguish in notation 
between $x\in \Ggamma$ and the 
induced automorphisms of 
\linebreak
$M_{\mathrm e}\otimes M_1$ 
nor between $\pi(x) \in Q$ and the induced automorphism
of $M_2$, yields a
semi-linear action of $\Ggamma$ on $J$;
hence the group extension  $\mathrm e_J$, cf. \eqref{eJ}, is defined
relative to $J$,  
and the $\Ggamma$-action on $J$, in turn,
induces a homomorphism $\Ggamma \to \Aut(J,Q)$ and
hence yields a congruence
$(1,\cdot,1) \colon {\mathrm e} \to \mathrm e_J $ of group extensions;
in particular, $[J]\in (\Pic (S))^Q$. 
The congruence
$(1,\cdot,1) \colon {\mathrm e} \to \mathrm e_J $ of group extensions
entails that
$\dDelta [J] = [\mathrm e]\in \mathrm H^2(Q,\mathrm U(S))$.

\noindent{\em Exactness at $\mathrm{EB}(S,Q)$\/}:
Since for a group extension $\mathrm e$ of 
$\mathrm U(S)$ by $Q$ the $Q$-equivariant structure
$\tau_{\mathrm e}$ on the algebra $\End_S(M_{\mathrm e})$ 
of $S$-linear endomorphisms of
the free $S$-module $M_{\mathrm e}$ that underlies
the corresponding crossed product algebra
$(S,Q,\mathrm e, \kappaQ\circ \pi)$
is an induced $Q$-equivariant structure, it is 
obvious that the composite $\res \circ \,\mathrm{cpr}$ is
zero.

Now we  show  that $\ker (\res ) \subset \im (\mathrm{cpr} )$.
Thus, let $(A,\tau )$ 
be a $Q$-equivariant Azumaya $S$-algebra, and suppose that
the class 
of its associated $Q$-normal algebra
$(A,\sigma_\tau)$
goes to zero in $\mathrm{XB}(S,Q)$.
In view of normal Brauer equivalence, there
are induced $Q$-normal split algebras
$(\End_S(M_1),\sigma_1)$ and $(\End_S(M_2),\sigma_2)$
such that 
$
(A,\sigma_{\tau})\otimes (\End_S(M_1),\sigma_1)$
and 
$(\End_S(M_2),\sigma_2)$
are isomorphic as $Q$-normal $S$-algebras.
Replacing $M_1$ and $M_2$ by
$S^{|Q|}\otimes M_1$ and $S^{|Q|}\otimes M_2$, respectively,
and adjusting the notation accordingly,
since  $(\End_S(M_1),\sigma_1)$ has zero Teichm\"uller class,
by Theorem \ref{fouro},
the $Q$-normal structure $\sigma_1$ on
$\End_S(M_1)$ lifts to a $Q$-equivariant structure
$\tau_1\colon Q \to
\End_S(M_1)$ on $\End_S(M_1)$
such that
$\sigma_1=\sigma_{\tau_1}$.  Thus
$
(A\otimes \End_S(M_1),\sigma_{\tau \otimes \tau_1})
$
and $(\End_S(M_2),\sigma_2)$
are isomorphic as $Q$-normal $S$-algebras.
Since $\sigma_{\tau \otimes \tau_1}$ is equivariant,
so is $\sigma_2$;
more precisely, the $Q$-equivariant structure
$\tau \otimes \tau_1$ induces, 
on $\End_S(M_2)$,
via the
isomorphism
$A\otimes \End_S(M_1)\cong \End_S(M_2)$,
a $Q$-equivariant structure
$\tau_2\colon Q \to \End_S(M_2)$ such that
$(A\otimes \End_S(M_1),\tau \otimes \tau_1)
\cong (\End_S(M_2),\tau_2)$, and
$\sigma_2=\sigma_{\tau_2}$.
Since
$\sigma_2$ is an induced $Q$-normal structure,
the  homomorphism
$\pi^{\substack{\mbox{\tiny{$\Aut(M_2,Q)$}}}}\colon \Aut(M_2,Q)\to Q$ is surjective;
by Lemma \ref{8.00},
the induced homomorphism 
$\pi^{\substack{\mbox{\tiny{$\Aut(M_2,Q,\tau_2)$}}}}\colon 
\Aut(M_2,Q,\tau_2)\to Q$ 
is surjective
as well
whence the $Q$-equivariant structure
$\tau_2$ on $M_2$ is induced.

By Proposition \ref{diagact}, with respect to associated group
extensions 
$\mathrm e_1\colon 
\mathrm U(S) \stackrel{i_1}\rightarrowtail \Ggamma_1 \stackrel{\pi_1}
\twoheadrightarrow Q$
and
$\mathrm e_2\colon 
\mathrm U(S) \stackrel{i_1}\rightarrowtail \Ggamma_2 \stackrel{\pi_2}
\twoheadrightarrow Q$
and morphisms
\begin{align*}
(j_1,\beta_1)&\colon (\mathrm U(S),\Ggamma_1,i_1) \longrightarrow
(\Aut_S(M_1),\Aut(\End_S(M_1)),\partial_1),
\\
(j_2,\beta_2)&\colon (\mathrm U(S),\Ggamma_2,i_2) \longrightarrow
(\Aut_S(M_2),\Aut(\End_S(M_2)),\partial_2)
\end{align*}
of crossed modules 
inducing the $Q$-equivariant structures $\tau_1$ and $\tau_2$, respectively,
the $Q$-equivariant $S$-algebra
$(\End_S(M_1),\tau_1)$
is equivariantly Brauer equivalent
to a $Q$-equivariant $S$-algebra of the kind
$(\End_S(M_{\mathrm e_1}),\tau_{\mathrm e_1})$
and the $Q$-equivariant $S$-algebra
$(\End_S(M_2),\tau_2)$
to a $Q$-equivariant $S$-algebra of the kind
$(\End_S(M_{\mathrm e_2}),\tau_{\mathrm e_2})$.
Consequently 
the $Q$-equivariant $S$-algebra
$(A,\tau)$
is equivariantly Brauer equivalent
to a $Q$-equivariant $S$-algebra of the kind
$(\End_S(M_{\mathrm e}),\tau_{\mathrm e})$, for some
 group
extension 
$\mathrm e\colon 
\mathrm U(S) \rightarrowtail \Ggamma \twoheadrightarrow Q$
such that 
\[
[\mathrm e]+[\mathrm e_1]=[\mathrm e_2]\in \mathrm H^2(Q,\mathrm U(S))
\]
whence 
$
{\mathrm{cpr}([\mathrm e]) = [(A,\tau)]  \in 
\mathrm {EB}(S,Q)}$.

Suppose now that $S|R$ is a Galois extension of commutative rings
with Galois group $Q$. Then $(A,Q,\mathrm e,\kappaQ\circ\pi^\Ggamma)$ is an Azumaya $R$-algebra, see, e.~g., 
Proposition \ref{tpf}(xi) above; moreover,
by Proposition \ref{tpf}(ix) above,
the obvious homomorphism 
\begin{equation*}
S\otimes_R (A,Q,\mathrm e,\kappaQ\circ\pi^\Ggamma)^{\mathrm{op}} 
\longrightarrow \End_S(M_{\mathrm e})
\end{equation*}
is then
an isomorphism of $S$-algebras and, by  Proposition \ref{tpf}(x),
the $Q$-equivariant structure $\tau_{\mathrm e}$ comes from
scalar extension; hence 
the canonical homomorphism $\mathrm B(R) \to \mathrm{EB}(S,Q)$ 
is then an isomorphism.
Likewise, again by Galois descent, cf. Subsection \ref{galext}~(ii), 
the canonical homomorphisms $\mathrm H^1(Q,\mathrm U(S))\to \Pic (S|R)$ and
$\Pic (R) \to \mathrm{EPic} (S,Q)$ are isomorphisms.

We leave the  proofs of the naturality assertions to the 
reader. \end{proof}

\chapter{Crossed pairs and the relative case}
\label{cII}
\begin{abstract}
Using a suitable notion of normal Galois extension of
commutative rings,
we develop the relative theory of the generalized Teichm\"uller cocycle
map.
We interpret the theory in terms of the Deuring embedding problem,
construct an eight term exact sequence involving
the relative Teichm\"uller cocycle map
and suitable relative versions of generalized Brauer groups
and compare the theory with the
group cohomology eight term exact sequence 
involving crossed pairs.
We also develop somewhat more sophisticated 
versions of the ordinary, equivariant 
and 
crossed relative Brauer groups and show that the resulting exact sequences
behave better with regard to comparison of the theory with group cohomology 
than do 
the naive notions
of the generalized relative Brauer groups.
\end{abstract}

\section{Introduction}

We explore further the approach to the
\lq\lq Teichm\"uller cocycle map\rq\rq\ 
developed in 
\cref{cI}
in terms of crossed $2$-fold extensions.
We keep the numbering of items from the introduction
to \cref{cI}. 

\smallskip

\noindent (6)  In 
\cref{normalr},
we  introduce the concept of a $Q$-{\em normal Galois
extension of commutative rings\/}; associated with such a $Q$-normal
Galois extension  $T|S$ of commutative rings
is 
a {\em structure extension\/}
$\mathrm e_{(T|S)}\colon N\rightarrowtail G \twoheadrightarrow Q$
of $Q$ by the Galois group $N=\Aut(T|S)$  of $T|S$
and an action $G \to \Aut(T)$ of $G$ on $T$ by ring automorphisms.
In \cref{cpal},
we associate to a {\em crossed pair} $(\mathrm e, \psi )$
with respect to $\mathrm e_{(T|S)}$ and $\mathrm U(T)$, 
endowed with the $G$-module structure coming from the $G$-action on $T$,
see 
\cite{MR597986}
or \cref{cpal}
for details on the crossed pair concept, 
a $Q$-normal crossed product algebra
$(A_\mathrm e , \sigma_\psi )$, and
we refer to such a $Q$-normal algebra as a {\em crossed pair algebra}. 
The crossed pair algebra $(A_\mathrm e , \sigma_\psi )$ represents
a member of the kernel  $\mathrm{XB}(T|S; G, Q)$ of the obvious 
homomorphism from
$\mathrm{XB}(S,Q)$ to $\mathrm{XB}(T,G)$; this homomorphism 
exists and is unique, in view of the
functoriality of the crossed Brauer group. Actually,  the assignment to
$(\mathrm e , \psi )$ of $(A_\mathrm e , \sigma_\psi )$ yields a natural 
homomorphism
\begin{equation}
\mathrm{Xpext} (G,N;\mathrm U(T)) \longrightarrow \mathrm{XB}(T|S;G,Q)
\label{nat1}
\end{equation}
of abelian groups from 
the 
corresponding
abelian
group $\mathrm{Xpext} (G,N;\mathrm U(T)) $ of
congruence classes of crossed pairs introduced in
\cite{MR597986} to the subgroup $\mathrm{XB}(T|S;G,Q)$ of the
crossed Brauer group.
Concerning the crossed pair concept, suffice it to mention at this stage that
the notion of crossed pair generalizes that of crossed module.

\smallskip

\noindent (7) Let $T|S$ be a $Q$-normal Galois extension of commutative rings, 
with structure extension 
$\mathrm e_{(T|S)}\colon N\rightarrowtail G \twoheadrightarrow Q$
of $Q$ by $N=\Aut(T|S)$
and associated $G$-action on $T$. 
\cref{normalcrossed}
says that 
{\em a class  $k \in \mathrm H^3(Q, \mathrm U(S))$ 
is the Teich\-m\"uller class of some crossed pair algebra 
$(A_\mathrm e , \sigma_\psi )$
with respect to the data if and only if $k$ is split in}  $T| S$
in the sense that, under inflation 
$\mathrm H^3(Q,\mathrm U(S)) \to \mathrm H^3(G,\mathrm U(T))$,
the class $k$ goes to zero.
This yields a partial answer to the question as to which classes in 
$\mathrm H^3(Q,\mathrm U(S))$ are Teich\-m\"uller classes.

Now, in the classical situation, 
$T|S$ is a Galois field extension with Galois group $Q$, 
and we denote by $\mathrm B(T|S)$ the subgroup
of the Brauer group $\mathrm B(S)$ of $S$ 
that consists of the classes split by $T$;
the group
$\mathrm{Xpext}(G,N;\mathrm U(T))$ of crossed pair extensions
now comes down to  the subgroup $\mathrm H^2(N,\mathrm U(T))^Q$
of 
$\mathrm H^2(N,\mathrm U(T))$ that consists of the classes
fixed by $Q$,
the crossed 
Brauer group
$\mathrm{XB}(T|S;G,Q)$  then boils down to the subgroup
$\mathrm B(T|S)^Q$ of $\mathrm B(T|S)$ that consists of the classes
fixed by $Q$---this is related with Speiser's principal genus 
theorem (which is often 
referred to as Hilbert's Satz~90)---, and the homomorphism \eqref{nat1} 
amounts to 
the 
classical isomorphism
\begin{equation}
\mathrm H^2(N,\mathrm U(T))^Q
\longrightarrow
\mathrm B(T|S)^Q.  
\label{nat2}
\end{equation}
Hence~(6) generalizes the 
corresponding result of Eilenberg and Mac Lane \cite{MR0025443}, and 
so does~(7):
 For Eilenberg-Mac Lane's characterization of the Teich\-m\"uller classes 
given in \cite[Theorem 7.1, Theorem 10.1] {MR0025443} 
is actually a crossed product theorem which
covers all the Teich\-m\"uller classes since classically each $Q$-normal algebra
class contains a $Q$-normal crossed product algebra, but such a result does not
hold in the general situation considered here.

\smallskip

\noindent (8)  In 
\cref{7},
given a $Q$-normal Galois extension $T|S$ of commutative rings, 
with structure extension 
$\mathrm e_{(T|S)}\colon N\rightarrowtail G \twoheadrightarrow Q$
of $Q$ by $N=\Aut(T|S)$,
we again focus our attention on
the Deuring embedding problem of a central $T$-algebra
into a central $S$-algebra and establish two
somewhat technical results, 
\cref{7.1} and \cref{7.2}.
These results entail,
in particular that, if 
a class
$k \in \mathrm H^3(Q,\mathrm U(S))$ goes under inflation to the 
Teich\-m\"uller class
in $\mathrm H^3(G,\mathrm U(T))$ of some $G$-normal central $T$-algebra
$A$, 
then $k$ is itself
the Teich\-m\"uller class of some $Q$-normal central $S$-algebra $B$
in such a way that,
when
$A$ is a $G$-normal Azumaya $T$-algebra, 
the algebra $B$ may be taken to be a $Q$-normal Azumaya $S$-algebra.

\smallskip

\noindent (9) Given a $Q$-normal
Galois extension  $T|S$ 
of commutative rings
with associated structure extension
$\mathrm e_{(T|S)}\colon \Aut(T|S)\rightarrowtail G \twoheadrightarrow Q$
and $G$-action on $T$,
let $\mathrm{EB}(T|S;G,Q)$ denote the kernel of the induced homomorphism
from $\mathrm{EB}(S,Q)$ to $\mathrm{XB}(T,G)$;
the exact sequence 
\eqref{AAAAA}
involving the Teichm\"uller map $t$ now
yields an extension
of the kind
\begin{equation*}
\ldots 
\to
\mathrm H^2(Q,\mathrm U(S))
\to
 \mathrm{EB}(T|S;G,Q)
\to
\mathrm{XB}(T|S;G,Q) 
\stackrel{t}
\to
\mathrm H^3(Q,\mathrm U(S))
\stackrel{\inf}\to \mathrm H^3(G,\mathrm U(T))
\end{equation*}
of the corresponding classical low degree four term exact sequence
by four more terms.
We  refer to the resulting theory as the
{\em naive relative theory\/}.
The exactness of that sequence
at $\mathrm H^3(Q,\mathrm U(S))$ follows
from (4) above and from the naturality of the Teich\-m\"uller map.  
In 
\cref{eighttermcomp}
we compare that exact sequence 
with the eight
term exact sequence in the cohomology of the group extension 
$\mathrm e_{(T|S)}$
with coefficients in $\mathrm U(T)$ constructed in
\cite{MR597986}.
\smallskip

\noindent
(10) A more sophisticated variant
of the relative theory involves
certain abelian groups which we denote by
$\mathrm B_{\mathrm{fr}}(T|S)$,
$\mathrm{EB}_{\mathrm{fr}}(T|S;G,Q)$,
$\mathrm{XB}_{\mathrm{fr}}(T|S;G,Q)$, and
$k\mathcat{Rep}(Q,\mathcat B_{T|S;G,Q})$,
see 
Subsection \ref{variantrt};
these groups fits into 
seven and eight term exact sequences
 similar to
those involving
$\mathrm B(S)$,
$\mathrm{EB}(S,Q)$,
$\mathrm{XB}(S,Q)$, 
$\mathrm B(T|S)$,
$\mathrm{EB}(T|S;G,Q)$,
$\mathrm{XB}(T|S;G,Q)$, and
$k\mathcat{Rep}(Q,\mathcat B_{S,Q})$.
This more sophisticated variant behaves better with 
regard to comparison of the theory
with group cohomology than does the naive relative theory; see
Theorems \ref{CTC} - \ref{CTCC} and Theorem \ref{withoutp}.
\smallskip

In \cite[Theorem 3 p.~303]{MR641328}, we developed a conceptual description 
of the differential 
$d_3 \colon \mathrm E^{0,2}_3 \to \mathrm E^{3,0}_3$ of
the Lyndon-Hochschild-Serre spectral sequence
associated with a group extension and a module over the extension group,
and in \cite{MR597986} we developed an eight term exact sequence, given there
as (1.7) and (1.8), that involves a certain
map $\Delta$, cf. 
\eqref{13.11} below;
in view of
\cite[Subsection 1.4]{MR597986},  
the map $\Delta$ lifts that differential to a map.
These observations entail that
the exactness of 
\eqref{AAAAA}
generalizes the result of Hochschild
and Serre~\cite[p.~130]{MR0052438} saying 
that, in the classical case when the rings under discussion are fields, 
the Teich\-m\"uller map coincides with
the transgression map in the low degree five term exact sequence in the
cohomology of the corresponding group extension.

The appendix recollects some material from the theory 
of stably graded symmetric monoi\-dal categories.

\section{Normal ring extensions}
\label{normalr}

As in \cref{cI},
$S$ denotes a commutative ring and $\kappaQ\colon Q \to \Aut(S)$
an action of a group $Q$ on $S$.
Let $T|S$ be a Galois extension of commutative rings with Galois group $N=\Aut(T|S)$.
We refer to $T|S$ as being $Q$-{\em normal} when each automorphism
$\kappaQ(q)$ of $S$, as $q$ ranges over $Q$, extends to an automorphism of $T$.

Somewhat more formally, 
given a Galois extension $T|S$ of commutative rings with Galois group
$N$,
denote by $\Aut^S(T)$ the group of those automorphisms
of $T$ that map $S$ to itself, let 
$\res \colon \Aut^S(T) \to \Aut(S)$ denote the obvious restriction map,
so that  $N=\Aut(T|S)$ is the kernel of $\res$,
let
$G$ denote the fiber product group
$G = \Aut^S(T) \times_{\Aut(S)} Q$ relative to $\kappaQ\colon Q \to \Aut(S)$,
and let $\pi_Q \colon G \to Q$ denote the canonical homomorphism
and
$i^N\colon N \to G$ the obvious injection.
The obvious homomorphism $\llambda\colon G \to \Aut^S(T)$
makes the diagram 
\begin{equation}
\begin{CD}
1 @>>> N 
@>{i^N}>> G @>{\pi_Q}>> Q 
\\
@.
@|
@V{\llambda}VV
@V{\kappaQ}VV
\\
1 @>>> \Aut(T|S) 
@>>>  \Aut^S(T)  @>{\res}>> \Aut(S) 
\end{CD}
\label{CD1}
\end{equation}
commutative,
 where the unlabeled 
arrow is the obvious
homomorphism. This diagram is a special case of a diagram of the kind
\eqref{changea2}.
The Galois extension
$T|S$ of commutative rings is plainly $Q$-normal if and only if
the homomorphism $\pi_Q\colon G \to Q$ is surjective,
that is, if and only if the sequence
\begin{equation}
\mathrm e_{(T|S)} \colon 1 
\longrightarrow N 
\stackrel{i^N} \longrightarrow G 
\stackrel{\pi_Q}
\longrightarrow Q \longrightarrow 1
\label{eTS0}
\end{equation}
is exact, i.e., an extension of $Q$ by $N$.
Given a $Q$-normal Galois extension $T|S$ of commutative rings,
we refer to the corresponding group extension 
\eqref{eTS0} as the {\em associated structure extension\/}
and to the corresponding homomorphism $\llambda \colon G \to \Aut^S(T)$
as the {\em associated structure homomorphism\/}.
It is immediate
that a $Q$-normal Galois extension
$T|S$ with structure extension \eqref{eTS0} and
structure homomorphism
$\llambda\colon G \to \Aut^S(T)$,
the injection $S \subseteq T$ being denoted by $i\colon S \subseteq T$,
yields the morphism
\begin{equation}
(i,\pi_Q )\colon(S,Q,\kappaQ) \longrightarrow (T,G,\llambda)
\label{mqng}
\end{equation}
in the change of actions category $\mathcat{Change}$ introduced in 
Subsection \ref{coa} above.

\begin{examp}
\label{exa1}
Let $K|P$ be a Galois extension of algebraic number fields, and denote 
by $G$ the Galois group of 
$K|P$. Let $Z$ be a subfield of $K$ that contains $P$ and is
a normal extension of $P$, and let $N = \Gal (K|Z)$ and $Q = \Gal (Z|P)$.
Let $T$, $S$ and $R$ denote the rings of integers in, respectively, $K$, $Z$
and $P$. Suppose that $K|Z$ is unramified but 
that $Z|P$ is ramified. Then $T|S$ is
a $Q$-normal Galois extension of commutative rings but $T|R$ and $S|R$
are not Galois extensions  of commutative rings, cf. 
\cref{exatwo}.
\end{examp}

Let $(S, Q, \kappa)$ and $(\hat S, \hat Q, \hat \kappa)$ 
be objects of the change of actions category $\mathcat{Change}$ introduced
in Subsection \ref{coa},
and let $T|S$ and $\hat T|\hat S$ be normal Galois extension of commutative rings with respect
to $Q$ and $\hat Q$, with structure extensions  
\[
\mathrm {\mathrm e}_{(T|S)} \colon
N\rightarrowtail G \twoheadrightarrow Q,\quad
\mathrm e_{(\hat T|\hat S)} \colon 
\hat N \rightarrowtail \hat G \twoheadrightarrow \hat Q
\]
and structure homomorphisms $\llambda\colon G \to \Aut^S(T)$
and
$\hat \llambda\colon \hat G \to \Aut^{\hat S}(\hat T)$,
respectively. Then a {\em morphism}
\begin{equation*}
(h,\phi) \colon T|S \longrightarrow \hat T|\hat S
\end{equation*}
{\em of normal Galois extensions}
 consists of a ring homomorphism 
$h \colon T \to \hat T$ and a 
group homomorphism $\phi\colon \hat G \to G$
such that

(i) $f = h|S$ is a ring homomorphism $S \to \hat S$,

(ii) the values of $\phi|\hat N$ lie in $N$, that is,
$\phi|\hat N$ is a homomorphism $\hat N \to N$, and

(iii) $h ({}^{\phi (\hat x)} t) = {}^{\hat x}(h(t)), 
\ \hat x \in \hat G, \ t\in T$.

\section{Crossed pair algebras}
\label{cpal}

As before, $S$ denotes a commutative ring and $\kappaQ\colon Q \to \Aut(S)$
an action of a group $Q$ on $S$.
In this section we  use the results of \cite{MR597986} 
to offer a partial
answer to the question as to 
which classes in $\mathrm H^3(Q,\mathrm U(S))$
are Teich\-m\"uller classes. Our result extends the classical answer 
of Eilenberg and Mac Lane \cite{MR0025443} (reproduced in 
\cite{MR0052438}); later in the paper we shall give a complete answer.

\subsection{Crossed pairs}
\label{6.11}
For intelligibility, we recall that notion from \cite[p.~152]{MR597986}.

Let
\begin{equation}
\begin{CD}
1 @>>> N 
@>{i^N}>> G @>>> Q @>>> 1
\label{ge}
\end{CD}
\end{equation}
be a group extension and $M$ a $G$-module;
we write the $G$-action $G \times M \longrightarrow M$ on $M$ as
$(x,y) \mapsto  {}^xy$, for $x \in G$ and $y \in M$.
Further, let
$\mathrm e \colon M \rightarrowtail \Gamma \stackrel{\piN} \twoheadrightarrow N$ 
be a group extension 
whose class $[\mathrm e] \in \mathrm H^2 (N,M)$ is fixed under the standard $Q$-action
on $\mathrm H^2(N,M)$. 
Given $x\in G$, we write
\begin{equation*}
\ell_x(y) = {}^xy, \ y \in M,\quad 
i_x(n) = xnx^{-1},\ n \in N.
\end{equation*}
Write $\Aut_G(\mathrm e) $ for the subgroup
of $\Aut (\Gamma) \times G$ that consists of those pairs $(\alpha , x)$ which
make the diagram
\begin{equation*}
\begin{CD}
0
@>>> 
M 
@>>> 
\Gamma 
@>>> 
N 
@>>> 
1
\\
@.
@V{\ell_x}VV
@V{\alpha}VV
@V{i_x}VV
@.
\\
0
@>>> 
M 
@>>> 
\Gamma 
@>>> 
N 
@>>> 
1
\end{CD}
\end{equation*}
commutative.

The homomorphism 
\begin{equation*}
\beta\colon \Gamma \longrightarrow \Aut_G(\mathrm e ),\ 
\beta(y)=(i_y,i^N(\piN(y))), \ y \in \Gamma,
\end{equation*}
together with the obvious action
of $\Aut_G(\mathrm e )$ on $\Gamma$, yields a crossed module
$(\Gamma,\Aut_G(\mathrm e ),\beta)$ whence, in particular,
$\beta(\Gamma)$ is a normal subgroup of $\Aut_G(\mathrm e )$;
we denote by $\Out_G(\mathrm e)$
the cokernel of $\beta$ and write
the resulting crossed 2-fold extension as 
\begin{equation}
\begin{CD}
\hat {\mathrm e} \colon 0 @>>>  
M^N
 @>>> \Gamma @>{\beta}>> \Aut_G(\mathrm e ) @>>> 
\Out_G(\mathrm e) @>>> 1.
\end{CD}
\label{res1}
\end{equation}
The map
$\mathrm{Der}(N,M) \longrightarrow \Aut_G(\mathrm e)$
given by the association
\begin{equation*}
\mathrm{Der}(N,M) \ni d\longmapsto (\alpha_d,1),\ \alpha_d(y)= (d\piN(y))y,\ 
y \in \Gamma,
\end{equation*}
is an injective homomorphism; 
this homomorphism and the obvious map $\Aut_G(\mathrm e) \to G$
yield the group extension
\begin{equation*}
0 \longrightarrow \Der (N,M) \longrightarrow \Aut_G(\mathrm e ) \longrightarrow G \longrightarrow 1,
\end{equation*}
the map $\Aut_G(\mathrm e) \to G$ being surjective,
since the  class $[\mathrm e ] \in \mathrm H^2 (N,M)$ is supposed to be
fixed under $Q$.
Further, let $\zeta\colon M \to \mathrm{Der}(N,M)$ be the homomorphism
defined by
\linebreak
$(\zeta(m))(n)=m({}^nm)^{-1}$, as $m$ ranges over $M$ and $n$ over $N$.
With these preparations out of the way, the data fit into the commutative
diagram
\begin{equation}
\begin{CD}
@. 0 @. 0 @.@.
\\
@.
@VVV
@VVV
@.
@.
\\
@. 
M^N 
@= 
M^N
@. 
1
@. 
\\
@.
@VVV
@VVV
@VVV
@.
\\
0
@>>> 
M 
@>>> 
\Gamma 
@>{\piN}>> 
N 
@>>> 
1
\\
@.
@V{\zeta}VV
@V{\beta}VV
@V{i^N}VV
@.
\\
0 @>>> \Der (N,M) @>>>       \Aut_G(\mathrm e ) @>>> G @>>> 1
\\
@.
@VVV
@VVV
@VVV
@.
\\
0 @>>> \mathrm H^1(N,M) @>>> \Out_G(\mathrm e ) @>>> Q @>>> 1
\\
@.
@VVV
@VVV
@VVV
@.
\\
@. 1 @. 1 @.1 @.
\end{CD}
\label{diag1}
\end{equation}
with exact rows and columns.
We use the notation
\begin{equation*}
\begin{CD}
\overline {\mathrm e} \colon 
0 @>>> \mathrm H^1(N,M) @>>> \Out_G(\mathrm e ) @>>> Q @>>> 1
\end{CD}
\end{equation*}
for the bottom row extension of \eqref{diag1}.
This extension is the cokernel, in the category
of group extensions with abelian kernel,
of the morphism
$(\zeta,\beta,i)$ of group extensions.

Suppose now that the extension $\overline{\mathrm e}$ 
splits; we  then say that 
$\mathrm e$ {\em admits a crossed pair structure\/},
and we  refer to 
a section $\psi\colon Q \to  \Out_G(\mathrm e )$
of $\overline{\mathrm e}$ as
a {\em crossed pair structure
on the group extension\/} 
$\mathrm e\colon M \rightarrowtail \Gamma \stackrel{\piN} \twoheadrightarrow N$ 
{\em with respect to the group extension\/}
\eqref{ge}.
By definition, a {\em crossed pair\/}
$(\mathrm e,\psi)$ {\em with respect to the group extension\/}
\eqref{ge} {\em and the\/} $G$-{\em module\/} $M$
consists of a group extension
$\mathrm e\colon M \rightarrowtail \Gamma \twoheadrightarrow N$
whose class $[\mathrm e]\in \mathrm H^2(N,M)$ is fixed under $Q$
such that  the associated extension $\overline{\mathrm e}$ splits,
together with a section 
$\psi\colon Q \to  \Out_G(\mathrm e )$
of $\overline{\mathrm e}$ \cite[p.~152]{MR597986}.

Suitable
classes of crossed pairs with respect to 
\eqref{ge} and the $G$-module $M$
 constitute an abelian group 
$\mathrm{Xpext}(G,N; M)$ \cite[Theorem~1]{MR597986}. 
Moreover, 
cf. 
 \cite[Theorem~2]{MR597986},
suitably defined homomorphisms
\begin{equation*}
j\colon \mathrm H^2(G,M)\longrightarrow
\mathrm{Xpext} (G,N;M),\ 
\Delta\colon \mathrm{Xpext} (G,N;M)
\longrightarrow \mathrm H^3(Q,M^N)
\end{equation*}
yield an extension of the classical five term exact sequence to an
eight term exact sequence of the kind
\begin{equation}
\begin{aligned}
0\longrightarrow &\mathrm H^1(Q,M^N) 
\stackrel{\inf}\longrightarrow
\mathrm H^1(G,M) 
 \stackrel{\res}\longrightarrow 
\mathrm H^1(N,M)^Q
\stackrel{\Delta}\longrightarrow \mathrm H^2(Q,M^N) 
\\
\stackrel{\inf}\longrightarrow &\mathrm H^2(G,M)
 \stackrel{j}\longrightarrow
\mathrm{Xpext} (G,N;M)
 \stackrel{\Delta }\longrightarrow \mathrm H^3(Q,M^N) 
\stackrel{\inf}\longrightarrow
\mathrm H^3(G,M).
\end{aligned}
\label{13.11}
\end{equation}
For later reference, we recall the construction of $\Delta$.
To this end,
given a crossed pair
\begin{equation*}
\left(\mathrm e\colon 0 \to M \to \Gamma  \to N \to 1, 
\ \psi\colon Q \to  \Out_G(\mathrm e )\right)
\end{equation*}
with respect to the group extension \eqref{ge} and the $G$-module $M$,
let $B^\psi $ denote the fiber product group
$\Aut_G (\mathrm e ) \times_{\Out_G (\mathrm e)} Q$
with respect to the crossed pair structure map
${\psi\colon Q \to \Out_G (\mathrm e)}$
and, furthermore, let
$\partial^\psi\colon
\Gamma \to B^\psi$ denote
the obvious homomorphism; together with the obvious
action of $B^\psi$ on $\Gamma$ induced by the canonical homomorphism
 $B^\psi \to \Aut_G (\mathrm e )$,
the exact sequence
\begin{equation}
\mathrm e_\psi \colon 0 \longrightarrow M^N \longrightarrow \Gamma 
\stackrel{\partial^\psi}{\longrightarrow} B^\psi \longrightarrow
Q \longrightarrow 1
\label{epsi}
\end{equation}
is a crossed 2-fold extension and hence represents a class
in $\mathrm H^3(Q,M^N)$.
We refer to
$\mathrm e_\psi$ {\em as the crossed\/} 2-{\em fold extension associated to
the crossed pair\/} $(\mathrm e,\psi)$.
The homomorphism $\Delta\colon 
\mathrm{Xpext} (G,N;M) \to\mathrm H^3(Q,M^N)$
is given by the assignment to
a crossed pair
$(\mathrm e,\psi)$ of its associated crossed 2-fold extension
$\mathrm e_\psi$.

\begin{rema} By 
\cite[Theorem~1]{MR641328},
the association $\mathrm e \mapsto \overline {\mathrm e}$
yields a conceptual 
description of the differential 
$d_2 \colon \mathrm E^{0,2}_2 \to \mathrm E^{2,1}_2$
of 
the Lyndon-Hochschild-Serre spectral sequence 
$(\mathrm E^{p,q}_r,d_r)$
associated with the group extension \eqref{ge} and the $G$-module $M$.
\end{rema}

\begin{prop} \label{cmcp}
In the special case where the $N$-action on $M$ is trivial,
given a group extension 
$\mathrm e\colon M \rightarrowtail \Gamma \stackrel{\piN} 
\twoheadrightarrow N$ 
that admits a crossed pair structure,
crossed pair structures 
\linebreak
$\psi\colon Q \to \Out_G(\mathrm e)$
on the group extension $\mathrm e$
correspond bijectively to actions of $G$ on $\Gamma$ that turn
\[
i^N \circ \piN\colon \Gamma \longrightarrow G 
\]
into a crossed module in such a way that
the canonical homomorphism 
\[
 G\longrightarrow B^\psi =\Aut_G (\mathrm e ) \times_{\Out_G (\mathrm e)} Q
\]
is an isomorphism. \qed
\end{prop}

\begin{rema}
\label{pseudo}
Given the group extension \eqref{ge}, consider a group extension
\begin{equation*}
\mathrm e \colon 1 \longrightarrow X \longrightarrow K \stackrel{\piN} \longrightarrow N \longrightarrow 1,
\end{equation*}
the group $X$ not necessarily being abelian,
let $\phi= i^N\circ \piN \colon K \to G$ denote the composite
of $i$ and $\piN$, and
let $\Aut(\mathrm e) $ denote the subgroup
of $\Aut (K)$ that consists of the automorphisms 
of $K$ that map $X$ to itself;
such a homomorphism $\phi$ is referred to in \cite{MR0059911} as 
a {\em normal homomorphism\/}.
Conjugation in $K$ yields a homomorphism 
$\beta\colon K \to \Aut(\mathrm e)$ from $K$ onto a normal subgroup $\beta(K)$
of $\Aut(\mathrm e)$, and the restriction $\zeta$
of $\beta$ to $X$, that is,
conjugation in $K$
with elements of $X$, yields a homomorphism 
$\zeta\colon X \to \Aut(\mathrm e)$ from $X$ onto a normal subgroup $\zeta(X)$
of $\Aut(\mathrm e)$ as well;
let
$\mathrm{can}\colon\Aut(\mathrm e)\to \Aut(\mathrm e)/\zeta(X)$
denote the canonical surjection.
A {\em modular structure on\/} $\phi$ is a
homomorphism $\theta\colon G \to \Aut(\mathrm e)/\zeta(X)$ making the diagram
\begin{equation*}
\xymatrix{
K  \ar[d]^{\beta} \ar[r]^{\piN}\ar[dr]^{\phi} &N \ar[d]^{i^N}\\
\Aut(\mathrm e)  \ar[d]^{\mathrm{can}}  &G \ar[dl]^{\theta}\\
\Aut(\mathrm e)/\zeta(X)
}
\end{equation*}
commutative  \cite{MR0059911}. A 
{\em pseudo-module\/} is defined to be
a pair 
$(\phi,\theta)$ that consists of a normal homomorphism $\phi$
and a modular structure $\theta$ on $\phi$ \cite{MR0059911}.

Let $(\phi,\theta)$ be a pseudo-module and consider
the two abstract kernels
${G \to \Out(X)}$ and ${Q\to \Out(K)}$ induced by that pseudo-module.
Now, fix an abstract $G$-kernel structure 
\linebreak
$\omega\colon 
G \to \Out(X)$ on $X$ in advance and consider the group
$\Aut_G(\mathrm e)$
that consists of the pairs $(\alpha,x)\in \Aut(\mathrm e)\times G$
 which
make the diagram
\begin{equation*}
\begin{CD}
1
@>>> 
X 
@>>> 
K 
@>>> 
N 
@>>> 
1
\\
@.
@V{\alpha|X}VV
@V{\alpha}VV
@V{i_x}VV
@.
\\
1
@>>> 
X 
@>>> 
K 
@>>> 
N 
@>>> 
1
\end{CD}
\end{equation*}
commutative in such a way that the image of $\alpha|X$
in $\Out(X)$ coincides with $\omega(x) \in \Out(X)$.
Then the modular structures on $\phi$ that induce, in particular,
the abstract $G$-kernel structure $\omega$ on $X$
are given by
homomorphisms $\theta\colon G \to \Aut_G(\mathrm e)/\zeta(X)$.
In the special case where $X$ is abelian,
an abstract $G$-kernel structure on $X$ is an ordinary $G$-module structure,
and those modular structures $\theta\colon G \to \Aut_G(\mathrm e)/\zeta(X)$
correspond bijectively to
crossed pair structures $\psi\colon Q \to \Out_G(\mathrm e)$
on $\mathrm e$.
\end{rema}

\subsection{Crossed pairs and normal algebras}
\label{cpna}

Let $T|S$ be a $Q$-normal Galois extension of commutative rings, 
with structure extension 
\begin{equation*}
\begin{CD}
\mathrm e_{(T|S)}\colon
1 @>>> N 
@>{i^N}>> G @>>> Q @>>> 1
\end{CD} 
\end{equation*}
and structure homomorphism $\llambda \colon G \to \Aut^S(T)$;
in particular, the group $N$ is finite.
Let 
$\left(\mathrm e\colon \mathrm U(T) \rightarrowtail \Gamma  \twoheadrightarrow N, 
\ 
\psi\colon Q \to  \Out_G(\mathrm e )\right)$
be a crossed pair with respect to the group
extension $\mathrm e_{(T|S)}$ and the $G$-module $\mathrm U(T)$.
The corresponding crossed 2-fold extension 
\eqref{res1}
now takes the form
\begin{equation*}
\hat {\mathrm e} \colon 0 \longrightarrow  \mathrm U(S) \longrightarrow \Gamma \longrightarrow \Aut_G(\mathrm e ) \longrightarrow 
\Out_G(\mathrm e) \longrightarrow 1.
\end{equation*}
To the crossed pair $(\mathrm e , \psi )$, associate
a $Q$-normal $S$-algebra $(A_\mathrm e , \sigma_\psi)$ as follows.

The composite
$\vartheta \colon \Gamma \to N \to \Aut (T)$ 
yields an action of $\Gamma$ on $T$; let 
$A_\mathrm e $ denote the crossed product  algebra 
$(T,N, \mathrm e , \vartheta)$.
Since the group $N$ is finite, $A_\mathrm e $ is an 
Azumaya $S$-algebra;  this fact also follows from 
\cref{tpf}~(xi).
Recall that 
there is an obvious injection $i \colon \Gamma \to \mathrm U(A_\mathrm e )$. The following
is immediate.

\begin{prop}
\label{6.1}
Setting
\begin{equation}
{}^{i_{\sharp} (\alpha ,x)}{(ty)} = ({}^x t)({}^\alpha y),
\end{equation}
as $t$ ranges over $T$, $y$ over $\Gamma$, and $(\alpha , x)$ 
over $\Aut_G(\mathrm e )\  (\subseteq \Aut(\Gamma) \times G)$,  
we obtain a morphism
\begin{equation*}(i,i_{\sharp} ) \colon (\Gamma, \Aut_G(\mathrm e ), \beta ) 
\longrightarrow (\mathrm U(A_\mathrm e), \Aut (A_\mathrm e , Q), \partial )
\end{equation*}
of crossed modules which, in turn, induces the morphism
\begin{equation*}
\begin{CD}
\phantom{{\mathrm e}_{(A_\mathrm e ,Q)}}
\hat{\mathrm e} \colon 
0 @>>> \mathrm U(S) @>>> \Gamma
@>{\beta}>>\Aut_G(\mathrm e)  @>>> \Out_G(\mathrm e) @>>> 1
\\
@.
@|
@V{i}VV
@V{i_{\sharp}}VV
@V{i_\flat}VV
@.
\\
\phantom{\hat{\mathrm e}\ }
{\mathrm e}_{(A_\mathrm e ,Q)}  \colon 
0 @>>> \mathrm U(S) @>>> \mathrm U(A_{\mathrm e}) 
@>{\partial}>> \Aut(A_{\mathrm e},Q) @>>>\Out(A_{\mathrm e},Q)  @>>> 1
\end{CD}
\end{equation*}
of crossed $2$-fold extensions, where $i_\flat$ denotes the induced
homomorphism.
\end{prop}

Given a crossed pair
$\left(\mathrm e\colon 0 \to \mathrm U(T) \to \Gamma  \to N \to 1, 
\ 
\psi\colon Q \to  \Out_G(\mathrm e )\right)$
with respect to the group
extension $\mathrm e_{(T|S)}$ and the $G$-module $\mathrm U(T)$,
let
\begin{equation*}
\sigma _\psi = i_{\flat} \circ \psi \colon Q \longrightarrow 
\Out_G(\mathrm e) \longrightarrow
\Out (A_\mathrm e , Q);
\end{equation*}
it is then obvious 
that $(A_\mathrm e , \sigma_\psi )$ is a $Q$-normal (Azumaya)
$S$-algebra, and 
we   refer to $(A_\mathrm e , \sigma_\psi )$ as a $Q$-{\em normal
crossed pair algebra with respect to the\/} 
$Q$-{\em normal Galois extension\/}
$T|S$ of commutative rings.

\begin{thm}
\label{normalcrossed}
 Let $T|S$ be a $Q$-normal Galois extension of commutative rings, 
with structure 
extension 
$\mathrm {\mathrm e}_{(T|S)} \colon
N\rightarrowtail G \twoheadrightarrow Q$
and structure homomorphism $\llambda \colon G \to \Aut^S(T)$,
cf. Section {\rm \ref{normalr}} above.
Then a class $k\in \mathrm H^3 (Q,\mathrm U(S))$ is the Teich\-m\"uller class
of some crossed pair algebra $(A_\mathrm e , \sigma_\psi)$ with respect
to the $Q$-normal Galois extension
$T|S$ if and only if $k$ is split in $T|S$ in the sense that 
$k$ goes to zero
under inflation 
$\mathrm H^3(Q,\mathrm U(S)) \to \mathrm H^3(G,\mathrm U(T))$.
\end{thm}

With $M=\mathrm U(T)$
and $M^N=\mathrm U(S)$, the theorem
is a consequence of 
the exactness, at $\mathrm H^3(Q,\mathrm U(S))$,
of the sequence \eqref{13.11}.
Indeed, by construction, the homomorphism $\Delta$
is given by the assignment to a crossed pair
 $
\left (\mathrm e \colon  \mathrm U(T)\rightarrowtail \Gamma \twoheadrightarrow N,\ 
\ppsi\colon Q \to \Out_G(\mathrm e)\right)$
with respect to the group
extension $\mathrm e_{(T|S)}$ and the $G$-module $\mathrm U(T)$
of the corresponding crossed 2-fold extension \eqref{epsi}, which now takes the
form
\begin{equation*}
\mathrm e_\psi \colon 0 \longrightarrow \mathrm U(S) \longrightarrow \Gamma \stackrel{\partial^\psi}{\longrightarrow} B^\psi \longrightarrow
Q \longrightarrow 1.
\end{equation*}

Theorem \ref{normalcrossed} is therefore a consequence of the following, which
is again immediate.

\begin{prop}
\label{6.3}
Given a crossed pair $(\mathrm e, \psi)$
with respect to the group
extension $\mathrm e_{(T|S)}$ and the $G$-module $\mathrm U(T)$, the morphism
$(i,i_{\sharp})$ of crossed modules 
in Proposition {\rm \ref{6.1}} above
induces a congruence morphism
\begin{equation*}
\begin{CD}
\phantom{{\mathrm e}_{(A_\mathrm e ,\sigma_{\psi})}}
{\mathrm e}_{\psi}
\colon 
0 @>>> \mathrm U(S) @>>> \Gamma
@>{\partial^{\psi}}>>B^\psi  @>>> Q @>>> 1
\\
@.
@|
@V{i}VV
@V{\hat i}VV
@|
@.
\\
\phantom{\mathrm e_{\psi}\ }
{\mathrm e}_{(A_\mathrm e ,\sigma_{\psi})}  \colon 
0 
@>>> 
\mathrm U(S) 
@>>> 
\mathrm U(A_{\mathrm e}) 
@>{\partial^{\sigma_\psi}}>> 
B^{\sigma_\psi}
@>>>
Q @>>> 1
\end{CD}
\end{equation*}
of crossed $2$-fold extensions.
\end{prop}

\begin{proof}[Proof of Theorem {\rm \ref{normalcrossed}}]
By exactness, it is immediate that
the Teich\-m\"uller class
of any crossed pair algebra $(A_\mathrm e , \sigma_\psi)$
with respect
to $T|S$ is split in $T|S$. Hence the condition is necessary.
To establish sufficiency,
consider
a class $k\in \mathrm H^3 (Q,\mathrm U(S))$ 
which is split in $T|S$, that is,
goes to zero
under inflation 
$\mathrm H^3(Q,\mathrm U(S)) \to \mathrm H^3(G,\mathrm U(T))$.
By exactness, $k$ then arises from some crossed pair $(\mathrm e, \psi)$
with respect to the group
extension $\mathrm e_{(T|S)}$ and the $G$-module $\mathrm U(T)$, that
is,
\begin{equation*}
k=[{\mathrm e}_{\psi}]\in \mathrm H^3 (Q,\mathrm U(S)).
\end{equation*}
By Proposition \ref{6.3},
the Teich\-m\"uller class
of  the associated 
crossed pair algebra $(A_\mathrm e , \sigma_\psi)$ with respect
to $T|S$ 
coincides with $[{\mathrm e}_{\psi}]=k$.
\end{proof}

\section{Normal Deuring embedding and Galois descent for
Teich\-m\"uller classes}
\label{7}

As before, $S$ denotes a commutative ring
and $\kappaQ\colon Q \to \Aut(S)$ an action of a group $Q$ on $S$.
Let $T|S$ be a $Q$-normal Galois extension of commutative rings, with 
structure extension
\begin{equation}\mathrm e_{(T|S)} \colon 1 \longrightarrow N 
\stackrel{i^N}
\longrightarrow G \stackrel{\piQ}{\longrightarrow} Q \longrightarrow 1
\label{eTS}
\end{equation}
and structure homomorphism $\llambda\colon G \to \Aut^S(T)$,
cf. \eqref{CD1}.
In this section, we  prove, among others, that if a class
$k \in \mathrm H^3(Q,\mathrm U(S))$ goes under inflation to the Teich\-m\"uller class
in $\mathrm H^3(G,\mathrm U(T))$ of some $G$-normal $T$-algebra, then $k$ is itself
the Teich\-m\"uller class of some $Q$-normal $S$-algebra. To this end,
we  reexamine Deuring's embedding problem, cf.
Subsection \ref{eaas} and \cref{nazet}.

\subsection{The definitions} 

Let $A$ be a central $T$-algebra,
$(C, \sigma_Q\colon Q \to \Out(C) )$ a $Q$-normal $S$-algebra, and
$A\subseteq C$ an embedding of $A$ into $C$. 
We refer to the embedding of $A$ into $C$
as
a $Q$-{\em normal Deuring embedding with
respect to\/} $\sigma_Q\colon Q \to \Out(C) $ 
{\em and\/} {\rm\eqref{eTS}}
if
each automorphism $\llambda (x)$ of $T$, 
as $x$ ranges over $G$, extends to an
automorphism $\alpha $ of $C$ in such a way 
that 
\begin{itemize}
\item[(i)] 
$[\alpha ] = \sigma_Q (\piQ (x))\in \Out (C)$, and
\item[(ii)] 
$\alpha$ maps $A$ to itself.
\end{itemize}

\begin{rema}
In the special case where $Q$ is the trivial group,
the group $G$ boils down to the group $N=\Aut(S|R)$ and,
since each automorphism $\alpha$ of $C$ 
that extends some $x\in N$ is required to 
map $A$ to itself and to
map to the trivial element
of $\Out(C)$,
that automorphism $\alpha$ 
necessarily extends to an inner automorphism of $C$ that normalizes $A$;
thus the notion of normal Deuring embedding
then comes down to the notion of Deuring embedding introduced in 
Subsection \ref{eaas}.
\end{rema}

\begin{rema} Given an embedding of $A$ into $C$ 
such that $A$ coincides with the centralizer of $T$ in $C$,
an automorphism $\alpha$ of $C$ extending an automorphism
$\llambda(x)$ of $T$ for $x \in G$
necessarily  maps $A$ to itself. Thus, in the definition
of a $Q$-normal Deuring embedding, condition (ii)
is then redundant.
\end{rema}

For technical reasons, we need a stronger concept
of a normal Deuring embedding. We  now prepare for this definition.

Let $A$ be a central $T$-algebra, $C$ a central $S$-algebra,
and suppose the algebra $A$ to be embedded into $C$.
Recall the crossed module $(\mathrm U(C),\Aut(C), \partial_C)$
associated to the central $S$-algebra $C$,
and consider the associated crossed 2-fold extension
\begin{equation}
\mathrm e_C \colon 0 \longrightarrow \mathrm U(S) 
\longrightarrow \mathrm U(C) \stackrel{\partial_C}\longrightarrow \Aut(C) \longrightarrow \Out(C) 
\longrightarrow 1,
\label{eC}
\end{equation}
cf. \eqref{eA}.
The normalizer $N^{\mathrm U(C)}(A)$
of $A$ in $\mathrm U(C)$ and the centralizer $C^{\mathrm U(C)}(T)$ of $T$
in $\mathrm U(C)$, together with 
$\mathrm U(A)$ and $\mathrm U(C)$, constitute an ascending sequence
\begin{equation*}
\mathrm U(A) \subseteq C^{\mathrm U(C)}(T) \subseteq N^{\mathrm U(C)}(A)
\subseteq
\mathrm U(C)
\end{equation*}
of groups.
When $A$ coincides with the centralizer of $T$ in $C$, 
the inclusion $\mathrm U(A) \subseteq C^{\mathrm U(C)}(T)$
is the identity.

We continue with the general case where $A$ 
does not necessarily coincide with the centralizer of $T$ in $C$. 
Let $\Aut^A(C)$ denote the group of automorphisms of $C$ that map $A$ 
to itself. 
The action of  $\Aut(C)$ on $\mathrm U(C)$
induces an action of $\Aut^A(C)$ on 
each of the groups $\mathrm U(A)$,
$C^{\mathrm U(C)}(T)$, and $N^{\mathrm U(C)}(A)$,
and the restrictions of the homomorphism $\partial_C$
together with the actions yield three crossed modules
\begin{align}
&(N^{\mathrm U(C)}(A),\Aut^A(C), \partial^N_C)
\label{cr1}
\\
&(C^{\mathrm U(C)}(T),\Aut^A(C), \partial^T_C)
\label{cr2}
\\
&(\mathrm U(A),\Aut^A(C), \partial^A_C),
\label{cr3}
\end{align}
each homomorphism $\partial^N_C$, $\partial^T_C$, $\partial^A_C$
being the corresponding restriction of the homomorphism
$\partial_C\colon \mathrm U(C) \to \Aut(C)$.
We write the associated 
crossed 2-fold extensions as
\begin{align}
\mathrm e_C^A &\colon 0
\longrightarrow
\mathrm U(S) 
\longrightarrow
\mathrm U(A) 
\stackrel{\partial^A_C}
\longrightarrow
\Aut^A(C)
\longrightarrow
\Out(C,A)
\longrightarrow
1
\label{ce3}
\\
\mathrm e_C^T &\colon 0
\longrightarrow
\mathrm U(S) 
\longrightarrow
C^{\mathrm U(C)}(T) 
\stackrel{\partial^T_C}
\longrightarrow
\Aut^A(C)
\longrightarrow
\Out(C,T)
\longrightarrow
1
\label{ce2}
\\
\mathrm e_C^N &\colon 0
\longrightarrow
\mathrm U(S) 
\longrightarrow
N^{\mathrm U(C)}(A) 
\stackrel{\partial^N_C}
\longrightarrow
\Aut^A(C)
\longrightarrow
\Out^A(C)
\longrightarrow
1,
\label{ce1}
\end{align}
the groups
$\Out(C,A)$,
$\Out(C,T)$, and
$\Out^A(C)$ being defined by exactness.
The
inclusions
$\mathrm U(A) \subseteq C^{\mathrm U(C)}(T) \subseteq N^{\mathrm U(C)}(A)$
induce a commutative diagram
\begin{equation*}
\begin{xymatrix}
{
\mathrm e_C^A \colon 0
\ar[r]&
\mathrm U(S) 
\ar@{=}[d]
\ar[r]&
\mathrm U(A) 
\ar[d]
\ar[r]^{\partial^A_C}&
\Aut^A(C)
\ar@{=}[d]
\ar[r]&
\Out(C,A)
\ar[d]
\ar[r]&
1
\\
\mathrm e_C^T \colon 0
\ar[r]&
\mathrm U(S) 
\ar@{=}[d]
\ar[r]&
C^{\mathrm U(C)}(T) 
\ar[d]
\ar[r]^{\partial^T_C}&
\Aut^A(C)
\ar@{=}[d]
\ar[r]&
\Out(C,T)
\ar[d]
\ar[r]&
1
\\
\mathrm e_C^N \colon 0
\ar[r]&
\mathrm U(S) 
\ar[r]&
N^{\mathrm U(C)}(A) 
\ar[r]^{\partial^N_C}
&
\Aut^A(C)
\ar[r]&
\Out^A(C)
\ar[r]&
1
}
\end{xymatrix}
\end{equation*}
of morphisms of crossed 2-fold extensions and, by diagram chase,
the induced homomorphisms
$\Out(C,A) \to \Out(C,T)$ and
$\Out(C,T)\to 
\Out^A(C)$ are surjective.

Restriction induces canonical homomorphisms
\begin{equation*}
\res\colon \Out(C,A) \longrightarrow \Out^S(A),
\
\res\colon \Out(C,T) \longrightarrow \Aut^S(T)
\end{equation*}
(where the notation \lq\lq res\rq\rq\ is slightly abused) 
in such a way that the diagram
\begin{equation*}
\begin{xymatrix}
{
&\Out(C,A) \ar[d]\ar[r]^{\res} &\Out^S(A)\ar[d]^{\res}
\\
&\Out(C,T) \ar[r]^{\res} &\Aut^S(T)
}
\end{xymatrix}
\end{equation*}
is commutative. Moreover, the obvious homomorphism
$\Out^A(C) \to\Out(C)$
is injective, and we identify
$\Out^A(C)$ with its isomorphic image in $\Out(C)$ if need be.

Now,  
given a homomorphism $\chi_G\colon G \to \Out(C,A)$, its composite
with  
the restriction map 
$\res\colon \Out(C,A) \to \Out(A)$
yields a $G$-normal structure on $A$.
However, in order for such a homomorphism to match
the other data, in particular the given $Q$-normal structure
$\sigma_Q\colon Q \to \Out(C)$, we must impose further conditions.
We now spell out the details.

Let $\partial^A_{C,\sharp}\colon \mathrm U(A)/\mathrm U(S) \to \Aut^A(C)$
denote the (injective) homomorphism induced by 
the crossed module structure map $\partial^A_C$ in the crossed module \eqref{cr3}.
The crossed modules \eqref{cr1} and \eqref{cr3} yield the commutative diagram
\begin{equation*}
\begin{xymatrix}
{
& &1\ar[d]&1\ar[d] &
\\
0
\ar[r]
&
\mathrm U(S) 
\ar@{=}[d]
\ar[r]
&
\mathrm U(A) 
\ar[d]
\ar[r] 
&
\mathrm U(A)/\mathrm U(S)
\ar[d]^{\partial^A_{C,\sharp}}
\ar[r]
&
1
&
\\
0\ar[r]
&
\mathrm U(S) 
\ar[r]&
N^{\mathrm U(C)}(T)
\ar[d]
\ar[r]^{\partial^N_C}&
\Aut^A(C)
\ar[d]
\ar[r]&
\Out^A(C)
\ar@{=}[d]
\ar[r]&
1
\\
&
1
\ar[r]&
N^{\mathrm U(C)}(A)/\mathrm U(A) 
\ar[d] 
\ar[r] 
&
\Out(C,A)
\ar[d]
\ar[r]
&
\Out^A(C)
\ar[r]&
1
\\
& &1&1 &
}
\end{xymatrix}
\end{equation*}
with exact rows and columns, the third row being defined by exactness.
This third row is an ordinary group extension, and we  denote it by
\begin{equation}
\begin{xymatrix}
{
\mathrm e_{(A,C)}\colon
1
\ar[r]&
N^{\mathrm U(C)}(A)/\mathrm U(A)  
\ar[r] 
&
\Out(C,A)
\ar[r]
&
\Out^A(C)
\ar[r]&
1.
}
\end{xymatrix}
\label{eAC}
\end{equation}
We define a {\em strong\/} $Q$-{\em normal Deuring embedding\/}
of $A$ into $C$ 
{\em with respect to the\/} $Q$-{\em normal structure\/}
$\sigma_Q\colon Q \to \Out(C)$ {\em and the structure extension\/} {\rm\eqref{eTS}}
to consist of an embedding of $A$ into $C$ together with
a homomorphism $\chi_G\colon G \to\Out(C,A)$
that is compatible with the other data in the following sense:
\begin{itemize}
\item 
The restriction $\chi_N\colon N \to N^{\mathrm U(C)}(A)/\mathrm U(A)$
 to $N=\Aut(T|S)$ of the homomorphism $\chi_G$
turns the embedding of  $A$ into $C$ 
into a strong
 Deuring embedding 
relative to the action 
$\mathrm{Id}\colon N \to \Aut(T|S)$ 
of $N$ on $T$
in such a way that  the diagram
\begin{equation}
\begin{gathered}
\begin{xymatrix}
{
\mathrm e_{(T|S)}\colon
1
\ar[r]&
N  
\ar[d]^{\chi_N} 
\ar[r]^{i^N} 
&
G
\ar[d]^{\chi_G}
\ar[r]^{\piQ}
&
Q
\ar[d]^{\sigma_Q}
\ar[r]&
1
\\
\mathrm e_{(A,C)}\colon
1
\ar[r]&
N^{\mathrm U(C)}(A)/\mathrm U(A)  
\ar[r] 
&
\Out(C,A)
\ar[r]
&
\Out^A(C)
\ar[r]&
1
}
\end{xymatrix}
\end{gathered}
\end{equation}
is commutative.

\item
The 
composite
\begin{equation}
G \stackrel{\chi_G}\longrightarrow\Out (C,A)\stackrel{\mathrm{res}}
\longrightarrow \Aut^S(T)
\label{comp00}
\end{equation}
 coincides with 
$\llambda\colon G \to \Aut^S(T)$.

\end{itemize}

\begin{rema}
In the special case where $Q$ is the trivial group,
this notion of strong normal Deuring embedding
comes down to the notion of strong 
Deuring embedding introduced in 
Subsection \ref{eaas}.
\end{rema}

Given a strong $Q$-normal Deuring embedding 
$(A\subseteq C,\chi_G)$ with 
respect to the $Q$-normal structure
$\sigma_Q\colon Q \to \Out(C)$  and the 
group extension {\rm\eqref{eTS}},
the composite of $\chi_G$
with 
the restriction map 
$\res\colon \Out(C,A) \to \Out(A)$
yields a $G$-normal structure 
\begin{equation}
\sigma_G\colon G \longrightarrow \Out(A)
\label{assG}
\end{equation}
on $A$ relative to the
action $\llambda\colon G \to \Aut^S(T)$
of $G$ on $T$; we  refer to this structure as being
{\em associated to the strong \/} $Q$-{\em normal  Deuring embedding}.

\subsection{Discussion of the notion of normal Deuring embedding}

Recall that $G$ denotes the fiber product group
${\Aut^S(T) \times_{\Aut(S)} Q}$ relative to the action 
$\kappaQ\colon Q \to \Aut(S)$ of $Q$ on $S$,
that 
$\llambda\colon G \to \Aut^S(T)$
is the associated obvious homomorphism, 
and that $\llambda$, restricted to $N$, boils down to the identity
$N \to \Aut(T|S)$,
cf. \eqref{CD1} above.

Let $A$ be a central $T$-algebra,
consider an embedding of 
$A$ into a central $S$-algebra $C$, and
let $\sigma_Q\colon Q \to \Out(C)$
be a $Q$-normal structure
on $C$. 
Consider the fiber product group
$B^{A,\sigma_Q} =\Aut^A(C)\times_{\Out (C)} Q$
relative to the $Q$-normal structure $\sigma_Q$ on $C$.
The following is immediate.

\begin{prop}
\label{abs1} Abstract nonsense identifies the kernel of
the canonical homomorphism $B^{A,\sigma_Q} \to Q$
with the normal subgroup 
$\mathrm {IAut}^A(C)$
of $\Aut^A(C)$ that consists
of the inner automorphisms of $C$ that map $A$ to itself.
Consequently the data determine a crossed module
$(N^{\mathrm U(C)}(A), B^{A,\sigma_Q},
\partial^{A,\sigma_Q})$,
the requisite action  of $B^{A,\sigma_Q}$ on 
$N^{\mathrm U(C)}(A)$ being induced from
the canonical homomorphism $B^{A,\sigma_Q} \to \Aut^A(C)$,
in such a way that the sequence
\begin{equation}
\begin{xymatrix}{
0 \ar[r] &\mathrm U(S) \ar[r] &N^{\mathrm U(C)}(A) 
\ar[r]^{{\phantom{(T)}}\partial^{A,\sigma_Q}} &
B^{A,\sigma_Q}
\ar[r] &Q
}
\end{xymatrix}
\label{eCsQ0}
\end{equation}
is exact. \qed
\end{prop}

Since $G 
={\Aut^S(T) \times_{\Aut(S)} Q}$ 
(relative to the action $\kappaQ\colon Q \to \Aut(S)$ of $Q$ on $S$),
and since the composite
$
Q \stackrel{\sigmaQ} \longrightarrow \Out(C) \stackrel{\res}\longrightarrow    \Aut(S)$ coincides with 
$\kappaQ\colon Q \to \Aut(S)$,
by abstract nonsense,  the 
 combined homomorphism
\begin{equation*}
B^{A,\sigma_Q} \stackrel{\mathrm{can}}\longrightarrow 
\Aut^A(C) \stackrel{\res}\longrightarrow \Aut^S(T)
\end{equation*}
 and the
canonical homomorphism 
$\mathrm{can}\colon B^{A,\sigma_Q} \to Q$ 
induce a homomorphism
\begin{equation}
\ttheta\colon B^{A,\sigma_Q} =\Aut^A(C)\times_{\Out (C)} Q
\longrightarrow 
{\Aut^S(T) \times_{\Aut(S)} Q}= G .
\label{can21}
\end{equation}
The following is again immediate.

\begin{prop}
\label{prop1}
The embedding of $A$ into $C$ is a $Q$-normal Deuring embedding
with respect to the $Q$-normal structure $\sigma_Q \colon Q \to \Out(C)$
on $C$ and the group extension {\rm\eqref{eTS}}
if and only if 
the homomorphism 
$\ttheta\colon B^{A,\sigma_Q} \to G$ is surjective. \qed
\end{prop}

Whether or not the homomorphism $\ttheta$ is surjective,
we  now determine the kernel of $\ttheta$.
To this end,
let $\Aut^A(C|T)$ denote the subgroup of $\Aut^A(C)$
that consists of the automorphisms in $\Aut^A(C)$
that are the identity on $T$.
Since $T$
coincides with the center of $A$,
restriction induces a homomorphism from
$\Aut^A(C)$ to $\Aut(T)$, and since $S$ coincides with the center of $C$, 
the values of this restriction map
lie in the subgroup $\Aut^S(T)$ of $\Aut(T)$
that consists of the automorphisms of $T$ which map $S$ to itself.
Thus, all told, restriction yields an exact sequence
\begin{equation}
\begin{xymatrix}
{
1
\ar[r]
&\Aut^A(C|T)
\ar[r]
&\Aut^A(C)
\ar[r]^{\res}
&\Aut^S(T)
}
\end{xymatrix}
\label{exact1}
\end{equation}
of groups.

Consider the fiber product groups
\begin{equation*}
B^{A,\llambda}= \Aut^A(C) \times_{\Aut^S(T)}G,\ 
B^{A,\kappaQ}= \Aut^A(C) \times_{\Aut(S)}Q,
\end{equation*}
relative to the homomorphisms $\llambda\colon G \to \Aut^S(T)$
and $\kappaQ\colon Q \to \Aut(S)$, respectively, and let
$\mathrm{can}\colon B^{A,\llambda} \to G$
denote the canonical homomorphism.
Since $G$ is the fiber product group
${\Aut^S(T) \times_{\Aut(S)} Q}$ with respect to the homomorphism
$\kappaQ\colon Q \to \Aut(S)$,
by abstract nonsense, 
the canonical homomorphism from $B^{A,\llambda}$ to $B^{A,\kappaQ}$ is an 
isomorphism.
Moreover, the exact sequence \eqref{exact1}
induces an exact sequence 
\begin{equation}
\begin{xymatrix}
{
1
\ar[r]
&\Aut^A(C|T)
\ar[r]
&B^{A,\llambda}
\ar[r]^{\mathrm{can}}
&G
}
\end{xymatrix}
\label{exact2}
\end{equation}
of groups in such a way that
\begin{equation*}
\begin{xymatrix}
{
1
\ar[r]
&\Aut^A(C|T)
\ar@{=}[d]\ar[r]
&B^{A,\llambda}
\ar[d]
\ar[r]^{\mathrm{can}}
&G
\ar[d]^{\llambda}
\\
1
\ar[r]
&\Aut^A(C|T)
\ar[r]
&\Aut^A(C)
\ar[r] ^{\res}
&\Aut^S(T)
}
\end{xymatrix}
\end{equation*}
is a commutative diagram with exact rows.

Abstract nonsense yields
a canonical homomorphism 
\[
\Aut^A(C)\times_{\Out (C)} Q= B^{A,\sigma_Q} \longrightarrow 
B^{A,\kappaQ}=\Aut^A(C) \times_{\Aut(S)}Q
\]
and hence
a canonical homomorphism 
$B^{A,\sigma_Q} \to B^{A,\llambda}$ 
whose composite
$B^{A,\sigma_Q} \to B^{A,\llambda} \stackrel{\mathrm{can}}\to G$
with $\mathrm{can}\colon B^{A,\llambda}\to G$ coincides with 
$\ttheta\colon B^{A,\sigma_Q} \to G$.

\begin{prop}
\label{kernel}
\rm{(i)} The homomorphism
$B^{A,\sigma_Q} \to B^{A,\llambda}$ is injective.

\noindent
\rm{(ii)}
Under the identification of $B^{A,\sigma_Q}$ with its isomorphic
image in the group $B^{A,\llambda}$, 
the group $\Aut^A(C|T)$ being identified with its isomorphic image
in $B^{A,\llambda}$ via {\rm{\eqref{exact2}}},
the kernel of 
$\ttheta\colon B^{A,\sigma_Q} \to G$ gets identified with  
the normal subgroup of 
$\Aut^A(C|T)$ that consists of the automorphisms
in $\Aut^A(C|T)$ that are inner automorphisms of $C$.

\noindent
\rm{(iii)} Consequently the canonical homomorphism
from the centralizer $C^{\mathrm U(C)}(T)$
of $T$ in $\mathrm U(C)$
to $\Aut^A(C|T)$ yields a surjective homomorphism
\begin{equation*}
C^{\mathrm U(C)}(T) \longrightarrow \mathrm{ker}(\ttheta\colon B^{A,\sigma_Q} \to G).
\end{equation*}
\end{prop}

\begin{proof}
Since the canonical homomorphism
$B^{A,\llambda} \to B^{A,\kappaQ}$ is an 
isomorphism,
the right-hand square in the the commutative diagram
\begin{equation*}
\begin{xymatrix}
{
&B^{A,\sigmaQ}
\ar[d]\ar[r]
&B^{A,\llambda}
\ar[d]
\ar[r]^{\mathrm{can}}
&Q
\ar[d]^{\kappaQ}
\\
&\Aut^A(C)
\ar@{=}[r]
&\Aut^A(C)
\ar[r] ^{\res}
&\Aut(S)
}
\end{xymatrix}
\end{equation*}
is a pull back square, and hence
inspection of the diagram reveals that 
the homomorphism 
$B^{A,\sigmaQ} \to B^{A,\llambda}$
is injective. This establishes (i).

To justify (ii), we note first that 
the kernel  of 
$\Aut^A(C)\to \Out(C)$ is the normal subgroup 
$\mathrm {IAut}^A(C)$
of $\Aut^A(C)$ that consists
of the inner automorphisms of $C$ that map $A$ to itself.
Since the group $B^{A,\sigma_Q}$ is the fiber product group
$B^{A,\sigma_Q} =\Aut^A(C)\times_{\Out (C)} Q$,
abstract nonsense identifies the kernel of
the canonical homomorphism $B^{A,\sigma_Q} \to Q$
with $\mathrm {IAut}^A(C)$,
and it is immediate that
$\mathrm{ker}(\ttheta)$ is a subgroup of 
$\mathrm {IAut}^A(C)=\mathrm{ker}(B^{A,\sigma_Q} \to Q)$.
On the other hand, $B^{A,\sigma_Q}$ being identified with
the corresponding subgroup of $B^{A,\llambda}$,
the kernel of 
$\ttheta\colon B^{A,\sigma_Q} \to G$ gets identified with  
the intersection
$B^{A,\sigma_Q} \cap \Aut^A(C|T) \subseteq B^{A,\llambda}$
and hence with the intersection
\[
\mathrm {IAut}^A(C) \cap \Aut^A(C|T) \subseteq B^{A,\llambda}.
\]
Consequently the kernel of $\ttheta$ gets identified with the 
normal subgroup of 
$\Aut^A(C|T)$ that consists of the automorphisms
in $\Aut^A(C|T)$ that are inner automorphisms of $C$.

Finally, statement (iii) is an immediate consequence of (ii).
\end{proof}

\begin{prop}
\label{prop11}
Suppose that the embedding of $A$ into $C$ is  a $Q$-normal Deuring embedding
with respect to the $Q$-normal structure $\sigma_Q \colon Q \to \Out(C)$
on $C$  and  the group extension {\rm\eqref{eTS}}.

\noindent
{\rm (i)} 
The surjective homomorphism {\rm \eqref{can21}}
yields a crossed $2$-fold extension
\begin{equation}
\begin{xymatrix}
{
\mathrm e^{A,T}_{(C,\sigma_Q)}\colon 0\ar[r]
&\mathrm U(S)
\ar[r]
& C^{\mathrm U(C)}(T) 
\ar[r]^{{\phantom{(TTT)}}\partial^{A,T,\sigma_Q}{\phantom{(T)}}}
&B^{A,\sigma_Q}
\ar[r]^{\ttheta}
&G
\ar[r]
&1 .
}
\end{xymatrix}
\label{eCAT}
\end{equation}

\noindent
{\rm(ii)} The values of the $Q$-normal structure
$\sigma_Q\colon Q \to \Out(C)$ on $C$ lie in  the subgroup $\Out^A(C)$
$(=\mathrm{coker}( {\partial^N_C}\colon
N^{\mathrm U(C)}(A) 
\longrightarrow
\Aut^A(C)) $, cf. \eqref{ce1}$)$.
\end{prop}

\begin{proof} Statement (i) is an immediate  consequence of
\cref{prop1} 
and \cref{kernel}(iii).
Moreover, the diagram
\begin{equation*}
\begin{xymatrix}
{
&N^{\mathrm U(C)}(A)
\ar@{=}[d]\ar[r]^{\partial^{A,\sigma_Q}}
&B^{A,\sigmaQ}
\ar[d]
\ar[r] 
&Q
\ar[d]^{\sigmaQ}
\\
&N^{\mathrm U(C)}(A)
\ar[r]^{\partial^N_C}
&\Aut^A(C)
\ar[r] ^{\mathrm{can}}
&\Out(C)
}
\end{xymatrix}
\end{equation*}
is commutative and, in view of \cref{prop1},  the 
canonical homomorphism 
$B^{A,\sigmaQ} \to Q$ is surjective. Consequently the values
of 
$\sigma_Q\colon Q \to \Out(C)$ on $C$ lie in the subgroup  $\Out^A(C)$
($=\mathrm{coker}( {\partial^N_C}\colon
N^{\mathrm U(C)}(A) 
\longrightarrow
\Aut^A(C)) $).
\end{proof}

Given a $Q$-normal Deuring embedding of $A$ into $C$
with respect to the $Q$-normal structure $\sigma_Q \colon Q \to \Out(C)$
on $C$ and the group extension {\rm\eqref{eTS}},
in view of \cref{prop11}(ii),
let
\begin{equation}
\begin{xymatrix}{
\mathrm e^A_{(C,\sigma_Q)}\colon
0 \ar[r] &\mathrm U(S) \ar[r] &N^{\mathrm U(C)}(A) 
\ar[r]^{{\phantom{(T)}}\partial^{A,\sigma_Q}} &
B^{A,\sigma_Q}
\ar[r] &Q  \ar[r] &1
}
\end{xymatrix}
\label{eCsQ1}
\end{equation}
denote the associated crossed 2-fold extension
induced from {\rm\eqref{ce1}}
via the $Q$-normal structure
${\sigma_Q\colon Q \to \Out^A(C)}$
on $C$; the underlying sequence
of groups and homomorphisms plainly coincides with \eqref{eCsQ0}.
Recall that
 the 
Teich\-m\"uller complex $\mathrm e_{(C,\sigma_Q)}$ of the kind
\eqref{pb1}
associated to the $Q$-normal $S$-algebra $(C,\sigma_Q)$ is
the crossed $2$-fold extension 
\begin{equation}
\begin{xymatrix}{
\mathrm e_{(C,\sigma_Q)}\colon
0 \ar[r] &\mathrm U(S) \ar[r] &\mathrm U(C) 
\ar[r]^{{\phantom{(T)}}\partial^{\sigma_Q}} &
B^{\sigma_Q}
\ar[r] &Q \ar[r] &1
}
\end{xymatrix}
\label{pbC}
\end{equation}
induced from 
{\rm \eqref{eC}} via the $Q$-normal structure
$\sigmaQ \colon Q \to \Out(C)$ on $C$.
The following is again immediate.

\begin{prop}
\label{prop2}
Suppose that the embedding of $A$ into $C$ is a $Q$-normal Deuring embedding
with respect to the $Q$-normal structure $\sigma_Q \colon Q \to \Out(C)$
on $C$ and the group extension {\rm\eqref{eTS}}.

\noindent
{\rm(i)} The inclusion maps $N^{\mathrm U(C)}(A) \to \mathrm U(C)$ and
$B^{A,\sigma_Q} \to B^{\sigma_Q}$ yield a congruence 
\begin{equation}
\begin{gathered}
\begin{xymatrix}{
\mathrm e^A_{(C,\sigma_Q)}\colon
0 \ar[r] &\mathrm U(S) \ar@{=}[d]\ar[r] &N^{\mathrm U(C)}(A) 
\ar[d]
\ar[r]^{{\phantom{(T)}}\partial^{A,\sigma_Q}} &
B^{A,\sigma_Q}
\ar[d]
\ar[r] &Q \ar@{=}[d] \ar[r] &1
\\
\mathrm e_{(C,\sigma_Q)}\colon
0 \ar[r] &\mathrm U(S) \ar[r] &\mathrm U(C) 
\ar[r]^{{\phantom{(T)}}\partial^{\sigma_Q}} &
B^{\sigma_Q}
\ar[r] &Q \ar[r] &1
}
\end{xymatrix}
\end{gathered}
\end{equation}
of crossed $2$-fold extensions from the crossed $2$-fold extension 
{\rm\eqref{eCsQ1}} to the crossed $2$-fold extension
{\rm\eqref{pbC}}.

\noindent
{\rm(ii)} The injection $C^{\mathrm U(C)}(T) \to N^{\mathrm U(C)}(A)$
 yields the morphism
\begin{equation}
\begin{gathered}
\begin{xymatrix}
{
\mathrm e^{A,T}_{(C,\sigma_Q)}\colon 0\ar[r]
&\mathrm U(S)
\ar@{=}[d]
\ar[r]
& C^{\mathrm U(C)}(T) 
\ar[d]\ar[r]^{{\phantom{(T)}}\partial^{A,T,\sigma_Q}}
&B^{A,\sigma_Q}
\ar@{=}[d]
\ar[r]^{}
&G
\ar[d]^{\piQ}
\ar[r]
&1
\\
\mathrm e^A_{(C,\sigma_Q)}\colon 0\ar[r]
&\mathrm U(S)
\ar[r]
& N^{\mathrm U(C)}(A) 
\ar[r]^{{\phantom{(T)}}\partial^{A,\sigma_Q}}
&B^{A,\sigma_Q}
\ar[r]^{}
&Q
\ar[r]
&1
}
\end{xymatrix}
\end{gathered}
\label{kind1}
\end{equation}
of crossed $2$-fold extensions
from the crossed $2$-fold extension 
{\rm{\eqref{eCAT}}}
to the crossed $2$-fold extension
{\rm \eqref{eCsQ1}}.   \qed
\end{prop}

\subsection{Results related with the two notions of normal Deuring embedding}

\begin{thm}
\label{7.1}
Let $A$ be a central $T$-algebra,  
$C$ a central $S$-algebra, and $A \subseteq C$ 
an embedding of $A$ into $C$
having the property that $A$ coincides with the centralizer of $T$
in $C$. Furthermore, let $\sigma_Q \colon Q \to \Out(C)$ be a $Q$-normal
structure on $C$, and suppose that
the embedding of $A$ into $C$ is a $Q$-normal Deuring embedding 
with respect to $\sigma_Q$  and  the group extension {\rm\eqref{eTS}}. 
Then the data determine 
a unique homomorphism $\chi_G \colon G \to \Out(C,A)$
that turns the given $Q$-normal
Deuring embedding of $A$ into $C$ into a strong $Q$-normal Deuring 
embedding of $A$ into $C$ with respect 
to the given data in such a way that, relative to
the associated
$G$-normal structure 
\begin{equation*}
\sigma_G\colon G \stackrel{\chi_G} \longrightarrow \Out(C,A)
\stackrel{\res}\longrightarrow \Out(A)
\end{equation*}
on $A$, cf. {\rm{\eqref{assG}}},
\begin{equation*}[\mathrm e_{(A,\sigma_G)} ] = \inf [{\mathrm e}_{(C,\sigma_Q)}] \in \mathrm H^3(G,\mathrm U(T)).
\end{equation*}
\end{thm}

\begin{proof}
Recall that the Teich\-m\"uller complex 
$\mathrm e_{(C,{\sigma_Q} )}$ 
of the $Q$-normal $S$-algebra $(C,{\sigma_Q} )$,
spelled out above as
\eqref{pbC},
represents the Teich\-m\"uller class 
$[{\mathrm e}_{(C,\sigma_Q)}] \in \mathrm H^3(Q,\mathrm U(S))$
of the $Q$-normal central $S$-algebra $(C,{\sigma_Q})$.

Suppose that
the embedding of $A$ into $C$ is a $Q$-normal Deuring embedding 
with respect to the $Q$-normal structure
$\sigma_Q \colon Q \to \Out(C)$ on $C$ and  the 
group extension  \eqref{eTS}.
By \cref{prop2}(i),
the crossed $2$-fold extension 
${\mathrm e}^A_{(C,\sigma_Q)}$, cf. \eqref{eCsQ1},
is available and
is congruent to ${\mathrm e}_{(C,\sigma_Q)}$, whence
\begin{equation*}
[{\mathrm e}_{(C,\sigma_Q)}] =
[{\mathrm e}^A_{(C,\sigma_Q)}]
\in \mathrm H^3(Q,\mathrm U(S)).
\end{equation*}
Moreover, 
by \cref{prop2}(ii),
the crossed 2-fold extension \eqref{eCAT}
is available and,
since the centralizer 
of $A$ in $C$ coincides with $T$,
the inclusion $\mathrm U(A)\subseteq  C^{\mathrm U(C)}(T)$
identifies the group $\mathrm U(A)$ 
of invertible elements of $A$ 
with the centralizer $C^{\mathrm U(C)}(T)$
of $T$ in $\mathrm U(C)$.
Hence the crossed 2-fold extension \eqref{eCAT}
has the form
\begin{equation}
\begin{xymatrix}
{
\mathrm e^{A,T}_{(C,\sigma_Q)}\colon 0\ar[r]
&\mathrm U(S)
\ar[r]
& \mathrm U(A)
\ar[r]^{{\phantom{(T)}}\partial^{A,T,\sigma_Q}{\phantom{(T)}}}
&B^{A,\sigma_Q}
\ar[r]^{}
&G
\ar[r]
&1 ,
}
\end{xymatrix}
\label{eCAT2}
\end{equation}
and the injection $\iota\colon \mathrm U(A) \to N^{\mathrm U(C)}(A)$
induces the morphism \eqref{kind1} of crossed 2-fold extensions in
\cref{prop2}(ii); this is
a morphism of crossed $2$-fold extensions of the kind
$(1,\iota,1,\piQ)\colon 
{\mathrm e}^{A,T}_{(C,\sigma_Q)} \to {\mathrm e}^A_{(C,\sigma_Q)}$.

Denote by $i \colon \mathrm U(S) \to \mathrm U(T)$
 the inclusion. The canonical homomorphism 
\begin{equation*}
{B^{A,\sigma_Q} =\Aut^A(C)\times_{\Out (C)} Q \longrightarrow \Aut^A(C)}
\end{equation*}
induces a morphism
\begin{equation*}
(\mathrm{Id},\,\cdot\,)\colon (\mathrm U(A),   B^{A,\sigma_Q}, \partial^{A,T,\sigma_Q})
\longrightarrow
(\mathrm U(A), \Aut^A(C),\partial^{A}_C)
\end{equation*}
of crossed modules and hence
a homomorphism 
$\chi_G \colon G \to \Out(C,A)$  
such that
\begin{equation*}
\begin{xymatrix}
{
\mathrm e^{A,T}_{(C,\sigma_Q)}\colon 0\ar[r]
&\mathrm U(S)
\ar[d]^{i}
\ar[r]
& \mathrm U(A) 
\ar@{=}[d]\ar[r]^{{\phantom{(TT)}}\partial^{A,T,\sigma_Q}{\phantom{(TT)}}} 
&B^{A,\sigma_Q}
\ar[d]
\ar[r]^{}
&G
\ar[d]^{\chi_G} 
\ar[r]
&1
\\
{\phantom{(T)}}\mathrm e_A
\colon 0\ar[r]
&\mathrm U(T)
\ar[r]
& \mathrm U(A)
\ar[r]^{{\phantom{(TT)}}\partial^{A}_C{\phantom{(TTTT)}}}
&\Aut^A(C)
\ar[r]^{}
&\Out(C,A)
\ar[r]
&1
}
\end{xymatrix}
\end{equation*}
is a morphism of crossed $2$-fold extensions from
\eqref{eCAT2} to
\eqref{eA}.
The homomorphism $\chi_G$ turns the given $Q$-normal
Deuring embedding of $C$ into $A$ into a strong $Q$-normal Deuring 
embedding of $C$ into $A$ with respect to the given data.

The $G$-normal structure $\sigma_G\colon
G \stackrel{\chi_G}\to \Out(C,A)\stackrel{\res}\to \Out(A)$ 
associated to the strong $Q$-normal Deuring embedding,
in turn, induces a
morphism
\begin{equation*}
\begin{xymatrix}
{
\mathrm e^{A,T}_{(C,\sigma_Q)}\colon 0\ar[r]
&\mathrm U(S)
\ar[d]^{i}
\ar[r]
& \mathrm U(A) 
\ar@{=}[d]\ar[r]^{{\phantom{(TT)}}\partial^{A,T,\sigma_Q}{\phantom{(TT)}}} 
&B^{A,\sigma_Q}
\ar[d]
\ar[r]^{}
&G
\ar@{=}[d]
\ar[r]
&1
\\
 \mathrm e_{(A,\sigma_G)}
\colon 0\ar[r]
&\mathrm U(T)
\ar[r]
& \mathrm U(A)
\ar[r]^{{\phantom{(TT)}}\partial^{\sigma_G}{\phantom{(TT)}}}
&B^{\sigma_G}
\ar[r]^{}
&G
\ar[r]
&1
}
\end{xymatrix}
\end{equation*}
of crossed $2$-fold extensions from \eqref{eCAT2} to
the corresponding crossed 2-fold extension $\mathrm e_{(A,\sigma_G)}$  
of the kind \eqref{pb1}.
Consequently $[\mathrm e_{(A,\sigma_G)}] 
= \inf[{\mathrm e}_{(C,{\sigma_Q})}]$. \end{proof}

Theorem \ref{7.1} has a converse; this converse
sort of a characterizes the
Teich\-m\"uller classes in $\mathrm H^3(Q,\mathrm U(S))$.

\begin{thm}
\label{7.2}
Let $k \in \mathrm H^3(Q,\mathrm U(S))$, let $A$ be a central 
$T$-algebra, and let 
$\rrr\colon G \to \Out(A)$
be a  $G$-normal structure on $A$
relative to the action $\llambda\colon G \to \Aut^S(T)$
of $G$ on $T$.
Suppose that
\begin{equation*}
\inf(k) = [{\mathrm e}_{(A,\rrr)}]\in \mathrm H^3(G,\mathrm U(T)).
\end{equation*}
Then there is a $Q$-normal $S$-central crossed product algebra 
\begin{equation*}
(C,\sigma_Q)=((A,N,\mathrm e,\vartheta),\sigma_Q )
\end{equation*}
related with the other data in the following way.

\begin{itemize}
\item
The $Q$-normal algebra $(C,\sigma_Q)=((A,N,\mathrm e,\vartheta),\sigma_Q )$ 
has
Teich\-m\"uller class $k$;
\item once the $Q$-normal algebra $((A,N,\mathrm e,\vartheta),\sigma_Q )$ 
has been fixed,
the data determine a homomorphism
$\chi_G\colon G \to \Out(C,A)$ that turns
the
obvious embedding of $A$ into $(A,N,\mathrm e, \vartheta)$ 
into a strong $Q$-normal Deuring
embedding with respect to 
$\sigma_Q \colon Q \to \Out (A,N,\mathrm e, \vartheta)$
 and  the group extension {\rm\eqref{eTS}};
\item
the associated
$G$-normal structure 
\begin{equation*}
G \stackrel{\chi_G} \longrightarrow \Out(C,A)
\stackrel{\res}\longrightarrow \Out(A)
\end{equation*}
on $A$, cf. {\rm{\eqref{assG}}}, and the given 
$G$-normal
structure $\rrr \colon G \to \Out(A)$ on $A$ coincide.
\end{itemize}
\end{thm}

\begin{comp} 
\label{compl1}
In the situation of Theorem {\rm \ref{7.2}}, if $A$ is an Azumaya $T$-algebra,
the algebra $(A,N,\mathrm e, \vartheta)$ is an Azumaya $S$-algebra.
\end{comp}

\begin{rema}
In the special case where $\inf (k) = 0$, the argument to be given 
comes down to that given for the statement of Theorem~\ref{normalcrossed}, 
and this theorem
is in fact a special case of Theorem~\ref{7.2}.
\end{rema}

\begin{proof}[Proof of Theorem {\rm \ref{7.2}}]
For convenience, we split the reasoning into 
Propositions
\ref{7.3} - \ref{7.61} below.

Consider a $G$-normal central $T$-algebra $(A,\rrr )$,
 and denote
by ${\rrn\colon N \to \Out(A)}$ 
the restriction of ${\rrr \colon G \to \Out(A)}$ to $N$
so that $(A,\rrn )$
is an $N$-normal central $T$-algebra.
The obvious unlabeled vertical arrow and the injection $i^N$ turn
\begin{equation*}
\begin{CD}
\mathrm e_{(A,\rrn )} \colon 
0 @>>> \mathrm U(T) @>>> \mathrm U(A) 
@>{\partial^{\rrn}}>> B^{\rrn} @>>> N @>>> 1
\\
@.
@|
@|
@VVV
@V{i^N}VV
@.
\\
\mathrm e_{(A,\rrr)} \colon 
0 @>>> \mathrm U(T) @>>> \mathrm U(A) 
@>{\partial^{\rrr}}>> B^{\rrr} @>>> G @>>> 1
\end{CD}
\end{equation*}
into a commutative diagram 
having as its rows the 
(exact)
Teich\-m\"uller complexes 
${\mathrm e}_{(A,\rrn )}$ and $\mathrm e_{(A,\rrr)}$
of $(A,\rrn)$ and $(A,\rrr)$, respectively.
Consequently the combined homomorphism
\begin{equation*}
\begin{CD}
B^{\rrr} @>>> G @>{\piQ}>> Q
\end{CD}
\end{equation*}
yields a
group
extension
\begin{equation}
1 \longrightarrow B^{\rrn} \longrightarrow B^{\rrr} \longrightarrow Q \longrightarrow 1.
\label{ext1}
\end{equation}
Let
\begin{equation}
\begin{CD}
\hat{\mathrm e} \colon 0 @>>> \mathrm U(T) @>>> \mathrm U(A) @>>> \mathrm U(A)/\mathrm U(T) @>>> 1
\end{CD}
\label{ehat}
\end{equation}
and
\begin{equation*}
\begin{CD}
 1 @>>> \mathrm U(A)/\mathrm U(T) @>{\pphi}>> B^{\rrn}@>>> N@>>> 1
\end{CD}
\end{equation*}
be
the obvious group extensions so that splicing them yields the
Teich\-m\"uller complex 
\begin{equation*}
\begin{CD}
{\mathrm e}_{(A,\rrn)}
\colon
 0 @>>> \mathrm U(T) @>>> \mathrm U(A) @>>>  B^{\rrn}@>>> N@>>> 1
\end{CD}
\end{equation*}
of $(A,\rrn)$.
We denote the resulting
morphism
\begin{equation}
\begin{gathered}
\begin{xymatrix}{
1 \ar[r]   &\mathrm U(A)/\mathrm U(T) \ar[d]^{\pphi} \ar[r] &B^{\rrr} \ar@{=}[d] \ar[r]
&G \ar[d]^{\piQ} \ar[r] &1 \\
1\ar[r]    &B^{\rrn} \ar@{=} \ar[r]      &B^{\rrr}\ar[r]&Q \ar[r] &1 
}
\end{xymatrix}
\end{gathered}
\label{PHI}
\end{equation}
of group extensions by $\Pphi $.

Consider
the Teich\-m\"uller complex 
\begin{equation*}
\mathrm e_{(A,\rrr )} \colon 0 \longrightarrow 
\mathrm U(T) \longrightarrow \mathrm U(A) 
\stackrel{\partial^{\rrr}}\longrightarrow B^{\rrr} \longrightarrow G \longrightarrow 1
\end{equation*}
associated to the given $G$-normal structure $\rrr\colon G \to \Out(A)$
on $A$, cf. \eqref{pb1}.
Since $\mathrm U(T)$ is a central subgroup of $\mathrm U(A)$,
the group extension $\hat {\mathrm e}$ spelled out above as
\eqref{ehat}
is a central extension and, as noted in Proposition \ref{cmcp},
$G$-crossed pair structures on $\hat{\mathrm e}$
are equivalent to $B^{\rrr}$-actions on $\mathrm U(A)$
that turn $\mathrm U(A)\to B^{\rrr}$
into a crossed module. Thus the action of $B^{\rrr}$ on $\mathrm U(A)$
that results from the given $G$-normal structure $\rrr\colon G \to \Out(A)$
via the associated crossed 2-fold extension 
$\mathrm e_{(A,\rrr )}$ induces a crossed pair structure
$\hat\ppsi \colon G \to \Out_{B^{\rrr}}(\hat{\mathrm e})$
on $\hat{\mathrm e}$ with respect to the group extension
\begin{displaymath}
\xymatrix{
1 \ar[r]   &\mathrm U(A)/\mathrm U(T)  \ar[r] &B^{\rrr}  \ar[r]
&G  \ar[r] &1
}\end{displaymath}
and the $G$-module $\mathrm U(T)$.
Then the canonical homomorphism
${\hat \bbeta\colon \Aut_{B^{\rrr}}(\hat {\mathrm e}) \to \Aut(A)}$
yields a morphism
\begin{equation*}
\xymatrix{
0\ar[r]          
&\mathrm U(T)\ar@{=}[d]\ar[r]   
&\mathrm U(A) \ar@{=}[d] \ar[r] 
&\Aut_{B^{\rrr}}(\hat{\mathrm e})\ar[d]^{\hat \bbeta}\ar[r]
&\Out_{B^{\rrr}}(\hat {\mathrm e})\ar[d]^{\hat \bbeta_{\sharp}} \ar[r] &1\\ 
0\ar[r] &\mathrm U(T)\ar[r] &
\mathrm U(A)  \ar[r]
&\Aut(A)\ar[r]
&
\Out(A)
\ar[r]   &1\\ 
}
\end{equation*}
of crossed $2$-fold extensions such that the 
composite
\begin{equation}
G  \stackrel{\hat \psi}\longrightarrow 
\Out_{B^{\rrr}}(\hat {\mathrm e}) 
\stackrel{\hat\bbeta_{\sharp}}\longrightarrow  \Out(A)
\label{comp101}
\end{equation}
coincides with $\sigmaG\colon G \to \Out(A)$.

\begin{prop}
\label{7.3}
Let $k \in \mathrm H^3(Q,\mathrm U(S))$, let $(A,\rrr)$ be a $G$-normal 
central $T$-algebra, and suppose that 
$\inf (k) = [{\mathrm e}_{(A,\rrr )}]\in \mathrm H^3(G,\mathrm U(T))$.
Then there is a group extension
\begin{equation*}
\tilde{\mathrm e} \colon 0 \longrightarrow \mathrm U(T) \longrightarrow \Gamma \longrightarrow B^{\rrn} \longrightarrow 1
\end{equation*}
together with a crossed pair structure
$\tilde \psi\colon Q \to \Out_{B^{\rrr}}(\tilde {\mathrm e})$ 
on $\tilde{\mathrm e}$ 
with respect to the group extension
{\rm \eqref{ext1}}
and the $B^{\rrr}$-module
$\mathrm U(T)$,
the requisite  module structure being induced by the map
$B^{\rrr} \to G$ in $\mathrm e_{(A,\rrr)}$,
 related with the
other data in the following way,
where $B^{\tilde \psi}$ denotes the fiber product group
$\Aut_{B^{\rrr}}(\tilde {\mathrm e})\times_{\Out_{B^{\rrr}}(\tilde {\mathrm e})}Q$
with respect to $\tilde \psi\colon Q \to \Out_{B^{\rrr}}(\tilde {\mathrm e})$.

\noindent
{\rm (i)}
The crossed $2$-fold extension
\begin{equation*}{\mathrm e}_{\tilde \psi} \colon 0 \longrightarrow \mathrm U(S) \longrightarrow \Gamma \longrightarrow B^{\tilde \psi} \longrightarrow Q \longrightarrow 1\end{equation*}
associated to the crossed pair 
$(\tilde {\mathrm e},\tilde \psi)$,
cf. {\rm \eqref{epsi}},
represents $k$.

\noindent
{\rm (ii)} Relative to the obvious actions of the group
$\Aut_{B^{\rrr}}(\tilde{\mathrm e})\, ( \subseteq \Aut (\Gamma) \times B^{\rrr})$ 
on the groups 
$\mathrm U(T), \mathrm U(A), \mathrm U(A)/\mathrm U(T)$, $\Gamma$, $B^{\rrn}$ 
and $N$,
the extension group $\Gamma $ in 
$\tilde {\mathrm e}$
fits into a commutative diagram
of $\Aut_{B^{\rrr}}(\tilde{\mathrm e})$-groups with exact rows and columns as follows:
\begin{equation}
\begin{CD}
@. @. 1 @. 1@.
\\
@. @. @VVV @VVV @.
\\
\hat{\mathrm e} \colon 0 @>>> \mathrm U(T) @>>> \mathrm U(A) @>>> \mathrm U(A)/\mathrm U(T) @>>> 1
\\
@. @|   @VVV @VV{\pphi}V @.
\\
\tilde{\mathrm e} \colon 0 @>>> \mathrm U(T) @>>> \Gamma @>>> B^{\rrn} @>>> 1
\\
@. @. @VVV  @VVV @.
\\
 @.  @. N @= N @.
\\
@. @. @VVV @VVV @.
\\
@. @. 1 @. 1@.
\end{CD}
\label{7.3.2.1}
\end{equation}
\end{prop}

\begin{proof} By 
\cite[Theorem~2]{MR597986},
the morphism \ref{PHI} of group extensions
induces a morphism for the
corresponding eight term exact sequences in group cohomology constructed
in 
\cite{MR597986}. In
particular, $\Pphi $ induces the commutative diagram
\begin{displaymath}
\scalefont{0.8}{
\xymatrix{ \mathrm H^2(B^{\rrr}, \mathrm U(T))  \ar@{=}[d] \ar[r]    
          &\mathrm{Xpext}(B^{\rrr} ,B^{\rrn} ; 
          \mathrm U(T))\ar[d]^{\Pphi^\ast }
          {\ar[r]^{\kern10mm \Delta }} 
          &\mathrm H^3(Q,\mathrm U(S)) \ar[d]^{\inf} \ar[r]
          &\mathrm H^3(B^{\rrr} ,\mathrm U(T)) \ar@{=}[d] \\ 
\mathrm H^2(B^{\rrr} , \mathrm U(T))  \ar[r] 
&\mathrm{Xpext} (B^{\rrr} ,\mathrm U(A)/\mathrm U(T);\mathrm U(T)){\ar[r]^{\kern15mm \Delta }}
&\mathrm H^3(G,\mathrm U(T))\ar[r]    
&\mathrm H^3(B^{\rrr} , \mathrm U(T)) .
}
}
\end{displaymath}
By the construction of 
$\Delta $, cf. Subsection~\ref{6.11} above  or 
\cite[Subsection 1.2]{MR597986}, 
\linebreak
$\Delta[(\hat{\mathrm e}, \hat\ppsi)] = [{\mathrm e}_{(A,\rrr )}]$, 
and so, by
exactness,
$\inf(k) = [\mathrm e_{(A,\rrr)}]$ 
goes to zero in $\mathrm H^3(B^{\rrr} ,\mathrm U(T))$. Therefore 
$k$ goes to zero in $\mathrm H^3(B^{\rrr}, \mathrm U(T))$, and hence
there is  
a group extension
\begin{equation*}
\tilde{\mathrm e} \colon 0 \longrightarrow \mathrm U(T) \longrightarrow \Gamma \longrightarrow B^{\rrn} \longrightarrow 1
\end{equation*}
of the asserted kind
together with a crossed pair structure
$\tilde \psi\colon Q \to \Out_{B^{\rrr}}(\tilde {\mathrm e})$ 
on $\tilde{\mathrm e}$ 
with respect to the group extension
{\rm \eqref{ext1}}
 and the $B^{\rrr}$-module
$\mathrm U(T)$ 
whose $B^{\rrr}$-module structure is
induced by the projection $B^{\rrr} \to G$ in $\mathrm e_{(A,\rrr)}$
so that
\begin{equation*}
\Delta [(\tilde{\mathrm e}, \tilde\ppsi)] = k \in \mathrm H^3(Q,\mathrm U(S));
\end{equation*}
moreover, making a suitable choice of 
$(\tilde{\mathrm e}, \tilde \ppsi)$
by means of some diagram chase if need be, we can arrange
for $[(\tilde{\mathrm e}, \tilde \ppsi)]$ to go to 
$[(\hat{\mathrm e}, \hat \ppsi)]$
in the sense that
\begin{equation*}
\Pphi^\ast [(\tilde{\mathrm e}, \tilde \ppsi)] 
= [(\hat{\mathrm e}, \hat \ppsi)]\in 
\mathrm{Xpext} (B^{\rrr} , \mathrm U(A)/\mathrm U(T) ; \mathrm U(T)).
\end{equation*}
The crossed pair $(\tilde{\mathrm e}, \tilde \ppsi)$ has the asserted 
properties.
For $\Delta [(\tilde{\mathrm e}, \tilde \ppsi)] = [{\mathrm e}_{\tilde \ppsi}]$ by definition,
and so assertion (i)  
holds. Moreover, since
$\Pphi^\ast [(\tilde{\mathrm e}, \tilde \ppsi)] = 
[(\hat{\mathrm e}, \hat \ppsi)]$, 
assertion (ii)
holds as well. The details are as follows, cf.  
\cite[Subsection 2.2]{MR597986}.

Since 
$\Pphi^\ast[(\tilde{\mathrm e}, \tilde\ppsi)] 
= [(\hat{\mathrm e}, \hat \ppsi)]$, 
we may
identify $(\hat{\mathrm e} , \hat \ppsi)$ with 
the induced crossed pair
$(\tilde{\mathrm e} \Pphi, \tilde \ppsi^\Pphi )$, cf.
\cite{MR597986}.
Recall that
$\tilde {\mathrm e}\Pphi $ is the group extension
induced from $\tilde{\mathrm e}$
via the injective homomorphism
$\pphi\colon \mathrm U(A)/\mathrm U(T) \to B^{\rrn}$
and let $\mathrm U =\ker (\Gamma \longrightarrow N)$; since $\pphi$
identifies $\mathrm U(A)/\mathrm U(T)$ with the kernel of 
$ B^{\rrn} \to N$, 
we can write the induced
group  extension $\tilde {\mathrm e}\Pphi $ as
\begin{equation*}
\tilde{\mathrm e} \Pphi \colon 0 \longrightarrow \mathrm U(T) \longrightarrow  \mathrm U \longrightarrow \mathrm U(A)/\mathrm U(T) \longrightarrow 1.
\end{equation*}
To explain the induced
crossed pair structure
$\tilde \ppsi^\Pphi\colon G \to   \Out_{B^{\rrr}}(\tilde {\mathrm e}\Pphi)$,
we note first that
the injection $\mathrm U \to \Gamma$ induces a  morphism
\begin{displaymath}
\xymatrix{
0\ar[r] &\mathrm U(S)\ar@{=}[d] \ar[r]& \mathrm U\ar[d]\ar[r] 
&\Aut_{B^{\rrr}}(\tilde{\mathrm e})\ar@{=}[d]\ar[r]
&\Out_{B^{\rrr} }(\tilde{\mathrm e})\times_Q G  \ar[d]\ar[r]& 1\\
0\ar[r] &\mathrm U(S) \ar[r]&\Gamma \ar[r] 
&\Aut_{B^{\rrr}}(\tilde{\mathrm e})\ar[r] 
 &\Out_{B^{\rrr} }(\tilde{\mathrm e})\ar[r]& 1 
}
\end{displaymath}
of crossed 2-fold extensions. 
Moreover,
restriction of the operators on $\Gamma $ 
to $\mathrm U$
yields a homomorphism
\begin{equation*}
\res\colon
\Aut_{B^{\rrr}}(\tilde{\mathrm e}) \longrightarrow 
\Aut_{B^{\rrr}}(\tilde{\mathrm e}\Pphi ),
\end{equation*}
and this homomorphism, in turn, yields a morphism
\begin{displaymath}
\xymatrix{
0\ar[r] &\mathrm U(S)\ar[d] \ar[r] &\mathrm U\ar@{=}[d]\ar[r] 
        &\Aut_{B^{\rrr}}(\tilde{\mathrm e})\ar[d]^{\res}\ar[r]
&\Out_{B^{\rrr} }(\tilde{\mathrm e})\times_Q G  \ar[d]^{\res_{\flat}}\ar[r]& 1\\
0\ar[r] &\mathrm U(T) \ar[r]&\mathrm U \ar[r] 
&\Aut_{B^{\rrr}}(\tilde{\mathrm e}\Pphi)\ar[r] 
 &\Out_{B^{\rrr} }(\tilde{\mathrm e}\Pphi)\ar[r]& 1 
}
\end{displaymath}
of crossed 2-fold extensions.
The 
crossed pair structure
$\tilde \ppsi^\Pphi\colon G \to \Out_{B^{\rrr} } (\tilde{\mathrm e}\Pphi)$
is the composite
\begin{equation*}
G  \stackrel{\tilde \ppsi_G}
\longrightarrow
\Out_{B^{\rrr} } (\tilde{\mathrm e}) \times_Q G 
\stackrel{\res_{\flat}}
\longrightarrow
 \Out_{B^{\rrr} } (\tilde{\mathrm e}\Pphi)
\end{equation*}
of ${\res_{\flat}}$ with 
the canonical lift of the crossed pair structure
$\tilde \ppsi \colon Q \to \Out_{B^{\rrr}}(\tilde{\mathrm e})$ 
on $\tilde{\mathrm e}$ to a homomorphism
$
\tilde \ppsi_G\colon G \to \Out_{B^{\rrr} } (\tilde{\mathrm e}) \times_Q G;
$
see 
\cite[Propositions 2.3 and 2.4]{MR597986}.
The identity $\Pphi^\ast[(\tilde{\mathrm e}, \tilde\ppsi)] 
= [(\hat{\mathrm e}, \hat \ppsi)]$ means that
the two crossed pairs
 $(\hat{\mathrm e} , \hat \ppsi)$ and
$(\tilde{\mathrm e} \Pphi, \tilde \ppsi^\Pphi )$
are congruent as crossed pairs.
Thus we may take $\mathrm U$ to be $\mathrm U(A)$ such that the following 
hold:
\begin{itemize}
\item
The injection
$\mathrm U(A) \to \Gamma $ induces a 
morphism $\hat{\mathrm e} \to \tilde{\mathrm e}$ of group extensions whose 
restriction to $\mathrm U(T)$ is the identity, 
as displayed in diagram \eqref{7.3.2.1} above,
and
\item
the crossed pair structure
$\hat\ppsi \colon G \to \Out_{B^{\rrr} } (\hat{\mathrm e})$ on $\hat{\mathrm e}$
is the composite
\begin{equation}
G \stackrel{\tilde \ppsi_G} \longrightarrow \Out_{B^{\rrr} } (\tilde{\mathrm e}) 
\stackrel{\res_{\sharp}} \longrightarrow \Out_{B^{\rrr} } (\hat{\mathrm e}) 
\label{comp102}
\end{equation}
of $\tilde \ppsi_G$ with the homomorphism
$\res_{\sharp}\colon \Out_{B^{\rrr} } (\tilde{\mathrm e}) \times_Q G \to \Out_{B^{\rrr} } (\hat{\mathrm e})$
induced by the obvious restriction homomorphism
$\res\colon
\Aut_{B^{\rrr}}(\tilde{\mathrm e}) \longrightarrow 
\Aut_{B^{\rrr}}(\hat{\mathrm e} )$.

\end{itemize}
The morphism
$\hat{\mathrm e}\to \tilde{\mathrm e}$ of group extensions
yields the commutative diagram \eqref{7.3.2.1} and,
by construction, this 
is a commutative diagram of 
$\Aut_{B^{\rrr}}(\tilde{\mathrm e})$-groups.
\end{proof}

We continue the proof of Theorem \ref{7.2}. Maintaining the hypotheses of 
Proposition \ref{7.3},
we write
\begin{equation*}\mathrm e \colon 1 \longrightarrow \mathrm U(A) \stackrel{j}{\longrightarrow} \Gamma \longrightarrow N \longrightarrow 1
\end{equation*}
for the group extension that arises as the middle column of 
diagram \eqref{7.3.2.1}
and denote by $\vartheta \colon \Gamma \to \Aut (A)$ the combined homomorphism
\begin{equation*}
\Gamma \longrightarrow B^{\rrn}\longrightarrow \Aut (A).
\end{equation*}
Consider
the crossed product algebra $(A,N,\mathrm e,\vartheta )$.
By construction
\begin{equation*} (A,N,\mathrm e,\vartheta ) = 
A^t \Gamma / <a-j(a), a\in \mathrm U(A)>,\end{equation*}
cf. \cref{three}. By 
\cref{threet}(iv),
since $T|S$ 
is a Galois extension of commutative rings
with Galois group $N$,
the group $\Gamma$ now gets identified with the
normalizer $N^{\mathrm U(A,N,\mathrm e,\vartheta )}(A)$
of $A$ in the crossed product algebra $(A,N,\mathrm e,\vartheta )$.

Recall  
the notation $B^{\sigmaG}$ 
for the fiber product group ${\Aut(A) \times_{\Out(A)}G }$
with respect to the given $G$-normal structure 
$\sigmaG \colon G \to \Out(A)$ on $A$,
cf. 
Subsection \ref{twothree}.
Furthermore, 
recall from Subsection~\ref{6.11} above 
that 
$\Aut_{B^{\rrr}}(\tilde{\mathrm e})$ denotes  the subgroup
of $\Aut (\Gamma) \times B^{\rrr}$ that consists of the pairs 
$(\alpha , x)$ which
render the diagram
\begin{equation*}
\begin{CD}
\tilde {\mathrm e}\colon 0
@>>> 
\mathrm U(T) 
@>>> 
\Gamma
@>>> 
B^{\rrn}
@>>> 
1
\\
@.
@V{\ell_x}VV
@V{\alpha}VV
@V{i_x}VV
@.
\\
\tilde {\mathrm e}\colon 0
@>>> 
\mathrm U(T) 
@>>> 
\Gamma
@>>> 
B^{\rrn} 
@>>> 
1
\end{CD}
\end{equation*}
commutative;
here, given $x \in  B^{\sigmaG}$, the notation  
$i_x\colon B^{\rrn} \to B^{\rrn}$
refers to conjugation by $x \in  B^{\sigmaG}$
and $\ell_x\colon \mathrm U(T)\to  \mathrm U(T)$
 to the canonical action of
$B^{\sigmaG}$ on $\mathrm U(T)$
(recall that $T$ denotes the center of $A$)
induced from the 
action of $B^{\sigmaG}$ on $A$ and hence on $\mathrm U(T)$
via the canonical
homomorphism $B^{\sigmaG}\to \Aut(A)$.

\begin{prop}
\label{7.4}
Setting
\begin{equation}{}^{(\alpha , x)}(ay) = 
{}^xa
{}^{\alpha}y,
\ a \in A,\ y \in \Gamma,
\label{rule1}
\end{equation}
where 
$(\alpha , x)
\in  \Aut_{B^{\rrr}}(\tilde {\mathrm e})\, 
(\subseteq \Aut (\Gamma) \times B^{\rrr} 
\subseteq \Aut (\Gamma) \times \Aut (A)\times G)$,
we obtain
a homomorphism
\begin{equation*}
\bbeta \colon  
 \Aut_{B^{\rrr}}(\tilde {\mathrm e})
\longrightarrow \Aut^A(A, N,\mathrm e,\vartheta ),
\end{equation*}
and this homomorphism, in turn, yields  morphisms
\begin{equation*}
\xymatrix{
0\ar[r]          
&\mathrm U(S)\ar@{=}[d]\ar[r]   
&\Gamma \ar@{=}[d] \ar[r] 
&\Aut_{B^{\rrr}}(\tilde {\mathrm e})\ar[d]^{\bbeta}\ar[r]
&\Out_{B^{\rrr}}(\tilde {\mathrm e})\ar[d]^{\bbeta_{\sharp}} \ar[r] &1\\ 
0\ar[r] &\mathrm U(S)\ar[r] &
N^{\mathrm U(A, N,\mathrm e,\vartheta )}(A)  \ar[r]
&\Aut^A(A, N,\mathrm e,\vartheta )\ar[r]
&
\Out^A(A, N,\mathrm e,\vartheta )
\ar[r]   &1\\ 
}
\end{equation*}
and
\begin{equation*}
\xymatrix{
0\ar[r]          
&\mathrm U(S)\ar@{=}[d]\ar[r]   
&\mathrm U(A) \ar@{=}[d] \ar[r] 
&\Aut_{B^{\rrr}}(\tilde {\mathrm e})\ar[d]^{\bbeta}\ar[r]
&\Out_{B^{\rrr}}(\tilde {\mathrm e})\times_QG\ar[d]^{\bbeta_{\flat}} \ar[r] &1\\ 
0\ar[r] &\mathrm U(S)\ar[r] &
\mathrm U(A)  \ar[r]
&\Aut^A(A, N,\mathrm e,\vartheta )\ar[r]
&
\Out((A, N,\mathrm e,\vartheta ),A)
\ar[r]   &1\\ 
}
\end{equation*}
of crossed $2$-fold extensions.
Furthermore, the homomorphisms $\bbeta_{\sharp}$, $\bbeta_{\flat}$,
and the obvious unlabeled homomorphisms render the diagram
\begin{equation}
\begin{CD}
\Out_{B^{\rrr}}(\tilde {\mathrm e})\times_QG
@>>>
\Out_{B^{\rrr}}(\tilde {\mathrm e})
\\
@V{\bbeta_{\flat}}VV
@V{\bbeta_{\sharp}}VV
\\
\Out((A, N,\mathrm e,\vartheta ),A)
@>>>
\Out^A(A, N,\mathrm e,\vartheta )
\end{CD}
\label{diag2}
\end{equation}
commutative.
\end{prop}

\begin{proof}
The left $A$-module that underlies the twisted group ring
$A^t \Gamma$ is the free $A$-module having $\Gamma$ as an $A$-basis,
whence it is manifest that \eqref{rule1} 
yields an action of the group
$\Aut_{B^{\rrr}}(\tilde {\mathrm e})$
on that left $A$-module.

Next we  show that the 
$\Aut_{B^{\rrr}}(\tilde {\mathrm e})$-action on the 
left $A$-module that underlies the twisted group ring
$A^t \Gamma$ is compatible with the multiplicative structure
of $A^t \Gamma$. To this end,
consider the crossed module 
$(\Gamma,\Aut_{B^{\rrr}}(\tilde {\mathrm e}),\beta)$, cf. 
the middle columns of the
commutative diagram \eqref{diag1} above.
Since
$\beta\colon \Gamma \to  \Aut_{B^{\rrr}}(\tilde {\mathrm e})$ is a morphism of
$\Aut_{B^{\rrr}}(\tilde {\mathrm e})$-groups,
given  $y\in \Gamma$ and
$(\alpha,x) \in \Aut_{B^{\rrr}}(\tilde {\mathrm e})$, 
\begin{equation*}
\beta({}^{(\alpha,x)}y) =(\alpha,x) \beta(y)(\alpha,x)^{-1}\in 
\Aut_{B^{\rrr}}(\tilde {\mathrm e}).
\end{equation*}
Let $\mathrm{can}\colon
 \Aut_{B^{\rrr}}(\tilde {\mathrm e})
\to \Aut(A)$ denote
the canonical homomorphism.
It  is now manifest that the action $\vartheta\colon \Gamma \to \Aut(A)$
of $\Gamma$ on $A$
factors through $\beta$, that is,
$\vartheta$ 
coincides with the combined homomorphism
\begin{equation*}
\Gamma \stackrel{\beta}\longrightarrow  \Aut_{B^{\rrr}}(\tilde {\mathrm e})
\stackrel{\mathrm{can}}\longrightarrow \Aut(A).
\end{equation*}
 Hence, given
${(\alpha,x) \in 
\Aut_{B^{\rrr}}(\tilde {\mathrm e})\,
(\subseteq \Aut (\Gamma) \times B^{\rrr})}$,  
$b \in A$, and $y\in \Gamma$, 
\begin{equation}
{}^{x\vartheta(y)x^{-1}} b =
{}^{\vartheta ({}^{\alpha }y)} b.
\label{rule2}
\end{equation}
Thus,
given $(\alpha , x) \in \Aut_{B^{\rrr}}(\tilde {\mathrm e})$, $y\in \Gamma $, 
$a\in A$, in view of \eqref{rule2}
we conclude  
\begin{equation*}
{}^{(\alpha , x)}(y a)
    = {}^{(\alpha ,x)} ({}^{\vartheta(y)}a\, y) 
    = ({}^{x\vartheta(y)x^{-1} x} a)^\alpha y 
= ({}^{\vartheta ({}^{\alpha }y) x} a)^\alpha y 
= {}^\alpha y \, {}^x a .
\end{equation*} 
Consequently \eqref{rule1} yields an action of
$\Aut_{B^{\rrr}}(\tilde {\mathrm e})$ on the algebra
$A^t \Gamma$.

Finally, to show that the 
action 
of  $\Aut_{B^{\rrr}}(\tilde {\mathrm e})$ on the algebra $A^t \Gamma$
preserves the two-sided ideal
$<a-j(a), a\in \mathrm U(A)>$ in $A^t \Gamma$,
let $a \in \mathrm U(A)$ and $(\alpha , x) \in \Aut_{B^{\rrr}}(\tilde {\mathrm e})$.
In view of 
\cref{7.3}(ii), 
$j({}^x a) = {}^\alpha (j(a))$,
whence
\begin{equation*}{}^{(\alpha , x)}(a- j(a)) =
({}^x a - j({}^x a)). \qedhere
\end{equation*}
\end{proof}

With respect to the crossed pair structure 
$\tilde \psi\colon Q \to \Out_{B^{\rrr}}(\tilde {\mathrm e})$ 
on $\tilde{\mathrm e}$,
the fiber product group
$B^{\tilde \ppsi} 
=\Aut_{B^{\rrr}}(\tilde {\mathrm e})\times_{\Out_{B^{\rrr}}(\tilde {\mathrm e})}Q$ is defined.
As before, we denote by
$
\tilde \ppsi_G\colon G \to \Out_{B^{\rrr} } (\tilde{\mathrm e}) \times_Q G
$
the canonical lift, into the fiber product group with respect to the
surjection $\piQ\colon G \to Q$,
 of the crossed pair structure
$\tilde \ppsi \colon Q \to \Out_{B^{\rrr}}(\tilde{\mathrm e})$ 
on $\tilde{\mathrm e}$.
Define
 $\chi_G \colon G \to  \Out((A, N,\mathrm e,\vartheta ),A)$ 
to be the combined homomorphism 
\begin{equation}
\chi_G \colon G\stackrel{\tilde \psi_G} \longrightarrow
\Out_{B^{\rrr}}(\tilde {\mathrm e})
\stackrel{\bbeta_{\flat}}
\longrightarrow
  \Out((A, N,\mathrm e,\vartheta ),A).
\label{comp103}
\end{equation}
Moreover, 
the composite homomorphism 
\begin{equation*}
\sigmaQ \colon Q \stackrel{\tilde \psi}
\longrightarrow 
\Out_{B^{\rrr}}(\tilde {\mathrm e})
\stackrel{\bbeta_{\sharp}}
\longrightarrow
 \Out^A(A, N,\mathrm e,\vartheta )
\end{equation*} 
yields a $Q$-normal structure
$\sigmaQ\colon Q \to \Out^A(A, N,\mathrm e,\vartheta )$
on the central $S$-algebra $(A, N,\mathrm e,\vartheta )$.
Denote by $i\colon \Gamma \to \mathrm U(A,N, \mathrm e, \vartheta )$
the inclusion and by $\tilde \bbeta$ the combined homomorphism
\begin{equation}
\tilde \bbeta
\colon B^{\tilde \ppsi} 
\stackrel{\mathrm{can}}
\longrightarrow
\Aut_{B^{\rrr}}(\tilde {\mathrm e})
\stackrel{\bbeta}
\longrightarrow
\Aut^A (A, N, \mathrm e, \vartheta).
\label{tildegamma}
\end{equation}

\begin{prop}
\label{7.61} 
Write $C=(A, N,\mathrm e,\vartheta )$.  
The homomorphisms $\sigmaQ$, $\llambda$, $\chi_G$,
$\sigmaG$, $i$, and $\tilde \gamma$
match in the following sense.

\noindent
{\rm(i)}
The homomorphisms $\sigmaQ$ and $\chi_G$
yield a commutative diagram
\begin{equation}
\begin{CD}
\mathrm e_{(T|S)}\colon
1
@>>>
N  
@>>>
G
@>{\piQ}>>
Q
@>>>
1
\\
@.
@V{\chi_N}VV
@V{\chi_G}VV
@V{\chi_Q}VV
@.
\\
\mathrm e_{(A,C)}\colon
1
@>>>
N^{\mathrm U(C)}(A)/\mathrm U(A)  
@>>>
\Out(C,A)
@>>>
\Out^A(C)
@>>>
1
\end{CD}
\label{diag4}
\end{equation}
with exact rows.

\noindent
{\rm(ii)}
The composite homomorphism
\begin{equation}
G \stackrel{\chi_G}\longrightarrow 
\Out((A, N,\mathrm e,\vartheta ),A) 
\stackrel{\res}\longrightarrow
 \Out(A)
\label{comp104}
\end{equation}
coincides with 
$\sigmaG \colon G \to \Out(A)$.

\noindent
{\rm(iii)}
The two homomorphisms 
$i$ and $\tilde \bbeta$ 
yield a
morphism
of crossed $2$-fold extensions
\begin{equation*}
\xymatrix{
1 \ar[r] &\mathrm U(S) \ar@{=}[d]\ar[r]&\Gamma\ar[d]^i\ar[r] & B^{\tilde \ppsi}\ar[d]^{\tilde \bbeta}\ar[r]& Q\ar[d]^{\sigmaQ} \ar[r]& 1\\
1 \ar[r] &\mathrm U(S) \ar[r]&  \mathrm U(A,N, \mathrm e, \vartheta )\ar[r]  & \Aut^A (A, N, \mathrm e, \vartheta) \ar[r]&  \Out^A(A, N,\mathrm e,\vartheta ) \ar[r]& 1
}
\end{equation*}
whence $(i, \tilde \bbeta )$ induces a congruence
$(1, i, \cdot , 1) \colon \mathrm e_{\tilde\ppsi} \longrightarrow \mathrm e_{((A,N,\mathrm e , \vartheta ),\sigma_Q)}$
of crossed $2$-fold extensions.

\noindent
{\rm(iv)} The homomorphism
$\chi_N\colon N \to N^{\mathrm U(C)}(A)/\mathrm U(A)$
turns the embedding of $A$ into 
\linebreak
$C=(A, N,\mathrm e,\vartheta )$ into a strong
$N$-normal Deuring embedding with respect to
$\mathrm{Id}\colon N \to \Aut(T|S)$.
\end{prop}

\begin{proof} 
(i) It is obvious that the diagram
\begin{equation*}
\begin{CD}
G
@>{\piQ}>> 
Q
\\
@V{\tilde \psi_G}VV
@V{\tilde \psi}VV
\\
\Out_{B^{\rrr}}(\tilde {\mathrm e})\times_QG
@>>>
\Out_{B^{\rrr}}(\tilde {\mathrm e})
\end{CD}
\end{equation*}
is commutative. Combining this diagram with the commutative diagram 
\eqref{diag2}, we obtain the right-hand square of \eqref{diag4}.
Since the lower row of that diagram is exact,
the homomorphisms $\chi_G$ and $\sigmaQ$
induce the requisite homomorphism
$\chi_N\colon N \to N^{\mathrm U(C)}/\mathrm U(A)$. 

\noindent
(ii)
Consider the diagram
\begin{equation*}
\xymatrix{
G  \ar[dr]_{\hat \psi} \ar[r]^{\tilde \ppsi_G{\phantom{TTTTTT}}} 
&\Out_{B^{\rrr} } (\tilde{\mathrm e}) \times_Q G \ar[d]^{\res_{\sharp}}
\ar[r]^{\bbeta_{\flat}}
&\Out((A, N,\mathrm e,\vartheta ),A)
\ar[d]^{\res}
\\
& 
\Out_{B^{\rrr}}(\hat {\mathrm e}) 
\ar[r]^{\bbeta_{\sharp}}& \Out(A)
}
\end{equation*}
The right-hand square is commutative in an obvious manner.
The left-hand triangle is 
commutative since, as noted earlier, 
the composite \eqref{comp102} coincides with 
$\hat \psi$.
The upper row
yields the homomorphism
$\chi_G \colon G \to \Out((A, N,\mathrm e,\vartheta ),A)$,
by the very definition \eqref{comp103} of $\chi_G$.

As noted above, the composite \eqref{comp101}, viz.
$G  \stackrel{\hat \psi}\longrightarrow 
\Out_{B^{\rrr}}(\hat {\mathrm e}) 
\stackrel{\hat\bbeta_{\sharp}}\longrightarrow  \Out(A)$,
 yields the
given $G$-normal structure ${\sigmaG\colon G \to \Out(A)}$ on $A$.
Consequently \eqref{comp104}
coincides with the structure map
$\sigmaG \colon G \to \Out(A)$ as asserted.

\noindent
(iii) This is obvious.

\noindent
(iv) Consider the commutative diagram
\begin{equation*}
\begin{CD}
\mathrm e_{(T|S)} \colon 1 @>>> N 
@>>> G @>>> Q @>>> 1
\\
@.
@V{\chi_N}VV
@V{\chi_G}VV
@V{\sigmaQ}VV
@.
\\
\mathrm e_{(A,C)}\colon 1 @>>>
N^{\mathrm U(C)}(A)/\mathrm U(A) 
@>>>
\Out(C,A)
@>>>
\Out^A(C)
@>>> 1
\\
@.
@V{\eta_{\flat}}VV
@V{\res}VV
@V{\res}VV
@.
\\
\phantom{\mathrm e_{(T|S)} \colon}
1 @>>> \Aut(T|S) 
@>>>  \Aut^S(T)  @>{\res}>> \Aut(S). @. 
\end{CD}
\end{equation*}
By construction, the outer-most
diagram coincides with the commutative diagram \eqref{CD1},
and the left-most column is the composite 
\eqref{compelev},
with $N$ substituted for $Q$ and $\Aut(T|S)$ for $\Aut(S)$.
Consequently the composite 
$\eta_{\flat}\circ \chi_N\colon N \to \Aut(T|S)$
is the identity.
Since $T|S$ is a Galois extension of commutative rings with Galois group $N$,
by \cref{threet}(ii),
the algebra $A$ coincides with the centralizer of $T$ in 
$C=(A,N,\mathrm e,\vartheta)$ whence,
by \cref{twosixo}(iii),
the homomorphism
$\eta_{\flat}$ is injective.
Consequently 
$\eta_{\flat}$ and $\chi_N$ are isomorphisms,  and
\[
\chi_N\colon N \longrightarrow N^{\mathrm U(C)}(A)/\mathrm U(A)
\]
turns the embedding of $A$ into $C$ into a strong
$N$-normal Deuring embedding with respect to
$\mathrm{Id}\colon N \to \Aut(T|S)$.
\end{proof}

We can now complete the proof of \cref{7.3}:
Since the structure homomorphism $\llambda\colon G \to \Out(A)$
is a $G$-normal structure relative to the 
action $\llambda\colon G \to \Aut^S(T)$ of $G$ on $T$,
by definition, the 
composite homomorphism 
$G  \stackrel{\sigmaG} \longrightarrow
\Out (A)\stackrel{\res}\longrightarrow 
\Aut^S(T)$ coincides with $\llambda\colon G \to \Aut^S(T)$; 
 since, by \cref{7.61}(ii),
the homomorphism \eqref{comp104}
coincides with 
$\sigmaG \colon G \to \Out(A)$,
we conclude that the composite
\begin{equation*}
G \stackrel{\chi_G}\longrightarrow \Out ((A,N,\mathrm e, \vartheta),A)
\stackrel{\mathrm {res}} \longrightarrow \Aut^S(T) 
\end{equation*}
coincides with $\llambda$, cf. \eqref{comp00}.

By Proposition \ref{7.61}(i),
the diagram \eqref{diag4} is commutative, 
and by Proposition \ref{7.61}(iv), the homomorphism
$\chi_N\colon N \to N^{\mathrm U(C)}(A)/\mathrm U(A)$
turns the embedding of $A$ into $C$ into a strong
$N$-normal Deuring embedding with respect to
$\mathrm{Id}\colon N \to \Aut(T|S)$.
Consequently, cf. \eqref{comp103},
the homomorphism 
$\chi_G\colon G \to \Out((A, N,\mathrm e,\vartheta ),A)$
turns the embedding
of $A$ into $(A, N,\mathrm e,\vartheta )$ into a strong $Q$-normal 
Deuring embedding with respect to the $Q$-normal structure
\linebreak
$\sigmaQ\colon Q \to \Out(A, N,\mathrm e,\vartheta )$ on 
$(A, N,\mathrm e,\vartheta )$
and the
structure extension {\rm\eqref{eTS}}.

\cref{7.61}(ii) says that
the  $G$-normal structure
$G \to \Out(A)$ on $A$
associated to the strong $Q$-normal Deuring embedding, 
cf. {\rm \eqref{assG}}, coincides with the given
$G$-normal structure $\sigmaG\colon G \to \Out(A)$ on $A$.

\cref{7.3}(i) and \cref{7.61}(iii) 
together entail that the $Q$-normal $S$-algebra
$((A,N,\mathrm e, \vartheta ), \sigma_Q )$
has Teich\-m\"uller class
$k$ as asserted since
the crossed 2-fold extension $\mathrm e_{\tilde \ppsi}$ represents $k$.

The proof of Theorem \ref{7.2} is now complete. \end{proof}

\begin{proof}[Proof of Complement {\rm{\ref{compl1}}}]
This is a consequence of 
\cref{tpf}(xi).
\end{proof}

Recall that 
$B^{\tilde \ppsi}$
denotes the fiber product group
$B^{\tilde \ppsi}=
\Aut_{B^{\sigmaG}}(\tilde {\mathrm e}) \times_{\Out_{B^{\sigmaG}}(\tilde {\mathrm e})}Q$
with respect to the crossed pair structure 
$\tilde \psi\colon Q \to \Out_{B^{\sigmaG}}(\tilde {\mathrm e})$
on $\tilde {\mathrm e}$, 
and that, likewise,
$B^{A,\sigmaQ}$ denotes the fiber product group
$B^{A,\sigmaQ}=
\Aut^A(A,N,\mathrm e,\vartheta ) \times_{\Out(A,N,\mathrm e,\vartheta )}Q$
with respect to the $Q$-normal structure
$\sigmaQ\colon Q \to \Out(A,N,\mathrm e,\vartheta )$
on $(A,N,\mathrm e,\vartheta )$.

\begin{comp} 
The 
canonical homomorphism $B^{\tilde \ppsi}\longrightarrow B^{A,\sigmaQ}$ induced by the
action 
$\tilde \bbeta \colon B^{\tilde \ppsi}\longrightarrow 
\Aut^A(A, N,\mathrm e,\vartheta )$
of $B^{\tilde \ppsi}$ on the crossed product algebra
$(A, N,\mathrm e,\vartheta)$, cf. {\rm \eqref{tildegamma}} above,
and the 
surjection $B^{\tilde \ppsi}\longrightarrow Q$
is an isomorphism.
\end{comp}

\begin{proof}
The homomorphism $B^{\tilde \ppsi}\to B^{A,\sigmaQ}$
makes the diagram
\begin{equation*}
\xymatrix{
0\ar[r]          &\mathrm U(S)\ar@{=}[d]\ar[r]   &\Gamma \ar@{=}[d] \ar[r] 
                                            &B^{\tilde\ppsi}\ar[d]\ar[r]
                                                 &Q\ar@{=}[d] \ar[r] &1\\ 
0\ar[r] &\mathrm U(S)\ar[r] &
N^{\mathrm U(A, N,\mathrm e,\vartheta )}(A)  \ar[r]
&B^{A,\sigmaQ}\ar[r]&Q\ar[r]   &1\\ 
}
\end{equation*}
commutative
whence the
homomorphism $B^{\tilde \ppsi}\to B^{A,\sigmaQ}$
is an isomorphism. 
\end{proof}

\section{Behavior of the crossed Brauer group under $Q$-normal Galois extensions}
\label{behavior}

Consider
a $Q$-normal Galois extension $T|S$ of commutative rings, with structure 
extension 
$\mathrm {\mathrm e}_{(T|S)} \colon
N\rightarrowtail G \stackrel{\pi_Q}
\twoheadrightarrow Q$ 
and structure homomorphism
$\llambda \colon G \to \Aut^S(T)$, cf. Section \ref{normalr} above, and
denote the injection of $S$ into $T$ by $i\colon S \to T$.
Then the abelian group $\mathrm{XB}(T|S;G,Q)$ is defined relative to
the associated morphism $(i,\pi_Q)\colon (S,Q,\kappaQ) \to (T,G,\llambda)$
in the change of actions category $\mathcat{Change}$, cf. \eqref{mqng} above.

\begin{thm}
\label{9.2}The sequence
\begin{equation*}\mathrm{XB}(T|S;G,Q) \stackrel{t}{\longrightarrow} \mathrm H^3 (Q,\mathrm U(S)) \stackrel{\inf}{\longrightarrow}
\mathrm H^3 (G,\mathrm U(T))\end{equation*}
is exact and, furthermore, natural in the data. 
Moreover, each class in the image of $t$ is 
also the Teich\-m\"uller class of some crossed pair algebra.
\end{thm}

\begin{proof}
The naturality of the constructions entails that $\inf \circ\, t = 0$.
Moreover, by Theorem~\ref{normalcrossed}, 
$\ker (\inf) \subset \mathrm{im}(t)$, 
and each class in
the image of $t$ comes from some crossed pair algebra. 
\end{proof}

Let $\Pic (T|S)$ denote the kernel of the homomorphism 
$\Pic (S) \to \Pic(T)$ induced by 
\linebreak
${i \colon S \to T}$.
Our next aim is to construct
a homomorphism
from
$\mathrm H^1(Q, \Pic (T|S))$ to 
\linebreak
$\mathrm{XB}(T|S;G,Q)$.
To this end, view $T$ as an $S$-module in the obvious way and let 
$A = \End_S(T)$. Now, given an automorphism $\alpha $ of $A$ so that
$\alpha | S$ is the identity, as above we can turn $T$ into a new $A$-module
${}^\alpha T$ be means of $\alpha $, and 
$J(\alpha ) = \Hom_A({}^\alpha T, T) $ is a 
faithful
finitely generated 
projective rank one  $S$-module; since $A\otimes T$ is
a matrix algebra, $J(\alpha )$ represents a member of $\Pic(T|S)$, 
and the association
$\alpha \mapsto [J(\alpha )]$
yields a  homomorphism $\Aut(A|S) \to \Pic (T|S)$
which we claim to be surjective. In order to justify this claim,
we first observe that the obvious map $j \colon T^t N \to A$, as 
explained in 
\cref{one},
is an isomorphism, since $T|S$ is a Galois
extension of commutative rings
with Galois group $N$. 
Now, given a derivation $d \colon N \to \mathrm U(T)$, define the automorphism
$\alpha_d$ of $T^t N$ by
\[
\alpha_d(tn) = d(n)tn,\  t\in T,\, n\in N.
\]
Then
\begin{equation*}
\Der (N,\mathrm U(T)) \longrightarrow \Aut(T^t N|S), \ 
d\longmapsto \alpha_d,
\end{equation*}
is a homomorphism,
and $[J(\alpha_d)]\in \Pic (T|S)$ is the image of
$[d]\in \mathrm H^1(N, \mathrm U(T))$ under the standard isomorphism 
$\mathrm H^1(N,\mathrm U(T)) \to \Pic(T|S)$ 
(with $N$ and $T$ substituted for $Q$ and $S$, respectively,
this is, e.g., a consequence of 
the exactness of 
\eqref{twelvet}
at the second term).
Hence the homomorphism $\Aut(A|S) \to \Pic (T|S)$ is surjective as asserted.
Consequently the obvious homomorphism
from $\Aut(A|S)$ to $\Aut(A,Q)$ fits into
a commutative diagram
\begin{equation*}
\xymatrix{
\Aut(A|S)      \ar[d]\ar[r]   &\Pic(T|S) \ar[d] \ar[r]^{\cong} & \Out(A|S) 
\\
\Aut (A,Q)            \ar[r]  &\Out(A,Q)\\
}
\end{equation*}
where the horizontal maps are surjective.
Since 
the $G$-action on $T$ and that on $N$ induce a canonical section 
$\sigma_0 \colon Q \to \Out (A,Q)$,
the canonical homomorphism
$\Out (A,Q) \to Q$ is surjective as well.
Consequently
the
sequence
\begin{equation*}
0 \longrightarrow \Pic(T|S) \longrightarrow \Out (A,Q) \longrightarrow Q \longrightarrow 1
\end{equation*}
is exact.
Now, given a derivation 
$d \colon Q \to \Pic (T|S)$, 
define the homomorphism
\begin{equation*} \sigma_d \colon Q \longrightarrow \Out (A,Q)\end{equation*}
by $\sigma (q) = d(q) \sigma_0 (q)$, as $q$ ranges over $Q$. Then
$(A,\sigma_d)$ is a $Q$-normal Azumaya $S$-algebra. 

We mention without proof the following. 

\begin{thm}
\label{10.10} 
The association $d\mapsto (\End_S(T) , \sigma _d)$,
as $d$ ranges over derivations from $Q$ to $\Pic (T|S)$,
 yields a natural isomorphism
\begin{equation*}
\mathrm H^1(Q, \Pic (T|S)) \longrightarrow \XBSQ \cap \mathrm{XB}(T|S;G,Q)
\end{equation*}
of abelian groups in such a way that the resulting
sequence
\begin{equation}
0
\longrightarrow
\mathrm H^1(Q,\Pic (T|S)) \longrightarrow
\mathrm{XB}(T|S;G,Q) \longrightarrow
\mathrm H^0(Q,\mathrm B(T|S))
\label{nat12}
\end{equation}
is exact.
\end{thm}

\section{Relative theory and equivariant Brauer group}
\label{releqb}

Given a morphism $(f,\varphi ) \colon (S,Q,\kappa ) \to (T,G,\lambda )$
in the change of actions category 
\linebreak
$\mathcat{Change}$ introduced 
in Subsection \ref{coa},
we denote by $\mathrm{EB}(T|S; G,Q)$ the kernel of the
combined map
\begin{equation*}
\mathrm{EB}(S,Q) 
\longrightarrow \mathrm{XB}(S,Q) \longrightarrow \mathrm{XB}(T,G) ; 
\end{equation*}
this kernel 
$\mathrm{EB}(T|S; G,Q)$
is the subgroup of $\mathrm{EB}(S,Q)$ that consists of classes 
of $Q$-equivariant
$S$-algebras $(A,\tau )$ so that $(A\otimes T, \tau_{(f,\varphi)})$
is an induced $G$-normal split algebra
and hence, in view of \cref{eighto},
an induced $G$-equivariant split algebra; see 
\cref{2.5.1}(ii)
for the notation
$\tau_{(f,\varphi)}$.
Thus, in particular,
$\mathrm{EB}(S|S;Q,Q)$ is the 
kernel of the canonical homomorphism from
$\mathrm{EB}(S,Q)$ to $\mathrm{XB}(S,Q)$
 whereas
$\mathrm{EB}(S|S;\{e\},Q)$ is the kernel of the forgetful homomorphism 
from $\mathrm{EB}(S,Q)$ to $\mathrm{B}(S)$.
It is obvious that
the restriction homomorphism
$\res$ from $\mathrm{EB}(S,Q)$ to $\mathrm{XB}(S,Q)$ 
induces a homomorphism
$\res \colon \mathrm{EB}(T|S;G, Q) \longrightarrow \mathrm{XB}(T|S;G,Q)$.

Consider
a $Q$-normal Galois extension $T|S$ of commutative rings, with structure 
extension 
$\mathrm {\mathrm e}_{(T|S)} \colon
N\rightarrowtail G \stackrel{\pi_Q}
\twoheadrightarrow Q$
and structure homomorphism
$\llambda \colon G \to \Aut^S(T)$, cf. Section \ref{normalr} above, and
denote the injection of $S$ into $T$ by $i\colon S \to T$.
Then the abelian groups $\mathrm{EB}(T|S;G,Q)$ 
and $\mathrm{XB}(T|S;G,Q)$
are defined relative to
the morphism 
$
(i,\pi_Q)\colon (S,Q,\kappaQ) \longrightarrow (T,G,\llambda)
$
in the change of actions category $\mathcat{Change}$
associated with the data, cf. \eqref{mqng} above.

\begin{thm}
\label{11.3}
Suppose that $Q$ is a finite group.
Then
the sequence 
\begin{equation}
\label{11.4}
\mathrm{EB}(T|S;G,Q) \stackrel{\res}\to \mathrm{XB}(T|S;G,Q) 
\stackrel{t}\to \mathrm H^3(Q,\mathrm U(S)) \stackrel{\inf}\to \mathrm H^3(G, \mathrm U(T)) 
\end{equation}
is exact and natural.
\end{thm}

\begin{proof}
The statement of the theorem is a consequence of  
Theorems 
\ref{fouro},
\ref{nineo}~(ii),
\ref{elevo},
\ref{7.1},
 and \ref{9.2}.

For if $(A,\sigma )$ represents a member of $\mathrm{XB}(T|S;G,Q)$ 
with zero Teich\-m\"uller class, by 
\cref{fouro},  we may assume
$(A,\sigma )$ to be equivariant, i.e., $\sigma = \sigma_\tau $ for some
equivariant structure $\tau $. Now
the $G$-normal algebra
 $(A\otimes T, \sigma_{(i,\pi^{\substack{\mbox{\tiny{$G$}}}})})$
represents zero in $\mathrm{XB}(T,G)$ and hence is an induced $G$-normal split
algebra, by \cref{nineo} (ii).
By 
\cref{eighto},
$(A\otimes T, \tau_{(i,\pi^{\substack{\mbox{\tiny{$G$}}}})})$ is an induced
$G$-equivariant split algebra. \end{proof}

Let $R=S^Q$, let
$\mathrm e_G \colon \mathrm U(T) \stackrel{\iEG}
\rightarrowtail \EG 
\stackrel{\piEG}
\twoheadrightarrow G$
be a group extension, and 
denote the restriction to $N$ of the group extension ${\mathrm e}_ G$
by
${\mathrm e_N \colon 
\mathrm U(T) \rightarrowtail \EN 
\stackrel{\piEN} \twoheadrightarrow N}$. 
Then the crossed product 
$S$-algebra $A=(T,N,{\mathrm e}_N, \piEN)$ and
the crossed product $R$-algebra
$(T,G,{\mathrm e}_G,\llambda \circ \piEG)$ are defined,
the former being an Azumaya $S$-algebra,
since $T|S$ is a Galois extension of commutative rings with Galois 
group $N$ (cf. 
\cref{tpf}~(xi)),
and 
$(T,G,{\mathrm e}_ G,\llambda \circ \piEG)$ contains $A$ 
as a subalgebra.
Consider the resulting group extension
$\mathrm e_Q\colon \EN \stackrel{\jEQ}\rightarrowtail \EG 
\stackrel{\piEQ}\twoheadrightarrow Q$, of the kind 
\eqref{getw},
 and introduce
 the notation
$i^{\substack{\mbox{\tiny{$\EN$}}}}\colon \EN \to \mathrm U(A)$
for the obvious injection. 
Conjugation in $\EG$ induces an action 
$\thetaEG \colon \EG \to \Aut(A)$
of $\EG$ on $A$ such that
the pair 
$(i^{\substack{\mbox{\tiny{$\EN$}}}},\thetaEG)$
is a morphism 
$(\EN,\EG,\jEQ) \to (\mathrm U(A),\Aut(A),\partial)$
of crossed modules of the kind 
\eqref{morthr},
and this morphism, in turn, induces
a $Q$-normal structure 
${\sigma_{\thetaEG}
\colon Q \to \Out(A)}$
on $A$; thus
the crossed product $R$-algebra 
of $(T,G,\mathrm e_G,\llambda \circ \piEG)$
can now be written as the crossed product $R$-algebra
$(A,Q,\mathrm e_Q,\thetaEG)$ relative to the group extension
$\mathrm e_Q$
and the morphism 
$(i^{\substack{\mbox{\tiny{$\EN$}}}},\thetaEG)$
of crossed modules,
cf. 
\cref{three}.
In particular,
the left 
$A$-module  
$\BbB$
that underlies the algebra 
$(T,G,{\mathrm e}_ G,\llambda \circ \piEG)
\cong (A,Q,\mathrm e_Q,\vartheta_{\mathrm e_G})$
is free with basis in one-one 
correspondence with the elements of $Q$,
and the $Q$-equivariant structure 
$\tau_{\mathrm e_Q}\colon Q \longrightarrow \Aut({}_A\End (\BbB))$
given as 
\eqref{tauo}
is defined.
When the group $Q$ is finite, the algebra 
${}_A\End(\BbB)$ is an Azumaya $S$-algebra.

\begin{prop}
Suppose that the group $Q$ is finite. Then the assignment to a group extension
$\mathrm e_G$ of $G$ by $\mathrm U(T)$ of the $Q$-equivariant algebra 
$({}_A\End(\BbB)^{\mathrm{op}},\tau^{\mathrm{op}}_{\mathrm e_Q})$
yields a homomorphism
\begin{equation}
\label{13.3}
\mathrm{cpr}\colon \mathrm H^2(G,\mathrm U(T)) \longrightarrow 
\mathrm{EB}(T|S;G,Q) 
\end{equation}
of abelian groups that is natural on the change of actions category
$\mathcat{Change}$.
In the special case where $T=S$ and $N$ is the trivial group, 
the homomorphism {\rm \eqref{13.3}} comes essentially down to
{\rm \eqref{inftwty}},
viz.
\begin{equation}
\mathrm{cpr} \colon 
\mathrm H^2(Q,\mathrm U(S)) \longrightarrow \mathrm{EB}(S|S;Q,Q). \qed
\label{inf2}
\end{equation}
\end{prop}

\section{The eight term exact sequence}
\label{eightterme}

Given a morphism
 $(f,\varphi ) \colon (S,Q, \kappa ) \to (T,G,\lambda )$ 
in the change of actions category $\mathcat{Change}$ introduced in 
Subsection \ref{coa}, the group $Q$ being finite,
the corresponding relative version of the exact sequence 
\eqref{twelvet}
takes the following form:
\begin{equation}
\ldots \stackrel{\dDelta}\longrightarrow \mathrm H^2(Q,\mathrm U(S))
\stackrel{\mathrm{cpr}} \longrightarrow \mathrm{EB}(T|S;G,Q) 
\stackrel{\res}\longrightarrow
\mathrm{XB}(T|S;G,Q) 
\stackrel{t}\longrightarrow \mathrm H^3(Q,\mathrm U(S))
\label{12.31}
\end{equation}

\begin{rema}
In the special case where $T = S$ and $G$ is the trivial group,
in view of the isomorphism 
\eqref{tenthr})
from $\XBSQ$ onto $\mathrm H^1(Q,\Pic(S))$,
the exact sequence \eqref{12.31} has the form of
the C(hase-)R(osenberg-)A(us\-lander-)B(rumer) sequence
\cite[Theorem 7.6 p.~62]{MR0195923}, 
\cite{MR0191894}. 
Other versions of the CRAB-sequence were obtained by Childs 
\cite[Theorem 2.2]{MR0311701},
Fr\"ohlich and Wall \cite[Theorem 1]{MR0409424},  \cite{genbrauer},
\cite[Theorem 4.2]{MR1803361}
(upper and middle long sequence),
Hattori \cite{MR539592}, Kanzaki \cite{MR0241469},
Ulbrich \cite{MR558913}, Yokogawa \cite{MR0480466}, and 
Villamayor-Zelinski \cite{MR0460308}.
\end{rema}

Consider
a $Q$-normal Galois extension $T|S$ of commutative rings, with structure 
extension 
$\mathrm {\mathrm e}_{(T|S)} \colon
N\rightarrowtail G \stackrel{\pi_Q}
\twoheadrightarrow Q$
and structure homomorphism
$\llambda \colon G \to \Aut^S(T)$, cf. Section \ref{normalr} above, and
denote the injection of $S$ into $T$ by $i\colon S \to T$.
Then the abelian groups $\mathrm{EB}(T|S;G,Q)$ 
and $\mathrm{XB}(T|S;G,Q)$
are defined relative to
the morphism 
$
(i,\pi_Q)\colon (S,Q,\kappaQ) \longrightarrow (T,G,\llambda)
$
in the change of actions category $\mathcat{Change}$
associated with the data, cf. \eqref{mqng} above.

\begin{thm}
\label{12.4}
The group $Q$ being finite, 
the extension
\begin{equation}
\begin{aligned}
0\longrightarrow &\mathrm H^1(Q,\mathrm U(S)) 
\stackrel{\jp}\longrightarrow\mathrm{EPic} (S,Q) 
\stackrel{\mup}\longrightarrow (\Pic (S))^Q
\stackrel{\dDelta}\longrightarrow \mathrm H^2(Q,\mathrm U(S))
\\
\stackrel{\mathrm{cpr}} \longrightarrow
&\mathrm{EB}(T|S;G,Q) \stackrel{\res}\longrightarrow
\mathrm{XB}(T|S;G,Q) 
\stackrel{t}\longrightarrow \mathrm H^3(Q,\mathrm U(S))
\stackrel{\inf} \longrightarrow  \mathrm H^3(G,\mathrm U(T))
\end{aligned}
\label{AAAAA}
\end{equation}
of the exact sequence 
{\rm \eqref{ldes}}
is defined
and yields an eight term exact sequence that is natural in terms
of the data.
If, furthermore, $S|R$ and $T|R$ are Galois extensions of commutative rings
over $R=S^Q=T^G$,
with Galois groups $Q$ and $G$, respectively,  
 then,
with $\Pic (S|R)$,
$\Pic (R)$ and $\mathrm B(T|R)$ substituted for,
respectively 
$\mathrm H^1(Q,\mathrm U(S))$,
$\mathrm{EPic} (S,Q)$ and $\mathrm{EB}(S,Q)$,
where $R = S^Q$,
the homomorphisms {\rm cpr} and {\rm res} being
modified accordingly,
the sequence is exact as well.
\end{thm}

\begin{proof}
This is an immediate consequence of \cref{twelvetw} and
\cref{11.3}. 
\end{proof}

\begin{rema}
 In terms of  the notation $B_0(R;\Gamma)$
for the group that corresponds to our 
\linebreak
$\mathrm{EB}(S|S;Q,Q)$
(where our notation $Q$ and 
$S$
corresponds to $\Gamma$ and 
$R$, respectively),
a homomorphism of the kind \eqref{inf2} above is given in
\cite[Theorem 4.2]{MR1803361}.
After the statement of Theorem 4.2, the authors of \cite{MR1803361}
remark that there is no direct construction for the map from
$\mathrm H^2(\Gamma; \mathrm U(R))$ to $B_0(R;\Gamma)$.
Our construction of \eqref{inf2}
is direct, however.
\end{rema}

\begin{rema}
\label{hochserre}
In the special case where $T|S|R$ are ordinary Galois extensions of fields,
the exact sequence boils down to the classical low degree five 
term exact sequence
\begin{equation}
0
\to
\Ho^2(Q,\mathrm U(S))
\to
\Ho^2(G,\mathrm U(T))
\to
\Ho^2(N,\mathrm U(T))^Q
\to
\Ho^3(Q,\mathrm U(S))
\to
\Ho^3(G,\mathrm U(T)),
\label{AAAAAA}
\end{equation}
see \cite[p.~130]{MR0052438}.
\end{rema}

\section{Relationship with the eight term exact sequence in
the cohomology of a group extension}
\label{13}

Let $T|S$ be a $Q$-normal Galois extension of commutative rings, 
 with structure extension
$\mathrm {\mathrm e}_{(T|S)} \colon
N\rightarrowtail G \stackrel{\pi_Q}\twoheadrightarrow Q$
and structure homomorphism 
$\kappa_G\colon G \to \Aut^S(T)$, 
cf. Section \ref{normalr} 
above; in particular, $N$ is a finite group.
Since $\mathrm U(T)^N$  
coincides with $\mathrm U(S)$,
the eight term exact sequence 
in
\cite{MR597986} associated with the group extension
$\mathrm e_{(T|S)}$ and the $G$-module $\mathrm U(T)$,
reproduced as \eqref{13.11} above,
has the following form:
\begin{equation}
\begin{aligned}
0\longrightarrow &\mathrm H^1(Q,\mathrm U(S)) 
\stackrel{\inf}\longrightarrow
\mathrm H^1(G,\mathrm U(T)) 
 \stackrel{\res}\longrightarrow 
\mathrm H^1(N,\mathrm U(T))^Q
\stackrel{\Delta}\longrightarrow \mathrm H^2(Q,\mathrm U(S))
\\ 
\stackrel{\inf}\longrightarrow &\mathrm H^2(G,\mathrm U(T)) \stackrel{j}\longrightarrow
\mathrm{Xpext} (G,N;\mathrm U(T))
 \stackrel{\Delta }\longrightarrow \mathrm H^3(Q,\mathrm U(S)) \stackrel{\inf}\longrightarrow
\mathrm H^3(G,\mathrm U(T)).
\end{aligned}
\label{BBBBB}
\end{equation}

\subsection{Relationship between  the two long exact sequences}

Consider the morphism
$(i,\pi_Q )\colon(S,Q,\kappaQ) \longrightarrow (T,G,\llambda)$
 associated to the
given $Q$-normal Galois extension, cf. \ref{mqng}, 
in the change of actions category $\mathcat{Change}$ introduced in 
Subsection \ref{coa}.
The  abelian groups
$\mathrm{EB}(T|S;G,Q)$ and
$\mathrm{XB}(T|S;G,Q)$ are now defined relative to 
this morphism.

The assignment to a crossed pair  $
\left (\mathrm e \colon  \mathrm U(T)\rightarrowtail \Ggamma \twoheadrightarrow N,\ 
\ppsi\colon Q \to \Out_G(\mathrm e)\right)$
with respect
to ${\mathrm e}_ {(T|S)}$ and $\mathrm U(T)$
of its associated crossed pair algebra 
$(A_{\mathrm e},\sigma_\ppsi )$,
cf. Section~\ref{cpal} above,  
yields a homomorphism
\begin{equation}
\mathrm{cpa}\colon
\mathrm{Xpext} (G,N;\mathrm U(T)) \longrightarrow \mathrm{XB}(T|S;G,Q). 
\label{13.2}
\end{equation}
Let $\mathrm{EPic}(T|S,Q)$  denote the kernel
of the induced homomorphisms 
\[
\mathrm{EPic} (S,Q)
\stackrel{\mumu_{\substack{\mbox{\tiny{$\mathcat{Pic}_{S,Q}$}}}}} 
\longrightarrow \Pic(S)\stackrel{i_*} \longrightarrow \Pic(T) 
\]
and 
$\Pic(T|S)$ that of the induced homomorphism 
$i_*\colon \Pic(S) \to \Pic (T)$.
With $T$ and $G$ substituted for $S$ and $Q$, respectively,
the isomorphism 
\eqref{picf}
takes the form
\begin{equation}
j_{\substack{\mbox{\tiny{$\mathcat{Pic}_{T,G}$}}}}\colon
\mathrm H^1(G,\mathrm U(T)) \longrightarrow \mathrm{EPic} (T|T,G), 
\label{form1}
\end{equation}
and Galois descent, cf. Subsection \ref{galext}~(ii), yields an isomorphism 
\[
\mathrm{EPic} (T|S,Q) \longrightarrow
\mathrm{EPic} (T|T,G)
\]
whence \eqref{form1} induces a homomorphism
\begin{equation}
\label{13.4}
\mathrm H^1(G,\mathrm U(T)) \longrightarrow\mathrm{EPic} (T|S,Q) 
\end{equation}
of abelian groups.
The homomorphism \eqref{13.4} admits, of course, a straightforward
direct description.
Likewise,
with $T$ and $N$ substituted for $Q$ and $S$, respectively,
the isomorphism 
\eqref{picf}
takes the form
\begin{equation}
j_{\substack{\mbox{\tiny{$\mathcat{Pic}_{T,N}$}}}}\colon
\mathrm H^1(N,\mathrm U(T)) \longrightarrow \mathrm{EPic} (T|T,N), 
\label{form2}
\end{equation}
and Galois descent yields an isomorphism 
$\mathrm{Pic} (T|S) \to
\mathrm{EPic} (T|T,N)$
whence \eqref{form2} induces an isomorphism
\begin{equation}
\label{13.5}
\mathrm H^1(N,\mathrm U(T)) \longrightarrow \Pic (T|S) 
\end{equation}
of abelian groups, necessarily compatible with the $Q$-module structures;
the isomorphism \eqref{13.5} is entirely classical.
Below we do not distinguish in notation
between \eqref{13.4}
and its composite $\mathrm H^1(G,\mathrm U(T)) \to \mathrm{EPic}(T,Q)$
with the canonical injection of
$\mathrm{EPic}(T|S,Q)$ into $\mathrm{EPic}(T,Q)$, nor
between \eqref{13.5}
and its composite $\mathrm H^1(N,\mathrm U(T)) \to \Pic(S)$
with the canonical injection
$\Pic(T|S) \to \Pic(S)$.
Direct inspection establishes the following.

\begin{thm}
\label{eighttermcomp}
The group  $Q$ being finite, the homomorphisms 
{\rm \eqref{13.4}}, {\rm \eqref{13.5}}, 
{\rm \eqref{13.3}}, and {\rm \eqref{13.2}}
of abelian groups are natural on the category $\mathcat{Change}$
and induce a morphism of exact sequences 
from
{\rm \eqref{BBBBB}} to  
{\rm \eqref{AAAAA}}. \qed
\end{thm}

\begin{rema}
Consider the classical case where $R$, $S$, and $T$ are fields.
Now the group 
$\mathrm{Xpext} (G,N;\mathrm U(T))$ 
comes down to $\mathrm H^2(N,\mathrm U(T))^Q$ and
$\mathrm{XB}(T|S;G,Q)$ to $\mathrm B(T|S)^Q$, and
\eqref{13.2} boils down to the classical isomorphism
$\mathrm H^2(N,\mathrm U(T))^Q \to \mathrm B(T|S)^Q$.
Furthermore, the groups 
$\mathrm H^1(N,\mathrm U(T))$,
$\mathrm H^1(G,\mathrm U(T))$, 
$\mathrm{EPic}(T|S,Q)$, and
$\mathrm{Pic}(T|S)$ are zero, and \eqref{13.3}
is an isomorphism.
Thus the morphism {\rm \eqref{BBBBB}} $\to $ 
{\rm \eqref{AAAAA}}
of exact sequences in Theorem \ref{eighttermcomp} above
is then an isomorphism of exact sequences.
\end{rema}

\subsection{An application}
\label{speccase}
Let $T|S$ be a Galois extension of commutative rings, with Galois group $N$,
suppose that $T$ carries a $Q$-action that extends the given $Q$-action
on $S$, and define
the group $\mathrm {EB}(T|S,Q)$ to be the kernel of the induced 
homomorphism  $\mathrm {EB}(S,Q) \to \mathrm {EB}(T,Q)$.
Relative to the induced $Q$-action on $N$, the 
semi-direct product group $N \rtimes Q$ is defined, and
$T|S$ is a $Q$-normal Galois extension of rings,
having as structure extension the split extension
$\mathrm {\mathrm e}_{(T|S)} \colon
N\rightarrowtail N \rtimes Q \twoheadrightarrow Q$.
Consider the commutative diagram
\begin{equation}
\begin{CD}
@.
0
@.
0
@.
0
\\
@.
@VVV
@VVV
@VVV
\\
0
@>>>
\widetilde{\mathrm{EB}(T|S,Q)}
@>>>
\mathrm{EB}(T|S;N \rtimes Q,Q)
@>>>
\mathrm{EB}(T|T;Q,Q)
\\
@.
@VVV
@VVV
@VVV
\\
0
@>>>
\mathrm{EB}(T|S,Q)
@>>>
\mathrm{EB}(S,Q) 
@>>> 
\mathrm{EB}(T,Q) 
\\
@.
@VVV
@VVV
@VVV
\\
0
@>>>
\mathrm{XB}(T;Q,N \rtimes Q) 
@>>>
\mathrm{XB}(T,N \rtimes Q) 
@>>>
\mathrm{XB}(T,Q) 
\end{CD}
\label{appl1}
\end{equation}
of abelian groups with exact rows and columns,
the 
subgroup 
$\widetilde{\mathrm{EB}(T|S,Q)}$ of $\mathrm{EB}(T|S,Q)$
being defined
by the requirement that the upper row be exact.

The group $N$ being finite, suppose now that $Q$ is a finite group as well.
The  corresponding homomorphism
\eqref{13.3}, viz.
$\mathrm{cpr}\colon \mathrm H^2(N \rtimes Q,\mathrm U(T)) \to \mathrm {EB}(T|S;N \rtimes Q,Q)$,
and the
homomorphism \eqref{inf2}, with $T$ substituted for $S$, viz.
$\mathrm{cpr}\colon \mathrm H^2(Q,\mathrm U(T)) \to \mathrm {EB}(T|T;Q,Q)$,
yield
the commutative diagram
\begin{equation}
\begin{CD}
0
@>>>
\mathrm{ker}(\mathrm{res})
@>>>
\mathrm H^2(N \rtimes Q,\mathrm U(T))
@>{\mathrm{res}}>>
\mathrm H^2(Q,\mathrm U(T)) 
\\
@.
@VVV
@V{\mathrm{cpr}}VV
@V{\mathrm{cpr}}VV
\\
0
@>>>
\widetilde{\mathrm{EB}(T|S,Q)}
@>>>
\mathrm{EB}(T|S;N \rtimes Q,Q)
@>>>
\mathrm{EB}(T|T;Q,Q)  
\end{CD}
\label{appl2}
\end{equation}
with exact rows and hence a homomorphism
$\mathrm{ker}(\mathrm{res}) \to \mathrm{EB}(T|S,Q)$
of abelian groups.
Suppose, furthermore, that $S$ and $T$ are fields.
Then the homomorphism from
$\mathrm{XB}(T,N \rtimes Q)$
to
$\mathrm{XB}(T,Q)$
in the lower row of the diagram \eqref{appl1}
comes down to the obvious injection
$\mathrm B(T)^{N \rtimes Q} \to \mathrm B(T)^Q$
whence the group $\mathrm{XB}(T;Q,N \rtimes Q)$ is now trivial
and the inclusion
$\widetilde{\mathrm{EB}(T|S,Q)}\subseteq \mathrm{EB}(T|S,Q)$
is the identity. Moreover,
the right-hand and the middle vertical arrow 
in \eqref{appl2}
are isomorphisms
whence the induced homomorphism
$\mathrm{ker}(\mathrm{res}) \to \mathrm{EB}(T|S,Q)$
is an isomorphism. This observation
recovers and casts new light on the main result of
\cite{MR2191717}, obtained there via relative group cohomology.
Our argument is elementary and does not invoke relative group cohomology.
Indeed, the main point of our reasoning is the identification
of the group cohomology group
$\mathrm H^2(N \rtimes Q,\mathrm U(T))$
with the group $\mathrm{EB}(T|S;N \rtimes Q,Q)$;
under the present circumstances, this group is the subgroup
of the $Q$-equivariant Brauer group $\mathrm {EB}(S,Q)$ of $S$ 
that consists of
classes of $Q$-equivariant central simple $S$-algebras $A$
such that $A\otimes T$ is a matrix algebra over $T$.
Likewise,
the group $\mathrm{EB}(T|T;Q,Q)$
is the subgroup
of the $Q$-equivariant Brauer group $\mathrm {EB}(T,Q)$ of $T$ 
that consists of
classes of $Q$-equivariant matrix algebras over $T$.
The group 
$\mathrm{EB}(T|S,Q)$ then appears as the kernel of the canonical 
homomorphism 
$\mathrm{EB}(T|S;N \rtimes Q,Q)
\to
\mathrm{EB}(T|T;Q,Q)$ and,
in view of the identifications of
$\mathrm H^2(N \rtimes Q,\mathrm U(T))$
with $\mathrm{EB}(T|S;N \rtimes Q,Q)$
and of
$\mathrm H^2(Q,\mathrm U(T))$
with $\mathrm{EB}(T|T;Q,Q)$,
the identification of
$\mathrm{ker}(\mathrm{res}\colon \mathrm H^2(N \rtimes Q,\mathrm U(T))
\to \mathrm H^2(Q,\mathrm U(T)))$
with $\mathrm{EB}(T|S,Q)$ is immediate.
In particular,  
when the group $Q$ is trivial,
that result comes down to the classical Brauer-Hasse-Noether isomorphism
between the corresponding second group cohomology group
and the corresponding subgroup of the ordinary Brauer group.

\subsection{A variant of the relative theory}
\label{variantrt}

In the situation of the relative versions
\eqref{12.31} and \eqref{AAAAA}
of the long exact sequence  
\eqref{twelvet},
in general, there is no obvious reason for 
a homomorphism
$\ome$ from $\mathrm H^0(Q,\mathrm B(T|S))$ to 
\linebreak
$\mathrm H^2(Q,\Pic(T|S))$
to exist 
that would complete
\begin{equation*}
\xymatrix{
\mathrm H^0(Q,\mathrm B(T|S)) \ar[d] 
&\mathrm H^2(Q,\Pic(T|S))  \ar[d] 
\\ 
\mathrm H^0(Q,\mathrm B(S))\ar[r]^{\ome}
&\mathrm H^2(Q,\Pic(S))
}
\end{equation*}
to a commutative square
and hence would 
complete the exact sequence \eqref{nat12}
to a corresponding relative version
of an exact sequence of the kind 
\eqref{FW}.
We  now show that a variant of the relative theory
includes such a homomorphism.

The object $(S,Q,\kappaQ)$ of the category
$\mathcat{Change}$ being given,
let $(T,G,\llambda)$ be another object of $\mathcat{Change}$,
and let $(f,\varphi)\colon (S,Q,\kappaQ) \to (T,G,\llambda)$ be a 
morphism in  $\mathcat{Change}$ having $\varphi\colon G \to Q$ 
surjective, cf. 
Subsection \ref{coa}.

\subsubsection{The standard approach}
\label{standa}

We  say that
two $Q$-normal Azumaya $S$-algebras $(A_1,\sigma_1)$ and 
 $(A_2,\sigma_2)$ such that $T \otimes A_1$
and  $T \otimes A_2$ are matrix algebras over $T$
are {\em relatively Brauer equivalent\/}
if there are faithful finitely
generated projective  $S$-modules
modules $M_1$ and $M_2$ 
having the property that
 $T \otimes M_1$ and $T \otimes M_2$
are free as $T$-modules,
together with induced $Q$-normal structures
\[
\rho_1\colon Q \to \Out(B_1),
\
B_1 = \End_S(M_1),
\
\rho_2\colon Q \to \Out(B_2),
\
B_2 = \End_S(M_2),
\]
such that 
$(A_1 \otimes B_1, \sigma_1 \otimes \rho_1 )$
and $(A_2 \otimes B_2, \sigma_2 \otimes \rho_2)$
are isomorphic $Q$-normal {$S$-algebras}.
Just as for $\mathrm{XB}(S,Q)$,
under the operations of
tensor product and that of taking opposite algebras,
the equivalence classes constitute an abelian
group,
the identity element being represented by
$(S,\kappaQ)$.
We refer to this group
as the $T$-{\em relative $Q$-crossed Brauer group of\/} $S$ 
{\em with respect to the 
morphism\/} $(f,\varphi)$ in $\mathcat{Change}$,
denote this group 
by $\mathrm{XB}_{\mathrm{fr}}(T|S;G,Q)$,
and we refer to the construction just given as the 
{\em standard construction\/}.
The $T$-{\em relative 
$Q$-equivariant Brauer group\/} $\mathrm {EB}_{\mathrm{fr}}(T|S;G,Q)$ 
{\em with respect to the 
morphism\/} $(f,\varphi)$ in $\mathcat{Change}$
arises 
in the same way as the relative $Q$-crossed Brauer group,
save that, in the definition, \lq equivariant\rq\ is substituted for
\lq crossed\rq, and we  likewise say
that this construction is the {\em standard construction\/}.
In particular, when we forget the actions, that is, we take
the groups $G$ and
$Q$ to be trivial, this construction
yields an abelian group $\mathrm B_{\mathrm{fr}}(T|S)$ which we refer to as the
$T$-{\em relative Brauer group of\/} $S$,
obtained by the {\em standard construction\/}.

The group $\mathrm B_{\mathrm{fr}}(T|S)$ acquires a $Q$-module structure. 
Indeed, let $R=S^Q$.
Given an $S$-module $M$ 
and $x\in Q$,
let ${}^xM$ denote the $S$-module
whose underlying $R$-module is just $M$, and whose $S$-module structure
is given by
\[
S \otimes M \longrightarrow M,\ (s\otimes q)\longmapsto {}^x\!s\, q,
\ s \in S,\ q \in M.
\]
Consider a faithful finitely generated projective $S$-module
$M$ such that $T \otimes M$ is a free $T$-module, let $x \in Q$, and 
pick a pre-image $y \in G$ of $x\in Q$.
Then the
association
\begin{equation}
T \otimes {}^xM \longrightarrow {}^y(T \otimes M),\ 
t \otimes q \longmapsto {}^y t \otimes q,
\label{assoc1}
\end{equation}
yields an isomorphism of $T$-modules, and
since $T\otimes M$
 is a free $T$-module, so is
${}^y(T \otimes M)$; further,
\[
T \otimes {}^x\End_S(M) \cong {}^y(T \otimes \End_S(M))
\cong {}^y(T \otimes \End_S(M))\cong\End_S ({}^y(T \otimes M))
\]
is a matrix algebra over $T$.
Likewise,
given an
Azumaya $S$-algebra $A$ such that $T\otimes A$ is a matrix algebra 
over $T$
and $x\in Q$, to
 show that $T \otimes {}^xA$ is a matrix algebra over $T$,
pick a pre-image $y \in G$ of $x \in Q$ and note that
the corresponding association 
\eqref{assoc1}
yields an isomorphism of $T$-algebras. Since
$T \otimes A$ is a matrix algebra over $T$, so is  
${}^y (T \otimes A)$.

By construction,
the canonical homomorphism
\begin{equation*}
\mathrm B_{\mathrm{fr}}(T|S) \longrightarrow \mathrm B(T|S)
\end{equation*}
is a morphism 
of $Q$-modules but in general there is no reason for this homomorphism
to be injective nor to be surjective.
The assignment to a $Q$-equivariant Azumaya $S$-algebra 
representing a member of
$\mathrm {EB}_{\mathrm{fr}}(T|S;G,Q)$ 
of the associated $Q$-normal Azumaya $S$-algebra
yields a homomorphism 
$\mathrm{res}_{\mathrm{fr}}\colon \mathrm {EB}_{\mathrm{fr}}(T|S;G,Q)\to \mathrm {XB}_{\mathrm{fr}}(T|S;G,Q)$
of abelian groups,
the assignment to a $Q$-normal Azumaya $S$-algebra $(A, \sigma )$
representing a member of 
$\mathrm {XB}_{\mathrm{fr}}(T|S;G,Q)$ of its Teichm\"uller complex
 $\mathrm e_{(A,  \sigma )}$
yields a homomorphism
$t_{\mathrm{fr}}\colon
\mathrm {XB}_{\mathrm{fr}}(T|S;G,Q)
\longrightarrow \mathrm H^3(Q,\mathrm U(S))
$ of abelian groups and,
when the group $Q$ is finite,
the construction of the homomorphism 
$\mathrm{cpr}\colon \mathrm H^2(Q, \mathrm U(S))\to \mathrm {EB}(T|S;G,Q)$,
cf. \eqref{inf2} above,
lifts to a homomorphism
\[
\mathrm{cpr}_{\mathrm{fr}}\colon \mathrm H^2(Q, \mathrm U(S))\longrightarrow 
\mathrm {EB}_{\mathrm{fr}}(T|S;G,Q).
\]

\begin{rema}
The abelian groups $\mathrm {EB}(T|S;G,Q)$ and $\mathrm {XB}(T|S;G,Q)$
being defined relative to the given morphism $(f,\varphi)$ in
$\mathcat{Change}$,
the obvious maps yield homomorphisms
\begin{align}
\mathrm {EB}_{\mathrm{fr}}(T|S;G,Q) &\longrightarrow 
\mathrm {EB}(T|S;G,Q)
\label{can52}
\\ 
\mathrm {XB}_{\mathrm{fr}}(T|S;G,Q) &\longrightarrow \mathrm {XB}(T|S;G,Q) 
\label{can53}
\end{align}
of abelian groups
that
make the diagram
\begin{equation*}
\begin{CD}
\mathrm {EB}_{\mathrm{fr}}(T|S;G,Q)
@>{\res_{\mathrm{fr}}}>>
\mathrm {XB}_{\mathrm{fr}}(T|S;G,Q)
@>{t_{\mathrm{fr}}}>>
\mathrm H^3(Q,\mathrm U(S))
\\
@VVV
@VVV
@|
\\
\mathrm {EB}(T|S;G,Q)
@>{\res}>>
\mathrm {XB}(T|S;G,Q)
@>{t}>>
\mathrm H^3(Q,\mathrm U(S))
\end{CD}
\end{equation*}
commutative and, when the group $Q$ is finite,
the homomorphisms
$\mathrm{cpr}_{\mathrm{fr}}$ from $\mathrm H^2(Q,\mathrm U(S))$
to
$\mathrm {EB}_{\mathrm{fr}}(T|S;G,Q)
$ and
$\mathrm{cpr}$ from $\mathrm H^2(Q,\mathrm U(S))$
to 
$\mathrm {EB}(T|S;G,Q)
$
extend the diagram to a larger commutative diagram
having four terms in each row.
However, there is no reason for the homomorphisms
\eqref{can52} or \eqref{can53}
to be injective nor to be surjective,
nor is there a reason, when $Q$ is a finite group, for 
$\mathrm{cpr}_{\mathrm{fr}}\colon \mathrm H^2(Q, \mathrm U(S))\to \mathrm {EB}_{\mathrm{fr}}(T|S;G,Q)$ to be injective or surjective. In the classical situation
where $R$, $S$, $T$ are fields etc.,
these homomorphisms are, of course, isomorphisms.
\end{rema}

Let $\mathrm {Pic}(T|S)$
denote the kernel of the homomorphism
$\mathrm {Pic}(S) \to \mathrm {Pic}(T)$
induced by the ring homomorphism $f\colon S \to T$,
necessarily a morphism of $G$-modules when $G$ acts on
$S$ through $\varphi\colon G \to Q$ whence, in particular,
the abelian subgroup  $\mathrm {Pic}(T|S)^Q$
of $Q$-invariants is defined, and let
$\mathrm {EPic}(T|S,Q)$
denote the kernel of the homomorphism
$\mathrm {EPic}(S,Q) \to \mathrm {EPic}(T,G)$
induced by the morphism $(f,\varphi)$ in $\mathcat{Change}$.
It is immediate that the low degree exact sequence 
\eqref{FW}
restricts to the exact sequence
\begin{equation}
0 \longrightarrow \mathrm H^1(Q, \mathrm U (S)) 
\stackrel{\jp}
\longrightarrow 
\mathrm{EPic}(T|S,Q)
\stackrel{\mup|}
\longrightarrow \Pic(T|S)^Q
\stackrel{\dDelta|}
\longrightarrow \mathrm H^2(Q, \mathrm U (S)) 
\label{ldesf}
\end{equation}
of abelian groups.
In the appendix (cf. Subsection \ref{esgsmc} below), we shall show that,
with a suitably defined Picard category
$\mathcat {Pic}_{T|S;G,Q}$ substituted for $\mathcat C_Q$,
 the sequence
\eqref{ldesf}
is as well a special case of the exact sequence 
\eqref{fittwelve}.

\begin{thm}
\label{CTC}
Suppose that the group $Q$ is finite. Then
the extension 
\begin{equation}
\ldots
\stackrel{\dDelta}\longrightarrow \mathrm H^2(Q,\mathrm U(S))
\stackrel{\mathrm{cpr}_{\mathrm{fr}}} \longrightarrow
\mathrm{EB}_{\mathrm{fr}}(T|S;G,Q) \stackrel{\res_{\mathrm{fr}}}\longrightarrow
\mathrm{XB}_{\mathrm{fr}}(T|S;G,Q) 
\stackrel{t_{\mathrm{fr}}}\longrightarrow \mathrm H^3(Q,\mathrm U(S)) 
\label{CCCC}
\end{equation}
of the exact sequence
{\rm \eqref{ldesf}}
is defined and yields a seven term  exact sequence
that is natural in terms of the data.
\end{thm}

\begin{proof} 
Essentially the same reasoning as that for 
\cref{twelvetw}
establishes this theorem as well. We  explain only the requisite salient 
modifications.

\noindent
{\em Exactness at\/} $\mathrm{XB}_{\mathrm{fr}}(T|S;G,Q)$: This follows again 
from 
\cref{elevo}
or \cref{fouro}.

\noindent{\em Exactness at\/} $\mathrm H^2(Q,\mathrm U(S))$:
Let $J$ represent a class in $(\Pic (T|S))^Q$, 
and proceed as in the proof of 
the exactness at $\mathrm H^2(Q,\mathrm U(S))$ in
\cref{twelvetw}.
Now $T \otimes J$ is free as a $T$-module and, 
with reference to the associated group extension $\mathrm e_J$, 
cf. \eqref{eJ},
by construction,
$M_{\mathrm e_J}$ is free as an $S$-module whence
$T \otimes M_{\mathrm e_J}$ is free as a $T$-module.
Hence 
\[
T \otimes \Hom_S(J,M_{\mathrm e_J})\cong 
\Hom_T(T \otimes J,T\otimes M_{\mathrm e_J})
\]
is free as a $T$-module.
Consequently $(\End_S(M_{\mathrm e_J}),\tau_{\mathrm e_J})$
represents zero in $\mathrm{EB}_{\mathrm{fr}}(T|S;G,Q)$.

Conversely, let
$\mathrm e \colon \mathrm U(S) \rightarrowtail \Ggamma 
\twoheadrightarrow Q$
be a group extension, 
and proceed as in the proof of 
the exactness at $\mathrm H^2(Q,\mathrm U(S))$ in
\cref{twelvetw}.
Thus suppose that $(\End_S(M_{\mathrm e}), \tau_{\mathrm e})$ represents
zero in $\mathrm{EB}_{\mathrm{fr}}(T|S;G,Q)$. 
Then there are 
$S^tQ$-modules
 $M_1$ and $M_2$
whose underlying $S$-modules
are faithful and finitely generated projective 
such that the following hold, where we denote by
${\tau_1\colon Q \to \Aut(\End_S(M_1))}$
and
${\tau_2\colon Q \to \Aut(\End_S(M_2))}$
 the associated trivially induced
$Q$-equivariant structures:
The algebras
$
(\End_S(M_{\mathrm e}),\tau_{\mathrm e})\otimes 
(\End_S(M_1),\tau_1)
$
and
$(\End_S(M_2),\tau_2) $
are isomorphic 
as $Q$-equivariant $S$-algebras and, furthermore, the $T$-modules
 $T\otimes M_1$ and $T\otimes M_2$ are free as $T$-modules.
Consequently the 
$T$-module $T \otimes S$ arising
from the finitely generated and projective  rank one $S$-module
 $J 
= \Hom_{\End_S(M_{\mathrm e}\otimes M_1)}(M_{\mathrm e}\otimes M_1,M_2)$
is free of rank one
whence
$[J]\in \Pic (T|S)$.
The group extension  $\mathrm e_J$, cf. 
\eqref{eJ},
is now defined relative to $J$,  
whence $[J]\in (\Pic (T|S))^Q$,
and the $\Ggamma$-action on $J$
induces a homomorphism $\Ggamma \to \Aut(J,Q)$
which yields
a congruence
$(1,\cdot,1) \colon {\mathrm e} \to \mathrm e_J $ of group extensions,
and this congruence
entails that
$\dDelta [J] = [\mathrm e]\in \mathrm H^2(Q,\mathrm U(S))$.

\noindent{\em Exactness at $\mathrm{EB}_{\mathrm{fr}}(T|S;G,Q)$\/}:
The reasoning in the proof of \cref{twelvetw}
which shows that
the composite $\res \circ \,\mathrm{cpr}$ is
zero shows as well that the composite 
$\res_{\mathrm{fr}} \circ \,\mathrm{cpr}_{\mathrm{fr}}$ is zero.

To show  that 
$\ker (\res_{\mathrm{fr}}) \subset \im (\mathrm{cpr}_{\mathrm{fr}})$,
let $(A,\tau )$ 
be a $Q$-equivariant Azumaya $S$-algebra
representing a member of $\mathrm{EB}_{\mathrm{fr}}(T|S;G,Q)$, 
and suppose that
the class 
of its associated $Q$-normal algebra
$(A,\sigma_\tau)$
goes to zero in $\mathrm{XB}_{\mathrm{fr}}(T|S;G,Q)$. 
As in the proof of 
the exactness at $\mathrm{EB}(S,Q)$ in
\cref{twelvetw},
there are two induced $Q$-equivariant split algebras
$(\End_S(M_1),\tau_1)$ and $(\End_S(M_2),\tau_2)$
over faithful finitely generated projective $S$-modules
$M_1$ and $M_2$, respectively,
such that 
$(A,\tau) \otimes (\End_S(M_1),\tau_1)$ and $(\End_S(M_2),\tau_2)$
are isomorphic as $Q$-equivariant central $S$-algebras 
but now we may furthermore take 
$M_1$ and $M_2$ to have the property that
the $T$-modules $T \otimes M_1$ and $T \otimes M_2$ are free of finite rank.
Essentially the same reasoning 
as that in the proof of
the exactness at $\mathrm{EB}(S,Q)$ in
\cref{twelvetw}
yields a group extension
$\mathrm e\colon \mathrm U(S)\rightarrowtail \Ggamma \twoheadrightarrow Q$
such that
\[
{\mathrm{cpr}_{\mathrm{fr}}([\mathrm e]) = [(\End_S(M_{\mathrm e}),
\tau_{\mathrm e})]= [(A,\tau)]  \in 
\mathrm {EB}_{\mathrm{fr}}(T|S;G,Q)}. \qedhere
\]
\end{proof}

Consider a
$Q$-normal Galois
extension  $T|S$ of commutative rings, with structure extension
$\mathrm {\mathrm e}_{(T|S)} \colon
N\rightarrowtail G \stackrel{\pi_Q}\twoheadrightarrow Q$
and structure homomorphism $\kappa_G\colon G \to \Aut^S(T)$, cf. Section 
{\rm \ref{normalr}}  above, and
take the morphism $(f,\varphi)$ to be
the morphism
$(i,\pi_Q)\colon(S,Q,\kappaQ) \longrightarrow (T,G,\llambda)$
in $\mathcat{Change}$ associated to that
$Q$-normal Galois extension, cf. \eqref{mqng}.

\begin{thm}
\label{14.44}
Suppose that the group $Q$ is finite. Then the extension
\begin{equation}
\begin{aligned}
0\longrightarrow 
&\mathrm H^1(Q,\mathrm U(S)) 
\stackrel{j_{\substack{\mbox{\tiny{$\mathcat{Pic}_{S,Q}$}}}}|}
\longrightarrow
\mathrm{EPic} (T|S,Q) 
\stackrel{\omu_{\substack{\mbox{\tiny{$\mathcat{Pic}_{S,Q}$}}}}|}
\longrightarrow 
(\Pic (T|S))^Q
\stackrel{\ome_{\substack{\mbox{\tiny{$\mathcat{Pic}_{S,Q}$}}}}|}
\longrightarrow 
\mathrm H^2(Q,\mathrm U(S))
\\
\stackrel{\mathrm{cpr}_{\mathrm{fr}}} \longrightarrow
&\mathrm{EB}_{\mathrm{fr}}(T|S;G,Q) 
\stackrel{\res_{\mathrm{fr}}}\longrightarrow
\mathrm{XB}_{\mathrm{fr}}(T|S;G,Q) 
\stackrel{t_{\mathrm{fr}}}\longrightarrow \mathrm H^3(Q,\mathrm U(S))
\stackrel{\inf} \longrightarrow  \mathrm H^3(G,\mathrm U(T))
\end{aligned}
\label{CCCCC}
\end{equation}
of the exact sequence {\rm \eqref{ldesf}}
is defined and yields an eight term exact sequence that is natural in terms of
the data.
\end{thm}

\begin{proof} Essentially the same reasoning as that for Theorem \ref{12.4}
establishes this theorem as well. We leave the details to the reader.
\end{proof}

The homomorphism \eqref{13.2} now lifts to a homomorphism
\begin{equation}
\label{14.2}
\mathrm{Xpext} (G,N;\mathrm U(T)) \longrightarrow \mathrm{XB}_{\mathrm{fr}}(T|S;G,Q) 
\end{equation}
such that  \eqref{13.2} may be written as the composite
\begin{equation}
\mathrm{Xpext} (G,N;\mathrm U(T)) \longrightarrow \mathrm{XB}_{\mathrm{fr}}(T|S;G,Q)
\longrightarrow  \mathrm{XB}(T|S;G,Q)
\end{equation}
and, when $Q$ and hence $G$ is a finite group, 
the homomorphism \eqref{13.3}
lifts to a homomorphism
\begin{equation}
\label{14.3}
\mathrm H^2(G,\mathrm U(T)) \longrightarrow 
\mathrm{EB}_{\mathrm{fr}}(T|S;G,Q) 
\end{equation}
such that
 \eqref{13.3} may be written as the composite
\begin{equation*}
\mathrm H^2(G,\mathrm U(T)) \longrightarrow 
\mathrm{EB}_{\mathrm{fr}}(T|S;G,Q) 
\longrightarrow
\mathrm{EB}(T|S;G,Q).
\end{equation*}
Theorem \ref{eighttermcomp}, adjusted to the present circumstances,
takes the following form which, again, we spell out
without proof.

\begin{thm}
\label{CTCC}
The group  $Q$ being finite, the maps {\rm \eqref{14.2}},
{\rm \eqref{13.4}}, {\rm \eqref{13.5}},
and {\rm \eqref{14.3}}
are natural homomorphisms of abelian groups 
and induce a morphism {\rm \eqref{BBBBB}} $\to $ 
{\rm \eqref{CCCCC}}
of exact sequences.
\end{thm}

\subsubsection{The Morita equivalence approach}
\label{moritaa}

We define the  $Q$-{\em graded relative Brauer precategory
associated with the morphism\/} 
$(f,\varphi)$ in $\mathcat {Change}$
to be
the precategory $\mathpzc {PreB}_{T|S;G,Q}$ that has as its
 {\em objects\/} the
Azumaya $S$-algebras $A$ such that
$T \otimes A$ is a matrix algebra over $T$, 
 a {\em morphism\/} $([M],x)\colon A \to B$ in 
$\mathpzc {PreB}_{T|S;G,Q}$ {\em of grade\/} 
$x\in Q$ between two 
Azumaya algebras $A$ and $B$ in $\mathpzc B_{T|S;G,Q}$,
necessarily an isomorphism in $\mathpzc {PreB}_{T|S;G,Q}$, 
being a morphism in $\mathcat B_{S,Q}$, that is,
a pair
$([M],x)$ where $[M]$ is an isomorphism class of an invertible
$(B,A)$-bimodule $M$ of grade  $x\in Q$,
such that, furthermore, $T \otimes M$ is free as a $T$-module.
There is no reason for composition in the ambient category
 $\mathpzc B_{S,Q}$ to induce an operation of composition in
 $\mathpzc {PreB}_{T|S;G,Q}$ since,
given three Azumaya algebras $A$, $B$, $C$ 
in $\mathcat {PreB}_{T|S;G,Q}$
and morphisms
 $([{}_BM_A],x)\colon A \to B$
and $([{}_AM_C],x)\colon C\to A$ of grade $x \in Q$
in $\mathpzc {PreB}_{T|S;G,Q}$, while 
the composite 
$([{}_BM_A\otimes_A {}_AM_C],x)\colon C\to B$  
of grade $x\in Q$ in $\mathpzc B_{S,Q}$
is defined,
there is no reason for the
$(T \otimes B,T\otimes C)$-bimodule 
\[
T \otimes ({}_BM_A\otimes_A {}_AM_C) \cong
{}_{T\otimes B}(T \otimes M)_{T \otimes A}
\otimes_{(T\otimes A)} 
{}_{T \otimes A}(T\otimes M)_{T\otimes C}
\]
to be free as a $T$-module.
To overcome this difficulty,
we take
the  $Q$-{\em graded relative Brauer category
associated with the morphism\/} 
$(f,\varphi)$ in $\mathcat {Change}$
to be
the subcategory $\mathpzc B_{T|S;G,Q}$ 
of $\mathpzc B_{S,Q}$ 
generated by
 $\mathpzc {PreB}_{T|S;G,Q}$.
Thus a morphism  
in $\mathpzc B_{T|S;G,Q}$ 
of grade $x \in Q$ 
between two objects $A$ and $B$ of 
$\mathpzc B_{T|S;G,Q}$ is
a morphism  $([{}_BM_A],x)\colon A \to B$ 
in $\mathpzc B_{S,Q}$  
of grade $x \in Q$
such that there are objects
$A_1$,\ldots, $A_n$ of
$\mathpzc B_{T|S;G,Q}$ 
and morphisms 
$([{}_{A_{j+1}}M_{A_j}],x)\colon A_j \to A_{j+1}$ 
in $\mathpzc {PreB}_{T|S;G,Q}$
such that, when we write $A$ as $A_0$ and $B$ as $A_n$,
\begin{equation}
{}_BM_A \cong 
{}_{A_n}M_{A_{n-1}} \otimes_{A_{n-1}}
\ldots
\otimes_{A_2}
{}_{A_2}M_{A_1} \otimes_{A_1} {}_{A_1}M_{A_0} .
\label{chain}
\end{equation}
We then define composition, monoidal structure,
the operation of inverse, and the unit object
as in $\mathpzc B_{S,Q}$. 
The resulting category $\mathpzc B_{T|S;G,Q}$
is
a group-like stably $Q$-graded
symmetric monoidal category.
 Hence
 the category
$\mathpzc {Rep}(Q, \mathpzc B_{T|S;G,Q})$
is group-like and thence
\linebreak
$k\mathpzc {Rep}(Q, \mathpzc B_{T|S;G,Q})$
is an abelian group.
When the groups $G$ and $Q$ are trivial, 
that is, we consider merely the homomorphism $f \colon S \to T$
of commutative rings, the same construction 
yields a precategory $\mathcat {PreB}_{T|S}$ and, accordingly,
the corresponding group-like symmetric monoidal category
$\mathpzc B_{T|S}$ 
which we refer to as the
{\em relative Brauer category associated with the homomorphism\/}
$f\colon S \to T$ of commutative rings.
The category $\mathcat B_{T|S}$ has
$\mathrm U(\mathpzc{B}_{T|S})  
=\Pic(T|S)$ as its unit group,
is group-like,
and $k\mathcat B_{T|S}$ is therefore an abelian group.
The ring homomorphism $f\colon S \to T$ being a constituent of the morphism
$(f,\varphi)$ in $\mathcat{Change}$ having $\varphi$ surjective,
the category $\mathpzc{B}_{T|S;G,Q}$
has $\mathcat{Ker}(\mathpzc B_{T|S;G,Q})=\mathpzc B_{T|S}$
and $\mathrm U(\mathpzc{B}_{T|S;G,Q}) =\mathrm U(\mathpzc{B}_{T|S})  
=\Pic(T|S)$ as its unit group.

Given two objects $A$ and $B$ of 
$\mathpzc B_{T|S;G,Q}$
we {\em define,
with respect to the morphism $(f,\varphi)$
in\/} $\mathcat{Change}$,
 a {\em relative Morita equivalence 
of grade $x \in Q$
between $A$ and $B$}
to be a string of isomorphisms in
$\mathpzc {PreB}_{T|S;G,Q}$ 
of the kind \eqref{chain} above.
It is immediate that, as in the classical situation,
given two objects $A_1$ and $A_2$ of $\mathpzc B_{T|S}$,
a relative Brauer equivalence
\[
A_1\otimes \End_S(M_1) \cong A_2\otimes \End_S(M_2)
\] 
between  $A_1$ and $A_2$
induces a string 
\[
A_1\simeq A_1\otimes \End_S(M_1) \cong A_2\otimes \End_S(M_2)
\simeq A_2
\]
of isomorphisms in
$\mathpzc {PreB}_{T|S}$ 
and hence
a relative Morita equivalence between $A_1$ and $A_2$
(of grade $e\in Q$)
whence
the obvious association induces a homomorphism
\begin{equation}
\mathrm B_{\mathrm{fr}}(T|S) \longrightarrow k\mathpzc B_{T|S}
\label{mfr}
\end{equation}
of abelian groups, necessarily surjective.
Moreover, since
$\mathcat{Ker}(\mathpzc B_{T|S;G,Q})$
is stably graded, 
$k\mathpzc B_{T|S}=k\mathcat{Ker}(\mathpzc B_{T|S;G,Q})$ 
acquires a $Q$-module structure,
and the homomorphism \eqref{mfr} is
a morphism of $Q$-modules.

\begin{prop}
The homomorphism {\rm \eqref{mfr}} is an isomorphism, that is,
relative Brauer equivalence
is equivalent to relative Morita equivalence.
\end{prop}

\begin{proof} The classical argument, suitably rephrased, carries over:
Let $A$ and $B$ be  two 
Azumaya $S$-algebras $A$ in $\mathpzc B_{T|S}$ and
consider 
 a {\em morphism\/} $[M]\colon A \to B$ in 
$\mathpzc {PreB}_{T|S}$.
We must show that $A$ and $B$ are relatively Brauer equivalent.
Now $B^{\mathrm{op}}\cong {}_A\End(M)$ (the algebra of
left $A$-endomorphisms of $M$), and
\[
\End_S(M)\cong A \otimes ({}_A\End(M)) \cong A \otimes B^{\mathrm{op}}
\]
whence
\[
\End_S(M)\otimes B 
\cong A \otimes B^{\mathrm{op}}\otimes B 
\cong 
A \otimes \End_S(B).
\]
Since $T \otimes M$ and $T\otimes B$ are free as
$T$-modules,  $A$ and $B$ are relatively Brauer equivalent.
\end{proof}

With $N$, $T$, $S$ substituted for, respectively,
$Q$, $S$, $R$, 
the standard homomorphism
\eqref{ordinarycp}
 from $\mathrm H^2(N,U)$ to $\mathrm B(T|S)$,
necessarily a morphism of $Q$-modules,
lifts to a morphism
\begin{equation}
\mathrm H^2(N,U)) \longrightarrow \mathrm B_{\mathrm{fr}}(T|S)
\label{standardQ}
\end{equation}
of $Q$-modules.
By construction, then, 
the assignment to an automorphism in $\mathcat B_{T|S;G,Q}$
of an  Azumaya algebra $A$ in $\mathcat B_{T|S;G,Q}$ of its grade in $Q$
yields a homomorphism 
\begin{equation*}
\pi^{\substack{\mbox{\tiny{$\Aut_{\mathcat B_{T|S;G,Q}}(A)$}}}}
\colon \Aut_{\mathcat B_{T|S;G,Q}}(A) \longrightarrow Q
\end{equation*}
which
is surjective if and only if the Brauer class
$[A]\in \mathrm B_{\mathrm{fr}}(T|S)$ of $A$ in $\mathrm B_{\mathrm{fr}}(T|S)\cong k\mathpzc B_{T|S}$ 
is fixed under $Q$, and
the group $\Aut_{\mathcat B_{T|S;G,Q}}(A)$ associated to 
an Azumaya $S$-algebra 
$A$ in $\mathcat B_{T|S}$ 
whose Brauer class $[A]\in \mathrm B_{\mathrm{fr}}(T|S)$ 
is fixed under $Q$ fits into a group 
extension of the kind \eqref{eUCC},
viz.
\begin{equation}
\mathrm e^{\substack{\mbox{\tiny{$\Pic(T|S)$}}}}_A\colon
1
\longrightarrow
\Pic (T|S)
\longrightarrow
\Aut_{\mathcat B_{T|S;G,Q}}(A)
\stackrel{\pi^{\substack{\mbox{\tiny{$\Aut_{\mathcat B_{T|S;G,Q}}(A)$}}}}}
\longrightarrow
Q
\longrightarrow
1,
\label{eApicf}
\end{equation}
with abelian kernel in such a way that the assignment to $A$ of 
$\mathrm e^{\substack{\mbox{\tiny{$\Pic(T|S)$}}}}_A$
yields a homomorphism 
\begin{equation}
\ome_{\mathcat B_{T|S;G,Q}}\colon \mathrm H^0(Q,\mathrm B_{\mathrm{fr}}(T|S)) \longrightarrow \mathrm H^2(Q,\Pic(T|S)).
\label{d2picf}
\end{equation}
The sequence \eqref{fittwelve}
 now takes the  form
\begin{equation}
\scalefont{0.8}{
0 \longrightarrow \mathrm H^1(Q, \Pic (T|S)) 
\stackrel{j_{\mathpzc B_{T|S;G,Q}}}\longrightarrow 
k\mathcat{Rep}(Q,\mathcat B_{T|S;G,Q})
\stackrel{\omu_{\mathpzc B_{T|S;G,Q}}}\longrightarrow \mathrm B_{\mathrm{fr}}(T|S)^Q
\stackrel{\ome_{\mathpzc B_{T|S;G,Q}}}
\longrightarrow \mathrm H^2(Q, \Pic (T|S)) 
}
\label{FW2}
\end{equation}
and 
is an exact sequence of abelian groups 
since the category $\mathcat B_{T|S;G,Q}$ is group-like. 
Furthermore, the association that defines the homomorphism
\eqref{theta}
yields an injective homomorphism
\begin{equation}
\theta_{\mathrm{fr}}\colon
\mathrm {XB}_{\mathrm{fr}}(T|S;G,Q) \longrightarrow
k\mathcat{Rep}(Q,\mathcat B_{T|S;G,Q})
\label{thetaf}
\end{equation} 
in such a way that the diagram
\begin{equation*}
\begin{CD}
\mathrm {XB}_{\mathrm{fr}}(T|S;G,Q) 
@>{\theta_{\mathrm{fr}}}>>
k\mathcat{Rep}(Q,\mathcat B_{T|S;G,Q})
\\
@VVV
@VVV
\\
\mathrm {XB}(S,Q) 
@>{\theta}>>
k\mathcat{Rep}(Q,\mathcat B_{S,Q})
\end{CD}
\end{equation*} 
is commutative, the unlabeled vertical arrows being the obvious maps,
and the argument for 
\cref{tentwentytw}~(iii)],
adjusted to the present situation, shows that
if $Q$ (and hence $G$) is a finite group,
the homomorphism 
$\theta_{\mathrm{fr}}$ is surjective
and hence an isomorphism of abelian groups.
Thus when the group $Q$ is finite,
the exact sequence \eqref{FW2}
is available with 
$\mathrm {XB}_{\mathrm{fr}}(T|S;G,Q)$ 
substituted for
$k\mathcat{Rep}(Q,\mathcat B_{T|S;G,Q})$.

Consider a $Q$-normal Galois
extension  $T|S$ of commutative rings, with structure extension
$\mathrm {\mathrm e}_{(T|S)} \colon
N\rightarrowtail G \stackrel{\pi_Q}\twoheadrightarrow Q$
and structure homomorphism $\kappa_G\colon G \to \Aut^S(T)$, cf. Section 
{\rm \ref{normalr}}  above, and
take the morphism $(f,\varphi)$ to be
the morphism
$(i,\pi_Q)\colon(S,Q,\kappaQ) \longrightarrow (T,G,\llambda)$
in $\mathcat{Change}$ associated to that
$Q$-normal Galois extension, cf. \ref{mqng}.
Comparison of the exact sequences \eqref{FW}
and \eqref{FW2} with
\cite[(1.9)]{MR597986} yields the following result, which
we spell out without proof.

\begin{thm}
\label{withoutp}
Write $U=\mathrm U(T)$.
The various groups and homomorphisms fit into a commutative diagram
\begin{equation*}
\scalefont{0.8}{
\xymatrix{
0 \ar[r] 
&\mathrm H^1(Q,\mathrm H^1(N,U)) \ar[d]^{\cong}  \ar[r]
&\mathrm{Xpext}(G,N;U) \ar[d]  \ar[r]
&\mathrm H^0(Q,\mathrm H^2(N,U))\ar[d]  \ar[r]^{d_2} 
&\mathrm H^2(Q,\mathrm H^1(N,U)) \ar[d]^{\cong}
\\ 
0 \ar[r] 
&\mathrm H^1(Q,\Pic (T|S))  \ar[d]\ar[r]^{j}  
&k\mathcat{Rep}(Q,\mathcat B_{T|S;G,Q}) \ar[d] \ar[r]^{\mumu} 
&\mathrm H^0(Q, \mathrm B_{\mathrm{fr}}(T|S)) \ar[d] \ar[r]^{\ome}
&\mathrm H^2(Q,\Pic(T|S))  \ar[d] 
\\ 
0 \ar[r] 
&\mathrm H^1(Q,\Pic (S)) \ar[r]^{j}  
&k\mathcat{Rep}(Q,\mathcat B_{S,Q}) \ar[r]^{\mumu} 
&\mathrm H^0(Q,\mathrm B(S))\ar[r]^{\ome}
&\mathrm H^2(Q,\Pic(S))
}
}
\end{equation*}
with exact rows;
here the top row is the exact sequence
{\rm \cite[(1.9)]{MR597986}}, the middle row the sequence 
{\rm \eqref{FW2}},
 the bottom row  the exact sequence
{\rm \eqref{FW}},
the unlabeled arrow from
$\mathrm H^0(Q,\mathrm H^2(N,U))$ to $\mathrm H^0(Q,\mathrm B_{\mathrm{fr}}(T|S))$
is induced by the 
homomorphism
\eqref{standardQ}, 
and the other unlabeled arrows are either the obvious ones
or have been introduced before.
If, furthermore, the group $Q$ is a finite group, the above diagram
is available with 
$\mathrm {XB}_{\mathrm{fr}}(T|S;G,Q)$
substituted for
$k \mathcat{Rep}(Q, \mathcat B_{T|S;G,Q})$
and 
$\mathrm {XB}(S,Q)$ for
$k \mathcat{Rep}(Q, \mathcat B_{S,Q})$.
\end{thm}

\begin{rema}
The exact sequences \eqref{CCCCC}
and \eqref{FW2} 
are presumably related with an equivariant 
Amitsur cohomology 
spectral sequence of the kind
given in \cite[Sections 1 and 2]{MR0311701} and
\cite[Theorem 7.3 p.~61]{MR0195923} 
in the same way as the exact sequences \eqref{13.11}
and \cite[(1.9)]{MR597986} are related with the spectral sequence
associated with a group extension and a module over the extension group,
cf. also \cite{MR641328}.
\end{rema}

\section{Appendix}

As a service to the reader,
we recollect some more material
from the theory of stably graded
symmetric monoidal categories \cite{genbrauer},
\cite{MR0349804},
\cite{MR1803361}
and use it to illustrate some of the constructions in the present paper.

Recall that an
{\em $S$-progenerator\/} is
a faithful finitely generated projective $S$-module.
Given two $Q$-equivariant Azumaya $S$-algebras
$(A,\tau_A)$ and $(B,\tau_B)$,
a $(B,A,Q)$-{\em bimodule\/} $(M,\tau_M)$ 
is a $(B,A)$-{\em bimodule\/} $M$
together with
an $S^tQ$-module structure $\tau_M\colon Q \to \Aut(M)$ 
which is compatible with the 
$Q$-equivariant structures $\tau_A\colon Q \to \Aut(A)$ and 
$\tau_B\colon Q \to \Aut(B)$
in the sense that
\begin{equation}
{}^x(bya) ={}^x \!  b\, {}^x \! y \, {}^x \! a,\ x \in Q,
\, a \in A,\, b \in B.
\end{equation}

The object $(S,Q,\kappaQ)$ of the category
$\mathcat{Change}$ being given,
let $(T,G,\llambda)$ be another object of $\mathcat{Change}$,
and let $(f,\varphi)\colon (S,Q,\kappaQ) \to (T,G,\llambda)$ be a 
morphism in  $\mathcat{Change}$ having $\varphi\colon G \to Q$ 
surjective, cf. 
Subsection \ref{coa} above

\subsection{Examples of symmetric monoidal categories}
\label{ordinary}

\noindent
--- $\mathpzc{Mod}_S$: the category of $S$-modules, a symmetric monoidal 
   category under the operation of tensor product, with $S$ as unit object,
and $\mathrm U(\mathpzc{Mod}_S)=\mathrm U(S)$;

\noindent
--- $\mathpzc{Gen}_S$ \cite[\S 2 p.~17]{genbrauer}, \cite[p.~229]{MR0349804}, 
\cite[\S 2]{MR1803361}: the symmetric monoidal
subcategory of $\mathpzc{Mod}_S$, necessarily a groupoid,
whose objects are the {\em $S$-progenerators\/}, with
morphisms only the invertible ones,
having $S$ as its unit object
and $\mathrm U(\mathpzc{Gen}_S)=\mathrm U(S)$ as its unit group;

\noindent
---  $\mathpzc{Pic}_S$: the symmetric monoidal subcategory of
$\mathpzc{Gen}_S$, necessarily group-like,
of invertible modules,
written in \cite[\S 2 p.~17]{genbrauer}, \cite[\S 2]{MR1803361} as 
$\mathpzc{C}_R$, reproduced in 
Subsection \ref{piccat} above;

\noindent
---  $\mathpzc{Az}_S$: the symmetric monoidal subcategory of
$\mathpzc{Gen}_S$, necessarily a groupoid, 
having
the Azumaya $S$-algebras as objects,  
invertible algebra morphisms between Azumaya $S$-algebras
as morphisms, the ground ring
$S$ as its unit object, and unit group $\mathrm U(\mathpzc{Az}_S)$ trivial
\cite[\S 2 p.~18]{genbrauer}, \cite[p.~229]{MR0349804}, \cite[\S 2]{MR1803361};

\noindent
---  $\mathpzc{XAz}_S$: the quotient category of
$\mathpzc{Az}_S$, necessarily a groupoid, 
having the same objects
as $\mathpzc{Az}_S$, and having as  morphisms  
$A \to B$ between two objects $A$ and $B$
equivalence classes
of morphisms $h \colon A \to B$ 
in $\mathpzc{Az}_S$ under the equivalence relation
$h_1 \sim h_2 \colon A \to B$ if $h_1=h_2 \circ I_a$ for some
$a \in \mathrm U(A)$
\cite[\S 5 p.~43]{genbrauer},  
\cite[\S 2]{MR1803361}, where the notation $I_a$ refers to the inner 
automorphism of $A$ induced by $a\in \mathrm U(A)$; this category
has $S$ as its unit object, and its unit group
$\mathrm U(\mathpzc{XAz}_S)$ is trivial;

\noindent
--- 
$\mathpzc B_S$,  the {\em Brauer category\/}
of the commutative ring $S$, reproduced in 
Subsection \ref{bcacr} above;

\noindent
---  with respect to the ring homomorphism 
$f\colon S \to T$, with
the obvious interpretations,
the relative categories $\mathcat{Mod}_{T|S}$,
 $\mathcat{Gen}_{T|S}$,  $\mathcat{Pic}_{T|S}$,
 $\mathcat{Az}_{T|S}$,
 $\mathcat{XAz}_{T|S}$, taken as full subcategories of, respectively,
$\mathcat{Mod}_{S}$,
 $\mathcat{Gen}_{S}$,  $\mathcat{Pic}_{S}$,
 $\mathcat{Az}_{S}$,
 $\mathcat{XAz}_{S}$;

\noindent
--- $\mathcat{B}_{T|S}$, with respect to the ring homomorphism 
$f\colon S \to T$, the {\em relative Brauer category\/}, 
introduced in Subsection \ref{moritaa} above;

\noindent
--- 
$\mathpzc{EB}_{S,Q}$,
the equivariant Brauer category
$\mathpzc{EB}_{S,Q}$ of $S$ relative to the given
action 
of $Q$ on $S$,
written 
 as  $\mathpzc{B}(R,\Gamma)$
in \cite[\S 5 p.~41]{genbrauer} 
and \cite[\S 3]{MR1803361};
 its objects are the
$Q$-equivariant Azumaya algebras $(A,\tau)$; 
a {\em morphism\/} ${[(M,\tau_M)]\colon (A,\tau_A) \to (B,\tau_B)}$ 
in $\mathpzc{EB}_{S,Q}$
between two 
given 
$Q$-equivariant Azumaya algebras $(A,\tau_A)$ and 
 $(B,\tau_B)$, necessarily an isomorphism in $\mathpzc{EB}_{S,Q}$,  
is an isomorphism class $[(M,\tau_M)]$
of a $(B,A,Q)$-bimodule $(M,\tau_M\colon Q \to \Aut(M))$ 
whose underlying $(B,A)$-bimodule $M$
is invertible; the operations of tensor product and that of assigning to
a $Q$-equivariant Azumaya $S$-algebra its opposite algebra 
(as a  $Q$-equivariant Azumaya $S$-algebra)
turn
$\mathpzc{EB}_{S,Q}$ into a group-like 
symmetric monoidal category having
$(S,\kappaQ)$ is its unit object and
$\mathrm U(\mathpzc{EB}_{S,Q}) = \mathrm {EPic}(S)$ 
as its unit group
\cite[\S 3]{MR1803361},
 \cite[Proposition 3.1]{MR1803361}.

\noindent
--- $\mathpzc{EB}_{T|S;G,Q}$,
the {\em relative equivariant Brauer category
associated with the morphism\/} 
$(f,\varphi)$ in $\mathcat {Change}$; it
has as its objects the
$Q$-equivariant Azumaya algebras $(A,\tau)$
such that the  
$G$-equivariant Azumaya algebra
$(T\otimes A, \tau^G)$ that arises by scalar extension
has its underlying central $T$-algebra
$T\otimes A$ isomorphic to a matrix algebra;
given two 
$Q$-equivariant Azumaya algebras $(A,\tau_A)$ and 
 $(B,\tau_B)$ in $\mathpzc{EB}_{S,Q}$,  
a {\em morphism\/} $(A,\tau_A) \to (B,\tau_B)$ 
in the associated precategory
$\mathpzc{PreEB}_{T|S;G,Q}$, necessarily an isomorphism
in $\mathpzc{EB}_{T|S;G,Q}$,
is a morphism
$[M,\tau_M]\colon (A,\tau_A) \to (B,\tau_B)$
in $\mathpzc{EB}_{S,Q}$,  
that is, an isomorphism class of a
$(B,A,Q)$-bimodule $(M,\tau_M\colon Q \to \Aut(M))$ 
whose underlying $(B,A)$-bimodule $M$
is invertible, such that, furthermore, 
the resulting $T^{t}G$-module $T \otimes M$
is free as a $T$-module.
We then take $\mathpzc{EB}_{T|S;G,Q}$ to be the resulting
subcategory of  $\mathpzc{EB}_{S,Q}$
generated by $\mathpzc{PreEB}_{T|S;G,Q}$, that is, we define
morphisms and
composition of morphisms as finite strings in
 $\mathpzc{EB}_{S,Q}$, of morphisms in 
 $\mathpzc{PreEB}_{T|S;G,Q}$, and we define the
monoidal structure, the operation
of inverse, and the unit object as in $\mathpzc{EB}_{S,Q}$.
The resulting category 
$\mathpzc{EB}_{T|S;G,Q}$ is a group-like symmetric monoidal category
and
has
$\mathrm U(\mathpzc{EB}_{T|S;G,Q}) = \mathrm {EPic}(T|S)$.

\subsection{Examples of stably $Q$-graded symmetric monoidal categories}
\label{esgsmc}

\noindent
---  $\mathpzc{Mod}_{S,Q}$, 
a stably  $Q$-graded symmetric monoidal category
that arises from  $\mathpzc{Mod}_S$ 
as follows: Given two $S$-modules $M$ and $N$,
a {\em morphism $M\to N$ of\/} $S$-{\em modules of grade\/} 
$x \in Q$ 
is  a pair $(\varphi,x)$ having $\varphi \colon M \to N$
a morphism over $R=S^Q$ such that 
$\varphi (sy)=({}^x\!s)y$ ($s\in S$, $y \in M$)
\cite[p.~229]{MR0349804}, \cite[\S 2]{MR1803361}.

Enhancing each of the categories $\mathpzc C =\mathpzc{Gen}_S, 
\mathpzc{Pic}_S,
\mathpzc{Az}_S, \mathpzc{XAz}_S, \mathpzc{B}_S$ 
in Subsection \ref{ordinary} above to a stably $Q$-graded 
symmetric monoidal
category $\mathpzc C_Q$ in the same was as enhancing
the category
$\mathpzc{Mod}_S$  of $S$-modules to the stably $Q$-graded symmetric monoidal
category 
$\mathpzc{Mod}_{S,Q}$ just explained
 yields the following stably $Q$-graded 
symmetric monoidal
categories:

\noindent
---  $\mathpzc{Gen}_{S,Q}$,  
written in \cite{MR1803361} as
$\mathpzc{Gen}_{R}$;

\noindent
---  $\mathpzc{Pic}_{S,Q}$,  
written in \cite[\S 3]{MR1803361} as
$\mathpzc{C}_{R}$,  reproduced in 
Subsection \ref{piccat} above;

\noindent
---  $\mathpzc{Az}_{S,Q}$,  
written in \cite[\S 2]{MR1803361} as
$\mathpzc{Az}_{R}$;

\noindent
---  $\mathpzc{XAz}_{S,Q}$,  
written in \cite[\S 5 p.~43]{genbrauer} as
$Q-\widetilde{\mathpzc{Az}_R}$
and in \cite[\S 2]{MR1803361} as $Q\mathpzc{Az}_R$
(beware: the notation $Q$ in [op. cit.] has nothing to do with our notation $Q$
for a group, and
the  tilde-notation in \cite[\S 5 p.~43]{genbrauer}
refers to the additional structure of a twisting
and need not concern us here);
morphisms are now enhanced
via the $Q$-grading, that is to say,
a morphism $([h],x)\colon A \to B$ in $\mathpzc{Az}_{S,Q}$ 
of grade $x\in Q$ has $[h]$ an equivalence class of an isomorphism
$h\colon A \to B$  of algebras over
$R=S^Q$ such that $(h,x)$ is, furthermore, a morphism
in $\mathpzc{Mod}_{S,Q}$ of grade  $x\in Q$;

\noindent
---  $\mathpzc{B}_{S,Q}$,  the stably $Q$-graded Brauer category
associated with the commutative ring $S$ and the $Q$-action 
$\kappaQ \colon Q \to \Aut(S)$ on $S$,
reproduced
in 
Subsection \ref{sqgbc} above.

The morphism
${(f,\varphi)\colon (S,Q,\kappaQ) \to (T,G,\llambda)}$ 
in  $\mathcat{Change}$ having $\varphi$ surjective being given,
similarly to the construction
of the category
$\mathpzc B_{T|S;G,Q}$ in Subsection \ref{variantrt} above,
for each of the stably $Q$-graded symmetric monoidal categories 
$\mathpzc C_{S,Q} = 
\mathpzc {Mod}_{S,Q}, 
\mathpzc {Gen}_{S,Q},
\mathpzc {Pic}_{S,Q} 
$,
the stably $Q$-graded symmetric monoidal category
$\mathpzc C_{T|S;G,Q}$ 
is the  subcategory that
arises from the ambient category 
$\mathpzc C_{S,Q}$
in essentially the same way as
$\mathpzc B_{T|S;G,Q}$
arises from the ambient category 
$\mathpzc B_{S,Q}$ save that
there is no need to pass through a corresponding precategory:
The objects of $\mathpzc C_{T|S;G,Q}$ are those objects $C$ 
of $\mathpzc C_{S,Q}$
having the property that $T \otimes C$ is free as a $T$-module, and
$\mathpzc C_{T|S;G,Q} = 
\mathpzc {Mod}_{T|S;G,Q}$, 
$\mathpzc {Gen}_{T|S;G,Q}$,
$\mathpzc {Pic}_{T|S;G,Q}
$
is the respective full subcategory of $\mathpzc C_{S,Q}$.
Likewise,
for  the stably $Q$-graded symmetric monoidal categories 
$\mathpzc C_{S,Q} = 
\mathpzc {Az}_{S,Q}$
and  
$\mathpzc C_{S,Q} =\mathpzc {XAz}_{S,Q}$,
the stably $Q$-graded symmetric monoidal category 
$\mathpzc C_{T|S;G,Q}$ 
arises as the subcategory that
has
as its objects  
Azumaya $S$-algebras $A$ such that  $T\otimes A$
is a matrix algebra over $T$, and
$ \mathpzc {Az}_{T|S;G,Q}$
 is the corresponding 
full subcategory of 
$\mathpzc {Az}_{S,Q}$
and
$\mathpzc {XAz}_{T|S;G,Q}
$ 
that of $\mathpzc {XAz}_{S,Q}$.
Now, with $\mathpzc {Pic}_{T|S;G,Q}$ substituted for $\mathcat C_{S,Q}$,
the exact sequence
\eqref{fittwelve}
yields  the exact sequence \eqref{ldesf}.

\begin{rema}
For an object of $\mathpzc {Gen}_{S,Q}$, that is, for a 
faithful
finitely generated projective $S$-module $M$,
the group $\Aut(M,Q)$ introduced in 
\cref{eight}
is canonically isomorphic to the group
$\Aut_{\mathpzc {Gen}_{S,Q}}(M)$.
\end{rema}

\subsection{The standard constructions revisited}

The endomorphism functor
$\mathpzc {End}\colon  \mathpzc{Gen}_S \to \mathpzc{Az}_S$
induces an exact sequence 
\begin{equation}
0
\longrightarrow
\Pic(S)
\longrightarrow
k\mathpzc{Gen}_S \stackrel{\End}
\longrightarrow k\mathpzc{Az}_S
\longrightarrow 
\mathrm B(S)
\longrightarrow 0
\end{equation}
of abelian monoids
\cite[\S 5 p.~38]{genbrauer},
\cite[Introduction]{MR0349804}, 
\cite[\S 3]{MR1803361}.
This yields 
$\Pic(S)$
as the maximal subgroup of the abelian monoid
$k\mathpzc{Gen}_S$ and recovers the {\em standard construction\/} of
$\mathrm B(S)$, cf. 
Subsection \ref{disc},
as the cokernel of the homomorphism 
$\End$ of abelian monoids, the cokernel of a morphism of monoids being 
suitably interpreted (in terms of the associated equivalence relation
and \lq\lq cofinality\rq\rq, cf. \cite[\S 12]{MR0349804}).
The obvious functor $\Omega\colon \mathpzc{Az}_S \to
\mathpzc{B}_S$ induces the isomorphism
$\mathrm B(S) \to k\mathpzc{B}_S$ of abelian groups
\cite[\S 5 p.~38]{genbrauer},
\cite[\S 3, Theorem 3.2 (i)]{MR1803361} quoted in 
Subsection \ref{disc}.

Likewise, the endomorphism functor
$\mathpzc {End}\colon  \mathpzc{Gen}_{S,Q} \to \mathpzc{Az}_{S,Q}$
induces an exact sequence 
\begin{equation*}
0
\longrightarrow
\mathrm {EPic}(S,Q)
\longrightarrow
k\mathcat{Rep}(Q,\mathpzc{Gen}_{S,Q}) \stackrel{\End}
\longrightarrow 
k\mathcat{Rep}(Q,\mathpzc{Az}_{S,Q})
\longrightarrow 
\mathrm {EB}(S,Q)
\longrightarrow 0
\end{equation*}
of abelian monoids
\cite[\S 5 p.~38]{genbrauer},
\cite[Introduction]{MR0349804}, 
\cite[\S 3]{MR1803361}.
This yields 
\linebreak
$\mathrm {EPic}(S,Q)$
as the maximal subgroup of the abelian monoid
$k\mathcat{Rep}(Q,\mathpzc{Gen}_{S,Q})$
and recovers the {\em standard construction\/} of
$\mathrm {EB}(S,Q)$, cf. 
\cref{eleven},
as the cokernel
of the corresponding homomorphism 
$\End$ of abelian monoids.
The obvious functor $\Omega\colon \mathpzc{Az}_{S,Q} \to
\mathpzc B_{S,Q}$ induces an isomorphism
$\mathrm {EB}(S,Q) \to k\mathpzc{EB}_{S,Q}$ of abelian groups
\cite[\S 5 p.~38]{genbrauer},
\cite[\S 3, Theorem 3.2 (i)]{MR1803361}, that is, equivariant 
Brauer equivalence is equivalent to equivariant
Morita equivalence. Moreover, that obvious functor $\Omega$
factors as
\begin{equation}
\mathpzc{Az}_{S,Q} 
\stackrel{\Omega^{\mathpzc{Az}}}
\longrightarrow 
\mathpzc{XAz}_{S,Q} 
\stackrel{\Omega^{\mathpzc{XAz}}}
\longrightarrow 
\mathpzc{B}_{S,Q},
\end{equation}
and the functor $\Omega^{\mathpzc{XAz}}\colon
\mathpzc{XAz}_{S,Q} \longrightarrow 
\mathpzc{B}_{S,Q}$
induces the injection  
$\theta\colon \mathrm {XB}(S,Q)\to k\mathpzc{Rep}(Q,\mathpzc{B}_{S,Q})$
 of abelian groups spelled out as 
\eqref{theta}.

Recall that, given a stably $Q$-graded category $\mathpzc C_Q$,
the notation $k_Q\mathpzc C_Q$ refers to the monoid
$k\mathpzc C=k\mathpzc{Ker}(\mathpzc C_Q)=k\mathpzc C_Q$,  
viewed as a $Q$-monoid, cf.
Subsection \ref{stablygraded}.
The functor
\linebreak
$
\mathpzc{End}\colon \mathpzc{Gen}_{S,Q}
\to
\mathpzc{Az}_{S,Q}
$
induces, furthermore, a homomorphism
\begin{equation}
\mathrm H^0(Q,k_Q\mathpzc{Gen}_{S,Q})
\longrightarrow
k\mathpzc{Rep}(Q,\mathpzc{XAz}_{S,Q})
\end{equation}
of monoids \cite[\S 3]{MR1803361}. Indeed, let $M$ be an object of
$\mathpzc{Gen}_{S,Q}$. By construction, the 
grading homomorphism
$\Aut_{\mathpzc{Gen}_{S,Q}}(M) \to Q$ 
is surjective if and only if
the isomorphism class of $M$ in $k\mathpzc{Gen}_{S,Q}$
is fixed under $Q$. 
Hence  an object $M$ of
$\mathpzc{Gen}_{S,Q}$
whose isomorphism class in $k\mathpzc{Gen}_{S,Q}$
is fixed under $Q$
determines the exact sequence
\begin{equation}
1
\longrightarrow
\Aut_S(M) 
\longrightarrow
\Aut_{\mathpzc{Gen}_{S,Q}}(M)
\longrightarrow
Q
\longrightarrow
1,
\end{equation}
plainly congruent to
the exact sequence
\eqref{extseventyone};
in particular, the group
$\Aut_{\mathpzc{Gen}_{S,Q}}(M)$ is canonically isomorphic to
the group $\Aut(M,Q)$, cf. 
\eqref{eightthirtythree}.
Now, for any object $M$ of $\mathpzc{Gen}_S$,
the groups
$\Aut_{\mathpzc{Gen}_S}(M)$, $\Aut_S(M)$, and $\mathrm U(\End_S(M))$
coincide,
and the induced action of
$\Aut_{\mathpzc{Gen}_{S,Q}}(M)$
on $\End_S(M)$ yields a commutative diagram of the kind
\begin{equation*}
\scalefont{0.8}{
\begin{CD}
1
@>>>
\Aut_S(M) 
@>>>
\Aut_{\mathpzc{Gen}_{S,Q}}(M)
@>>>
Q
@>>>
1
\\
@.
@|
@VVV
@VVV
@.
\\
@.
\mathrm U(\End_S(M))
@>>>
\Aut(\End_S(M),Q)
@>>>
\Out(\End_S(M),Q)
@>>>
1
\end{CD}
}
\end{equation*}
and hence an induced $Q$-normal structure
$Q \to \Out(\End_S(M))$
on the split algebra $\End_S(M)$.
Thus
the endomorphism functor
$\mathpzc {End}\colon  \mathpzc{Gen}_{S,Q} \to \mathpzc{Az}_{S,Q}$
induces an exact sequence 
\begin{equation}
\scalefont{0.8}{
0
\longrightarrow
\mathrm H^0(Q,\Pic(S))
\longrightarrow
\mathrm H^0(Q,k_Q\mathpzc{Gen}_{S,Q})\stackrel{\End}
\longrightarrow
k\mathpzc{Rep}(Q,\mathpzc{XAz}_{S,Q})
\longrightarrow
\mathrm{XB}(S,Q)
\longrightarrow
0
}\label{Seq2}
\end{equation}
of abelian monoids \cite[\S 3]{MR1803361} which,
in turn, recovers the {\em standard construction\/}
of the crossed Brauer group $\mathrm{XB}(S,Q)$ of $S$
relative to $Q$ given in
\cref{cbg}.
The unit object of $\mathpzc{XAz}_{S,Q}$ is represented by
$(S,\kappaQ)$.
This kind of construction is given in
\cite[Theorem 4 p.~43]{genbrauer},
\cite[Section 3, a few lines before Theorem 3.2]{MR1803361}
(the cokernel of $\End$ being written as $QB(R,\Gamma)$).
In general,
for the \lq\lq crossed\rq\rq\ versions, the equivalence between
Brauer and Morita equivalence persists only when the group $Q$ is finite, 
that is the 
canonical homomorphism 
\[
\theta\colon \mathrm{XB}(S,Q) \longrightarrow 
k\mathpzc{Rep}(Q,\mathpzc B_{S,Q})
\]
of abelian groups given as
\eqref{theta} above
is injective, see 
\cref{tentwentytw}~(i),
but to prove that $\theta$ is surjective we need the additional hypothesis
that $Q$ be a finite group,  see 
\cref{tentwentytw} (iii).

The above constructions,
applied, with  respect to the morphism
$(f,\varphi)$ in $\mathcat{Change}$,
 to the functors
$\mathpzc {End}\colon  \mathpzc{Gen}_{T|S} \to \mathpzc{Az}_{T|S}$ and
$\mathpzc {End}\colon  \mathpzc{Gen}_{T|S;G,Q} \to \mathpzc{Az}_{T|S;G,Q}$,
yield the exact sequences
\begin{equation*}
\scalefont{0.8}{
\begin{aligned}
0
\longrightarrow
\Pic(T|S)
\longrightarrow
k\mathpzc{Gen}_{T|S} &\stackrel{\End}
\longrightarrow k\mathpzc{Az}_{T|S}
\longrightarrow 
\mathrm B_{\mathrm{fr}}(T|S)
\longrightarrow 0
\\
0
\longrightarrow
\mathrm {EPic}(T|S,Q)
\longrightarrow
k\mathcat{Rep}(Q,\mathpzc{Gen}_{T|S;G,Q}) &\stackrel{\End}
\longrightarrow k\mathcat{Rep}(Q,\mathpzc{Az}_{T|S;G,Q})
\longrightarrow 
\mathrm {EB}_{\mathrm{fr}}(T|S;G,Q)
\longrightarrow 0
\\
0
\longrightarrow
\mathrm H^0(Q,\Pic(T|S))
\longrightarrow
\mathrm H^0(Q,k_Q\mathpzc{Gen}_{T|S;G,Q}) &\stackrel{\End}
\longrightarrow
k\mathpzc{Rep}(Q,\mathpzc{XAz}_{T|S;G,Q})
\longrightarrow
\mathrm{XB}_{\mathrm{fr}}(T|S;G,Q)
\longrightarrow
0
\end{aligned}
}
\end{equation*}
of abelian monoids.
These recover the standard constructions of the abelian groups 
$\mathrm B_{\mathrm{fr}}(T|S)$, $\mathrm{EB}_{\mathrm{fr}}(T|S;G,Q)$, and 
$\mathrm{XB}_{\mathrm{fr}}(T|S;G,Q)$, cf. Subsection \ref{standa} above.

\chapter{Examples}
\label{cIII}

\begin{abstract}
We describe various non-trivial examples
that illustrate the approach to the 
\lq\lq Teichm\"uller cocycle map\rq\rq\ 
developed in \cref{cI} and \cref{cII}
in terms of crossed $2$-fold extensions and generalizations thereof.
\end{abstract}

\section{Introduction}

We recall the classical situation for number fields and show
how it extends to rings of integers in number fields.
We then construct
explicit examples of a non-trivial Teichm\"uller class
that arise in
 Grothendieck's theory of the Brauer group
of a topological space.
We finally interpret various group 3-cocycles
constructed in $\mathrm C^*$-algebra theory
as variants of the Teichm\"uller 3-cocycle.

With hindsight it is interesting to note that
versions of the Teichm\"uller cocycle
arise in $\mathrm C^*$-algebra theory,
see 
\cref{dynamical}
for details.
While such cocycles
were studied already in the 1980s, to our knowledge, their relationship
with the Teichm\"uller cocycle was never pointed out in the literature.

\section{Number fields}
\label{numberfields}

\subsection{General remarks}
\label{generalremarks}
Consider an algebraic number field $K$
(a finite-dimensional extension of the field $\mathbb Q$ of rational numbers).
Let $Q$ be a finite group of operators on $K$, let
$\fiel = K^Q$,  and consider the resulting
Galois extension $K|\fiel$.
Let $J_K$ denote the abelian group of {\em id\`eles\/}
of $K|\fiel$ and $C_K$ that of {\em id\`ele classes\/}, and
consider the familiar $Q$-module extension
\begin{equation}
0
\longrightarrow
\mathrm U(K)
\longrightarrow
J_K
\longrightarrow
C_K
\longrightarrow 
0 
\label{ext11}
\end{equation}
\cite[(III.2)  p.~117]{MR3058613}.
By the \lq\lq main theorem of class field theory\rq\rq,
$\Ho^2(Q,C_K)\cong \tfrac 1{[K:\fiel]}\mathbb Z/\mathbb Z$
\cite[\S VII.3 Lemma 6 p.~49]{MR2467155},
\cite[RESULT p.~196]{MR0220697},
\cite[(III.6.8) Theorem p.~150]{MR3058613},
the group $\Ho^2(Q,C_K)$ has a canonical generator, referred to as
the {\em fundamental class\/} of the extension $K|\fiel$ and written as
$u_{K|\fiel}\in \Ho^2(Q,C_K)$. As a side remark we note that,
given a group extension
$C_K \rightarrowtail  W_{K|\fiel} \twoheadrightarrow Q$
that represents the class $u_{K|\fiel}\in \Ho^2(Q,C_K)$,
the group $W_{K|\fiel}$ is referred to as the {\rm Weil group\/}
of the field extension $K|\fiel$ 
\cite[Ch. XV]{MR2467155}, \cite[\S 11.6 p.~200]{MR0220697},
\cite{MR546607}. The Weil group is uniquely determined
since $\Ho^1(Q,C_K)$ is zero \cite[Ch. XV]{MR2467155}.

Let 
$m$ denote the l.c.m. of the local degrees.
We summarize the results of
 \cite[Theorem 2, Theorem 3]{MR0025442}, 
\cite{MR0047700}, and others as follows, cf.
\cite[\S VII.4 Theorem 12 and Theorem 14 p.~53]{MR2467155},
\cite[\S 11.4 Case $r=3$ p.~199]{MR0220697}.
\begin{prop}
{\rm (i)} 
The boundary homomorphism
$\delta\colon\Ho^2(Q,C_K) \to \Ho^3(Q,\mathrm U(K))$ 
in the long exact cohomology sequence
associated with \eqref{ext11}
is surjective, and 
$\Ho^3(Q,\mathrm U(K))$ 
is cyclic of order $s=\tfrac {[K:\fiel]}m$, generated
by the image $t_{K|\fiel}= \delta(u_{K|\fiel})\in \Ho^3(Q,\mathrm U(K))$.

\noindent
{\rm (ii)} 
The class  $t_{K|\fiel}$
splits in some extension field $L$ of $K$
that is normal over $\fiel$, indeed, 
things may be arranged in such a way that
$L|K$ is cyclic.
\end{prop}

Under the present circumstances,
the eight term exact sequence \eqref{AAAAA}
boils down to the classical five term exact sequence,
cf. 
\eqref{AAAAAA},
given, e.~g., in
\cite[p.~130]{MR0052438},
combined with the canonical isomorphisms
\[
\Ho^2(Q,\mathrm U(K))\cong \mathrm B(K|\fiel),
\quad
\Ho^2(G,\mathrm U(L))\cong \mathrm B(L|\fiel),
\quad
\Ho^2(N,\mathrm U(L))^Q\cong \mathrm B(L|K)^Q,
\]
and the exactness of this sequence
entails that
$t_{K|\fiel}$ is the Teichm\"uller class associated to some
$Q$-normal crossed product central simple 
$K$-algebra having $L$ as maximal commutative
subalgebra.
In the literature, the generator
$t_{K|\fiel}=\delta(u_{K|\fiel})\in \Ho^3(Q,\mathrm U(K))$
is referred to as the {\em Teichm\"uller $3$-class\/}
\cite[\S VII.4 p.~52]{MR2467155},
\cite[\S 11.4 Case $r=3$ p.~199]{MR0220697}, here interpreted
as the {\em obstruction\/} to the global degree being computed
as the l.c.m. of the local degrees.

\subsection{Explicit examples}
Thus to get examples, all we need is
a Galois extension $K|\fiel$ having $s>1$.
While, in view of the Hilbert-Speiser Theorem,  
this is impossible when the Galois group $Q$ is cyclic,
for example,
the fields $K=\mathbb Q(\sqrt{13},\sqrt{17})$
or $K=\mathbb Q(\sqrt{2},\sqrt{17})$
have as Galois group $Q$ the four group and
$s=2$ \cite{MR0025442}, see also 
\cite[\S 11.4 p.~199]{MR0220697} and
\cite[Ch. VIII Exampe 4.5 p.~238]{milneCFT}.

Since it is hard to find truly explicit examples in the literature,
we  now briefly sketch a construction of such examples.
According to classical results
due to Albert, Brauer, Hasse, and E. Noether,
 every member of $\mathrm B(K)$
has a cyclic cyclotomic splitting field
\cite[Satz~4, Satz 5 p.~118]{zbMATH03255773},
\cite[VIII.2 Theorem 2.6 p.~229]{milneCFT},
\cite[10.5 Step 3. p.~191]{MR0220697}.
Indeed, the argument in the last reference shows
that, given a central simple $K$-algebra $A$,
there is a cyclic cyclotomic field $L|K$ 
such that $[L_{\mathcat P}:K_{\mathcat p}]\equiv 0(\mathcat m_{\mathcat p})$
for every prime $\mathcat p$ of $K$ and such that 
$[L:K]=\mathrm {l.c.m.}(\mathcat m_{\mathcat p})$.
Thus, consider a cyclic cyclotomic extension
$L=K(\zeta)$ having Galois group $N$ cyclic of order $n$ (say).
Let $\sigma$ denote a generator of $N$,  
let $\eta \in \mathrm U(K)$,
and 
consider the cyclic central simple $K$-algebra
$D(\sigma,\eta)$ generated by 
$L=K(\zeta)$ and some (indeterminate)
$u$ subject to the relations
\begin{equation}
u \lambda = \sigma(\lambda)u,\ u^n = \eta,\ \lambda \in L=K(\zeta),
\end{equation}
necessarily a crossed product of $N$ with 
$L$ relative  to the $\mathrm U(L)$-valued  
2-cocycle of $N$
determined by $\eta$. By construction,
 $D(\sigma,\eta)$ is split by $L$.
Moreover,
given $\vartheta \in \mathrm U(L)$,
the member $\eta_\vartheta=\eta \prod_{j=0}^{n-1} \vartheta^{\sigma^j}$ 
of $K$ yields
the algebra
 $D(\sigma, \eta_\vartheta)$, and
the association $u \longmapsto \vartheta u$
induces an isomorphism
\begin{equation}
D(\sigma, \eta) \longrightarrow D(\sigma, \eta_\vartheta)
\end{equation}
of central $K$-algebras.
The field $L|\fiel$
is the composite field
$\fiel(\zeta)K$. Hence
$L|\fiel$ is a Galois extension, and 
the Galois group $G$
of $L|\fiel$
 is a central extension of
$Q=\mathrm{Gal}(K|\fiel)$ by the cyclic group 
$N=\mathrm{Gal}(\fiel(\zeta)|\fiel)$ of order $n$,
a split extension if and only if 
$\fiel(\zeta)\cap K = \fiel$.
The member $\eta$ of $K$ represents the corresponding cohomology class
$[\eta]\in \Ho^2(N,\mathrm U(L))$, and
$[\eta]\in \Ho^2(N,\mathrm U(L))^Q$
if and only if, given $x \in Q=\mathrm{Gal}(K|\fiel)$, 
there is some $\vartheta_x\in \mathrm U(L)$
such that the association
 $u \longmapsto \vartheta_x u$
induces an automorphism
\begin{equation}
\Theta_x\colon D(\sigma, \eta) \longrightarrow D(\sigma, \eta)
\end{equation}
of central $K$-algebras that extends the automorphism
$x\colon K \to K$ over $\fiel$.

The sequence
\begin{equation}
0
\longrightarrow
\Ho^2(N,\mathrm U(L))
\longrightarrow
\Ho^2(N, J_L)
\stackrel{\mathrm{inv}_1}
\longrightarrow 
\tfrac 1{|N|}\mathbb Z /\mathbb Z
\longrightarrow
0
\label{neuk11}
\end{equation}
is well known to be exact,
cf., e.~g., \cite[III.5.6 Proposition p.~143]{MR3058613},
and taking $Q$-invariants, we obtain the injection
\begin{equation}
0
\longrightarrow
\Ho^2(N,\mathrm U(L))^Q
\longrightarrow
\Ho^2(N, J_L)^Q .
\end{equation}
Given a prime $\mathcat p$ of $K$, for each prime $\mathcat P$
of $L$ above $\mathcat p$, the local extension
$L_{\mathcat P}|  K_{\mathcat p}$
is likewise a
cyclic cyclotomic extension.
From a given system of local invariants 
in $\Ho^2(N, J_L)^Q$ that goes to zero under
$\mathrm{inv}_1\colon \Ho^2(N, J_L)
\to
\tfrac 1{|N|}\mathbb Z /\mathbb Z$,
at each prime
$\mathcat P$
of $L$ above $\mathcat p$ 
that occurs in that system of local invariants,
we can construct an explicit cyclic central
$K_{\mathcat p}$-algebra $D(\sigma_{\mathcat P},\eta_{\mathcat P})$
defined in terms of a prime element
of $K_{\mathcat p}$ and,
using, e.~g., the recipe in the proof of
 \cite[Satz 9 p.~119 ff.]{zbMATH03255773},
we can then construct
a member $\eta$ of $K$
such that the cyclic algebra $D(\sigma,\eta)$
has the given local invariants.
By construction, then, the class of 
$D(\sigma,\eta)$ in $\Ho^2(N,\mathrm U(L))$ is $Q$-invariant.
Hence $D(\sigma,\eta)$ acquires a $Q$-normal structure, necessarily 
non-trivial when its Teichm\"uller class is non-zero,
and the above reasoning classifies those cyclic $Q$-normal algebras
that are non-trivially $Q$-normal.

\section{Rings of integers and beyond}
\label{ringsofint}
Let $R$ be a regular domain, and let $K$ denote its quotient field.
By \cite[Theorem 7.2 p.~388]{MR0121392},
the induced homomorphism
$\mathrm B(R)\to \mathrm B(K)$
between the Brauer groups is a monomorphism.
It is known that, furthermore,
the canonical map $\mathrm B(R) \to \cap_{\mathcat p} 
\mathrm B(R_{\mathcat p})$
from the Brauer group $\mathrm B(R)$ to the intersection 
$\cap_{\mathcat p} \mathrm B(R_{\mathcat p})$ taken over all height one primes $\mathcat p$ is an isomorphism, cf., e.~g., 
\cite[Theorem 9.7 p.~64]{MR1692654}.

Let $K$ be an algebraic number field, $S$ its ring of integers,
and let $r$ denote the number of embeddings of $K$ into the reals.
The Brauer group $\mathrm B(S)$ of $S$ is zero when $r=1$ and
isomorphic to a direct product of $r-1$ copies of
the cyclic group with two elements
when $r \geq 2$.
This is a consequence of a result in
\cite{MR0228471},
see, e.~g.,
\cite[(6.49) p.~151]{MR1610222}.
While a central $S$-Azumaya algebra
representing a non-trivial member of 
$\mathrm B(S)$ need not be representable as
an ordinary crossed product with respect to a Galois extension
of $S$,
see, e.~g., \cite{MR909029} and the literature there,
a right $H$-Galois extension $T|S$ of rings of integers
with respect to a general finite-dimensional Hopf algebra $H$
which splits  all classes in the Brauer group
$\mathrm B(S)$ can easily be found
\cite[Proposition 2.1 p.~246]{MR909029}.
The question as to, whether or not, given a finite group $Q$ of 
operators on $K$ and hence
on $S$,
along these lines, $Q$-normal $S$-Azumaya algebras arise 
is a largely unexplored territory. 
The example \cite[Remark 2.6 p.~249]{MR909029}
yields a $Q$-equivariant $Q$-normal Azumaya algebra
for $Q$ the cyclic group with two elements.

Consider now an algebraic number field $K$, a finite group $Q$
of operators on $K$, let $\fiel =K^Q$, and let
$S$ be the ring of integers in $K$ and $R$ that in $\fiel$.
Consider a  field extension $L|K$ such that $K|\fiel$ is normal, with 
Galois group $G$, let $N=\mathrm{Gal}(L|K)$, so that the 
Galois groups fit into an extension $N \rightarrowtail G \twoheadrightarrow Q$,
and let $T$ denote the ring of integers in $L$.
Let $\mathbb S_{L|K}$ denote the finite set of primes
of $K$ 
 that ramify in $L$ and let
 $\mathbb S_{L}$ denote the finite set of primes of $L$
above the primes in $\mathbb S_{L|K}$. Inverting the primes
in 
 $\mathbb S_{L}$ 
and those in $\mathbb S_{L|K}$
we obtain a Galois extension $T_{\mathbb S_{L}}|S_{\mathbb S_{L|K}}$
of commutative rings with Galois group $N$.
Let, furthermore, 
 $\mathbb S_{K|\fiel}$
denote those primes of $\fiel$ such that
the primes in 
$\mathbb S_{L|K}$
are exactly the primes above
 $\mathbb S_{K|\fiel}$,
and let
$R_{\mathbb S_{K|\fiel}}$ denote
the corresponding ring that arises from $R$ by inverting the primes
in $\mathbb S_{K|\fiel}$.
Then the data constitute a $Q$-normal Galois extension of 
commutative rings but, while 
$R_{\mathbb S_{K|\fiel}}= S_{\mathbb S_{L|K}}^Q$,
the ring extension 
$S_{\mathbb S_{L|K}}|R_{\mathbb S_{K|\fiel}}$
need not be a Galois extension of commutative rings.
Recall the exact sequence 
\eqref{BBBBB},
for $T_{\mathbb S_{L}}|S_{\mathbb S_{L|K}}|R_{\mathbb S_{K|\fiel}}$
as well as for $L|K|\fiel$.
The inclusions into the quotient fields yield
a commutative diagram
\begin{equation}
\begin{CD}
\mathrm H^2(G,\mathrm U(T_{\mathbb S_{L}})) 
@>{j}>>
\mathrm{Xpext} (G,N;\mathrm U(T_{\mathbb S_{L}}))
@>{\Delta}>> \mathrm H^3(Q,\mathrm U(S_{\mathbb S_{L|K}})) 
\\
@VVV
@VVV
@VVV
\\
\mathrm H^2(G,\mathrm U(L)) 
@>{j}>>
\mathrm H^2(N,\mathrm U(L))^Q
@>{\Delta}>> \mathrm H^3(Q,\mathrm U(K)) 
\end{CD}
\end{equation}
Suitably interpreting the constructions in
Section \ref{numberfields} above,
we can then construct crossed pair extensions
that represent members of
$\mathrm{Xpext} (G,N;\mathrm U(T_{\mathbb S_{L}}))$
whose images in 
$\mathrm H^2(N,\mathrm U(L))^Q$
have non-zero values in 
$\mathrm H^3(Q,\mathrm U(K))$.
Hence the associated crossed pair algebras
then have non-zero Teichm\"uller class in 
$\mathrm H^3(Q,\mathrm U(S_{\mathbb S_{L|K}}))$.
This yields non-trivial examples 
of Teichm\"uller classes 
of normal Azumaya algebras
over rings of algebraic numbers
with finitely many primes inverted.
We intend to give the details at another occasion.
The Galois module structure
of groups like $\mathrm U(S)$ and $\Pic(S)$ is delicate, cf., e.~g.,
\cite{MR717033}, and the calculation of the relevant group cohomology
groups is not an easy matter. More work is called for in this area.

\section{Examples arising in algebraic topology}

\subsection{General remarks}
Let $\Bsp$ be a topological space, and let $S$ denote the algebra of 
continuous complex-valued functions on $\Bsp$.
Isomorphism classes of Azumaya $S$-algebras of rank $n>1$ correspond 
bijectively
to isomorphism classes of principal
$\mathrm {PGL}(n,\mathbb C)$-bundles.

When $\Bsp$ is a finite CW-complex,
by a Theorem of Serre
 \cite[Theorem 1.6]{MR1608798},
the Brauer group $\mathrm B(S)$ is canonically isomorphic to
the torsion part $\Ho^3(\Bsp)_{\mathrm{tors}}$ of the third
integral cohomology group $\Ho^3(\Bsp)$ of $\Bsp$. The isomorphism
is realized explicitly as follows:
Let $\mathcat{Map}(\Bsp, \mathbb C)$ 
denote the sheaf of germs of
continuous $\mathbb C$-valued functions on $\Bsp$
and
$\mathcat{Map}(\Bsp, \mathbb C^*)$
that of 
continuous $\mathbb C^*$-valued functions on $\Bsp$. 
The exponential exact sequence 
\begin{equation*}
0
\longrightarrow
\mathbb Z
\longrightarrow
\mathcat{Map}(\Bsp, \mathbb C)
\longrightarrow
\mathcat{Map}(\Bsp, \mathbb C^*)
\longrightarrow
0
\end{equation*}
of sheaves on $\Bsp$ yields an isomorphism
$\Ho^2(\Bsp,\mathcat{Map}(\Bsp, \mathbb C^*)) \cong \Ho^3(\Bsp)$
of sheaf cohomology groups
(valid more generally for paracompact $\Bsp$).
The theorem of Serre's just quoted says that, $\Bsp$ being a finite CW-complex,
the canonical map from 
the Brauer group
$\mathrm B(S)$ to $\Ho^2(\Bsp,\mathcat{Map}(\Bsp, \mathbb C^*))$
is an isomorphism onto 
$\Ho^2(\Bsp,\mathcat{Map}(\Bsp, \mathbb C^*))_{\mathrm{tors}}$.

Let $\prin\colon P \to \Bsp$ be a principal $\mathrm {PGL}(n,\mathbb C)$-bundle
and,
relative to the adjoint action of 
$\mathrm {PGL}(n,\mathbb C)$ on $\mathrm {M}_n(\mathbb C)$,
 let $\zeta$ denote the associated vector bundle
\[
\zeta\colon E=P \times _{\mathrm {PGL}(n,\mathbb C)}
\mathrm M_n(\mathbb C)\longrightarrow \Bsp
\] 
on $\Bsp$.
The $S$-module of continuous sections $A=\Gamma(\zeta)$
of $\zeta$ acquires the structure of an Azumaya $S$-algebra
in an obvious manner in such a way that
the group $\mathrm U(A)$ of units of $A$ gets naturally identified
with the space of sections of
the associated fiber bundle
\begin{equation*}
\mathrm u_\xi\colon 
P\times _{\mathrm {PGL}(n,\mathbb C)}\mathrm {GL}_n(\mathbb C)
\longrightarrow \Bsp
\end{equation*}
relative to the adjoint action of 
$\mathrm {PGL}(n,\mathbb C)$ on $\mathrm {GL}_n(\mathbb C)$, endowed with
the pointwise group structure.
Thus the  group $\mathrm U(A)$ of units of $A$ 
can be written as the group
\[
\overline{\mathcal G}_{\xi}\cong 
\mathrm{Map}_{\mathrm{PGL}(n,\mathbb C)}(P,\mathrm{GL}(n,\mathbb C))
\]
of $\mathrm{PGL}(n,\mathbb C)$-equivariant maps
from $P$ to $\mathrm{GL}(n,\mathbb C)$,
and the group $\overline{\mathcal G}_{\xi}$, in turn,
maps canonically onto the group 
$
\mathcal G_{\xi}\cong 
\mathrm{Map}_{\mathrm{PGL}(n,\mathbb C)}(P,\mathrm{PGL}(n,\mathbb C))
$
of gauge transformations of $\prin$.
The group $\Aut(\prin)$ of bundle automorphisms of $\prin$, i.~e., pairs
$(\Phi,\varphi)$ of homeomorphisms that make the diagram
\begin{equation}
\begin{CD}
P@>{\Phi}>> P
\\
@V{\xi}VV
@V{\xi}VV
\\
\Bsp
@>{\varphi}>>
\Bsp
\end{CD}
\label{diag101}
\end{equation}
commutative, yields, in a canonical way, a subgroup
of the group $\Aut(A)$ of ring automorphisms of $A$, and the assignment to a
 section of $\mathrm u_\xi$ of the induced gauge transformation
of the kind \eqref{diag101}
with $\varphi=\mathrm{Id}$
yields a homomorphism 
$\partial\colon \overline{\mathcal G}_{\xi} \to \Aut(\prin)$.
Denote by $\mathrm Z_n(\mathbb C)\cong \mathbb C^*$ 
the central diagonal subgroup of
$\mathrm {GL}_n(\mathbb C)$.
Identifying
the kernel of $\partial$  with the space 
of sections of the associated bundle
\[
P\times _{\mathrm {PGL}(n,\mathbb C)}\mathrm {Z}_n(\mathbb C)
\longrightarrow \Bsp, 
\]
necessarily trivial, since
$\mathrm Z_n(\mathbb C)$ is the center of $\mathrm {GL}_n(\mathbb C)$, we see that 
the  kernel of $\partial$ is canonically isomorphic to the abelian group
$\mathrm U(S)$ of continuous functions from $\Bsp$ to $\mathbb C^*$.
Denote the group
of homeomorphisms of $\Bsp$
by $\mathrm{Homeo}(\Bsp)$, 
and let $\Out(\prin)\subseteq \mathrm{Homeo}(\Bsp)$ denote the image of
 $\Aut(\prin)$ in $\mathrm{Homeo}(\Bsp)$
under the forgetful map
which assigns to a member
$(\Phi,\varphi)$ of $\Aut(\prin)$ 
the second component
$\varphi$. The group $\Out(\prin)$ is the group of homeomorphisms $\varphi$ of $\Bsp$
such that the induced princiapl bundle $\varphi^*\prin$ is isomorphic
to $\prin$.
Thus the principal bundle $\prin$ determines the crossed 2-fold extension
\begin{equation}
0
\longrightarrow
\mathrm{Map}(\Bsp,\mathbb C^*)
\longrightarrow
\overline{\mathcal G}_\prin
\stackrel{\partial}\longrightarrow
\Aut(\prin)
\longrightarrow
\Out(\prin)
\longrightarrow
1.
\end{equation}

Consider a group $Q$, and suppose that
$Q$ acts on $\Bsp$ via a homomorphism $\sigma\colon Q \to \Out(\prin)$.
Requiring that the action be via 
a homomorphism $\sigma\colon Q \to \Out(\prin)$
is equivalent to requiring that some group $\Gamma$
that maps onto $Q$ act on the total space $P$ of $\prin$ in such a way that
given $q \in Q$, there exists some $\gamma \in \Gamma$ such that 
\begin{equation*}
\begin{CD}
P@>{\gamma}>> P
\\
@V{\xi}VV
@V{\xi}VV
\\
\Bsp
@>{q}>>
\Bsp 
\end{CD}
\end{equation*}
is an automorphism
of principal $\mathrm {PGL}(n,\mathbb C)$-bundles.
Requiring that $\Gamma$ act by bundle automorphisms is equivalent
to requiring that the $\Gamma$-action on $P$ commute with
the $\mathrm {PGL}(n,\mathbb C)$-action.
The homomorphism $\sigma$ then induces the requisite $Q$-action
\[
\kappaQ \colon Q \longrightarrow \Aut(S)
= \Aut(\mathrm{Map}(\Bsp,\mathbb C))
\]
 on 
$S=\mathrm{Map}(\Bsp,\mathbb C)$, and
$(A,\sigma) =(\Gamma(\zeta),\sigma)$ 
is a $Q$-normal Azumaya $S$-algebra.

\subsection{Explicit examples involving metacyclic groups}
\label{revisited}

Consider a metacyclic group $G$ given by a presentation
\begin{equation}
G(r,s,t,f) = \langle x,y;\ y^r = 1,\ x^s = y^f,
\ xyx^{-1} = y^t \rangle
\label{0.22}
\end{equation}
where
\[
s > 1,\quad r > 1,
\quad t^s \equiv 1 \mod r,\quad tf \equiv f \mod r,
\]
so that, in particular, the numbers ${\tfrac {t^s-1}r}$ and
${\tfrac {(t-1)f}r}$ are positive integers.
The group $G$ is an extension 
\begin{equation}
1
\longrightarrow
N
\longrightarrow
G
\longrightarrow
Q
\longrightarrow
1
\label{gemm}
\end{equation}
of the cyclic 
group $Q=C_s$ of order $s$ by the cyclic 
group $N=C_r$ of order $r$ generated by $y$.
The upshot of the present subsection is
an explicit $Q$-normal crossed pair algebra
having a ring of the kind $S=\Map(\Bsp,\mathbb C)$ 
for some topological space $\Bsp$
as its center
and having
non-zero Teichm\"uller class in $\Ho^3(Q,\mathrm U(S))$, to be given as
\eqref{ntcpa} below.

Suppose that the g.c.d. $({\tfrac {t^s-1}r},r)$ is non-trivial, 
let $\ell>1$ denote a non-trivial divisor of $({\tfrac {t^s-1}r},r)$,
let
$C_{\ell r}$ denote the cyclic group of order $\ell r$, let $v$ denote a 
generator of $C_{\ell r}$, and let $C_\ell$ denote 
the cyclic subgroup of  $C_{\ell r}$ of order $\ell$ generated by $v^r$.
The assignment to
$v$ of $y$ yields a group extension
\begin{equation}
\mathrm e_{\ell r}\colon 
0
\longrightarrow
C_\ell
\longrightarrow
C_{\ell r}
\longrightarrow
C_r
\longrightarrow
1
\label{elr}
\end{equation}
representing the generator of 
$\Ho^2(C_r,\mathbb  Z/\ell)\cong \mathbb  Z/\ell$.

Since $\tfrac {t^s-1}{\ell r}$ is an integer, the association
$Q \times C_{\ell r} \longrightarrow C_{\ell r}$ given by 
$(x,v) \longmapsto v^t$
yields an action of the group $Q=C_s$ on 
$C_{\ell r}=\langle v; v^{\ell r}=1\rangle$.
The induced action 
\begin{equation}
G \times C_{\ell r} \longrightarrow C_{\ell r} 
\label{action22}
\end{equation}
of $G$ on $C_{\ell r}$  via the
projection $G \to C_s$
and the obvious homomorphism $\partial\colon C_{\ell r} \to G$
then constitute a crossed module.

\begin{prop}
\label{crossedvalue}
With respect to a suitable choice of the isomorphism
\[
\Ho^3(C_s,\MZ/\ell) \cong \MZ/(\ell,s),
\]
the resulting associated crossed $2$-fold extension
\begin{equation}
\mathrm e^2\colon 
0
\longrightarrow
C_\ell
\longrightarrow
C_{\ell r}
\longrightarrow
G
\longrightarrow
C_s
\longrightarrow
1
\label{crossed12}
\end{equation}
represents the class in
$\Ho^3(C_s,\MZ/\ell)$ that corresponds to
$\tfrac {(t-1)f}r \mod (\ell,s)$.
\end{prop}

\begin{proof}
Let $F_x$ denote the free group on $x$,
let 
$\mathbb Z C_s\langle b\rangle$
denote the free $C_s$-module on a single generator $b$,
view $\mathbb Z C_s\langle b\rangle$
as an $F_x$-group via the canonical projection $F_x \to C_s$,
define the morphism
$\partial \colon \mathbb Z C_s\langle b\rangle \to F_x$
of $F_x$-groups by
$\partial (b)=x^s$, and note that
$\partial \colon \mathbb Z C_s\langle b\rangle
\to
F_x
$
is the free crossed module associated to
the presentation $\langle x; x^s\rangle$
of the group $C_s$.
Using the familiar notation $\mathrm IC_s \subseteq \mathbb ZC_s$
for the augmentation ideal of $C_s$,
consider the associated crossed 2-fold extension
\begin{equation}
\begin{CD}
0
@>>>
\mathrm IC_s\langle b\rangle
@>>>
\mathbb Z C_s\langle b\rangle
@>{\partial}>>
F_x
@>>>
C_s
@>>>
1, 
\end{CD}
\end{equation}
and lift the identity of $C_s$
to a morphism
\begin{equation}
\begin{CD}
0
@>>>
\mathrm IC_s\langle b\rangle
@>>>
\mathbb Z C_s\langle b\rangle
@>{\partial}>>
F_x
@>>>
C_s
@>>>
1 
\\
@.
@V{\alpha_2}VV
@V{\alpha_1}VV
@V{\alpha_0}VV
@|
@.
\\
0
@>>>
C_\ell
@>>>
C_{\ell r}
@>{\partial}>>
G
@>>>
C_s
@>>>
1
\end{CD}
\label{morcrossed2}
\end{equation}
of crossed 2-fold extensions as follows: 
With an abuse of the notation $x$,
let $\alpha_0(x)=x$, and let
$\alpha_1(b)=v^f$.
Then
\begin{align*}
\alpha_1((x-1)b)&={}^x(v^f)v^{-f}=v^{(t-1)f}
=(v^r)^{\tfrac {(t-1)f}r} \in C_\ell \subseteq C_{\ell r}
\end{align*}
Consequently
$\alpha_2((x-1)b)=(v^r)^{\tfrac {(t-1)f}r} \in C_\ell \subseteq C_{\ell r}$, 
whence $\alpha_2$ represents
the member of $\Ho^3(C_s,\MZ/\ell) \cong \MZ/(\ell,s)$
that corresponds to $\tfrac {(t-1)f}r \mod (\ell,s)$.
\end{proof}

\begin{rema} 
\label{nontrivial}
It is immediate that, for a suitable choice of the
parameters,
$\tfrac {(t-1)f}r$ is non-trivial modulo $(\ell,s)$.
For example, as in the situation of 
\cite[Theorem E]{MR1014607}, suppose that
$p$ is a prime that divides $r$, $s$, 
$\tfrac {t^s-1}r$ and $f$, but that it does not
divide
$\tfrac {(t-1)f}r$.
Then, with $\ell = p$,
the 2-cocycle $\alpha_2$ 
and hence the crossed $2$-fold extension $\mathrm e^2$ 
represent a generator of
$\Ho^3(C_s,\mathbb  Z/p)\cong \MZ/p$. 
Indeed, for a suitable choice of the data,
in the notation of \cite[Theorem E]{MR1014607},
this class is that written there as $\omega_x c_x$.
\end{rema}

Since $\Ho^2(C_s,\mathbb C^*)$ is trivial,
the homomorphism
$\Ho^3(C_s,\mathbb  Z/\ell) \to \Ho^3(C_s,\mathbb C^*)$
induced by the canonical injection
$C_\ell \to \mathbb C^*$
is injective
whence, when $\alpha_2$ represents a non-trivial coho\-mology class,
the composite of $\alpha_2$ with
the canonical injection
$C_\ell \to \mathbb C^*$
yields a non-trivial coho\-mology class in 
$\Ho^3(C_s,\mathbb C^*)\cong \mathbb Z/s$.
To construct a crossed 2-fold extension
representing that cohomology class, let
$\widehat C_{\ell r}$ denote the universal group characterized by the 
requirement that the diagram
\begin{equation}
\begin{CD}
\mathrm e_{\ell r}
\colon
0
@>>>
C_\ell
@>>>
C_{\ell r}
@>>>
C_r
@>>>
1
\\
@.
@VVV
@VVV
@|
@.
\\
\mathrm e^*
\colon
0
@>>>
\mathbb C^*
@>>>
\widehat C_{\ell r}
@>>>
C_r
@>>>
1
\end{CD}
\label{diag22}
\end{equation}
be commutative with exact rows.
The $G$-action \eqref{action22} on $C_{\ell r}$
and the trivial $G$-action on $\mathbb C^*$ combine to 
a
 $G$-action
\begin{equation}
G \times \widehat C_{\ell r} \longrightarrow \widehat C_{\ell r}
\label{action32}
\end{equation}
on $\widehat C_{\ell r}$ that turns the 
obvious map $\widehat \partial\colon
\widehat C_{\ell r} \to G$
into a crossed module.
The cohomology class under discussion is represented by the resulting 
crossed 2-fold extension
\begin{equation}
\mathrm e^*_2 \colon
0
\longrightarrow
\mathbb C^*
\longrightarrow
\widehat C_{\ell r}
\stackrel{\widehat \partial}\longrightarrow
G
\longrightarrow
C_s
\longrightarrow
1 .
\label{crossed22}
\end{equation}
Since the group $\mathbb C^*$ is a divisible abelian group,
the bottom row extension $\mathrm e^*$ in \eqref{diag22} splits in the category
of abelian groups.
However, when the class represented by $\alpha_2$ is non-trivial, 
such a splitting
cannot be compatible with the $G$-module structures.

Under the present circumstances,
since the action of $N$ on $\mathbb C^*$ is trivial,
the bottom diagram of what corresponds to 
\eqref{diag1},
with $M=\mathbb C^*$,
takes the form
\begin{equation}
\begin{CD}
0 @>>> \Hom (N,\mathbb C^*) @>>>       \Aut_G(\mathrm e^* ) @>>> G @>>> 1
\\
@.
@V{\cong}VV
@VVV
@VVV
@.
\\
0 @>>> \mathrm H^1(N,\mathbb C^*) @>>> \Out_G(\mathrm e^* ) @>>> Q @>>> 1
\end{CD}
\label{diag10}
\end{equation}
with exact rows and,
cf. 
\cref{cmcp},
the $G$-action  \eqref{action32}
(turning $\widehat C_{\ell r}$ together with the canonical
homomorphism $\widehat \partial\colon
\widehat C_{\ell r} \to G$
into a $G$-crossed module)
determines and is determined by a crossed pair structure
$\psi\colon Q \to  \Out_G(\mathrm e^*)$ on 
the group extension $\mathrm e^*$.
The resulting crossed pair 
\[
\left(\mathrm e^*\colon \mathbb C^* \rightarrowtail \widehat C_{\ell r}
\twoheadrightarrow C_r,
\psi\colon Q \to  \Out_G(\mathrm e^* )\right)
\]
represents a non-trivial class 
\begin{equation}
[(\mathrm e^*,\psi)] \in
\mathrm{Xpext}(G,N;\mathbb C^*)
\label{class2}
\end{equation}
in the group $\mathrm{Xpext}(G,N;\mathbb C^*)$
of crossed pair extensions with respect to the
group extension \eqref{gemm} and the (trivial) $G$-module
$\mathbb C^*$,
cf. \cite[Theorem~1]{MR597986} and
Subsection \ref{6.11}
for these notions
and,
cf. 
 \cite[Theorem~2]{MR597986}
or Subsection \ref{6.11},
the homomorphism
\begin{equation*}
\Delta\colon \mathrm{Xpext} (G,N;\mathbb C^*)
\longrightarrow \mathrm H^3(Q,\mathbb C^*)
\end{equation*}
sends the class \eqref{class2}
to
$[\mathrm e^*_2] \in \mathrm H^3(Q,\mathbb C^*) \cong \mathbb Z/s$.

Let $\pi \colon \widetilde \Bsp \to \Bsp$
be a regular covering projection
having the group $N=C_r$ 
as deck transformation group, and let
$S= \mathrm{Map}(\Bsp,\mathbb C)$ and
$T= \mathrm{Map}(\widetilde \Bsp,\mathbb C)$.
Then $T|S$ is a Galois extension of 
commutative rings with Galois group $N$, cf. 
\cref{exathr}.
Suppose that
$\widetilde \Bsp$ is endowed with a  
$G$-action that extends the $N$-action.
Then the quotient group $Q = G/N$ acts on $\Bsp$ in an obvious manner, and
$T|S$ is a $Q$-normal Galois extension of commutative rings, 
with structure extension 
\eqref{gemm}
and structure homomorphism $\llambda \colon G \to \Aut^S(T)$.
By construction, 
$\mathrm U(S)= \Map(\Bsp,\mathbb C^*)$
and
$\mathrm U(T)=\Map(\widetilde \Bsp,\mathbb C^*)$.
Since the groups $N$, $G$, and $Q$ are finite,
the homomorphisms
\begin{align}
\Ho^*(N,\mathbb C^*)&\longrightarrow
\Ho^*(N,\mathrm U(T)),
\\
\Ho^*(Q,\mathbb C^*) &\longrightarrow
\Ho^*(Q,\mathrm U(S)),
\\
\mathrm{Xpext} (G,N;\mathbb C^*)
&\longrightarrow \mathrm{Xpext} (G,N;\mathrm U(T)),
\label{explicit}
\end{align}
induced by the canonical injections
$\mathbb C^* \to \Map(\widetilde \Bsp,\mathbb C^*)$
and
$\mathbb C^* \to \Map(\Bsp,\mathbb C^*)$
(induced by the assignments to a member of $\mathbb C^*$ of the
associated constant maps), respectively,
are isomorphisms,
the third homomorphism being an isomorphism
in view of the naturality of the exact sequence
{\rm \cite[(1.9)]{MR597986}}
(spelled out as the top sequence in the diagram in
\cref{withoutp}).

For later reference, we now give an explicit description
of a representative of the image of \eqref{class2}
under \eqref{explicit}.
To this end, let
$C_T$ denote the universal group characterized by the 
requirement that the diagram
\begin{equation}
\begin{CD}
\mathrm e^*
\colon
0
@>>>
\mathbb C^*
@>>>
\widehat C_{\ell r}
@>>>
C_r
@>>>
1
\\
@.
@VVV
@VVV
@|
@.
\\
\mathrm e_T
\colon
0
@>>>
\mathrm U(T)
@>>>
C_T
@>>>
C_r
@>>>
1
\end{CD}
\label{diag222}
\end{equation}
be commutative with exact rows.
The $G$-action \eqref{action32} on $\widehat C_{\ell r}$
and the $G$-action on $\mathrm U(T)$ combine to 
a
 $G$-action
\begin{equation}
G \times C_T \longrightarrow C_T
\label{action42}
\end{equation}
on $C_T$.
With $M=\mathrm U(T)$,
 diagram 
\eqref{diag1}
takes the form
\begin{equation}
\begin{CD}
@. 0 @. 0 @.@.
\\
@.
@VVV
@VVV
@.
@.
\\
@. 
\mathrm U(S)
@= 
\mathrm U(S)
@. 
1
@. 
\\
@.
@VVV
@VVV
@VVV
@.
\\
0
@>>> 
\mathrm U(T)
@>>> 
C_T
@>>> 
N 
@>>> 
1
\\
@.
@VVV
@VVV
@VVV
@.
\\
0 @>>> \Der (N,\mathrm U(T)) @>>>       \Aut_G(\mathrm e_T ) @>>> G @>>> 1
\\
@.
@VVV
@VVV
@VVV
@.
\\
0 @>>> \mathrm H^1(N,\mathrm U(T)) @>>> \Out_G(\mathrm e_T ) @>>> Q @>>> 1
\\
@.
@VVV
@VVV
@VVV
@.
\\
@. 1 @. 1 @.1 @. ,
\end{CD}
\label{diag100}
\end{equation}
with exact rows and columns, the $G$-action 
\eqref{action42}
on $C_T$ induces a section
$\Psi_T\colon G \to \Aut_G(\mathrm e_T)$
for the third row group extension in \eqref{diag100},
and this section, in turn,
induces a section
$\psi_T\colon Q \to \Out_G(\mathrm e_T )$
for the bottom row extension in \eqref{diag100} in such a way that
$(\mathrm e_T,\psi_T)$ is a crossed pair.
By construction, then, the image 
\[
\Delta[(\mathrm e_T,\psi_T)] \in \Ho^3(Q,\mathrm U(S)) 
\]
of the class 
\begin{equation}
[(\mathrm e_T,\psi_T)]\in \mathrm {Xpext}(G,N;\mathrm U(T))
\label{class3}
\end{equation}
is represented by the crossed two-fold extension 
that arises 
as the top row of the commutative diagram
\begin{equation}
\begin{CD}
\mathrm e^2_T \colon
0
@>>>
\mathrm U(S)
@>>>
C_T
@>{\partial_T}>>
B^{\psi_T}
@>>>
Q
@>>>
1
\\
@.
@|
@|
@VVV
@V{\psi_T}VV
@.
\\
\phantom{\mathrm e^2_T \colon}
0
@>>>
\mathrm U(S)
@>>>
C_T
@>>>
\Aut_G(\mathrm e_T)
@>>>
\Out_Q(\mathrm e_T)
@>>>
1,
\end{CD}
\label{101}
\end{equation}
the group $B^{\psi_T}$ being characterized by the requirement that
the right-hand square be a pull back square.

The naturality of the constructions entails that the diagram
\begin{equation}
\begin{CD}
\mathrm{Xpext} (G,N;\mathbb C^*)
@>{\Delta}>> 
\mathrm H^3(Q,\mathbb C^*)
\\
@V{\cong}VV
@V{\cong}VV
\\
\mathrm{Xpext}(G,N;\mathrm U(T))
@>{\Delta}>>
\Ho^3(Q,\mathrm U(S))
\end{CD}
\label{diag102}
\end{equation}
is commutative. 
In the case at hand the commutativity of \eqref{diag102}
is an immediate consequence of the observation that
the above homomorphism
$\Psi_T\colon G \to \Aut_G(\mathrm e_T )$
induces a homomorphism $G \to B^{\psi_T}$
which makes the diagram
\begin{equation}
\begin{CD}
\mathrm e_2^* \colon
0
@>>>
\mathbb C^*
@>>>
\widehat C_{\ell r}
@>{\widehat \partial}>>
G
@>>>
Q
@>>>
1
\\
@.
@|
@|
@VVV
@V{\psi_T}VV
@.
\\
\mathrm e^2_T \colon
0
@>>>
\mathrm U(S)
@>>>
C_T
@>{\partial_T}>>
B^{\psi_T}
@>>>
Q
@>>>
1
\end{CD}
\label{103}
\end{equation}
commutative.
This commutativity, in turn,
implies that (i) the class \eqref{class3}
yields a non-trivial class in
the group
$\mathrm{Xpext}(G,N;\mathrm U(T))$ of crossed pair extensions
with respect to
\eqref{gemm}
and $\mathrm U(T)$ and that (ii)
this class goes under $\Delta$ to the image 
in $\Ho^3(Q,\mathrm U(S))$ of the class
$[\mathrm e_2] \in \mathrm H^3(Q,\mathbb C^*) \cong \mathbb Z/s$,
non-trivial for suitable choices of the parameters, cf. Remark \ref{nontrivial}
above.

Recall
the homomorphism
\eqref{13.2},
of the kind 
\[
\mathrm{cpa}\colon \mathrm{Xpext}(G,N;\mathrm U(T)) 
\longrightarrow \mathrm{XB}(T|S;G,Q).
\]
This homomorphism fits into the
commutative diagram
\begin{equation}
\begin{CD}
\Ho^1(Q,\Ho^1(N,\mathrm U(T)))
@>{\mathrm d_2}>>
\Ho^3(Q,\mathrm U(S))
\\
@V{\alpha}VV
@|
\\
\mathrm{Xpext}(G,N;\mathrm U(T))
@>{\Delta}>>
\Ho^3(Q,\mathrm U(S))
\\
@V{\mathrm{cpa}}VV
@|
\\
\mathrm{XB}(T|S;G,Q)
@>{t}>>
\Ho^3(Q,\mathrm U(S)) .
\end{CD}
\end{equation}
Here the upper square is part of the diagram
in \cite[Subsection 1.4]{MR597986},
and
the lower square results from
\cref{eighttermcomp}.
Consequently the $Q$-normal crossed pair algebra
\begin{equation}
\left(A_{\mathrm e_T},\sigma_{\psi_T}\right)
\label{ntcpa}
\end{equation}
with respect to the $Q$-normal Galois extension $T|S$ of commutative rings
that
arises from the
crossed pair
$(\mathrm e_T,\psi_T)$
with respect to
\eqref{gemm}
and $\mathrm U(T)$ 
via the construction in 
Subsection \ref{cpna}
has non-zero Teichm\"uller class in $\Ho^3(Q,\mathrm U(S))$.

To realize this kind of example concretely,
consider a faithful unitary representation 
$E$ of complex dimension $n$
of the metacyclic group $G$
\cite[\S 47 ~p.~335]{MR2215618}.
Things can be arranged in such a way that
the unitary $G$-representation
yields an action of $G$ on the unit sphere 
$S^{2n-1}\subseteq \mathbb C^n$
so that the restriction of the action 
to $N=\mathbb C_r$ is free but, apart from
trivial cases,
 the $G$-action itself will not be free.
Thus we may take $\widetilde X=S^{2n-1}$ and
$X=S^{2n-1}/C_r$ (a lens space) and carry out the above construction.

\begin{rema}
The above observation that the bottom row in \eqref{diag22}
splits in the category of abelian groups
translates, 
in  view of the exactness of the sequence
{\rm \cite[(1.9)]{MR597986}}
(spelled out as the top sequence in the diagram in
\cref{withoutp})
 to the fact that the above homomorphism
$\alpha$ is an isomorphism.
\end{rema}

\begin{rema}
Apart from trivial cases,
while $C_T$ acquires a $G$-action,
this action does not turn
the obvious map $C_T \to G$ into a crossed module
since the action of $N$ on the kernel $\mathrm U(T)$
of $C_T \to G$ is non-trivial when $N$ is non-trivial.
Thus we cannot get away with the crossed pair concept,
more general than that of a crossed module.
\end{rema}

\section{Examples arising from $\mathrm C^*$-dynamical systems}
\label{dynamical}
The group 3-cocycle in \cite[II 3.1 p. 147]{MR574031}
with values in the group of units of the center of a von Neumann algebra
is an instance of a Teichm\"uller cocycle in the
von Neumann algebra context. The aim of this 3-cocycle was indeed 
to explore a crossed product construction formally
of the same kind as the crossed product
in 
\cref{three} above,
and these ideas were pushed further in
\cite{MR548118,MR587749}
to gain structural insight into von Neumann algebras
(but to our knowledge the relationship with the Teichm\"uller cocycle
was not observed in the literature).

Given a topological space $X$,
the results of \cite{MR0150608,MR0163182}
are nowadays well known to establish an isomorphism
\begin{equation}
\delta\colon \mathrm B(X) \longrightarrow \check{\Ho}^3(X,\mathbb Z)
\end{equation}
between the Brauer group $\mathrm B(X)$
of Morita equivalence classes $[A]$ in the sense of Rieffel 
of continuous-trace
$\mathrm C^*$-algebras $A$ having spectrum $X$
and the third $\v{C}ech$-cohomology group $\check{\Ho}^3(X,\mathbb Z)$,
see, e.~g., \cite{MR1446378, MR1239913}.
The continuous-trace
$\mathrm C^*$-algebras $A$ having spectrum $X$ can be characterized as the
$\mathrm C^*$-algebras which are locally Morita equivalent
to the commutative algebra $C_0(X)$ of continuous complex-valued functions
on $X$ that vanish at infinity, and the Dixmier-Douady class is the obstruction
to building a global equivalence with $C_0(X)$
from the local equivalences.

Let now $Q$ denote a group and suppose that $Q$ acts on $X$ 
and hence on $C_0(X)$.
Given a
 continuous-trace
$\mathrm C^*$-algebras $A$ having spectrum $X$,
just as before, we define a $Q$-normal structure on $A$ to be
a homomorphism $\kappa \colon Q \to \Out(A)=\Aut(A)/\mathrm{Inn}(A)$.
Then, with a suitable definition of
the group 
$\mathrm U(A)$ of units of $A$,
 the Teichm\"uller class in $\Ho^3(Q,\mathrm U(A))$ is defined, 
just as before in the ordinary algebraic case.
Since
the algebraic theory developed 
in \cref{cI} and \cref{cII}
involves only the objects themselves but
does not involve any cocycles,
it is now a laborious but most likely rather straightforward endeavor to
extend that theory
to the $\mathrm C^*$-algebra case.

In \cite{MR1446378}, the theory 
of $\mathrm C^*$-algebra Brauer groups
was extended so that
group actions can be accomodated, and 
a corresponding equivariant Brauer group
was defined. 
In  \cite[Lemma 4.6]{MR1446378},
even a version of a Teichm\"uller cocycle shows up
(but the authors did not recognize that the cocycle 
they constructed is a kind of Teichm\"uller cocycle).

In this area there are presumably many examples
of a non-trivial Teichm\"uller class
to be found and new phenomena are lurking behind.
See also 
\cite{MR2278062, MR1712519, MR946434, MR1119194, MR687968}. 

\section{Complements}

Other explicit examples of a non-trivial Teichm\"uller cocycle
can be found in \cite{MR1836363} and \cite{MR657422}.

\begin{rema}
In \cite{MR1255884},
the Teichm\"uller cocycle serves as a crucial means
for building a Galois theory of skew fields.
It is worthwhile noting that, in \lq\lq non-commutative Galois theory\rq\rq,
a counterexample in
\cite[p.~141]{MR0002858} serves as well as a counterexample in 
\cite[p.~558]{MR892917},
\cite[p.~298]{MR0037834},
and \cite[\S VI.11 p.~147]{MR0222106}.
\end{rema}

\section*{Acknowledgement}
We gratefully acknowledge
support by the Labex CEMPI (ANR-11-LABX-0007-01).

\nocite{MR549932} 
\addcontentsline{toc}{section}{References}

\bibliographystyle{alpha}
\def\cprime{$'$} \def\cprime{$'$} \def\cprime{$'$} \def\cprime{$'$}
  \def\cprime{$'$} \def\cprime{$'$} \def\cprime{$'$} \def\cprime{$'$}
  \def\dbar{\leavevmode\hbox to 0pt{\hskip.2ex \accent"16\hss}d}
  \def\cprime{$'$} \def\cprime{$'$} \def\cprime{$'$} \def\cprime{$'$}
  \def\cprime{$'$} \def\Dbar{\leavevmode\lower.6ex\hbox to 0pt{\hskip-.23ex
  \accent"16\hss}D} \def\cftil#1{\ifmmode\setbox7\hbox{$\accent"5E#1$}\else
  \setbox7\hbox{\accent"5E#1}\penalty 10000\relax\fi\raise 1\ht7
  \hbox{\lower1.15ex\hbox to 1\wd7{\hss\accent"7E\hss}}\penalty 10000
  \hskip-1\wd7\penalty 10000\box7}
  \def\cfudot#1{\ifmmode\setbox7\hbox{$\accent"5E#1$}\else
  \setbox7\hbox{\accent"5E#1}\penalty 10000\relax\fi\raise 1\ht7
  \hbox{\raise.1ex\hbox to 1\wd7{\hss.\hss}}\penalty 10000 \hskip-1\wd7\penalty
  10000\box7} \def\polhk#1{\setbox0=\hbox{#1}{\ooalign{\hidewidth
  \lower1.5ex\hbox{`}\hidewidth\crcr\unhbox0}}}
  \def\polhk#1{\setbox0=\hbox{#1}{\ooalign{\hidewidth
  \lower1.5ex\hbox{`}\hidewidth\crcr\unhbox0}}}
  \def\polhk#1{\setbox0=\hbox{#1}{\ooalign{\hidewidth
  \lower1.5ex\hbox{`}\hidewidth\crcr\unhbox0}}}

\end{document}